\documentclass[11pt]{amsart}

\makeatletter

\renewcommand\subsubsection{\@secnumfont}{\normalfont\bfseries}%
\renewcommand\subsubsection{\@startsection{subsubsection}{3}
	\z@{.5\linespacing\@plus.7\linespacing}{-.5em}%
	{\normalfont\bfseries}}

\makeatother

\usepackage{etoolbox}

\makeatletter
\patchcmd{\@setaddresses}{\indent}{\noindent}{}{}
\patchcmd{\@setaddresses}{\indent}{\noindent}{}{}
\patchcmd{\@setaddresses}{\indent}{\noindent}{}{}
\patchcmd{\@setaddresses}{\indent}{\noindent}{}{}
\makeatother

\usepackage{graphicx}
\usepackage{amssymb}
\usepackage{amsmath}
\usepackage{amsthm}
\usepackage{comment}
\usepackage{colortbl}
\usepackage{blkarray, bigstrut} %
\usepackage{xcolor}
\usepackage{bm}
\usepackage{placeins}
\usepackage{verbatim}
\usepackage{multirow}
\usepackage{float}
\usepackage{enumitem}
\usepackage[a4paper,left=1.2in,right=1.2in,top=1in,bottom=1in]{geometry}
\usepackage{url}
\usepackage{enumitem}
\usepackage{tikz-cd}
\usepackage[all]{xy}
\usepackage{url}
\usepackage{array,booktabs}
\usepackage[hidelinks]{hyperref}
\usepackage{etex}
\usepackage{tikz}
\usepackage{nicematrix}
\usepackage[width=0.9\textwidth]{caption}
\usetikzlibrary{decorations.pathmorphing} 
\usetikzlibrary{arrows}
\usetikzlibrary{shapes}
\usepackage{needspace} 
\usepackage{caption}
\DeclareCaptionType{algorithm}[Algorithm][List of Algorithms]
\usepackage{pinlabel}
\usepackage{subfig}
\makeatletter
\DeclareRobustCommand*{\bfseries}{%
	\not@math@alphabet\bfseries\mathbf
	\fontseries\bfdefault\selectfont
	\boldmath
}
\makeatother
\usepackage{xargs}                      
\usepackage[colorinlistoftodos,textwidth=2.5cm, textsize=tiny,backgroundcolor=red!10]{todonotes}

\makeatletter%
\@mparswitchfalse%
\makeatother%
\normalmarginpar

\setlist[itemize]{leftmargin=1.6em}
\setlist[enumerate]{leftmargin=1.6em} 

\newcommandx{\change}[2][1=]{\todo[linecolor=red,backgroundcolor=red!25,bordercolor=red,#1]{#2}}
\newcommandx{\unsure}[2][1=]{\todo[linecolor=blue,backgroundcolor=blue!25,bordercolor=blue,#1]{#2}}
\newcommandx{\info}[2][1=]{\todo[linecolor=OliveGreen,backgroundcolor=OliveGreen!25,bordercolor=OliveGreen,#1]{#2}}
\newcommandx{\improvement}[2][1=]{\todo[linecolor=Plum,backgroundcolor=Plum!25,bordercolor=Plum,#1]{#2}}
\newcommandx{\thiswillnotshow}[2][1=]{\todo[disable,#1]{#2}}
\makeatletter
\providecommand\@dotsep{5}
\renewcommand{\listoftodos}[1][\@todonotes@todolistname]{%
	\@starttoc{tdo}{#1}}
\makeatother

\numberwithin{equation}{section} 

\newtheorem{theorem}[equation]{Theorem}

\newtheorem*{thm:main}{Main Theorem}

\newtheorem{proposition}[equation]{Proposition}
\newtheorem{lemma}[equation]{Lemma}
\newtheorem{corollary}[equation]{Corollary}

\newtheorem{conjecture}[equation]{Conjecture}
\newtheorem{experimental}[equation]{Experimental Result}

\theoremstyle{definition}
\newtheorem{definition}[equation]{Definition}

\newtheorem{remark}[equation]{Remark}
\newtheorem{example}[equation]{Example}
\newtheorem*{example*}{Example}
\newtheorem{construction}[equation]{Construction}
\newtheorem{convention}[equation]{Convention}
\newtheorem{procedure}[equation]{Procedure}
\newtheorem{observation}[equation]{Observation}

\newtheorem*{remark*}{Remark}

\newtheorem{question}[equation]{Question}

\makeatletter
\theoremstyle{plain}
\newtheorem*{rep@theorem}{\rep@title}

\newcommand{\newreptheorem}[2]{%
	\newenvironment{rep#1}[1]{%
		\def\rep@title{#2~\ref{##1}}%
		\begin{rep@theorem}}%
		{\end{rep@theorem}}}
\makeatother

\newreptheorem{theorem}{Theorem}
\newreptheorem{proposition}{Proposition}
\newreptheorem{experimental}{Experimental Result}
\newreptheorem{conjecture}{Conjecture}

\newcommand{\rr}{\mathbb{R}}
\newcommand{\zz}{\mathbb{Z}}
\newcommand{\V}{\mathcal{V}}

\newcommand{\T}{\mathcal{T}}

\newcommand{\C}{\mathcal{C}}
\newcommand{\A}{\mathcal{A}}

\newcommand{\Q}{\mathcal{Q}}

\newcommand{\E}{\mathcal{E}}
\newcommand{\F}{\mathcal{F}}
\newcommand{\W}{\mathcal{W}}

\newcommand{\HH}{\mathcal{H}}

\newcommand{\liftedFlow}{\widetilde{\Psi}}
\newcommand{\univN}{\widetilde{N}}
\newcommand{\orbit}{\mathcal{O}}
\newcommand{\drilledOrbit}{\mathcal{P}\mathcal{O}}
\newcommand{\X}{\mathcal{X}}
\newcommand{\loom}{\mathcal{L}}

\newcommand{\completedLoom}{\overline{\overline{\mathcal{L}}}}
\newcommand{\completedLoomS}{\overline{\overline{\mathcal{L}^+}}}

\newcommand{\completedLoomSU}{\overline{\overline{\mathcal{L}^\pm}}}
\newcommand{\loomGamma}{\mathcal{L}_\gamma}
\newcommand{\pLoom}{\mathcal{P}\mathcal{L}}
\newcommand{\loomif}{\mathsf{Loom}}
\newcommand{\liftGamma}{\tilde{\tilde{\gamma}}}
\newcommand{\liftZeta}{	\tilde{\tilde{\zeta}}}
\newcommand{\eleZeta}{\widehat{\zeta}}

\newcommand{\hull}{\mathsf{Hull}}

\newcommand{\Sing}{\mathsf{Sing}}
\newcommand{\Per}{\mathsf{Per}}

\newcommand{\TRect}{\mathsf{TRect}}
\newcommand{\FTRect}{\mathsf{FTRect}}
\newcommand{\ERect}{\mathsf{ERect}}
\newcommand{\FERect}{\mathsf{FERect}}
\newcommand{\gammaTRect}{\gamma\text{-}\mathsf{TRect}}
\newcommand{\gammaFTRect}{\gamma\text{-}\mathsf{FTRect}}
\newcommand{\Ideal}{\mathsf{Ideal}}

\newcommand{\NL}{\mathrm{NL}}
\newcommand{\NN}{\mathrm{NN}}
\renewcommand{\NG}{\mathrm{NG}} 

\newcommand{\univV}{\widetilde{\V}}
\newcommand{\univT}{\widetilde{\T}}

\newcommand{\Rsharp}{R^\sharp}
\newcommand{\Stab}{\mathsf{Stab}}

\newcommand{\bez}{-}

\newcommand{\SSW}{S_{\text{SW}}}
\newcommand{\SSE}{S_{\text{SE}}}
\newcommand{\SNW}{S_{\text{NW}}}
\newcommand{\SNE}{S_{\text{NE}}}
\newcommand{\bNW}{b_{\text{NW}}}
\newcommand{\bSE}{b_{\text{SE}}}

\DeclareFontFamily{U}{mathx}{}
\DeclareFontShape{U}{mathx}{m}{n}{ <-> mathx10 }{}
\DeclareSymbolFont{mathx}{U}{mathx}{m}{n}
\DeclareFontSubstitution{U}{mathx}{m}{n}
\DeclareMathAccent{\widecheck}{0}{mathx}{"71}

\makeatletter
\@namedef{subjclassname@2020}{%
	\textup{2020} Mathematics Subject Classification}
\makeatother
\usepackage{doc}
\makeatletter
\def\maketitle{\par
	\begingroup
	\def\thefootnote{\fnsymbol{footnote}}%
	\setcounter{footnote}\z@
	\def\@makefnmark{\hbox to\z@{$\m@th^{\@thefnmark}$\hss}}%
	\long\def\@makefntext##1{\noindent
		\ifnum\c@footnote>\z@\relax
		\hbox to1.8em{\hss$\m@th^{\@thefnmark}$}##1%
		\else
		\hbox to1.8em{\hfill}%
		\parbox{\dimexpr\linewidth-1.8em}{\raggedright ##1}%
		\fi}
	\if@twocolumn\twocolumn[\@maketitle]%
	\else\newpage\global\@topnum\z@\@maketitle\fi
	\thispagestyle{titlepage}\@thanks\endgroup
	\setcounter{footnote}\z@
	\gdef\@date{\today}\gdef\@thanks{}%
	\gdef\@author{}\gdef\@title{}}

\makeatother

\makeatletter
\def\@settitle{%
	\begin{center}%
		\baselineskip=1.3\baselineskip
		\bfseries
		\uppercasenonmath\@title
		\@title
	\end{center}%
}
\makeatother

\hypersetup{
	colorlinks=false,
	pdfborder={0 0 0}
}

\begin{document}

	\numberwithin{equation}{section}
	\title{
		Drilling veering triangulations \\ with applications to pseudo-Anosov flows}
	
	\author{Anna Parlak}
	\address{Max Planck Institute for Mathematics in the Sciences, 04103 Leipzig, Germany}
	\email{parlak@mis.mpg.de}

	\author{Henry Segerman}
	\address{Department of Mathematics \\
		Oklahoma State University, Stillwater, OK 74078, USA}
	\email{henry.segerman@okstate.edu}

	

	\begin{abstract}
		
		We develop an algorithm that takes as input a veering triangulation $\V$ and a flow cycle $\zeta$ of $\V$, and outputs a \emph{veering parent} $\V_\zeta$ of $\V$. Constructing veering parents combinatorializes drilling out closed orbits of transitive pseudo-Anosov flows. The algorithm relies on locally approximating the structure of the loom space associated to~$\V$ in sufficient detail to carry out a certain modification of the Agol-Gu\'eritaud construction.
		
		We use an implementation of the drilling algorithm to construct layered parents of non-layered veering triangulations (and hence Birkhoff sections for transitive pseudo-Anosov flows) and geometric parents of non-geometric veering triangulations. We also find closed orbits of (drilled) pseudo-Anosov flows that are not ambiently isotopic to the geodesics in their free homotopy classes, and show that certain veering triangulations encode almost orbit equivalent flows. In particular, we provide new computational evidence for a conjecture of Ghys asserting that all transitive Anosov flows with orientable invariant foliations are almost orbit equivalent.

	\end{abstract}

	\maketitle%
	\setcounter{tocdepth}{1}
	
	\vspace{1cm}
	\tableofcontents

	\newpage
	
	\section{Introduction}
	
	Veering triangulations are ideal triangulations of hyperbolic 3-manifolds with torus cusps that combinatorially encode transitive pseudo-Anosov flows. More precisely, a veering triangulation $\V = \V(\Psi, \Lambda)$ is determined by a pair $(\Psi, \Lambda)$, where $\Psi$ is a transitive pseudo-Anosov flow on a closed, orientable 3-manifold and $\Lambda$ is a finite collection of closed orbits of $\Psi$ that satisfy a number of conditions; see Definition \ref{def:perfect:set}. We call $\Lambda$ a \emph{perfect set} for $\Psi$. The triangulation $\V$ can be explicitly constructed via the Agol-Gu\'eritaud construction; see \cite[Section~4]{LMT_flow} or Section \ref{subsec:loom:to:veering} here. If $N$ is the manifold underlying $\Psi$ then  $\V$ is a triangulation of $N \bez \Lambda$. The triangulation $\V$ encodes $\Psi$ in the sense that its combinatorics and the Dehn filling slopes that recover $N$ from $N\bez \Lambda$ are sufficient to reconstruct $\Psi$ \cite[Theorem 5.1]{Tsang-Agol}. Further results connecting  veering triangulations with pseudo-Anosov flows can be found in \cite{Agol_veer, LMT_flow, LMT-transverse-surfaces, SchleimSeg-dynamic-pairs, SchleimSegLinks, Tsang-branched-surfaces}. Moreover, veering triangulations have already been used as a tool to establish a few  new results about pseudo-Anosov flows; see \cite{Huang-Taylor-depth1, LMT-transverse-surfaces, Parlak-mutations, Tsang-Birkhoff, Tsang-genus-one}.
	
	Importantly, a perfect set for a transitive pseudo-Anosov flow is not unique. If $\Lambda$ is a perfect set for $\Psi$ then, for every closed orbit $\gamma$ of $\Psi$, the set $\Lambda_\gamma = \Lambda \cup \lbrace \gamma \rbrace$ is also a perfect set for $\Lambda$. Consequently, if $\gamma \notin \Lambda$, the Agol-Gu\'eritaud construction yields another veering triangulation $\V_\gamma = \V(\Psi, \Lambda_\gamma)$ that encodes $\Psi$. We call $\V_\gamma$ a \emph{veering parent} of $\V$, and refer to the process of constructing $\V_\gamma$ from $\V$ as \emph{drilling} $\V$ along $\gamma$. 
	
	While drilling a simple closed curve from a $3$-manifold, and even constructing a triangulation of the resulting manifold, is straightforward, constructing the particular veering triangulation $\V_\gamma$ is much more subtle. In particular, drilling can affect the veering structure globally and lead to significant changes that are difficult to predict a priori; see Appendix \ref{appendix:examples} for several examples of veering parents.
	
	There are multiple reasons why the ability to construct veering parents is worthwhile. 
	
	First, a veering parent $\V_\gamma$ may have `better' properties than the original triangulation~$\V$. For instance, $\V_\gamma$ may have fewer tetrahedra than $\V$ (see Appendix \ref{app:ex:decreasing}), or may be \emph{layered} when $\V$ is not (see Section \ref{subsec:layered:parents}). In particular, working with a veering parent may be helpful when studying quantitative properties of the flow whose bounds depend on the complexity of the veering triangulation that encodes it; see Section~\ref{further:descreasing}.
	
	Second, several questions concerning pseudo-Anosov flows naturally involve drilling out periodic orbits. These include the Ghys conjecture, which states that all transitive Anosov flows with transversely orientable invariant foliations are \emph{almost orbit equivalent}, and the problem of constructing  \emph{Birkhoff sections} for pseudo-Anosov flows. A good understanding of veering parents would be useful in approaching these questions using the machinery of veering triangulations; see Sections~\ref{subsec:layered:parents} and \ref{subsec:ghys} for a discussion of how veering parents arise in such setups.
	
	Finally, constructing veering parents provides experimental data for understanding how veering triangulations arising from the same pseudo-Anosov flow are related combinatorially.
	This is important if one hopes to use the veering census~\cite{VeeringCensus} to extract some form of comprehensive census of pseudo-Anosov flows. The precise relationship between the two censuses is somewhat involved, as every veering triangulation encodes infinitely many pseudo-Anosov flows, and every pseudo-Anosov flow can be encoded by infinitely many veering triangulations. But at a minimum one would like to be able to recognize when two veering triangulations can be constructed from the same flow. This holds, in particular, whenever the triangulations can be connected by a finite sequence of drillings. Until now, there was no way to certify the existence of such a sequence for a given pair of triangulations. Our drilling algorithm changes this: using its implementation, we were able to connect in this way many triangulations in the census; see Experimental Result~\ref{exp:ancestry:experiments}.
	
	\subsection{Theoretical results}
	The main goal of this paper is to describe an algorithm that constructs veering parents of an arbitrary finite veering triangulation. To specify a closed orbit along which to drill, we use \emph{flow cycles}, introduced by Landry--Minsky--Taylor in \cite[Section~3]{LMT_flow}. Roughly speaking, every flow cycle $\zeta$ of $\V=\V(\Psi,\Lambda)$ encodes a closed orbit $\gamma$ of $\Psi$. Our algorithm takes as input~$\V$ and a flow cycle $\zeta$, and outputs the corresponding veering parent $\V_\gamma$. For simplicity, we also denote this veering parent by $\V_\zeta$.

	The key steps of the algorithm are shown in Figure \ref{fig:flowchart}. The correctness of the algorithm is the main result of the paper.
	
	\begin{reptheorem}{thm:algorithm:correct}
		Given a veering triangulation $\V$ and a flow cycle $\zeta$ of $\V$, the algorithm presented in Figure~\ref{fig:flowchart} terminates and outputs
		\begin{itemize}[topsep=0pt, leftmargin=1.6em]
			\item $\V$ if $\zeta$ is peripheral;
			\item the veering parent $\V_\zeta$ if $\zeta$ is non-peripheral.
		\end{itemize}
	\end{reptheorem}

\emph{Peripheral} flow cycles mentioned in the above theorem encode closed orbits of the flow that are already drilled out in $\V$. Naturally, they cannot be drilled out again. The algorithm recognizes this case early in the process and terminates returning the original triangulation.
	
Much of the length of the paper comes from developing the technical framework underlying the algorithm and establishing various properties of the auxiliary objects involved in it. Below we outline the main ideas and highlight some of the results obtained along the way that may be of independent interest.

\subsubsection{Pseudo-Anosov flows and loom spaces}
Our first intermediate result connects pseudo-Anosov flows with \emph{loom spaces}, introduced by Schleimer and the second author in~\cite{SchleimSegLoom}. A loom space is a special type of a bi-foliated plane whose foliations satisfy several additional conditions; see Definition \ref{def:loom}.

Every pseudo-Anosov flow $\Psi$ determines a  triple $(\orbit, \orbit^+, \orbit^-)$, where $\orbit$ is a plane, and~$\orbit^\pm$ is a pair of transverse foliations of $\orbit$. The triple $(\orbit, \orbit^+, \orbit^-)$ is called the \emph{orbit space} of $\Psi$ and is never itself a loom space. We prove that if~$\Psi$ is transitive and~$\Lambda$ is a perfect set for $\Psi$ (see Definition \ref{def:perfect:set}) then  puncturing~$\orbit$ at the set $\widetilde{\Lambda}$ of points determined by $\Lambda$, and then taking the universal cover, yields a loom space.

\begin{repproposition}{prop:loom:space}
		Let $\Lambda$ be a perfect set  for a transitive pseudo-Anosov flow~$\Psi$. Let $(\loom, \loom^+, \loom^-)$ be the universal cover of the orbit space $(\orbit, \orbit^+, \orbit^-)$ of $\Psi$ punctured at~$\widetilde{\Lambda}$. 
	Then   $(\loom, \loom^+, \loom^-)$ is a loom space.
	\end{repproposition}

It is well-known that the bi-foliated space $(\loom, \loom^+, \loom^-)$ can be used to construct a veering triangulation $\V(\Psi, \Lambda)$ via the \emph{Agol-Gu\'eritaud construction}; see \cite[Section~4]{LMT_flow}. Showing more specifically that $(\loom, \loom^+, \loom^-)$ is a loom space allows us to freely use all properties of loom spaces established in \cite{SchleimSegLoom} to prove several technical results needed to justify the correctness of the drilling algorithm.

Proposition \ref{prop:loom:space} is known to those working with pseudo-Anosov flows and veering triangulations. However, to our knowledge, a proof that uses only the flow and a perfect set $\Lambda$ has not previously appeared in the literature. In \cite[Example 2.19]{SchleimSegLoom} the fact that one can get a loom space from  the pair $(\Psi, \Lambda)$ is stated in the special case when $\Psi$ has \emph{no perfect fits} and $\Lambda$ is the set of singular orbits of $\Psi$, with a reference to \cite[Theorem 10.51]{SchleimSegLinks}. While the generalization to the more general setup is immediate, that reference only contains a proof that the \emph{link space} associated to the veering triangulation $\V(\Psi, \Lambda)$ is a loom space, without a proof that this  link space is loom-isomorphic to $(\loom, \loom^+, \loom^-)$. Thus, strictly speaking, \cite[Theorem 10.51]{SchleimSegLinks} is not currently known to be equivalent to Proposition~\ref{prop:loom:space}. 
For us, the advantage of working directly with $(\loom, \loom^+, \loom^-)$, rather than with another loom space obtained through additional layers of combinatorial constructions, is that it is entirely clear what it means to puncture $(\loom, \loom^+, \loom^-)$ at points corresponding to additional periodic orbits of $\Psi$, and that doing so can be a step towards constructing other veering triangulations that encode $\Psi$.

Before continuing, we introduce some terminology related to loom spaces. The foliations $\loom^+, \loom^-$ are called the \emph{stable} and the \emph{unstable} foliation of $(\loom, \loom^+, \loom^-)$, respectively. A \emph{rectangle} in $(\loom, \loom^+, \loom^-)$ is the image $R$ of an embedding  $(0,1)^2 \hookrightarrow \loom$ that maps each vertical segment into a leaf of $\loom^+$ and each horizontal segment into a leaf of~$\loom^-$.  The embedding equips $R$ with a \emph{vertical} direction, along the leaves of the stable foliation $\loom^+ \cap R$, and a \emph{horizontal} direction, along the leaves of the unstable foliation $\loom^- \cap R$. We also fix a convention for assigning cardinal directions within $R$; see Convention~\ref{convention:cardinal:directions}. Thus, given two points in $R$ or its boundary, it makes sense to say that one lies to the east or west of the other and, analogously, to the north or south of the other. 

We emphasize that all rectangles in $\loom$ are open. Therefore, throughout the paper, statements about a point lying in a rectangle or a leaf intersecting a rectangle always refer to the open rectangle and not to its boundary. 

\subsubsection{Veering parents via a finite restriction of the Agol-Gu\'eritaud construction}\leavevmode\\ \label{subsubsec:finite:AG}
With $\Psi, \Lambda$, and $(\loom, \loom^+, \loom^-)$ as before, set $\V = \V(\Psi, \Lambda)$. One of the key features of $(\loom, \loom^+, \loom^-)$ is that every simplex $s$ of the universal cover $\univV$ of $\V$ determines a rectangle $\loomif(s) \subset \loom$ called the \emph{loomification} of $s$; see Section \ref{subsec:loomification}.  Any rectangle in $(\loom, \loom^+, \loom^-)$ arising in this way is called a \emph{simplex rectangle}. A simplex rectangle has two, three, or four ideal points in its boundary, depending on whether it is the loomification of an edge, a face, or a tetrahedron of $\univV$. It is called an \emph{edge}, \emph{face}, or \emph{tetrahedron rectangle}, respectively. The ideal points in the boundary of a simplex rectangle are called its \emph{cusps}. 
Given a collection $\mathcal{S}$ of simplices of~$\univV$, by $\loomif(\mathcal{S})$ we denote the union of the loomifications of $s \in \mathcal{S}$.

Let $M$ be the manifold underlying $\V$. 
Fix a closed orbit $\gamma$ of~$\Psi$ that is not in $\Lambda$. Set $\Lambda_\gamma = \Lambda \cup \lbrace \gamma \rbrace$ and $\V_\gamma = \V(\Psi, \Lambda_\gamma)$. Denote the loom space associated to $(\Psi, \Lambda_\gamma)$ by $(\loomGamma, \loomGamma^+, \loomGamma^-)$ and the manifold underlying $\V_\gamma$ by $M_\gamma$. The standard Agol-Gu\'eritaud construction of $\V_\gamma$ uses the action of $\pi_1(M_\gamma)$ on $(\loomGamma, \loomGamma^+, \loomGamma^-)$; see Construction \ref{const:finite:veering}. The key idea behind our drilling algorithm is to use the action of $\pi_1(M)$ on $(\loom, \loom^+, \loom^-)$ instead. Specifically, the orbit $\gamma$ determines a closed, discrete, $\pi_1(M)$-invariant set $\liftGamma \subset \loom$. Puncturing $\loom$ at $\liftGamma$ yields a bi-foliated space $(\pLoom, \pLoom^+, \pLoom^-)$ equipped with an action of $\pi_1(M)$. 
In Lemma \ref{lem:ploom} we show that although $(\pLoom, \pLoom^+, \pLoom^-)$ is not a loom space, it retains enough of the relevant structure to carry out an Agol-Gu\'eritaud-like construction. In Proposition \ref{prop:parent:from:intermediate:cover} we prove that the resulting veering triangulation is precisely the veering parent $\V_\gamma$. 

In Section \ref{subsec:finite:parent}, we furthermore show that it is not necessary to consider the entire loom space $\loom$ to construct $\V_\gamma$. Instead, it suffices to determine the points of $\liftGamma$ contained in a finite collection of tetrahedron rectangles in $\loom$, one for each tetrahedron of $\V$. We call these the \emph{fundamental tetrahedron rectangles}. Each fundamental tetrahedron rectangle $R$ contains finitely many \emph{$\gamma$-simplex rectangles}, defined as the images of simplex rectangles under the inclusion $\pLoom \hookrightarrow \loom$.
In Proposition \ref{prop:finite:veering:parent}, we prove that the tetrahedra of $\V_\gamma$ are in bijective correspondence with the $\pi_1(M)$-equivalence classes of $\gamma$-tetrahedron rectangles contained in the fundamental tetrahedron rectangles of $(\loom, \loom^+, \loom^-)$. Moreover, each face pairing of $\V_\gamma$ can be represented by a pair $\lbrace (R_\gamma, Q_\gamma), (R_\gamma', Q_\gamma')\rbrace$ where $R_\gamma$, $R_\gamma'$ are \mbox{$\gamma$-tetrahedron} rectangles contained in the fundamental tetrahedron rectangles, and $Q_\gamma$, $Q_\gamma'$ are $\pi_1(M)$-equivalent face subrectangles of $R_\gamma$, $R_\gamma'$, respectively, such that $Q_\gamma$ is as tall as $R_\gamma$ and $Q_\gamma'$ is as wide as $R_\gamma'$. This reduces the construction of the veering parent $\V_\gamma$ to finding all $\gamma$-tetrahedron rectangles contained in the fundamental tetrahedron rectangles and determining their $\pi_1(M)$-equivalence classes as well as the $\pi_1(M)$-equivalence classes of their $\gamma$-face subrectangles. We refer to this approach to constructing $\V_\gamma$ as a \emph{finite restriction of the Agol-Gu\'eritaud construction}. 

The finite restriction of the Agol-Gu\'eritaud construction is the key ingredient of our drilling algorithm. However, to carry it out without having direct access to $(\loom,\loom^+,\loom^-)$, we first need to approximate the local structure of $(\loom,\loom^+,\loom^-)$ using only the combinatorics of $\V$. This requirement to obtain a sufficiently detailed approximation of the loom space from purely combinatorial data accounts for most of the technical complexity of the algorithm.

\subsubsection{Locating the points of $\liftGamma$ using the combinatorics of $\univV$}

Any fundamental tetrahedron rectangle $R$ is the loomification of some tetrahedron $t$ of $\univV$. The veering structure of $t$ encodes  the horizontal and vertical orders on the cusps in the boundary of~$R$; see Figure \ref{fig:tet:rects}. Detecting the points of $\liftGamma$ inside $R$ and determining their horizontal and vertical orders  using only the combinatorics of $\univV$ is much more involved.

To encode the orbit $\gamma$ combinatorially, we use \emph{flow cycles}, introduced by Landry--Minsky--Taylor in \cite[Section~4.3]{LMT}. Roughly speaking, $\gamma$ can be encoded by a finite cyclic sequence $\zeta$ of tetrahedra of $\V$, called a flow cycle; see Definition \ref{defn:flow:cycle}. The preimage of~$\zeta$ under the universal covering map $\univV \rightarrow \V$ consists of a $\pi_1(M)$-invariant collection of bi-infinite sequences of tetrahedra in $\univV$, called \emph{elevations} of $\zeta$. Each such elevation $\eleZeta$ forms a \emph{flow line} (see Definition \ref{def:flow:line}). The intersection of the loomifications of all tetrahedra of~$\eleZeta$ is a point of $\liftGamma$ \cite[Fact 4.4 and Proposition 6.12]{LMT_flow}.
Thus  elevations of $\zeta$ represent the points of $\liftGamma$. Flow cycles of $\V$ and their relationship with the orbits of the flow encoded by $\V$ are discussed in more detail in Section \ref{sec:flow:cycles}.

Since our algorithm must terminate after finitely many steps, it can involve only finitely many tetrahedra of $\univV$. Therefore instead of considering the entire flow line $\eleZeta$, we consider only its finite subsequences, called \emph{flow intervals}. A disadvantage of this simplification is that intersecting the loomifications of the tetrahedra of a flow interval yields a rectangle in $\loom$ rather than a point. Thus we cannot precisely locate the point represented by $\eleZeta$. Instead, we only obtain its rectangular approximations. 

The flow intervals used in the algorithm are obtained by intersecting flow lines with a (sufficiently nice) finite subtriangulation $\C$ of $\univV$. We may assume that $\C$ contains at least one lift of every tetrahedron of $\V$, so that $\loomif(\C)$ contains a set of fundamental tetrahedron rectangles. Given $x \in \liftGamma$, denote by $\eleZeta_x$ the flow line representing $x$, by $\eleZeta_x^\C$  its restriction to $\C$, and by $N_\C(x)$ the resulting rectangular approximation of~$x$ in $\loom$.

The main technical difficulty of the algorithm lies in constructing a finite subtriangulation $\C \subset \univV$ that is sufficiently large to ensure that, for every fundamental tetrahedron rectangle \mbox{$R \subset \loomif(\C)$}, the following three conditions hold.
\begin{enumerate}[leftmargin=2.4em]
	\item $\eleZeta_x^\C \neq \emptyset$ for every $x \in R \cap \liftGamma$.
	\item The rectangular approximations $\lbrace N_\C(x) \ | \ x\in R \cap \liftGamma\rbrace$ are disjoint from the leaves of~$\loom^\pm$ originating from the cusps of $R$.
	\item There is no leaf of $\loom^\pm$ that intersects both $N_\C(x)$ and $N_\C(y)$ for distinct $x,y \in R \cap \liftGamma$.
\end{enumerate}

Condition (1) ensures that we do not miss any points of $\liftGamma$ contained in $R$. Since we do not know a priori how many such points there are or where they are located, it is not immediately clear which parts of $\univV$ must be included in $\C$ to guarantee that (1) holds.

We solve this problem as follows. Fix $x \in \liftGamma$. The stabilizer of $x$ in $\pi_1(M)$ is cyclic; let~$\gamma_x$ be its generator. We say that a subset $A \subset \loom$ is \emph{adapted to $x$} if every edge rectangle in $\loom$ containing $x$ has a $\langle \gamma_x \rangle$-translate in $A$. If~$A$ is adapted to $x$, then every edge rectangle $P$ containing a point $y \in \liftGamma$ has a $\pi_1(M)$-translate in $A$ that contains $x$; see Lemma \ref{lem:adapted}. Thus we can determine $P \cap \liftGamma$ by translating $x$ to $P$ from all $\pi_1(M)$-translates of $P$ contained in $A$. Since tetrahedron rectangles are covered by edge rectangles, an adapted set can  be used to find all points of $\liftGamma$ contained in each fundamental tetrahedron rectangle.

In Proposition \ref{prop:adapted:general} we give an explicit description of a set $A$ that is adapted to~$x$ and covered by finitely many edge rectangles. Starting from a  subtriangulation $\C$ such that $\loomif(\C)$ contains $A$, to ensure (1) it suffices to search through finitely many edge rectangles in $\loomif(\C)$ and perform finitely many expansions of $\C$, dictated by the necessary translations of $\eleZeta_x$. The theoretical background for this construction is developed in Section \ref{sec:adapted:set}, while its precise implementation is described in Section~\ref{subsec:adapted:continent}. 

Conditions~(2) and (3) require the rectangular approximations  of the points $R \cap \liftGamma$ to be `small enough' so that the horizontal and vertical orders of these points, relative both to one another and to the cusps of $R$, can be determined unambiguously. In Lemma \ref{lem:expanding:until:on:the:same:side} we prove that by \emph{expanding} $\C$ along the relevant flow lines (see Definition \ref{def:expanding:along:flowline}) one can always obtain approximations that are sufficiently small. The real difficulty lies in being able to decide, from the combinatorics of the subtriangulation $\C$ alone, whether conditions~(2) and (3) have already been achieved.

The sides of the rectangular approximations $N_\C(x)$ are always contained in some \emph{cusp leaves} of the foliations $\loom^\pm$. Therefore checking if these rectangles are sufficiently small is a special case of the more general problem of determining the horizontal order of the stable cusp leaves and the vertical order of the unstable cusp leaves intersecting a common rectangle. Solving this problem requires two ingredients: a way to encode the cusp leaves of $\loom^\pm$ using only the combinatorics of $\univV$, and a subtriangulation $\C \subset \univV$ that interacts with that encoding sufficiently well so that the relative order of two encoded cusp leaves can be determined from the combinatorics of $\C$ alone.

To encode the stable and unstable cusp leaves, we use \emph{upper} and \emph{lower rivers} in~$\univV$. These are certain sequences of faces of $\univV$ satisfying several conditions; see Section \ref{subsec:convexity}. Rivers were originally introduced in \cite[Section~4.4]{SchleimSegLinks}, although we use a more restrictive definition here. An upper (respectively, lower) river starting at a face $f$ encodes the unique stable (respectively, unstable) cusp leaf that originates at a cusp of $\loomif(f)$ and intersects $\loomif(f)$; see Lemma \ref{lem:bracketing:cusps}.

Rivers interact particularly well with \emph{continents} in $\univV$. These are a special type of subtriangulations that were introduced in \cite[Section~4.14]{SchleimSegLinks}. Below we outline several new results about continents, extending those of \cite{SchleimSegLinks}, that allow us to combinatorially order certain cusp leaves of $\loom^\pm$.

\subsubsection{Ordering cusp leaves using convex continents}

A \emph{continent} in $\univV$ is a subtriangulation $\C \subset \univV$ which is \emph{face-connected} and \emph{layered}; see Definitions \ref{def:face:connected} and~\ref{def:layering}. We first establish a key result about the loomification of $\C$.

\begin{repproposition}{prop:convex:disk}
	Let $\C$ be a nonempty continent in $\univV$. Then the closure $\overline{\loomif(\C)}$ of $\loomif(\C)$ in $\loom$ is a convex, finite-sided closed disk with finitely many boundary points removed.
\end{repproposition}

We refer the reader to Section \ref{subsec:convexity} for a discussion of convexity in loom spaces. Proposition \ref{prop:convex:disk}  implies that any leaf intersecting  $\loomif(\C)$ separates $\loomif(\C)$  into two connected components. In particular, if~$\ell$ is a cusp leaf that originates at a cusp of $\loomif(\C)$ and intersects $\loomif(\C)$ then $\ell$ intersects the boundary of $\loomif(\C)$ at a unique point~$\epsilon(\ell)$. We call $\ell$ a \emph{continental cusp leaf} and $\epsilon(\C)$ its \emph{(continental) exit point}.

Let $\Omega(\C)$ be the set of cusps of $\C$. 
Denote the set of exit points of all stable (respectively, unstable) continental cusp leaves of $\C$ by $\E^+(\C)$ (respectively, $\E^-(\C)$), and set $\E(\C) = \E^+(\C) \cup \E^-(\C)$. The disk structure on $\loomif(\C)$ and its orientation inherited from $(\loom, \loom^+, \loom^-)$  induces a circular order $\Theta: (\Omega(\C) \cup \E(\C))^3\rightarrow \lbrace-1, 0, 1\rbrace$, that we call the \emph{loom-compatible circular order} on $\Theta(\C) \cup \E(\C)$.  It extends the circular order on the cusps of $\C$ that was established in \cite[Corollary 4.20]{SchleimSegLinks}. We use $\Theta$ to horizontally order certain stable continental cusp leaves and vertically order certain unstable continental cusp leaves as follows.

For every continental cusp leaf~$\ell$ we define a \emph{side function} \mbox{$\Theta_\ell \colon \Omega(\C) \rightarrow \lbrace -1,0,1\rbrace$} by setting, for every $c \in \Omega(\C)$, \[\Theta_\ell(c)=\Theta(\iota(\ell), \epsilon(\ell), c),\] where $\iota(\ell)$ denotes the cusp from which $\ell$ originates.   The side function of $\ell$ detects which cusps of $\C$ are separated by~$\ell$; see Lemma \ref{lem:theta:l:function}. Thus, once cardinal directions are fixed near $\ell$, $\Theta_\ell$ can be used to determine whether a stable continental cusp leaf lies to the west or east of $\ell$ when $\ell$ is a leaf of~$\loom^+$, and whether an unstable continental cusp leaf lies to the south or north of $\ell$ when $\ell$ is a leaf of~$\loom^-$.
We never need to determine the relative position of a leaf of $\loom^+$ with respect to a leaf of $\loom^-$. Thus it suffices to consider the restrictions $\Theta^+$, $\Theta^-$ of $\Theta$ to $\Omega(\C) \cup \E^+(\C)$ and $\Omega(\C) \cup \E^-(\C)$, respectively.

There is a subtle distinction between the existence of the loom-compatible circular order~$\Theta$ associated to~$\C$ and the ability to determine $\Theta$ from the combinatorics of~$\C$ alone. The restriction of $\Theta$ to $\Omega(\C)$ is easy --- it can be recovered by following the boundary of the \emph{upper landscape} of $\C$. 
Ordering the continental exit points relative to each other and relative to the cusps of $\C$ is more involved. This is because continental cusp leaves, even though naturally associated to $\loomif(\C)$, are not directly visible in  $\C$.
In Figure \ref{fig:non:convex} we show that the loom-compatible circular order on $\Omega(\C) \cup \E(\C)$ may  depend on tetrahedra outside~$\C$. We prove that this issue disappears for \emph{convex continents}; see Definition~\ref{def:convex:continent}. 

	\begin{repproposition}{prop:conv:cont:circ:order}
	Let $\C\subset \univV$ be a convex continent. The loom-compatible circular orders \[\Theta^\pm:(\Omega(\C) \cup \E^\pm(\C))^3 \rightarrow \lbrace -1, 0, 1 \rbrace\] are uniquely determined by the veering structure on $\C$.
\end{repproposition}
Note that convexity of a continent is distinct from convexity of its loomification, which, by Proposition \ref{prop:convex:disk}, holds for every continent; see Remark \ref{rem:convexity:overlap}. 

Proposition \ref{prop:conv:cont:circ:order} implies that if $\C$ is a convex continent, $\ell$ is a continental cusp leaf of~$\C$, and $R$ is a simplex rectangle in $\loomif(\C)$, then we can find the value of $\Theta_\ell(\cdot)$ on any cusp of $R$ and hence determine if $\ell$ intersects $R$. 
 Furthermore, given two continental cusp leaves $\ell$ and $\ell'$ intersecting $R$ and belonging to the same foliation, we can similarly determine the relative  positions of $\ell$ and $\ell'$ in $R$. 
 
In particular, if a convex continent $\C \subset \univV$ satisfies condition~(1), we can determine whether conditions~(2) and (3) also hold. If they do not, we expand $\C$ upwards or downwards along the appropriate flow lines to obtain finer approximations of the points of $\liftGamma$ contained in the fundamental tetrahedron rectangles. By Lemma \ref{lem:expanding:until:on:the:same:side}, after finitely many such expansions we obtain a convex continent $\C$ satisfying all conditions~(1) -- (3). 
For every fundamental tetrahedron rectangle $R$, this gives sufficiently small rectangular approximations of the points of $R \cap \liftGamma$ to determine their horizontal and vertical orders relative to one another and to the cusps of $R$. We  then apply the finite restriction of the Agol-Gu\'eritaud construction, to construct the veering parent $\V_\gamma$. We refer the reader to Section \ref{sec:algorithm} for full details of the algorithm and its implementation.

\subsection{Experimental results and conjectures}	Our algorithm is not just theoretical; it has been implemented and incorporated into the Python package Veering \cite{VeeringGitHub}. Using that implementation, we obtained multiple experimental results concerning pseudo-Anosov flows and veering triangulations. The data from these experiments is available in the `Drilling data' section of the veering census \cite{VeeringCensus}.

	One application of our algorithm concerns the construction of \emph{layered} veering parents. Roughly speaking, a veering triangulation is \emph{layered} if it carries a fiber of a fibration over the circle; see Definition \ref{def:layered:tri}. In terms of the associated pseudo-Anosov flow, it means that the underlying drilled flow is a suspension flow. The existence of layered veering parents follows from results on \emph{Birkhoff sections} for pseudo-Anosov flows \cite{Brunella_surfaces, Fried_drill}. Moreover, Tsang proved that every veering triangulation with $n$ tetrahedra admits a layered veering parent obtained by drilling a single closed orbit whose square can be represented by a flow cycle of length at most $2\cdot 10^6n^{12}$ \cite[Theorem~6.3]{Tsang-Birkhoff}.
	We used our algorithm to experimentally  investigate how long a flow cycle is actually needed. An exhaustive search over all non-layered veering triangulations with at most $16$ tetrahedra gives the following.
	
		\begin{repexperimental}{experimental:layered:parents} \hfill 
		\begin{enumerate}[leftmargin=1.6em]
			\item Every non-layered veering triangulation with $n\leq 16$ tetrahedra admits a flow cycle of length at most $n-3$ that yields a layered veering parent.
			
			
			\item Among all non-layered veering triangulations with at most $16$ tetrahedra, approximately $82.78\%$ admit a flow cycle of length $1$ that yields a layered veering parent. Approximately $93.66\%$ admit such a flow cycle of length at most $2$.
		\end{enumerate}
	\end{repexperimental}

	See also Table \ref{tab:layered:parents} where more information is provided. The experimental bound of \mbox{$n-3$} is much smaller than the theoretical bound of $10^6n^{12}$; for $n=16$, the difference is almost 20 orders of magnitude. Of course, small triangulations may not be representative of the worst-case behavior and more complicated examples may appear as the number of tetrahedra grows. Still, it is encouraging that such short flow cycles suffice for the examples in the census. Perhaps a better understanding of the gathered experimental data could lead to more efficient ways of constructing Birkhoff sections than those discussed in~\cite{Tsang-Birkhoff}.
	Following Experimental Result \ref{experimental:layered:parents}, we propose the following conjecture.
		\begin{repconjecture}{conj:layered}
		Every non-layered veering triangulation with $n$ tetrahedra admits a layered veering parent obtained by drilling out a single flow cycle of length at most $n$.
	\end{repconjecture}
	
	Another `nice' property that a veering triangulation may have is geometricity. An ideal triangulation of a hyperbolic $3$-manifold is \emph{geometric} if, after lifting it to the universal cover and straightening its faces to totally geodesic ideal triangles (while keeping the ideal vertices fixed), none of the resulting tetrahedra flatten or overlap. While veering triangulations are known to exist only on hyperbolic $3$-manifolds \cite[Theorem 1.5]{veer_strict-angles}, many of them are not geometric \cite[Theorem 1.1]{FTW-random-veering}.  This somewhat limits our ability to precisely relate the combinatorics of veering triangulations, and hence the dynamics of pseudo-Anosov flows, to hyperbolic geometry. Although, some results in that direction are known; see \cite{Fenley_ideal, Fenley_Mosher, Huang-Taylor-hyperbolic}. 
	
	We used our algorithm to investigate experimentally whether lack of geometricity can be `fixed' by drilling. Although there is no theoretical result guaranteeing the existence of a geometric veering parent, we found one for every non-geometric veering triangulation in the veering census. The lengths of the flow cycles that we needed to drill were also relatively short; see Table \ref{tab:geometric:parents}.
	
		\begin{repexperimental}{experimental:geometric}
		Every non-geometric veering triangulation with up to 16 tetrahedra admits a flow cycle $\zeta$ such that the veering parent $\V_\zeta$ is geometric.
	\end{repexperimental}

Experimental Result~\ref{experimental:geometric} leads us to conjecture the following.

	\begin{repconjecture}{conj:geometric}
	Every transitive pseudo-Anosov flow can be encoded by a geometric veering triangulation.
\end{repconjecture}

We also used our drilling algorithm to find closed orbits of pseudo-Anosov flows (drilled along perfect sets) that are not ambiently isotopic to the unique geodesics in their respective free homotopy classes.

\begin{repexperimental}{experimental:geodesic}
	Every drilled pseudo-Anosov flow underlying a geometric veering triangulation  with up to 10 tetrahedra admits a closed orbit  that is not ambiently isotopic to the unique geodesic in its free homotopy class. 
\end{repexperimental}

The computation underlying the above statement is carried out entirely in terms of veering triangulations and their flow cycles. 
To check whether a flow cycle $\zeta$ of a veering triangulation $\V$  represents a closed orbit~$\gamma$ that is not ambiently isotopic to the freely homotopic geodesic, we consider two different drilling operations on $\V$. The first one constructs the veering parent $\V_\gamma$ via our drilling algorithm. The second drills out the geodesic in the free homotopy class of $\gamma$. A triangulation of the latter drilled manifold can be constructed using SnapPy \cite{snappea} (under certain additional conditions). If the manifolds obtained by these two drilling operations are not homeomorphic, then the closed orbit $\gamma$ is not ambiently isotopic to the freely homotopic geodesic; see Section \ref{subsec:orbit:geodesic}. Experimental Result~\ref{experimental:geodesic} follows from finding such a flow cycle for every geometric veering triangulation with at most $10$ tetrahedra. In Example~\ref{ex:figure:eight:geodesic} we describe one such flow cycle for the $2$-tetrahedron veering triangulation of the figure-eight knot complement.

Experimental Result \ref{experimental:geodesic} motivates the following conjecture.

\begin{repconjecture}{conj:geodesic}
	Every  (possibly drilled) pseudo-Anosov flow on a hyperbolic three-manifold admits closed orbits that are not ambiently isotopic to the unique geodesic in their respective free homotopy classes.
\end{repconjecture}

Our last experimental result relates to a conjecture of Ghys. Given flows $\Psi$ and $\Psi'$ on manifolds $N$ and $N'$, respectively, we say that $\Psi$ and $\Psi'$ are \emph{orbit equivalent} if there is a homeomorphism $N\to N'$ that sends orbits of $\Psi$ to orbits of~$\Psi'$. They are \emph{almost orbit equivalent} if they become orbit equivalent after removing finitely many closed orbits from each. The conjecture of Ghys states that all transitive Anosov flows with orientable invariant foliations are almost orbit equivalent. It is currently still open in full generality, although some partial results are known \cite{dehornoy-almost-equiv}.

To explore this conjecture experimentally, we define the \emph{veering ancestry graph} $\A$. Its vertices correspond to finite veering triangulations, with a directed edge from $\V$ to~$\V'$ whenever $\V'$ is a veering parent of $\V$. In Theorem \ref{thm:connected:cpts} we prove that connected components of $\A$ are in a bijective correspondence with almost orbit equivalence classes of transitive pseudo-Anosov flows. 
Since every transitive Anosov flow with orientable invariant foliations is encoded by an \emph{edge-orientable} veering triangulation with four \emph{ladders} in every boundary component, and every such triangulation encodes such a flow (see Section \ref{subsec:ghys}), we can restate the conjecture of Ghys as follows.

	\begin{repconjecture}{conj:veering:graph:component}
	Any two edge-orientable veering triangulations with four ladders in each boundary component are in the same connected component of $\A$. 
\end{repconjecture}

We used our algorithm to build a large subgraph $\A_{16,4,7}^+$ of $\A$. We started with all edge-orientable veering triangulations with at most $16$ tetrahedra and four ladders in each boundary component. We then added their veering parents obtained by drilling non-peripheral flow cycles of length at most $7$, as well as the edges corresponding to these drillings.

\begin{repexperimental}{exp:ancestry:experiments}
	\hfill
	\begin{enumerate}[label=(\Roman*), leftmargin=2.5em] 
		\item $\A^{+}_{16,4,7}$  has 8,767,114 vertices and 2247 connected components.
		\item The largest connected component of $\A^{+}_{16,4,7}$  has 8,212,337 vertices. 
		\item Every connected component of $\A^{+}_{16,4,7}$ has a vertex corresponding to a veering triangulation encoding an Anosov suspension flow.
	\end{enumerate}
\end{repexperimental}

Since all suspension Anosov flows are orbit equivalent \cite[Theorem A]{dehornoy-almost-equiv}, we deduce that 	all vertices of $\A^{+}_{16,4,7}$  lie in the same connected component of $\A$. This provides experimental evidence to support the Ghys' conjecture.

	\subsection{Outline of the paper}\label{subsec:intro:outline}
	In Section \ref{sec:pA} we discuss pseudo-Anosov flows and their orbit spaces, as well as perfect fits, perfect sets, and loom spaces. We prove that the orbit space of a transitive pseudo-Anosov flow drilled along a perfect set is a loom space.
	
	In Section \ref{sec:veering} we define veering triangulations and recall the Agol-Gu\'eritaud construction of a veering triangulation from a loom space. Although this construction is standard, we describe it here because our drilling construction is based on its modification  that is most easily explained by referring to the original. We also recall several further results connecting veering triangulations to pseudo-Anosov flows.
	
	In Section \ref{sec:parents} we discuss two abstract constructions of veering parents of $\V$: one that uses the entire loom space $(\loom, \loom^+, \loom^-)$ associated to $\V$, and another that involves only finitely many tetrahedron rectangles in $(\loom, \loom^+, \loom^-)$. The latter is what we refer to as a finite restriction of the Agol-Gu\'eritaud construction; it is the basis of our drilling algorithm.
	
	In Section \ref{sec:flow:cycles} we recall the results of Landry-Minsky-Taylor relating flow cycles of veering triangulations to periodic orbits of pseudo-Anosov flows.
	
	In Section \ref{sec:continents} we discuss convexity and convex hulls in loom spaces, as well as landscapes and continents in $\univV$. The key result is that the closure of the loomification of a continent is a convex disk with finitely many boundary points removed. From this we obtain the loom-compatible circular order on the set consisting of the cusps and the exit points of the continental cusp leaves of a continent.
	
	In Section \ref{sec:orders} we explain how to establish the horizontal order on certain stable continental cusp leaves and vertical order on certain unstable continental cusp leaves using the loom-compatible circular order. We then prove that if a continent $\C$ is convex then the associated circular order can be uniquely determined using only the combinatorics of~$\C$.
	
	In Section~\ref{sec:adapted:set} for any periodic point $x\in\loom$, we construct a subset of $(\loom^+,\loom^-,\loom^-)$ that is adapted to~$x$ and covered by finitely many edge rectangles. This is used to reduce the problem of finding all translates of $x$ in a given simplex rectangle in $(\loom,\loom^+,\loom^-)$ to a finite search through an explicit set of edge rectangles.
	
	In Section~\ref{sec:algorithm} we give a step-by-step description of the drilling algorithm, including how the theoretical results established in earlier sections are implemented. We also prove the correctness of the algorithm.
	
	In Section~\ref{sec:applications}, we use an implementation of our algorithm to study pseudo-Anosov flows experimentally. We explore four different directions, present the experimental results obtained in each, and formulate several conjectures suggested by these computations.
	
	In Section \ref{sec:questions}, we discuss several further questions concerning veering parents and related topics. These include the computational complexity of the algorithm (Section~\ref{subsec:q:time});  dependence of the combinatorics of a veering parent $\V_\zeta$ on those of $\V$ and $\zeta$ (Section~\ref{subsec:q:combinatorics});  relationship between drilling veering triangulations and \emph{vertical surgeries} (Sections~\ref{further:descreasing} and~\ref{subsec:q:surgery});  recognizing veering ancestors and, more generally, veering triangulations that are connected  in the veering ancestry graph (Section \ref{subsec:q:veering:graph}); and the operation of \emph{filling veering triangulations} (Section \ref{subsec:q:filling}).

	\subsection*{Acknowledgments}
	We thank Michael Landry and Sam Taylor for answering our questions about flow cycles of veering triangulations, Matthias Goerner for his help with SnapPy \cite{snappea}, which was crucial in establishing Experimental Result \ref{experimental:geodesic}, and Saul Schleimer for many helpful conversations. 
	
	Some of the research underlying this paper was carried out while the first author was affiliated with the University of Oxford and then the University of California, Davis. The manuscript was partially written while the authors were in residence at the Simons Laufer Mathematical Sciences Institute in Berkeley, California, during the Spring 2026 semester, supported by the National Science Foundation under Grant No. DMS-2424139. The second author was supported in part by National Science Foundation grant DMS-2203993.

\vspace*{-1em}
	\subsection*{Tool use declaration}
	The authors used SageMath \cite{sagemath}, Regina \cite{regina}, SnapPy \cite{snappea}, Python~\cite{python}, PyX \cite{pyx}, and Veering \cite{VeeringGitHub} for computational experiments, software development, and figure preparation. The first author used ChatGPT (OpenAI) for language editing during the final stages of manuscript preparation.
	
	\section{Pseudo-Anosov flows, orbit spaces, and loom spaces}\label{sec:pA}
	
	In this section we define pseudo-Anosov flows and discuss basic properties of their orbit spaces. We then define a \emph{perfect set} for a transitive pseudo-Anosov flow and show that a transitive pseudo-Anosov flow $\Psi$ and a perfect set $\Lambda$ for $\Psi$ determines a \emph{loom space}  $\loom = \loom(\Psi, \Lambda)$. This loom space is used in Section \ref{subsec:loom:to:veering} to construct a veering triangulation encoding $\Psi$.
	
	\subsection{Pseudo-Anosov flows}\label{subsec:pA}
	
	In this subsection we recall the definition of a pseudo-Anosov flow and discuss properties of its stable and unstable foliations. Both these foliations are singular. First, we describe the types of singularities that may occur.
	
	\begin{definition}\label{def:1D:prong}
		Let $H = \lbrace (x,y) \in \rr^2 \ |  \ y\geq0 \rbrace$. Equip $H$ with a one-dimensional foliation~$\F$ by horizontal lines. Set 
		\begin{align*}
		\ell^\rightarrow &= \lbrace (x,y) \in H \ | \ x \geq 0, y=0 \rbrace,\\
		\ell^\leftarrow &= \lbrace (x,y) \in H \ | \ x \leq 0, y=0 \rbrace.\end{align*} 
		We call $\ell^\rightarrow$ and $\ell^\leftarrow$ the \emph{boundary half-leaves} of $\F$. 
		
		Fix an integer $k\geq 3$. For every integer $1\leq i \leq k$ let $(H_i, \F_i)$ be a copy of the foliated half-plane $(H, \F)$, and let $\ell_i^\rightarrow$, $\ell_i^\leftarrow$ be the two boundary half-leaves of $\F_i$. Identify the negative boundary leaf $\ell_i^\leftarrow$ of $\F_i$ with the positive boundary leaf $\ell_{i+1}^\rightarrow$ of $\F_{i+1}$ using a homeomorphism $h_i: \ell_i^\leftarrow \rightarrow \ell_{i+1}^\rightarrow$ sending $(x,0) \in H_i$ to $(-x, 0) \in H_{i+1}$; the subscript $i$ is taken modulo~$k$. After all identifications, the resulting space $P$ is homeomorphic to a plane. Let $\F_P$ be the foliation of $P$ induced by the foliations $\F_i$. 
		The boundary half-leaves of $\F_i$'s project to $k$ rays in $P$ that meet at a single point $p \in P$. We say that $\F_P$ has a  \emph{$k$-prong singularity} at $p$.
		The foliation $\F_P$ is nonsingular around all other points of $P$.
	\end{definition}

	\begin{definition}
		Let $\F$ be a two-dimensional foliation of a three-manifold $N$. Let $\ell$ be a line or a simple closed curve contained in a leaf of $\F$. Let $k\geq 3$ be an integer. We say that $\ell$ is a \emph{$k$-prong singularity} of $\F$  if there exists a tubular neighborhood $U$ of $\ell$ in~$N$  such that for every meridional disk $D$ of $U$ the one-dimensional foliation $D \cap \F$ has precisely one $k$-prong singularity which appears at $D \cap \ell$. If furthermore $\ell$ is closed, we say that $\ell$ is a \emph{closed $k$-prong} of $\F$.
	\end{definition}
	
	In the following definition,  $\mathrm{Homeo}^+(\rr, 0)$ denotes the set of orientation-preserving homeomorphisms of $\rr$ that map 0 to 0. 
	\begin{definition}\label{def:pA:flow} \cite[Definition 7.1]{Fenley_geometry} 
		Let $N$ be a closed oriented three-manifold. A flow $\Psi: N\times \rr \rightarrow N$ is a \emph{(topological) pseudo-Anosov flow} if it has no point orbits and there is a pair of singular two-dimensional foliations $\W^+$, $\W^-$ in $N$ such that
		\begin{enumerate}
			\item $\W^+$, $\W^-$ have finitely many singularities, each of which is a closed $k$-prong  for some integer  $k\geq 3$, 
			\item  $\W^+$, $\W^-$ are transverse to each other away from the singularities, and coincide along their singular sets,
			\item The leaves of $\W^+$ and $\W^-$ intersect along orbits of $\Psi$,
			\item For any path metric $d$ on $N$ and for any points $x,y$ that are on the same leaf of $\W^+$ (respectively, $\W^-$) there is $h \in \mathrm{Homeo}^+(\rr, 0 )$ such that \mbox{$d\left(\Psi(x,t), \Psi(y, h(t))\right) \rightarrow 0$} as $t\rightarrow + \infty$ (respectively, $t \rightarrow - \infty$). Furthermore, if $x, y$ do not lie on the same orbit of~$\Psi$ then there exists $\delta > 0$ such that for every $h \in \mathrm{Homeo}^+(\rr, 0 )$ there exists $T<0$ (respectively, $T>0$) such that $d\left(\Psi(x,t), \Psi(y, h(t))\right) >\delta$ for every $t<T$ (respectively, $t>T$). 
		\end{enumerate}
	\end{definition}
	
	\begin{remark}
	We refer the readers to Section 1.1.2 of \cite{Barthelme-Mann-manuscipt} for a detailed discussion of the various definitions of pseudo-Anosov flows that can be found in the literature.
	\end{remark}

	We call the foliations $\W^+$, $\W^-$ associated to a pseudo-Anosov flow $\Psi$ its \emph{stable} and  \emph{unstable foliation}, respectively. A \emph{stable leaf} is a leaf of $\W^+$ and an \emph{unstable leaf} is a leaf of $\W^-$. When we write `$\W^\pm$' we mean `$\W^+$ or $\W^-$'.  
	Since $\W^\pm$ share singularities,  we may denote their common singular set by $\Sing(\Psi)$.  Elements of $\Sing(\Psi)$ are called the \emph{singular orbits} of~$\Psi$. We allow $\Sing(\Psi)$ to be empty, in which case $\Psi$ is an \emph{Anosov flow}.

	\begin{remark}
		When discussing flows, many authors use the terms `orbits' and `flow lines' interchangeably. We warn the reader that in this paper the meaning of a `flow line' will be different; see Definition~\ref{def:flow:line}.
	\end{remark}
	
	A pseudo-Anosov flow $\Psi$ is \emph{transitive} if it admits a dense orbit. Assuming that~$\Psi$ is transitive is not a strong restriction; for instance, every pseudo-Anosov flow on a hyperbolic manifold is transitive \cite[Proposition~2.7]{Mosher-dynamical1992-part1}. Transitivity implies two useful properties of the flow:  that all leaves of its stable/unstable foliation are dense \cite[Proposition 1.3]{Mosher-dynamical1992-part1}, and that periodic orbits are dense \cite[p. 493]{Brunella_surfaces}. See also Proposition~\ref{prop:orbit:space} below.
	
	From now on we assume that $\Psi$ is a transitive pseudo-Anosov flow on an oriented three-manifold $N$. A leaf $L$  of $\W^\pm$ is \emph{singular} if it contains a singular orbit of $\Psi$.  Otherwise, we say that $L$ is \emph{regular}. Every regular leaf  of $\W^\pm$ is homeomorphic to a plane, an open annulus, or an open M\"obius band \cite[Proposition~1.3]{Mosher-dynamical1992-part1}.  It follows from part (4) of Definition \ref{def:pA:flow} that every leaf  whose fundamental group is nontrivial contains a unique closed orbit of~$\Psi$. This orbit must be homotopically non-trivial \cite[proof of Proposition~4.2]{Fenley_Mosher}. A regular closed orbit $\gamma$ is \emph{twisted} if the leaf of $\W$ which contains it is homeomorphic to an open M\"obius band. Otherwise, we say that~$\gamma$ is \emph{untwisted}.

	Let $\liftedFlow$ be the lift of $\Psi$ to the universal cover $\univN$. The preimages of foliations $\W^+, \W^-$ form foliations $\widetilde{\W}^+$, $\widetilde{\W}^-$ of $\univN$ which are invariant under $\liftedFlow$. Each orbit of~$\liftedFlow$ is homeomorphic to~$\rr$.  The \emph{orbit space} $\orbit = \orbit(\Psi)$ of $\Psi$ is the quotient of $\widetilde{N}$ by the equivalence relation that identifies points lying on the same orbit of $\liftedFlow$.
	Thus each point $x \in \orbit$ corresponds to an orbit $\widetilde{\gamma}(x)$ of $\liftedFlow$. We say that $x$ is \emph{periodic} if the image of $\widetilde{\gamma}(x)$ under the covering map $\univN \rightarrow N$ is closed.  Fenley-Mosher proved that~$\orbit$  equipped with the quotient topology is homeomorphic to~$\rr^2$ \cite[Proposition~4.2]{Fenley_Mosher}. Thus~$\orbit$ is orientable. We orient $\orbit$ using the fixed orientation on $N$, the orientation of orbits of~$\liftedFlow$, and the right hand rule. 
	
	The foliations $\widetilde{\W}^+$, $\widetilde{\W}^-$ project to a pair of one-dimensional singular foliations $\orbit^+$, $\orbit^-$ in~$\orbit$, respectively. These foliations are transverse to each other away from the singularities.  We say that $(\orbit, \orbit^+, \orbit^-)$ is the \emph{bi-foliated plane associated to $\Psi$}. Singular and regular leaves of $\widetilde{\W}^\pm$ project to \emph{singular} and \emph{regular} leaves of $\orbit^\pm$, respectively. We define a \emph{half-leaf} of $\orbit^\pm$ to be a closed subset of a leaf of $\orbit^\pm$  that is homeomorphic to $\lbrack 0,+\infty)$. Below we list some well-known properties of $(\orbit, \orbit^+, \orbit^-)$ that will be used in later parts of this paper.

	\begin{proposition}\label{prop:orbit:space}
		Let $(\orbit, \orbit^+, \orbit^-)$ be the bi-foliated plane associated to a transitive pseudo-Anosov flow $\Psi$ on a closed three-manifold $N$. 
		\begin{enumerate}
			\item Periodic points are dense in $\orbit$. \label{orbit_space:periodic:points:dense}  
			\item Each leaf of $\orbit^\pm$ contains at most one periodic point. \label{orbit_space:one:periodic}
			\item The deck group action of $\pi_1(N)$ on $\widetilde{N}$ induces a continuous, orientation-preserving action of $\pi_1(N)$ on $\orbit$ which sends $\orbit^+$ to $\orbit^+$ and $\orbit^-$ to $\orbit^-$. Additionally, this action has the following properties. \label{orbit_space:group_action}
			\begin{enumerate}
				\item The stabilizer of a periodic point $x \in \orbit$ is infinite cyclic, generated by an element $g \in \pi_1(N)$  conjugate to the homotopy class of the primitive orbit  of~$\Psi$ whose elevation to $\univN$ projects to $x$. The half-leaves of $\orbit^\pm$  originating from $x$ are permuted by $g$. Moreover,  if $\ell$ is a leaf of $\orbit^+$ (respectively, $\orbit^-$) that contains $x$ then
				$g^{-n} (y) \rightarrow x$ (respectively, $g^n (y) \rightarrow x$) as $n\rightarrow +\infty$ for every $y \in \ell$. \label{orbit_space:fixed:points}
				\item The orbit of a periodic point does not have any accumulation points in $\orbit$.\label{orbit_space:no:accumulation:points} 
				\item The orbit of a half-leaf of $\orbit^\pm$ is dense in $\orbit$. \label{orbit_space:dense:leaves}
			\end{enumerate}
		\end{enumerate}
	\end{proposition}
	\begin{proof}

		The proof that periodic orbits of a transitive Anosov flow are dense can be found in \cite[Theorem 5.3.50]{HyperbolicFlows}. The same argument works in the general case \cite[p. 493]{Brunella_surfaces}. See also \cite[Proposition 1.4.9]{Barthelme-Mann-manuscipt}. This gives \eqref{orbit_space:periodic:points:dense}.

		Every regular leaf of $\W^\pm$ with nontrivial fundamental group is homeomorphic either to an open annulus or to an open M\"obius band  \cite[Proposition~1.3]{Mosher-dynamical1992-part1}. Every singular leaf of $\W^\pm$ is the union of finitely many open annuli glued along the singular orbit. Thus 
		part \eqref{orbit_space:one:periodic} follows from the asymptotic behavior of orbits contained in the same leaf of $\W^\pm$, stated in part (4) of Definition~\ref{def:pA:flow}. 
		Since every closed orbit of $\Psi$ is homotopically non-trivial \cite[proof of Proposition 4.2]{Fenley_Mosher}, we also get \eqref{orbit_space:fixed:points}.

		A periodic point in $\orbit$ corresponds to a closed orbit of $\Psi$. The preimage of such an orbit in $\widetilde{N}$ must be a closed and discrete set of orbits of $\liftedFlow$.  This gives  \eqref{orbit_space:no:accumulation:points}.

		Part \eqref{orbit_space:dense:leaves} is proved in \cite[Lemma 2.21]{BFM} for leaves, but the same argument works for half-leaves. 
	\end{proof}

	We call the action of $\pi_1(N)$ from Proposition \ref{prop:orbit:space} \eqref{orbit_space:group_action} the \emph{$\Psi$-action} of $\pi_1(N)$ on $\orbit$.

	\subsection{Rectangles in bi-foliated surfaces}
	In the previous section we defined the orbit space of a pseudo-Anosov flow, which is an example of a bi-foliated plane. More generally, a \emph{bi-foliated surface} is a triple $(\X, \X^+, \X^-)$ such that $\X$ is a surface and $\X^+, \X^-$ are  one-dimensional foliations of $\X$ which are transverse to each other. For us, $\X$ will always be either a plane or a plane punctured along a closed discrete set of points. The foliations $\X^\pm$ may or may not be singular. 
	In this subsection we consider subsets of $(\X, \X^+, \X^-)$  that are trivially bi-foliated as a product. 

	\begin{definition}\label{defn:rectangle}
		Let $(\X, \X^+, \X^-)$ be a bi-foliated surface. 
		Let  \mbox{$f_R: (0,1)^2\rightarrow \X$} be an embedding which maps vertical and horizontal segments in $(0,1)^2$ to leaves of $\X^+$ and~$\X^-$, respectively. We say that $f_R$ is a \emph{rectangle embedding} and that the image $R$ of $f_R$ is a \emph{rectangle} in $(\X, \X^+, \X^-)$.
		We denote the closure of~$R$ in $\X$ by $\overline{R}$. 
	\end{definition} 
	
	\begin{definition}
		Let  \mbox{$f_R: (0,1)^2\rightarrow \X$} be a  rectangle embedding into a bi-foliated surface $(\X, \X^+, \X^-)$. A \emph{side} of $R$ is the set of accummulation points of the family $\lbrace f_R(s,t)\rbrace_{s,t \in (0,1)}$, where either $s\to 0^+$, or $s\to 1^-$, or $t\to 0^+$, or $t\to 1^-$. 
		The union of all sides of $R$ forms the \emph{boundary} $\partial R = \overline{R} \bez R$ of $R$. 
	\end{definition}

	We may locally fix the cardinal directions in a rectangle $R \subset (\X, \X^+, \X^-)$ by declaring one side to be south (with the opposite side north) and one side to be west (with the opposite side east). We denote the west, east, south, and north sides of $R$ by $\partial_W R$, $\partial_E R$, $\partial_S R$, and $\partial_N R$, respectively. Note that the sides of $R$ are not necessarily homeomorphic to $\lbrack 0,1 \rbrack$; see Example \ref{ex:skewed} or Definition \ref{def:tet-rectangle}. 
	
	\begin{convention}\label{convention:cardinal:directions}
		For any rectangle $R$ in a bi-foliated surface $(\X, \X^+, \X^-)$ we pick the cardinal directions in $R$ as follows.
		\begin{enumerate}
			\item The west and east sides of $R$ are contained in the leaves of $\X^+$. Therefore, the south and north sides of $R$ are contained in the leaves of $\X^-$.
			\item If the foliations $\X^\pm$ are oriented, the orientation of $\X^+\cap R$ points from the south side of $R$ to the north side of $R$, and the orientation of $\X^-\cap R$ points from the west side of $R$ to the east side of $R$.
		\end{enumerate}
	\end{convention}
	
	Once the cardinal directions in $R$ are fixed by Convention~\ref{convention:cardinal:directions} (1), 
	the \emph{vertical sides} of~$R$ are its west and east sides, and the \emph{horizontal sides} are its south and north sides.

	
	\begin{definition}
		\label{def:expansion}
		Let $R$ be a rectangle in a bi-foliated surface $(\X, \X^+, \X^-)$. An \emph{expansion} of~$R$ is a rectangle $R'$ in $(\X, \X^+, \X^-)$ such that $R \subset R'$. 
		If cardinal directions in $R$ are fixed then we say that an expansion $R'$ of $R$ is 
		\begin{itemize}
			\item a \emph{western expansion} of $R$ if $\partial_E R = \partial_E R'$, $\partial_N R \subset \partial_N R'$, and $\partial_S R \subset \partial_S R'$;
			\item  an \emph{eastern expansion} of $R$ if $\partial_W R = \partial_W R'$, $\partial_N R \subset \partial_N R'$, and $\partial_S R \subset \partial_S R'$;
			\item a \emph{southern expansion} of $R$  if $\partial_N R = \partial_N R'$, $\partial_W R \subset \partial_W R'$, and $\partial_E R \subset \partial_E R'$;
			\item a \emph{northern expansion} of $R$  if $\partial_S R = \partial_S R'$, $\partial_W R \subset \partial_W R'$, and $\partial_E R \subset \partial_E R'$.
		\end{itemize}
		A \emph{horizontal expansion} of $R$ is a western and/or eastern expansion of $R$, and a \emph{vertical expansion} of $R$ is a southern and/or northern expansion of $R$. 		We say that an expansion of $R$ is \emph{maximal} if it is maximal under the inclusion. 
	\end{definition}

	\subsection{Perfect fits and perfect sets}
	
	In this subsection, we discuss certain rectangles in bi-foliated surfaces whose closures are not homeomorphic to $\lbrack 0,1 \rbrack^2$.

	\begin{example}\label{ex:skewed}
		Let $B = \lbrace (x,y) \in \rr^2 \ | \ x-1 < y < x+1 \rbrace$. Equip $B$ with foliations $B^+$, $B^-$ induced by the vertical and horizontal foliations of $\rr^2$, respectively. Then $(B, B^+, B^-)$ is a bi-foliated plane. The closure of $(0,1)^2$ in $B$ is equal to $\lbrack 0,1\rbrack^2 \bez \lbrace (0,1), (1,0)\rbrace$. The horizontal leaf segment $\ell_{h,1} = \lbrace (x,1) \rbrace_{x \in (0,1)}$ and the vertical leaf segment $\ell_{v,0} = \lbrace (0,y)\rbrace_{y \in (0,1)}$ do not intersect even though arbitrarily close to $\ell_{h,1}$ we can find horizontal leaves intersecting $\ell_{v,0}$, and vice versa. 
	\end{example}
	
	\begin{remark}\label{remark:skewed}
		There are Anosov flows whose orbit spaces are equivalent to the bi-foliated plane described in Example \ref{ex:skewed}. Such flows are called \emph{positively skewed}  \cite[Section 1.3]{Ioannis-phd}.	\end{remark}
	

	\begin{definition}\label{def:perfect:fit}
		A rectangle $R$ in a bi-foliated surface $(\X, \X^+, \X^-)$ is a \emph{perfect fit rectangle} if its closure~$\overline{R}$ in $\X$ is homeomorphic to $\lbrack 0,1 \rbrack - \lbrace (1,1) \rbrace$. 
	\end{definition}

	The missing boundary point of a perfect fit rectangle $R$ is called the \emph{ideal corner} of $R$.
	More generally, an \emph{ideal corner} of a rectangle $R$ in $(\X,\X^+,\X^-)$ is a point missing from the boundary of $R$ that does not disconnect any side of $R$. Such a rectangle may have additional ideal points in its boundary, either at its corners or in the interior of its sides; see Definition~\ref{def:tet-rectangle}.
	A \emph{material corner} of $R$ is an intersection point of two non-parallel sides of $R$.
	
	
	\begin{definition}
		A pseudo-Anosov flow $\Psi$ does \emph{not have perfect fits} if the orbit space of~$\Psi$ does not contain any perfect fit rectangles. Otherwise, we say that $\Psi$ has perfect fits.
	\end{definition}
	Examples of pseudo-Anosov flows with perfect fits include positively skewed Anosov flows mentioned in Remark \ref{remark:skewed}

	In Subsection \ref{subsec:loom:to:veering} we will outline a construction of a veering triangulation out of (the orbit space of) a transitive pseudo-Anosov flow. This construction is due to  Agol-Gu\`eritaud. Originally, it was described only for pseudo-Anosov flows without perfect fits. 
	However, it can be applied to all pseudo-Anosov flows after appropriately generalizing the `no perfect fits' condition.
	
	\begin{definition}\label{def:perfect:fits:relative} \cite[Definition 5.13]{Tsang-Agol}
		Let $(\orbit, \orbit^+, \orbit^-)$ be the bi-foliated plane associated to a pseudo-Anosov flow $\Psi$ on a three-manifold $N$. Given a collection~$\Lambda$ of closed orbits of~$\Psi$, let $\widetilde{\Lambda}$  be the preimage of $\Lambda$ under the universal covering map $\univN \rightarrow N$. We identify~$\widetilde{\Lambda}$ with a subset of $\orbit$.   We say that $\Psi$ does \emph{not have perfect fits relative to $\Lambda$} if every perfect fit rectangle in $\orbit$ has nonempty intersection with $\widetilde{\Lambda}$. 
	\end{definition}

	In the next subsection, we will use the bi-foliated plane $(\orbit, \orbit^+, \orbit^-)$ associated to a transitive pseudo-Anosov flow $\Psi$ and a suitable collection $\Lambda$ of closed orbits of $\Psi$ to construct another bi-foliated plane $(\loom, \loom^+, \loom^-)$ associated to the pair $(\Psi, \Lambda)$. We want $(\loom, \loom^+, \loom^-)$ to satisfy certain conditions: in particular,  $\loom^\pm$ should be non-singular, and every rectangle in $(\loom, \loom^+, \loom^-)$ should be contained in a rectangle of a particular type; see Definition \ref{def:loom}. To ensure this, we need $\Lambda$ to satisfy a few properties which we enlist in the following definition.
	
	\begin{definition}\label{def:perfect:set}
		Let $\Psi$ be a transitive pseudo-Anosov flow. We say that a collection~$\Lambda$ of closed orbits of $\Psi$ is a \emph{perfect set for $\Psi$}  if $\Lambda$ is finite, nonempty, contains $\Sing(\Psi)$, and $\Psi$ does not have perfect fits relative to $\Lambda$.
	\end{definition}
	
	If $\Psi$ does not have perfect fits in the usual sense, the set $\Sing(\Psi)$ is a perfect set for~$\Psi$. Otherwise, it is always possible to augment $\Sing(\Psi)$ by additional closed orbits of $\Psi$ to obtain a perfect set for $\Psi$. 
	
	\begin{proposition}\cite[Proposition 2.7]{Tsang-Birkhoff} \label{prop:tsang:perfect:set}
		If $\Psi$ is a transitive pseudo-Anosov flow then there exists a closed orbit $\gamma \notin \Sing(\Psi)$ such that $\Sing(\Psi) \cup \lbrace \gamma \rbrace$ is a perfect set for~$\Psi$. \qed
	\end{proposition}
	
	The choice of an additional orbit to be added to the set $\Sing(\Psi)$ to obtain a perfect set is not canonical.

	\subsection{A loom space out of the orbit space and a perfect set}\label{subsec:loom}

	A \emph{loom space} is another type of a bi-foliated plane that was introduced by Schleimer and the second author in \cite{SchleimSegLoom}. 
	The key notion in its definition is that of \emph{tetrahedron rectangles}, defined below. 
	
	\begin{definition} \label{def:tet-rectangle}
		Let $R$ be a rectangle in a bi-foliated surface $(\X, \X^+, \X^-)$ with a rectangle embedding $f_R: (0,1)^2 \rightarrow \X$. We say that $R$ is a \emph{tetrahedron rectangle} if there exists a continuous extension of $f_R$ to a homeomorphism \[\overline{f_R}: \lbrack 0,1\rbrack^2 \bez \lbrace (a,0), (b,1), (0,c), (1,d) \rbrace \rightarrow \overline{R}\] for some $a,b,c,d \in (0,1)$ such that $a\neq b$ and $c\neq d$.
	\end{definition}
	
	Observe that a tetrahedron rectangle in $(\X, \X^+, \X^-)$  cannot be expanded to a larger rectangle along the leaves of $\X^+$ or $\X^-$. Thus every tetrahedron rectangle is  maximal under the inclusion. However, the converse does not generally hold; see Example \ref{ex:skewed:2}.

	Recall from Definition \ref{def:perfect:fit} that if $R$ is a perfect fit rectangle in a bi-foliated space $(\X, \X^+, \X^-)$ then $\overline{R} \bez R$ consists of four leaf segments $s_1, s_2, s_3, s_4$ such that $s_i$ intersects $s_{i+1}$ exactly once for $1\leq i \leq 3$ and $s_1 \cap s_4 = \emptyset$. We call the closures of $s_1$ and $s_4$   the \emph{perfect fit sides} of $R$. If $s$ is a perfect fit side of $R$ then for every $0<\epsilon<1$ the image of $(0, \epsilon)$ under any homeomorphism $h:(0,1\rbrack\rightarrow s$ is called an \emph{initial segment} of $s$. 
	\begin{definition}\cite[Definition 2.11]{SchleimSegLoom} \label{def:loom}
		A \emph{loom space} is a plane $\X$ equipped with a pair of transverse non-singular foliations $\X^+$, $\X^-$ satisfying two conditions.
		\begin{enumerate}
			\item For every perfect fit rectangle $R$ in $(\X, \X^+, \X^-)$ and every perfect fit side $s$ of~$R$ some initial segment of $s$ is contained in some rectangle.		
			\item Every rectangle in $(\X, \X^+, \X^-)$ is contained in a tetrahedron rectangle.
		\end{enumerate}
	\end{definition}


	First, we discuss two examples of bi-foliated planes which are not loom spaces.
	
	\begin{example}\label{ex:skewed:2}
		The triple $(B, B^+, B^-)$ discussed in Example \ref{ex:skewed} is a plane foliated by a pair of transverse non-singular foliations. However, it fails both conditions of a loom space. If $c$ is the ideal corner of a perfect fit rectangle $R$ in $(B, B^+, B^-)$ then every open neighborhood of $c$ contains an ideal corner of a different perfect fit rectangle $R'$ such that neither $R\subset R'$ nor $R' \subset R$. Thus no initial segment of a perfect fit side of $R$ is contained in a rectangle. Furthermore, there are no tetrahedron rectangles in $(B, B^+, B^-)$. 
	\end{example}
	
	\begin{example}\label{ex:orbit:suspension}
		Let $\F_v, \F_h$ be the vertical and the horizontal foliation in $\rr^2$, respectively. Then there are no perfect fits in $(\rr^2, \F_v, \F_h)$. Consequently, $(\rr^2, \F_v, \F_h)$ satisfies condition (1) of a loom space. However, there are no tetrahedron rectangles in $(\rr^2, \F_v, \F_h)$. Hence $(\rr^2, \F_v, \F_h)$ is not a loom space.
	\end{example}
	
	Let $\Psi$ be a transitive pseudo-Anosov flow on an oriented three-manifold $N$. Let~$\Lambda$ be a perfect set for $\Psi$,  and let $\widetilde{\Lambda}$ be the preimage of~$\Lambda$ in the universal cover of $N$. Viewing $\widetilde{\Lambda}$ as a subset of the orbit space $\orbit = \orbit(\Psi)$ of~$\Psi$ allows us to define  $\drilledOrbit = \drilledOrbit(\Psi, \Lambda) = \orbit \bez \widetilde{\Lambda}$.  Each point of $\widetilde{\Lambda}$ gives precisely one \emph{puncture} of $\drilledOrbit$. By Proposition~\ref{prop:orbit:space} \eqref{orbit_space:no:accumulation:points}, $\widetilde{\Lambda}$ does not have any accumulation points in~$\orbit$. Thus each puncture of $\drilledOrbit$ has a neighborhood homeomorphic to a once punctured disk. 

	Let $\loom = \loom(\Psi, \Lambda)$ be the universal cover of~$\drilledOrbit$. It is a non-compact  simply connected \mbox{two-manifold} without boundary, and thus it is homeomorphic to a plane \cite[Corollary~1.8]{Epstein-2manifolds}.  
	A \emph{cusp neighborhood} in $\loom$ is an open subset of $\loom$ whose closure is homeomorphic to the universal cover of a closed two-dimensional disk with the center removed.
	The interior of each cusp neighborhood can be parametrized by \mbox{$\lbrace (r,\theta) \ | \ 0<r<1 , \ \theta \in \rr\rbrace$}. The ideal point approached by $(r, \theta)$ as $r \to 0^+$ is called a \emph{cusp} of $\loom$. 	The distinction between a puncture and a cusp will become important in Section~\ref{sec:parents}, where we consider a bi-foliated surface that has both punctures and cusps. 

	Let $\loom^+$, $\loom^-$  be the preimages of $\drilledOrbit^+, \drilledOrbit^-$, respectively, under $\loom \rightarrow \mathcal{PO}$. Since $\Lambda$ contains all singular orbits of $\Psi$, the foliations $\loom^\pm$ are non-singular. 
	Furthermore, every leaf $\ell$ of $\loom^\pm$ is an elevation of a connected component of $\ell' \bez \widetilde{\Lambda}$ for some leaf $\ell'$ of $\orbit^\pm$ and is thus homeomorphic to $\rr$. We distinguish two types of leaves, depending on whether $\ell'$ is disjoint from $\widetilde{\Lambda}$ or not.
	
	\begin{definition}\label{def:cusp:leaf}
		Let $\ell$ be a leaf of $\loom^\pm$. Let $\ell'$ be the leaf of $\orbit^\pm$ such that $\ell$ is an elevation of a connected component of $\ell' \bez \widetilde{\Lambda}$.  We say that $\ell$ is a \emph{regular leaf} of $\loom$ if $\ell'$ is disjoint from $\widetilde{\Lambda}$. Otherwise, we say that $\ell$ is a \emph{cusp leaf} of $\loom^\pm$.  
	\end{definition}

	By Proposition \ref{prop:orbit:space} \eqref{orbit_space:one:periodic}, each leaf of $\orbit^\pm$ intersects $\widetilde{\Lambda}$ at most once. Thus there is exactly one cusp $c(\ell)$ associated to each cusp leaf $\ell$ of $\loom^\pm$. An \emph{initial segment} of $\ell$ is the image of $(-\infty, s)$, for any $s \in \rr$, under any homeomorphism $h:\rr \rightarrow \ell$ such that $h(t)$ approaches $c(\ell)$ as $t \rightarrow -\infty$.
	We also say that $\ell$ \emph{originates from} $c(\ell)$ or that $\ell$ is a $c(\ell)$-leaf. This terminology corresponds to viewing cusp leaves as oriented away from the cusp (which need not agree with the orientation inherited from the chosen orientation of the foliation).

	\begin{definition}\label{def:orthant}
		Let $c$ be a cusp of $\loom$. We say that a subset $O_c \subset \loom$ is an \emph{orthant} at $c$ if 
		\begin{itemize}
			\item $O_c$ is a connected component of $\loom \bez (\ell^+ \cup \ell^-)$ for some stable $c$-leaf $\ell^+$ and some unstable $c$-leaf $\ell^-$.
			\item $O_c$ does not contain any $c$-leaves of $\loom^\pm$. 
		\end{itemize}
		We call $\ell^+$ and  $\ell^-$  the \emph{stable} and the \emph{unstable boundary} of $O_c$, respectively. We also say that they \emph{cobound} $O_c$.
	\end{definition}
	For a material point $x \in \loom$, we define an orthant $O_x$ at $x$ analogously, except that \mbox{$c$-leaves} from the above definition are replaced by half-leaves with endpoint $x$. Since there are exactly four orthants at $x$, we also call $O_x$ a \emph{quadrant} at $x$.

	Around the cusps of $\loom$ there are always perfect fits rectangles whose perfect fit sides are  initial segments of some cusp leaves of $\loom^\pm$. We call such rectangles \emph{cusp rectangles}.
	
	\begin{definition}
		We say that a perfect fit rectangle $R$ in $(\loom, \loom^+, \loom^-)$ is a \emph{cusp rectangle} if one perfect fit side of $R$ is an initial segment of some cusp leaf of $\loom^+$ and the other perfect fit side of $R$ is an initial segment of some cusp leaf of $\loom^-$. 
	\end{definition}
	
	Now we will prove that all perfect fit rectangles in $(\loom, \loom^+, \loom^-)$ are cusp rectangles.

	\begin{lemma}\label{lem:only:fake:pf}
		Every perfect fit rectangle in $(\loom, \loom^+, \loom^-)$ is a cusp rectangle.
	\end{lemma}
	
	\begin{proof}
		Let $\pi_\Lambda$ be the composition of the covering projection $\loom \rightarrow \drilledOrbit$ and the natural inclusion $\drilledOrbit \hookrightarrow \orbit$. Let $R$ be a perfect fit rectangle in $\loom$. Then $R$  does not contain any cusps of~$\loom$. Thus $\pi_\Lambda(R)$ is disjoint from $\widetilde{\Lambda}$, and hence, by Definition \ref{def:perfect:set}, is not a perfect fit rectangle in $\orbit$.
		But since $R$ is a perfect fit rectangle, exactly one corner of~$\pi_\Lambda(R)$ must belong to $\widetilde{\Lambda}$.
		Therefore the perfect fit sides of $R$ are contained in cusp leaves of~$\loom^\pm$. This means that $R$ is a cusp rectangle.
	\end{proof}
	
	The biggest difference between cusp rectangles and other perfect fit rectangles is that the perfect fits sides of the former can be `thickened up' to a rectangle; see Lemma \ref{lem:condition:1} below. Compare this with perfect fit rectangles from Example \ref{ex:skewed}, for which this is not possible. 

	\begin{lemma}\label{lem:condition:1} 
		The interior of every perfect fit side of every perfect fit rectangle in $(\loom, \loom^+, \loom^-)$ is contained in some rectangle. 
	\end{lemma}
	\begin{proof}	
		Let $R$ be a perfect fit rectangle in $(\loom, \loom^+, \loom^-)$ and let $s$ be a perfect fit side of~$R$. By Lemma \ref{lem:only:fake:pf}, the ideal corner $c$ of $R$ is a cusp  of $\loom$. 
		Let $\ell$ be the cusp leaf that contains~$s$. Without loss of generality, assume that $\ell$ is a leaf of $\loom^+$. Let $\ell^-_1$, $\ell^-_2$ be the cusp leaves of $\loom^-$ that cobound an orthant at $c$ with $\ell$. 
		By Proposition \ref{prop:orbit:space} \eqref{orbit_space:no:accumulation:points},  $\widetilde{\Lambda}$ does not have any accumulation points in $\orbit$. Thus there is no infinite sequence of distinct cusps of $\loom$ that accumulate on $s$. It follows that for $i=1,2$ there must be $x_i \in \ell_i^-$ such that the leaf $\ell_i^+$ of $\loom^+$ passing through $x_i$ intersects the leaf $\ell^-$ of $\loom^-$ passing through the endpoint of $s$. Then $\ell_1^-, \ell_2^-, \ell_1^+, \ell_2^-$ and $\ell^-$ cobound a rectangle $R'$ which contains the interior of $s$.
	\end{proof}

	Set $M = M(\Lambda) = N \bez \Lambda$. Observe that  $(\loom, \loom^+, \loom^-)$ is the orbit space of the drilled flow $\Psi \bez \Lambda$ on $M$.
	Thus, as in Proposition~\ref{prop:orbit:space}~\eqref{orbit_space:group_action}, we get an orientation-preserving action of $\pi_1(M)$ on $\loom$.  We call it the \emph{$(\Psi, \Lambda)$-action} of $\pi_1(M)$ on $(\loom, \loom^+, \loom^-)$.


	Since $\loom$ is simply connected and foliations $\loom^\pm$ are non-singular, they must be orientable. From now on we assume that they are oriented. Equip each rectangle~$R$ in $(\loom, \loom^+, \loom^-)$ with cardinal directions as in Convention \ref{convention:cardinal:directions}. We say that  $R$ has a cusp in its side $s$ if $s$ is disconnected or $s$ is an initial segment of some cusp leaf of $\loom^\pm$. If $R$ has a cusp in any of its sides we say that $R$ has a cusp in its boundary. 
	In Lemma \ref{lem:expansions}, we describe the structure of the maximal western expansion of $R$. 

	\begin{lemma}\label{lem:expansions}
		Every rectangle $R$  in $(\loom, \loom^+, \loom^-)$ admits a
		maximal western expansion~$R'$ that has at most one cusp in each side, and exactly one cusp in the interior of the west  side. Furthermore, no two cusps in the boundary of $R'$ are connected by a leaf of $\loom^\pm$. Analogous statements hold for eastern, southern and northern expansions of~$R$.
	\end{lemma}
	\begin{proof}

		We prove the statement concerning the maximal western expansion of $R'$. The remaining statements can be proved analogously.
		
		Let $s$ be the west side of $R$. If $s$ is disconnected, then $R' = R$ is a maximal western expansion of $R$.
		Otherwise, let $P$ be the union of all segments of leaves of~$\loom^+$ that can be obtained by sliding $s$ to the west along the leaves of $\loom^-$ while keeping it transverse to $\loom^-$. 
		Since periodic points are dense in $\orbit$  by Proposition~\ref{prop:orbit:space}~\eqref{orbit_space:periodic:points:dense}, and every periodic point in $\orbit \bez \widetilde{\Lambda}$ lifts to a periodic point in $\loom$, we may choose  $x \in \mathrm{int}(P)$ which is periodic. 
		By Proposition~\ref{prop:orbit:space} \eqref{orbit_space:one:periodic}, the leaf $\ell^-(x)$ of $\loom^-$ that passes through $x$ must be regular. 
		Let $\ell_W^-(x)$ be the western half-leaf of $\ell^-(x)$ starting at $x$. By Proposition \ref{prop:orbit:space} \eqref{orbit_space:dense:leaves},  for any cusp~$c$ of~$\loom$ and for any neighborhood~$U$ of $c$ there is $g \in \pi_1(M)$ such that $g(\ell_W^-(x))\cap U\neq \emptyset$. Therefore arbitrarily close to the half-leaf $\ell_W^-(x)$ there is a translate of $c$.
		This implies that $R$ cannot be expanded to the west indefinitely and thus~$P$ is bounded. The maximal western expansion $R'$ of $R$ is then the interior of the closure of $R\cup P$.
		
		By Proposition \ref{prop:orbit:space} \eqref{orbit_space:no:accumulation:points}, the cusps of $\loom$ do not accumulate on any point or cusp of~$\loom$. Therefore the west side of $R'$ must be disconnected. Proposition \ref{prop:orbit:space} \eqref{orbit_space:one:periodic} implies that no leaf of $\loom^\pm$ connects two cusps of $\loom$. 
		Thus each side of  $R'$ has at most one cusp and the cusps in the sides of $R'$ are not connected by a leaf of $\loom^\pm$.
	\end{proof}

	Finally, we use Lemmas \ref{lem:condition:1} and \ref{lem:expansions} to prove that $(\loom, \loom^+, \loom^-)$ is a loom space.

	\begin{proposition}\label{prop:loom:space}
		Let $\Lambda$ be a perfect set  for a transitive pseudo-Anosov flow~$\Psi$. Let $(\loom, \loom^+, \loom^-)$ be the universal cover of the orbit space $(\orbit, \orbit^+, \orbit^-)$ of $\Psi$ punctured at~$\widetilde{\Lambda}$. 
		Then   $(\loom, \loom^+, \loom^-)$ is a loom space.
	\end{proposition}
	\begin{proof}
		The space $\loom$ is the universal cover of an open disk $\orbit$ punctured at a closed and discrete set of points. Thus $\loom$ is a non-compact simply connected two-manifold without boundary. Hence it is homeomorphic to a plane \cite[Corollary~1.8]{Epstein-2manifolds}.  Since $\Sing(\Psi) \subset \Lambda$, the foliations $\loom^\pm$ are non-singular. 
		
		The fact that $(\loom, \loom^+, \loom^-)$ satisfies condition (1) of a loom space (Definition~\ref{def:loom}) follows from 
		Lemma \ref{lem:condition:1}.
		
		For a rectangle $R$ in  $(\loom, \loom^+, \loom^-)$, let $W(R)$, $E(R)$, $S(R)$, $N(R)$ be the maximal western, eastern, southern, and northern expansion of $R$, respectively. By Lemma~\ref{lem:expansions}, the rectangle $N(E(S(W(R))))$ is a tetrahedron rectangle that contains $R$. Thus 	$(\loom, \loom^+, \loom^-)$ satisfies condition (2) of Definition \ref{def:loom}.
	\end{proof}
	
	We call the loom space  determined by a transitive pseudo-Anosov flow~$\Psi$ and a perfect set $\Lambda$ for $\Psi$ the \emph{$(\Psi,\Lambda)$-loom space}.

	In \cite[Theorem 4.8]{Fenley_good_geometry} Fenley proved that if $\Psi$ does not have perfect fits then the stabilizers of distinct points of $\orbit$ under the $\Psi$-action of $\pi_1(N)$ intersect trivially. Below we use an analogous argument to establish the corresponding statement for the loom space~$\loom$. More precisely, we prove it for the compact space $\completedLoom$, obtained from $\loom$ by filling in each cusp with a single point; the double overline is used since a single overline denotes the closure of a set within $\loom$, as in Definition \ref{defn:rectangle}.  The non-singular foliations $\loom^\pm$ of $\loom$ extend to foliations $\completedLoomSU$ of $\completedLoom$, where each added point becomes a singularity with infinitely many prongs.  This gives leaves that are not homeomorphic to $\rr$.
	
	The action of $\pi_1(M)$ on $\completedLoom$ is the natural extension of the $(\Psi,\Lambda)$-action on $\loom$, since every element of $\pi_1(M)$ permutes the cusp neighborhoods of $\loom$.
	This action and the $\Psi$-action of $\pi_1(N)$ on $\orbit$ are equivariant with respect to the  branched cover $\completedLoom \rightarrow \orbit$ and the epimorphism $\pi_1(M) \rightarrow \pi_1(N)$ induced by Dehn filling $M$ to $N$. Thus all properties listed in Proposition \ref{prop:orbit:space} \eqref{orbit_space:group_action} hold also  for $\pi_1(N)\curvearrowright \completedLoom$.

	\begin{lemma}\label{lem:loom:stabilizers}
		The stabilizers of distinct points of $\completedLoom$  under the $(\Psi, \Lambda)$-action of $\pi_1(M)$ intersect trivially.
	\end{lemma}
	\begin{proof}

		Suppose that $x,y \in \completedLoom$ are distinct points of $\completedLoom$ that are stabilized by some nontrivial element $g \in \pi_1(M)$. Let $\ell_x^+, \ell_y^+$ be the leaves of $\completedLoomSU$ that pass through $x,y$, respectively. 
		Since, by Proposition \ref{prop:orbit:space} \eqref{orbit_space:one:periodic}, no leaf of $\completedLoom$ contains two periodic points, there must be an orthant $O_{x,y}$ at $x$ that contains $y$. 
		Let $h_x^-$ be the unstable boundary of $O_{x,y}$ . Since~$g$ sends every orthant at $x$ to an orthant at~$x$, and $y$ is fixed by $g$, we must have that $g(O_{x,y}) = O_{x,y}$. Consequently, $g(h_x^-) = h_x^-$. 
		
		Let $U_x^+$  be the union of all leaves of $\completedLoomS$ that intersect the interior of $h_x^-$.  Proposition~\ref{prop:orbit:space}~\eqref{orbit_space:fixed:points}  implies that no leaf of $\completedLoomS$ contained in $U_x^+$ is invariant under $g$. Since $\ell_y^+$ is fixed by $g$, we deduce that $\ell_y^+$ is disjoint from $U_x^+$. It follows that
		there is a boundary component $\delta$ of $U_x^+$ such that either $\delta$ separates $\ell_x^+$ from $\ell_y^+$ or $\delta \subseteq \ell_y^+$. Since $\delta$ must be contained in a leaf of $\completedLoomS$, in either case we get a leaf $\ell^+$ of $\completedLoomS$ that does not intersect~$h_x^-$, yet it is accumulated on by leaves of $\completedLoomS$ that do intersect $h_x^-$.
		Thus there is a perfect fit rectangle in $\completedLoom$ with perfect fit sides contained in $h_x^- \cup \ell^+$. This contradicts Lemma~\ref{lem:only:fake:pf}.
	\end{proof}

	\section{Veering triangulations}\label{sec:veering}
	In this section we define veering triangulations and discuss their basic properties. Agol and Gu\'eritaud were the first to observe that veering triangulations can be constructed from loom spaces (although this terminology did not exist at the time). The drilling algorithm described in Section~\ref{sec:algorithm} relies on a certain modification of the Agol--Gu\'eritaud construction. This modification is discussed in Section~\ref{subsec:finite:parent}. To explain why it works, in Subsection~\ref{subsec:loom:to:veering} we review the original construction in sufficient detail. This material appeared previously in \cite[Section~5]{SchleimSegLoom} and \cite[Section~4]{LMT_flow}.

	\subsection{Ideal triangulations}\label{subsec:ideal}
	
	An \emph{ideal tetrahedron} is a tetrahedron with all vertices removed, called \emph{ideal vertices}. It has six edges, and four two-dimensional faces. For brevity, we refer to the two-dimensional faces simply as `faces'.  Each face of an ideal tetrahedron is an ideal triangle --- that is, a triangle with all vertices removed. If an edge~$e$ of $t$ is embedded in the boundary of a face $f$ of $t$ then we say that $e$ is an edge of~$f$.

	Let $T$ be an arbitrary collection of ideal tetrahedra. Let $f$ and $f'$ be distinct faces of some $t, t' \in T$, respectively, where $t$ and $t'$ may be equal. Any homeomorphism $\phi:f\rightarrow f'$ is called a \emph{face pairing}.  
	We say that a collection $\Phi$ of face pairings is a \emph{complete set of face pairings on $T$} if for every $t \in T$ and  for every face $f$ of $t$ there is exactly one $t' \in T$, exactly one face $f'$ of $t'$, and exactly one $\phi \in \Phi$ such that $\phi$ is a homeomorphism between $f$ and $f'$.  The pair $\T = (T, \Phi)$ is called an \emph{ideal triangulation}. We say that $\T$ is \emph{oriented} if all tetrahedra from $T$ are positively oriented and all face pairings from $\Phi$ are orientation-reversing. Given a three-dimensional topological space $X$ we say that $\T$ is an \emph{ideal triangulation of $X$} if the quotient $T/\Phi$ is homeomorphic to $X$.

	It is customary to identify $\T$ with $T/\Phi$. Thus a  \emph{tetrahedron of $\T$} is the image of some $t \in T$ under the quotient map \mbox{$\Pi_{T,\Phi}:T \rightarrow T/\Phi$}.  
	An \emph{edge} (respectively, a \emph{face}) \emph{of} $\T$ is the image of some edge (respectively, some face) of some $t \in T$ under $\Pi_{T,\Phi}$. Conversely, given a simplex $s$ of~$\T$ a \emph{model} for $s$ is any simplex in the preimage of $s$ under   $\Pi_{T,\Phi}$. 
	Observe that every tetrahedron of $\T$ has precisely one model, and every face of $\T$ has precisely two models. If $T$ is infinite then an edge of~$\T$ can have infinitely many models. However, we will only consider triangulations for which edges have  only finitely many models; this condition is necessary for $\T$ to be a triangulation of a three-manifold.

	In Section \ref{sec:continents} we also consider \emph{subtriangulations} of ideal triangulations. These are determined by a subset of tetrahedra of a given triangulation together with a subcollection of the face pairings between them. 
	
	\begin{definition}\label{def:subtriangulation}
		Let $\T = (T, \Phi)$ be an ideal triangulation. A \emph{subtriangulation} of $\T$ is a pair $(T', \Phi')$ where $T' \subset T$ and $\Phi' \subset \Phi$ is such that a face pairing $\phi \in \Phi$ belongs to $\Phi'$ if and only if $\phi$ identifies faces of some tetrahedra $t, t' \in T'$.
	\end{definition}

	\subsection{Taut triangulations}\label{subsec:taut}
	A \emph{taut structure}  on an ideal tetrahedron $t$ is an assignment  of a  coorientation $\alpha_t(f)$ to each face $f$ of $t$ so that $\alpha_t(\cdot)$ points out of $t$ for exactly two faces of~$t$, and $\alpha_t(\cdot)$ points into $t$ for exactly two faces of $t$. 
	We encode $\alpha_t$ by drawing~$t$ as a quadrilateral with two diagonals — one on top of the other; see Figure~\ref{fig:taut_tet}. The convention is that coorientations on all faces point towards the reader. A \emph{taut tetrahedron} is  an ideal tetrahedron with a taut structure.

	\begin{figure}[h]
		\includegraphics[scale=1]{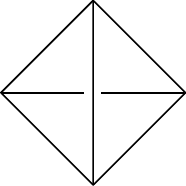}
		\caption{Taut tetrahedron.} 
		\label{fig:taut_tet}
	\end{figure}
	
	We say that a face $f$ of a taut tetrahedron $t$ is a \emph{top face} of $t$ if its coorientation points out of $t$. Otherwise, we say that $f$ is a \emph{bottom face} of $t$. The edge of $t$ which is the common edge of the two top faces (respecetively, two bottom faces) of $t$ is called the \emph{top diagonal} (respectively, the \emph{bottom diagonal}) of $t$. The remaining edges of $t$ are called its \emph{equatorial edges}. 
	
	\begin{definition}\label{def:taut} \cite[Introduction]{Lack_taut}
		Let $\T = (T,\Phi)$ be an oriented ideal triangulation. Suppose that each $t \in T$ is equipped with a taut structure $\alpha_t$. Then $\alpha = \lbrace \alpha_t \rbrace$ is a \emph{taut structure on $\T$} if
		\begin{enumerate}
			\item every $\phi \in \Phi$ identifies a top face of some $t \in T$ with a bottom face of some $t' \in T$, 
			\item for every edge $e$ of $\T$ there is precisely one $t \in T$ such that the top diagonal of~$t$ is a model for $e$, and exactly one $t' \in T$ such that the bottom diagonal of~$t'$  is a model for~$e$.
		\end{enumerate}
	\end{definition}
	A \emph{taut triangulation} is a pair $(\T, \alpha)$, where $\T =  (T, \Phi)$ is an ideal triangulation, and~$\alpha$ is a taut structure on $\T$. We often suppress $\alpha$ from the notation and refer simply to a taut triangulation $\T$, with the understanding that a taut structure on $\T$ has been fixed.
	
	\begin{remark}
		Taut triangulations are sometimes called \emph{transverse taut}; see for instance \cite[Section 2.10]{SchleimSegLinks}. 
	\end{remark}
	
	Condition (1) of Definition \ref{def:taut} ensures that the taut structures on individual tetrahedra match up across face pairings to give a well-defined upwards direction throughout the whole triangulation. 
	Recall that each face $f$ of an ideal triangulation $\T$ has precisely two models, each contained in a model tetrahedron of $\T$. If $\T$ is equipped with a taut structure, we say that a tetrahedron of $\T$ is \emph{immediately below} $f$ if one of the top faces of its model tetrahedron is a model for $f$. Similarly, a tetrahedron is \emph{immediately above}~$f$ if one of the bottom faces of its model tetrahedron is a model for $f$.
	
	Condition (2) of  Definition \ref{def:taut}  ensures that, for each edge $e$ of a taut triangulation $(\T,\alpha)$, exactly one model for $e$ is the top diagonal of a model tetrahedron and exactly one model for $e$ is the bottom diagonal of a model tetrahedron. We say that the corresponding tetrahedra of $\T$ are \emph{immediately below} and \emph{immediately above} $e$, respectively.

	\subsection{Veering triangulations}\label{subsec:veer:tri}
	
	Veering triangulations are taut triangulations whose edges are additionally equipped with a particular type of 2-coloring. We first describe this coloring on a single taut tetrahedron.
	
	\begin{definition} \cite[Definition 5.1]{SchleimSegLinks} \label{def:veering:tet}
		Let $t$ be a taut tetrahedron. We say that $t$ is \emph{veering} if all edges of $t$ are colored either red or blue so that the following two conditions hold.
		\begin{itemize}
			\item If $e_0, e_1 , e_2$ are edges of a top face of $t$, ordered counter-clockwise as viewed from above and so that $e_0$ is the top diagonal of $t$, then $e_1$ is red and $e_2$ is blue.
			\item If $e_0 , e_1 , e_2$ are edges of a bottom face of $t$, ordered counter-clockwise as viewed from above and so that $e_0$ is the bottom diagonal of $t$, then $e_1$ is blue and $e_2$ is red.
		\end{itemize}
	\end{definition}
	In particular, there are four different types of veering tetrahedra; see Figure \ref{fig:tet:rects}. 
	
	A \emph{veering triangulation} is a triple  $\V = (\T, \alpha, \beta)$ such that $(\T, \alpha)$ is a taut triangulation, and $\beta$ is a coloring on edges of $\T$ which makes every tetrahedron of~$\T$ veering. Equivalently, we may say that a veering triangulation is a triangulation composed of taut veering tetrahedra whose faces are paired up via homeomorphisms which preserve the colors of their edges. We also say that $\beta$ is a \emph{veering structure} on~$(\T, \alpha)$. Note that if a taut triangulation admits a veering structure, then this structure is unique \cite[Lemma 2.14]{Parlak-mutations}. Thus $\beta$ could be omitted from the notation. However, since not every taut triangulation admits a veering structure, it is reasonable to keep $\beta$ explicit.
	
	In Construction \ref{constr:loom:to:veering} we will show that certain triangulations are veering. For this purpose, it is helpful to collect all conditions that an ideal triangulation must satisfy to be veering in a single lemma.
	
	\begin{lemma}\label{lem:veering:conditions}
		Let $(T, \Phi)$ be an oriented ideal triangulation. Suppose that the following conditions are satisfied.
		\begin{enumerate}
			\item \label{condition:taut}
			Every $t \in T$ is taut (as defined in Subsection \ref{subsec:taut}). 
			\item \label{condition:veering}
			Every $t \in T$ is veering (in the sense of Definition \ref{def:veering:tet}). 
			\item \label{condition:respecting:coorientation} Every $\phi \in \Phi$ identifies a top face of some $t \in T$ with a bottom face of some $t' \in T$. 
			\item \label{condition:once:top:once:bottom}
			For every edge $e$ of the triangulation $\T = (T, \Phi)$ there is precisely one $t \in T$ such that the top diagonal of $t$ is a model for $e$ and exactly one $t' \in T$ such that the bottom diagonal of $t'$ is a model for $e$. 
			\item \label{condition:respecting:colors} Every $\phi \in \Phi$ preserves the colors of edges.  
		\end{enumerate}
		Then the taut and veering structures on the ideal tetrahedra from $T$ induce a taut and veering structure on $\T = (T, \Phi)$. \qed
	\end{lemma}

	\subsection{A veering triangulation from a loom space}\label{subsec:loom:to:veering}

	Let $(\loom, \loom^+, \loom^-)$ be the loom space determined by a transitive pseudo-Anosov flow $\Psi$ on a closed three-manifold $N$ and a perfect set $\Lambda$ for $\Psi$. Set $M = N\bez \Lambda$. In this subsection we follow \cite[Section~4]{LMT} to explain the construction of a veering triangulation of $M$.

	
	Let $R$ be a tetrahedron rectangle in $(\loom, \loom^+, \loom^-)$. By Definition \ref{def:tet-rectangle}, there are four cusps in the boundary of $R$, one in the interior of each side of $R$. Connecting each cusp in the boundary of $R$ with the remaining cusps in the boundary of $R$ by a line segment contained in $R$ yields a quadrilateral $q(R)$ with two diagonals. To view $q(R)$ as a taut tetrahedron $t(R)$ it suffices to decide which diagonal lies strictly above the other.  
	Our convention is that the diagonal connecting the horizontal sides of $R$ is the top diagonal of $t(R)$ and the diagonal connecting the vertical sides of $R$ is the bottom diagonal of $t(R)$; see Figure~\ref{fig:tet:rects}.


	\begin{figure}[h]
		\begin{center}
			\includegraphics[width=1\textwidth]{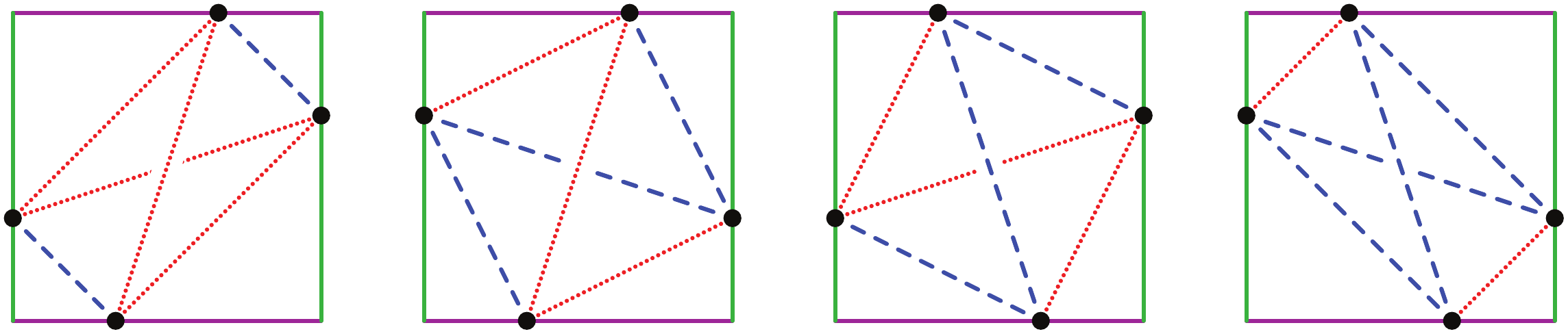} 
		\end{center}
		\caption{Four types of tetrahedron rectangles in a loom space and their associated veering tetrahedra. The black dots indicate the cusps.}
		\label{fig:tet:rects}
	\end{figure}
	
	By Proposition \ref{prop:orbit:space} \eqref{orbit_space:one:periodic}, the cusps that are contained in the horizontal (respectively, vertical) sides of~$R$ are not connected by a leaf of $\loom^+$ (respectively, $\loom^-$). Therefore each edge $e$ of $t(R)$ corresponds to a non-degenerate subrectangle $P_e \subset R$ with exactly two ideal points in the boundary, located at the diagonally opposite corners. Such rectangles are called \emph{edge rectangles}.  We say that the ideal corners of an edge rectangle \emph{span} the rectangle.
	
	
	\begin{definition}\label{def:edge:rect}
		Let $R$ be a rectangle in a bi-foliated surface $(\X, \X^+, \X^-)$. We say that~$R$ is an \emph{edge rectangle} if $R$ has exactly two ideal points in the boundary and they are located at the diagonally opposite corners of $R$. Furthermore, if the foliations $\X^\pm$ are oriented and the cardinal directions in $R$ are fixed according to Convention \ref{convention:cardinal:directions}, we say that $R$ is a \emph{red edge rectangle} if its ideal points are at the south-west and north-east corners. Otherwise, we say that $R$ is a \emph{blue edge rectangle}. 
	\end{definition}
	
	Every tetrahedron rectangle contains exactly six edge rectangles; see Figure \ref{fig:tet:rects}  or a formal proof \cite[Lemma 2.30]{SchleimSegLoom}. There is a natural bijection between the edge rectangles contained in $R$ and the edges of $t(R)$. We color an edge of $t(R)$ red if the corresponding edge rectangle is red, and blue otherwise.
	With that coloring, $t(R)$ is a veering tetrahedron; see Definition \ref{def:veering:tet}. 
	
	To further discuss subrectangles of tetrahedron rectangles, we introduce the following terminology.
	
	\begin{definition}\label{def:taller:etc}
		Let $R, R'$ be rectangles in a bi-foliated surface $(\X, \X^+, \X^-)$. 
		
		\begin{enumerate}
			\item If every leaf of $\X^+$ intersecting $R$ intersects $R'$ we say that $R$ is \emph{thinner} than $R'$  or that $R'$ is \emph{wider} than $R$. 
			If furthermore there are leaves of $\X^+$ intersecting $R'$ that do not intersect $R$, we say that $R$ is \emph{strictly thinner} than $R'$ or that $R'$ is \emph{strictly wider} than $R$. 
			\item If every leaf of $\X^-$ intersecting $R$ intersects $R'$ we say that $R$ is \emph{shorter} than $R'$  or that $R'$ is \emph{taller} than $R$. 
			If furthermore there are leaves of $\X^-$ intersecting $R'$ that do not intersect $R$, we say that $R$ is \emph{strictly shorter} than $R'$ or that $R'$ is \emph{strictly taller} than $R$. 
		\end{enumerate}
	\end{definition}

	Motivated by Definition~\ref{def:taller:etc}, we call the edge rectangle spanned by the two cusps contained in the horizontal (respectively, vertical) sides of $R$ the \emph{tallest edge (sub)rectangle} (respectively, the \emph{widest edge (sub)rectangle}) of $R$. The tallest edge subrectangle of $R$ corresponds to the top diagonal of $t(R)$ and the widest edge subrectangle of $R$ corresponds to the bottom diagonal of~$t(R)$.

	\begin{lemma}\cite[Lemma 5.14]{SchleimSegLoom}\label{lem:edge:rectangle:once:bottom:top}
		Let $P$ be an edge rectangle in a loom space $(\loom, \loom^+, \loom^-)$. Then there is a unique tetrahedron rectangle $R$ in $(\loom, \loom^+, \loom^-)$ such that $P$ is the widest edge subrectangle of $R$, and a unique tetrahedron rectangle $R'$ in $(\loom, \loom^+, \loom^-)$ such that $P$ is the tallest edge subrectangle of $R'$.\qed
	\end{lemma}
	
	One can prove Lemma \ref{lem:edge:rectangle:once:bottom:top} using Lemma \ref{lem:expansions} about maximal expansions of rectangles in $(\loom, \loom^+, \loom^-)$. The maximal vertical expansion of an edge rectangle $P$ is the tetrahedron rectangle whose widest edge rectangle is $P$. The  maximal horizontal expansion of $P$ is the tetrahedron rectangle whose tallest edge rectangle is $P$.

	\begin{definition}\label{def:face:rect}
		Let $R$ be a rectangle in a bi-foliated space $(\X, \X^+, \X^-)$. We say that~$R$ is a \emph{face rectangle} if $R$ has exactly three material corners and every side  of $R$ with two material endpoints contains a cusp in its interior. 
	\end{definition}
	
	Every tetrahedron rectangle contains exactly four face rectangles, and every face rectangle contains exactly three edge rectangles \cite[Lemma~2.30]{SchleimSegLoom}.  Thus there is a natural bijection  between face rectangles contained in $R$ and two-dimensional faces of~$t(R)$. 
	Moreover, precisely two face rectangles contained in~$R$ are strictly thinner than $R$; we call them the \emph{tall face (sub)rectangles} of $R$. They correspond to the top faces of $t(R)$. Similarly, there are precisely two face rectangles contained in~$R$ that are strictly shorter than $R$; we call them the \emph{wide face (sub)rectangles} of $R$. They correspond to the bottom faces of $t(R)$.
	
	\begin{lemma}\label{lem:face:rectangles:in:two:tet}
		Let $Q$ be a face rectangle in a loom space $(\loom, \loom^+, \loom^-)$. Then there are exactly two tetrahedron rectangles in $(\loom, \loom^+, \loom^-)$ which contain~$Q$.  The face rectangle $Q$ is a tall face subrectangle of one of them and a wide face subrectangle of the other.
	\end{lemma}
	\begin{proof}
		Let $c$ be the ideal corner of $Q$. The only sides of $Q$ that do not contain any cusps in their interior are those that are initial segments of cusp leaves that originate from~$c$. 
		Thus $Q$ can be expanded to a larger rectangle in exactly two transverse directions, but not both at the same time.  By Lemma \ref{lem:expansions}, the maximal horizontal expansion of $Q$ is a tetrahedron rectangle having $Q$ as a tall face subrectangle. Similarly, the maximal vertical expansion of $Q$ is a tetrahedron rectangle having $Q$ as a wide face subrectangle.
	\end{proof}

	We say that a rectangle $R$ in bi-foliated surface $(\X, \X^+, \X^-)$ is a \emph{simplex rectangle} if~$R$ is a tetrahedron, face, or an edge rectangle. The \emph{dimension} of a simplex rectangle is the dimension of the simplex that corresponds to it.

	In Construction \ref{constr:loom:to:veering} we outline the connection between loom spaces and veering triangulations, drawing primarily from \cite[Section 4]{LMT} and \cite[Section 5]{SchleimSegLoom}. 
	Note that \cite[Section 4]{LMT_flow} considers only the case of a transitive pseudo-Anosov flow $\Psi$ that does not have perfect fits in the usual sense and $\Lambda=\Sing(\Psi)$. However, the arguments extend readily to the more general setup.

	\begin{construction}[Infinite veering triangulation from a  loom space] \label{constr:loom:to:veering}
		Let $(\loom, \loom^+, \loom^-)$ be a loom space.
		Tetrahedron rectangles in $(\loom, \loom^+, \loom^-)$ determine a collection $\widetilde{T}$ of positively oriented ideal tetrahedra, each of which is taut and veering, as indicated in Figure \ref{fig:tet:rects}. Thus $\widetilde{T}$ satisfies conditions \eqref{condition:taut} and \eqref{condition:veering} of Lemma~\ref{lem:veering:conditions}. 
		
		Let $Q$ be a face rectangle in $(\loom, \loom^+, \loom^-)$. By Lemma~\ref{lem:face:rectangles:in:two:tet}, there is a unique tetrahedron rectangle $R$ in  $(\loom, \loom^+, \loom^-)$ such that $Q$ is a tall face subrectangle of $R$.  The face $f$ of $t(R)$ corresponding to $Q$ is a top face of $t(R)$. 
		Similarly, let  $R'$ be the unique tetrahedron rectangle in  $(\loom, \loom^+, \loom^-)$ such that $Q$ is a wide face subrectangle of $R'$.  Then the face $f'$ of $t(R')$ corresponding to $Q$ must be a bottom face of $t(R)$.
		
		Let $c_1$ be the ideal corner of $Q$, and let $c_2, c_3$ be the remaining cusps of $Q$. For $i=1,2,3$ let $v_i, v_i'$ be the ideal vertices of $f, f'$, respectively, corresponding to the cusps $c_i$. Up to isotopy, there is a unique orientation-reversing homeomorphism $\phi: f \rightarrow f'$ which satisfies condition~\eqref{condition:respecting:coorientation} of Lemma \ref{lem:veering:conditions} and sends to  $v_i$ to $v_i'$ for all $i=1,2,3$.
		Since $f, f'$ come from the same face rectangle $Q$ in $(\loom, \loom^+, \loom^-)$ and $\phi$  respects the cusps of~$Q$, it is clear that $\phi$ must preserve the colors on edges. Thus $\phi$ satisfies condition \eqref{condition:respecting:colors} of Lemma \ref{lem:veering:conditions}. 
		
		Let $\widetilde{\Phi}$ be the collection of face pairings determined by the face rectangles in $(\loom, \loom^+, \loom^-)$ as explained above. Lemma \ref{lem:face:rectangles:in:two:tet} implies that $\widetilde{\Phi}$ is a complete  collection of face pairings on~$\widetilde{T}$.
		It follows that $\widetilde{\T} = (\widetilde{T}, \widetilde{\Phi})$ is an ideal triangulation. 
		
		By Lemma \ref{lem:edge:rectangle:once:bottom:top},  every edge rectangle is the tallest edge rectangle of exactly one tetrahedron rectangle, and the widest edge rectangle of exactly one tetrahedron rectangle. This, combined with the fact that each $\phi \in \widetilde{\Phi}$ satisfies condition \eqref{condition:respecting:coorientation} of Lemma \ref{lem:veering:conditions}, implies that the taut structures on individual tetrahedra induce a taut structure $\widetilde{\alpha}$ on $\widetilde{\T}$.  
		
		Since each $\phi \in \widetilde{\Phi}$ satisfies condition \eqref{condition:respecting:colors} of Lemma \ref{lem:veering:conditions}, the veering structures on individual tetrahedra of $\widetilde{T}$ induce a veering structure $\widetilde{\beta}$ on $(\widetilde{\T}, \widetilde{\alpha})$. Thus $\univV  = (\widetilde{\T}, \widetilde{\alpha}, \widetilde{\beta})$ is a veering triangulation. \hfill$\lozenge$
	\end{construction}

	We call the veering triangulation $\univV  = (\widetilde{\T}, \widetilde{\alpha}, \widetilde{\beta})$ from Construction \ref{constr:loom:to:veering}  the \emph{infinite Agol-Gu\'eritaud triangulation determined by $(\loom, \loom^+, \loom^-)$}. If $\loom = \loom (\Psi, \Lambda)$ for some transitive pseudo-Anosov flow $\Psi$ and a perfect set $\Lambda$ for $\Psi$, we also say that $\univV$ is determined by~$(\Psi, \Lambda)$. 

	
	In Construction \ref{const:finite:veering} we follow \cite[Section 4]{LMT} to describe a finite veering triangulation determined by  the $(\Psi,\Lambda)$-action of $\pi_1(M)$ on the loom space $\loom(\Psi, \Lambda)$.

	\begin{construction}(Finite veering triangulation from the $(\Psi, \Lambda)$-action) \label{const:finite:veering}
		Suppose that~$\Psi$ is a transitive pseudo-Anosov flow on a closed three-manifold $N$ and $\Lambda$ is a perfect set for $\Psi$. 
		Let $\widetilde{\V}(\Psi, \Lambda)$ be the infinite Agol-Gu\'eritaud veering triangulation determined by $(\Psi, \Lambda)$, as described in Construction \ref{constr:loom:to:veering}.
		Set $M = N\bez \Lambda$.	The $(\Psi,\Lambda)$-action of $\pi_1(M)$ on $\loom(\Psi, \Lambda)$ sends simplex rectangles to simplex rectangles of the same dimension. 
		Thus it induces a simplicial action of $\pi_1(M)$ on  $\univV(\Psi, \Lambda)$. Let $\V(\Psi, \Lambda)$ be the quotient of $\univV(\Psi, \Lambda)$ by this action. Then $\V(\Psi, \Lambda)$ is a finite ideal triangulation \cite[Lemma 4.5]{LMT}. Since the combinatorics of simplex rectangles is preserved under the $(\Psi, \Lambda)$-action, $\V(\Psi, \Lambda)$ is taut. It is also veering, because the action preserves the colors of edge rectangles.

		By Lemma \ref{lem:loom:stabilizers}, the stabilizers of distinct cusps of $\loom(\Psi, \Lambda)$ have trivial intersection. Moreover, $\pi_1(M)$ is torsion-free, and hence every simplex of $\univV(\Psi, \Lambda)$ has trivial stabilizer. Since $\univV(\Psi, \Lambda)$ triangulates $\rr^3$ \cite[Theorem 6.50]{SchleimSegLoom}, it follows that the underlying space of $\V(\Psi, \Lambda)$ is a three-manifold with fundamental group isomorphic to $\pi_1(M)$. Finally, the peripheral subgroups of $\pi_1(M)$ act peripherally on $\univV(\Psi, \Lambda)$. This identifies the cusps of $\V(\Psi, \Lambda)$ with those of $M$ and shows that $\V(\Psi, \Lambda)$ triangulates $M$ \cite[Section 4]{LMT}. \hfill$\lozenge$
	\end{construction}

	We call $\V(\Psi, \Lambda)$ from Construction \ref{const:finite:veering} the \emph{finite veering triangulation determined by the $(\Psi,\Lambda)$-action on the $(\Psi, \Lambda)$-loom space}. 
	The key property of $\V(\Psi, \Lambda)$ is that its combinatorics alone is sufficient to reconstruct the flow~$\Psi$. More precisely, Agol and Tsang proved that  $\V(\Psi, \Lambda)$ and the Dehn filling slopes on $M$ that recover $N$ from $M$ determine a transitive pseudo-Anosov flow $\Psi'$ on $N$ \cite[Theorem 5.1]{Tsang-Agol}. Furthermore, Tsang used a result of Barthelm\'e-Frankel-Mann \cite[Theorem 1.1]{BFM} to show that~$\Psi'$ is orbit equivalent to~$\Psi$ via a homeomorphism that is isotopic to identity \cite[Theorem~2.4.5]{tsang-thesis}. For this reason, we say that $\V(\Psi, \Lambda)$ \emph{(combinatorially) encodes}~$\Psi$.  See also \cite{SchleimSeg-flows}. 
	
	For later reference, we collect the above results relating pseudo-Anosov flows and veering triangulations into the following four theorems.

	\begin{theorem}[Agol-Gu\'eritaud]\label{theorem:AG}
		Let $\Psi$ be a transitive pseudo-Anosov flow on a closed oriented three-manifold $N$. Let $\Lambda$ be a perfect set for $\Psi$. Then $M = N\bez \Lambda$ admits a veering triangulation that encodes $\Psi$.\qed
	\end{theorem}
	
	\begin{theorem}[Agol-Tsang]\label{thm:AT}
		Let $\V$ be a finite veering triangulation of an oriented three-manifold~$M$. For almost every collection $s$ of Dehn filling slopes on the boundary components of $M$ the Dehn filled manifold $M(s)$ admits a transitive pseudo-Anosov flow~$\Psi$ for which the core curves of the filling solid tori form a perfect set $\Lambda$ such that  \mbox{$\V = \V(\Psi, \Lambda)$}. 
		\hspace*{\fill}\cite[Theorem 5.1]{Tsang-Agol} 
	\end{theorem}
Although the following result is not explicitly stated in \cite{Tsang-Agol}, it follows directly from the construction.
\begin{theorem}[Agol-Tsang] \label{thm:unique:drilled:flow}
	Let $\V$ be a finite veering triangulation of an oriented three-manifold~$M$. For any admissible collections $s, s'$ of Dehn filling slopes on the boundary components of $M$, the constructed pseudo-Anosov flows $\Psi_s$, $\Psi_{s'}$ on $M(s), M(s')$, respectively, agree on the complements of the core curves of the filling solid tori. \qed
\end{theorem}
	
	\begin{theorem}[Tsang]\label{thm:tri:flow:id}
		Let $\Psi, N$ and $\Lambda$ be as in Theorem \ref{theorem:AG}. Set $\V = \V(\Psi, \Lambda)$. Let $\Psi'$ be the pseudo-Anosov flow built via the Agol-Tsang construction from $\V$ and a collection of Dehn filling slopes that recover $N$ from $N \bez \Lambda$. 
		Then $\Psi'$ is orbit equivalent to $\Psi$ via a homeomorphism that is isotopic to identity. \hfill \cite[Theorem~2.4.5]{tsang-thesis}
	\end{theorem}

	
	\subsection{Loomification} \label{subsec:loomification}
	
	Let $\V$ be a finite veering triangulation of a three-manifold $M$. By Theorem \ref{thm:AT}, there is a transitive pseudo-Anosov flow $\Psi$ on some Dehn filling $N$ of~$M$ and a perfect set $\Lambda$ for $\Psi$ such that \mbox{$\V(\Psi, \Lambda) = \V$}. Set $\loom = \loom(\Psi, \Lambda)$. Fix an orientation on each foliation $\loom^\pm$ and the cardinal directions for each rectangle in $\loom$ according to Convention~\ref{convention:cardinal:directions}. Furthermore, fix a bijection between the ideal vertices of $\univV$ and the cusps of $\loom$ such that whenever two ideal vertices $v, v'$ of $\univV$ are distinct endpoinst of the same edge $e$ of $\univV$, the corresponding cusps span an edge rectangle which gives rise to $e$ under Construction \ref{constr:loom:to:veering}. To avoid overcomplicating the notation, we do not introduce explicit notation for this bijection. However, from now on, it is always understood to be fixed.

	Let $s$ be a simplex of~$\univV$. Set $\loomif(s)$ to be the unique rectangle in $\loom$ that gives rise to~$s$ under Construction~\ref{constr:loom:to:veering}. We call $\loomif(s)$ the \emph{loomification} of $s$. Although we cannot assign concrete dimensions to $\loomif(s)$, as we have not equipped $\loom$ with any metric, the edge colors of $s$ encode the horizontal (from west to east) and vertical (from south to north) orders on the set of cusps of $\loomif(s)$. Thus, $\loomif(s)$ may be viewed as a rectangle that is well defined only up to these orders. In particular, there is a subtle difference between $\loomif(s)$ viewed as a subset of $\loom$ and our model for $\loomif(s)$. This distinction is, however, largely  supressed in the discussion.
	
	For a collection $\mathcal{S}$ of simplices of $\univV$ set
	\[\loomif(\mathcal{S}) = \bigcup\limits_{s \in \mathcal{S}} \loomif(s).\]
	When $\mathcal{S}$ is viewed as a subcomplex of $\univV$ and some of its simplices have common subsimplices, their loomifications overlap. Thus we might be able to determine the relative positions of certain pairs of cusps of $\loomif(\mathcal{S})$ that do not meet the same rectangle. 
	
	In the case of a collection of simplices, the distinction between $\loomif(\mathcal{S})$ and our model for $\loomif(\mathcal{S})$ is even more profound, because the colors of edges of simplices in $\mathcal{S}$ might not be sufficient to determine the relative position of all pairs of cusps of $\loomif(\mathcal{S})$; see Figure~\ref{fig:non:convex} for an example. We will show that for any finite collection $W$ of cusps of $\loomif(\mathcal{S})$, one can always find a finite collection $\mathcal{S}'$ of simplices of $\univV$ that contains $\mathcal{S}$ and such that the relative positions of cusps from $W$ can  be determined using $\loomif(\mathcal{S}')$; see Proposition~\ref{prop:continents:exists}, Lemma \ref{lem:exist:convex:continent} and Proposition~\ref{prop:conv:cont:circ:order}. 
	
	Relative positions of cusps play a central role in the drilling algorithm described in Section~\ref{sec:algorithm} for the following reason. As described in Section \ref{subsec:finite:parent}, to construct the veering parent $\V_\gamma$ we need to find all points of $(\loom, \loom^+, \loom^-)$ corresponding to the lifts of $\gamma$ that lie in a finite collection of tetrahedron rectangles. 
	Given a rectangle $R \subset (\loom, \loom^+, \loom^-)$ and a point $p\in R$, we approximate the position of $p$ by a small rectangle $N(p)$ whose boundary is contained in a collection of cusp leaves. 
	Given another point $q \in R$, we then determine the relative horizontal and vertical positions of $p$ and $q$ by comparing the positions of the cusps from which the cusp leaves bounding $N(p)$ and $N(q)$ originate. Thus the ability to compare the positions of points ultimately rests on our ability to compare the relative positions of cusps; see Section~\ref{sec:algorithm} for details.


	\section{Veering parents}\label{sec:parents}
	
	Let $\Psi$ be a transitive pseudo-Anosov flow on an oriented three-manifold $N$, and let~$\Lambda$ be a perfect set for~$\Psi$. By Theorem \ref{theorem:AG}, there is a veering triangulation~$\V$ on $M = N \bez \Lambda$ that encodes $\Psi$. If $\gamma$ is a closed orbit of $\Psi$ that is not an element of~$\Lambda$ then $\Lambda_\gamma = \Lambda \cup \lbrace \gamma \rbrace$ is also a perfect set for~$\Psi$. Thus, again by Theorem~\ref{theorem:AG}, there is a veering triangulation $\V_\gamma$ of $M_\gamma = N \bez \Lambda_\gamma$ that encodes~$\Psi$. We call $\V_\gamma$ a \emph{veering parent} of $\V$.  
	
	The standard method for building $\V_\gamma$  utilizes the $(\Psi, \Lambda_\gamma)$-action of $\pi_1(M_\gamma)$ on the $(\Psi, \Lambda_\gamma)$-loom space; see Constructions \ref{constr:loom:to:veering} and \ref{const:finite:veering}. We aim to construct~$\V_\gamma$ directly from $\V$ and some combinatorial object encoding $\gamma$ that we will describe in Section~\ref{sec:flow:cycles}. In this section, we focus on one of the ingredients ---  constructing $\V_\gamma$ using the $(\Psi, \Lambda)$-action of $\pi_1(M)$ on the $(\Psi, \Lambda)$-loom space.

	\subsection{Veering parent from the punctured loom space}\label{subsec:punctured:loom:space}

	Set $\loom = \loom(\Psi, \Lambda)$. Recall that $(\loom, \loom^+, \loom^-)$ can be seen as the orbit space of the drilled flow $\Psi \bez \Lambda$ on $M$. 
	Since~$\gamma$ is a closed orbit of $\Psi$ that is not in $\Lambda$, the preimage $\liftGamma$ of $\gamma$ under the universal covering map $\widetilde{M} \rightarrow M$ is a nonempty set of orbits of $\widetilde{\Psi} \bez \widetilde{\Lambda}$.  Identify $\liftGamma$ with the corresponding subset of $\loom$ and set  $\pLoom = \pLoom(\Psi, \Lambda, \gamma) =  \loom \bez \liftGamma$. Denote the restrictions of $\loom^\pm$  to $\pLoom$ by $\pLoom^\pm$. 
	The bi-foliated surface $(\pLoom, \pLoom^+, \pLoom^-)$ is not a loom space, because $\pLoom$ is not a plane. In Lemma \ref{lem:ploom} we show that nonetheless $(\pLoom, \pLoom^+, \pLoom^-)$ retains many key properties of a loom space. Figure~\ref{diagram:bifoliated:spaces} illustrates relationships between various bi-foliated surfaces that we consider in this section.

	\begin{figure}[h]
		\begin{tikzcd}
		\loom(\Psi, \Lambda_\gamma)\arrow[d]& & \\
		\pLoom(\Psi, \Lambda, \gamma) \arrow[d] \arrow[r, hookrightarrow]& \loom(\Psi, \Lambda)\arrow[d]& \\
		\drilledOrbit(\Psi, \Lambda_\gamma)\arrow[r, hookrightarrow] & \drilledOrbit(\Psi, \Lambda) \arrow[r, hookrightarrow]& \orbit(\Psi)
		\end{tikzcd}
		\caption{Relationship between bi-foliated surfaces obtained from the orbit space of a transitive pseudo-Anosov flow $\Psi$, a perfect set $\Lambda$ for $\Psi$, and a periodic orbit $\gamma$ of $\Psi$ that is not in $\Lambda$. All horizontal maps are inclusions. All vertical maps are covering maps. 
		}
		\label{diagram:bifoliated:spaces}
	\end{figure}
	
	The surface $\pLoom$ admits two distinct types of \emph{ideal points}: punctures and cusps. The distinction between the two is such that a puncture admits a neighborhood homeomorphic to a once punctured disk, while a cusp admits a neighborhood homeomorphic to the universal cover of a once punctured disk.

	\begin{lemma}\label{lem:ploom}
		The bi-foliated surface $(\pLoom, \pLoom^+, \pLoom^-)$ satisfies the following properties.
		\begin{enumerate}
			\item The foliations $\pLoom^\pm$ are orientable. \label{ploom:oriented:foliations}
			\item Every rectangle in $\pLoom^\pm$ is contained in a tetrahedron rectangle. \label{ploom:everything:in:tet}
			\item 	Every edge rectangle in $(\pLoom, \pLoom^+, \pLoom^-)$ is the widest edge rectangle of exactly one tetrahedron rectangle in $(\pLoom, \pLoom^+, \pLoom^-)$ and the tallest edge rectangle of exactly one tetrahedron rectangle in $(\pLoom, \pLoom^+, \pLoom^-)$. \label{ploom:edge:rectangles}
			\item Every face rectangle in $(\pLoom,\pLoom^+,\pLoom^-)$ is contained in precisely two distinct tetrahedron rectangles, and is a tall face subrectangle of one of them and a wide face subrectangle of the other. \label{ploom:face:rectangles}
			\item There is an orientation-preserving action of $\pi_1(M)$ on $(\pLoom, \pLoom^+, \pLoom^-)$ that sends simplex rectangles to simplex rectangles of the same dimension and is free on simplex rectangles. \label{ploom:action}
		\end{enumerate}
	\end{lemma}
	\begin{proof}
		Property \eqref{ploom:oriented:foliations} is satisfied because the foliations $\pLoom^\pm$ can be equipped with orientations inherited from $\loom^\pm$.
		
		Let $\loom_\gamma = \loom(\Psi, \Lambda_\gamma)$. Then $\loom_\gamma$ is the universal cover of $\pLoom$. We claim that the covering projection $\pi_\gamma: \loom_\gamma\rightarrow \pLoom$ sends rectangles to rectangles. Since it sends $\loom_\gamma^+$ to $\pLoom^+$ and $\loom_\gamma^-$ to $\pLoom^-$, it suffices to show that $\pi_\gamma$ is injective on rectangles. 
		
		Equip the foliations $\loom_\gamma^\pm$ with the orientation induced from $\pLoom^\pm$. 
		Let $x,y \in \loom_\gamma$ be distinct points contained in a rectangle $R \subset \loom_\gamma$. 
		Then $x,y$ can be connected by an injective path $\alpha:\lbrack 0,1 \rbrack \rightarrow R$ that is contained in the union of the stable leaf through~$x$ and the unstable leaf through~$y$; see Figure \ref{fig:2:points:rectangle}. The path must advance a non‑zero distance along at least one of those leaves because $x\neq y$. Since the image of $\alpha$ lies within two transverse leaves and $\alpha$ is injective, the direction of the path never reverses: if it ever moves west (respectively north, east, south), it never moves east (respectively, south, west, north). The same must be true for $\pi_\gamma \circ \alpha$, because 
		$\pi_\gamma$ sends $\loom_\gamma^\pm$ to $\pLoom^\pm$ respecting the orientations of leaves. Therefore $\pi_\gamma(x) \neq \pi_\gamma(y)$. It follows that for any rectangle embedding $f_R : (0,1)^2 \rightarrow \loom_\gamma$, the composition $\pi_\gamma \circ f_R$ is a rectangle embedding. 
		
		\begin{figure}[h]
			\includegraphics[scale=0.35]{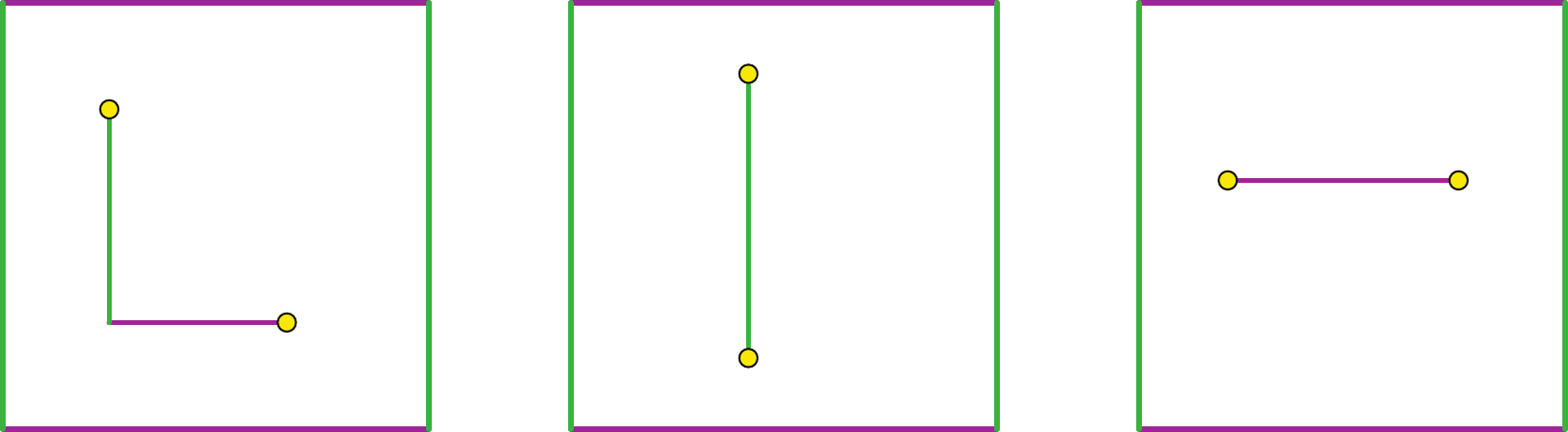}
			\caption{Two distinct points contained in a rectangle can be connected by an injective path contained in the union of one stable leaf and one unstable leaf.}
			\label{fig:2:points:rectangle}
		\end{figure}
		
		The covering projection $\pi_\gamma$ sends cusp neighborhoods in $\loom_\gamma$ to either cusp neighborhoods in $\pLoom$ or punctured disks in $\pLoom$.
		Thus~$\pi_\gamma$ preserves the number and location of ideal points in the boundary of each rectangle.
		With the orientations on $\loom_\gamma^\pm$ induced from $\pLoom^\pm$ and the cardinal directions in each rectangle determined by Convention \ref{convention:cardinal:directions},  the west (east, south, north, respectively) side of a rectangle maps to the west (east, south, north, respectively) side of the rectangle in the image.
		It follows that every rectangle in $(\pLoom, \pLoom^+, \pLoom^-)$ admits a maximal western expansion $R'$ that has at most one ideal point in each side, exactly one ideal point in the west side, and such that no two cusps in the boundary of $R'$ are connected by a leaf segment of $\pLoom^\pm$, because --- by~Lemma \ref{lem:expansions} --- such property holds in $(\loom_\gamma, \loom_\gamma^+, \loom_\gamma^-)$. Analogous statements hold for eastern, southern, and northern expansions. This together with Proposition \ref{prop:orbit:space} \eqref{orbit_space:no:accumulation:points} imply properties \eqref{ploom:everything:in:tet}, \eqref{ploom:edge:rectangles}, and \eqref{ploom:face:rectangles}; see Lemmas \ref{lem:edge:rectangle:once:bottom:top}, \ref{lem:face:rectangles:in:two:tet} and Proposition \ref{prop:loom:space} where these properties are proved for loom spaces.
		
		Recall that $\pLoom$ is obtained from the loom space $\loom = \loom(\Psi, \Lambda)$ by deleting an orbit of a periodic point under the $(\Psi, \Lambda)$-action of $\pi_1(M)$. Thus the $(\Psi, \Lambda)$-action of $\pi_1(M)$ on $(\loom, \loom^+, \loom^-)$ restricts to an action of $\pi_1(M)$ on $(\pLoom, \pLoom^+, \pLoom^-)$ which sends simplex rectangles to simplex rectangles of the same dimension. By Lemma \ref{lem:loom:stabilizers}, the stabilizers of distinct ideal points of $\pLoom$ have trivial intersection. Since moreover $\pi_1(M)$ is torsion-free, \eqref{ploom:action} follows.
	\end{proof}

	Parts \eqref{ploom:oriented:foliations} -- \eqref{ploom:face:rectangles}  of Lemma \ref{lem:ploom}  imply that  Construction \ref{constr:loom:to:veering}, which produces an infinite veering triangulation of $\rr^3$ from a loom space, can be generalized to produce an infinite veering triangulation from $(\pLoom, \pLoom^+, \pLoom^-)$. We denote this triangulation by $\widehat{\V}_\gamma$. 
	By Lemma~\ref{lem:ploom}~\eqref{ploom:action}, the group $\pi_1(M)$ acts simplicially on $\widehat{\V}_\gamma$. We now show that the resulting quotient triangulation is the veering parent $\V_\gamma$.

	\begin{proposition}\label{prop:parent:from:intermediate:cover}
		The quotient of $\widehat{\V}_\gamma$ by the action induced by the $(\Psi, \Lambda)$-action of $\pi_1(M)$ on $(\pLoom, \pLoom^+, \pLoom^-)$ is the veering parent $\V_\gamma$.
	\end{proposition}
	\begin{proof}
		As before, set $\loom_\gamma = \loom(\Psi, \Lambda_\gamma)$.  Let $\widetilde{\V}_\gamma$ be the veering triangulation determined by the action of $\pi_1(M_\gamma)$ on $\loom_\gamma$ as described in  Construction \ref{constr:loom:to:veering}. The veering parent $\V_\gamma$ is defined as the quotient of $\widetilde{\V}_\gamma$  by the action of $\pi_1(M_\gamma)$.
		
		Let $\delta_\gamma: \pi_1 (M_\gamma) \rightarrow \pi_1(M)$ be the epimorphism induced by Dehn filling $M_\gamma$ to~$M$. Observe that $\pLoom = \loom_\gamma/ \ker \delta_\gamma$ and thus
		$\widehat{\V}_\gamma = \widetilde{\V}_\gamma/\ker \delta_\gamma$. Since $\pi_1(M) \cong \pi_1(M_\gamma)/ \ker \delta_\gamma$, we get an identification between $\V_\gamma = \widetilde{\V}_\gamma/\pi_1(M_\gamma)$ and the iterated quotient $ \left(\widetilde{\V}_\gamma/\ker \delta_\gamma \right)\Big/\Big(\pi_1(M_\gamma)/\ker \delta_\gamma\Big)  = \widehat{\V}_\gamma/\pi_1(M)$. \end{proof}

	\subsection{Veering parent from punctured tetrahedron rectangles}\label{subsec:finite:parent}

	In Proposition~\ref{prop:parent:from:intermediate:cover} we showed that the veering parent $\V_\gamma$ can be constructed using the $(\Psi, \Lambda)$-action of $\pi_1(M)$ on the punctured loom space $(\pLoom, \pLoom^+, \pLoom^-)$ rather than the $(\Psi, \Lambda_\gamma)$-action of $\pi_1(M_\gamma)$ on the loom space $(\loom_\gamma, \loom_\gamma^+, \loomGamma^-)$. 
	In this section, we use the fact that a fundamental domain for the action of $\pi_1(M)$ on $\univV$ contains only finitely many tetrahedra to reduce this construction to one involving only finitely many tetrahedron rectangles in $(\loom,\loom^+,\loom^-)$.

	Let $R$ be a rectangle in $(\pLoom, \pLoom^+, \pLoom^-)$. We denote the image of $R$ under the inclusion $(\pLoom, \pLoom^+, \pLoom^-) \hookrightarrow (\loom, \loom^+, \loom^-)$ by $R^\sharp$.

	\begin{lemma}\label{lem:everything:in:tet:rect}
		Let $R$ be a rectangle in $(\pLoom, \pLoom^+, \pLoom^-)$. There is a tetrahedron rectangle~$Q$ in $(\loom, \loom^+, \loom^-)$ such that $\Rsharp \subset Q$. 
	\end{lemma}
	\begin{proof}
		The image $\Rsharp$ of $R$ under the inclusion $(\pLoom, \pLoom^+, \pLoom^-) \hookrightarrow (\loom, \loom^+, \loom^-)$ is clearly a rectangle. Since  $(\loom, \loom^+, \loom^-)$ is a loom space, $\Rsharp$ must be contained in a tetrahedron rectangle by condition (2) of Definition \ref{def:loom}.
	\end{proof}
	
	Since images of simplex rectangles are especially  important, we introduce a special term for them.
	
	\begin{definition}
		A $\gamma$-\emph{simplex} ($\gamma$-\emph{tetrahedron}, $\gamma$-\emph{face}, $\gamma$-\emph{edge}, respectively) \emph{rectangle} in $(\loom, \loom^+, \loom^-)$ is the image of a simplex (tetrahedron, face, edge, respectively) rectangle in $(\pLoom, \pLoom^+, \pLoom^-)$ under the inclusion $(\pLoom, \pLoom^+, \pLoom^-) \hookrightarrow (\loom, \loom^+, \loom^-)$. 
	\end{definition}
	
	Note that a $\gamma$-simplex rectangle might not be a simplex rectangle; see Figure \ref{fig:subrectangles}.

	Recall from Construction \ref{const:finite:veering} that the tetrahedra of $\V$ are in bijection with the orbits of tetrahedron rectangles in $(\loom,\loom^+,\loom^-)$ under the $(\Psi,\Lambda)$-action of $\pi_1(M)$. Choose one representative from each such orbit and call these  tetrahedron rectangles \emph{fundamental}. This terminology is motivated by the fact that fundamental tetrahedron rectangles bijectively correspond to the tetrahedra of a triangulated fundamental domain for the action of $\pi_1(M)$ on $\widetilde{\V}$. In particular, there are only finitely many of them.
	
	Once a collection of fundamental tetrahedron rectangles is fixed, we say that a $\gamma$-simplex rectangle in $(\loom, \loom^+, \loom^-)$ is \emph{fundamental} if it is contained in a fundamental tetrahedron rectangle. 
	In Proposition \ref{prop:finite:veering:parent} we show that to construct the veering parent $\V_\gamma$ it suffices to know which fundamental $\gamma$-simplex rectangles are translates of one another under the $(\Psi, \Lambda)$-action of $\pi_1(M)$ on $(\loom, \loom^+, \loom^-)$. To simplify the statement of the proposition, we introduce the following terminology.
	
	\begin{definition} \label{def:pi-equivalence}
		Let $Q, Q'$ be rectangles in either $(\pLoom, \pLoom^+, \pLoom^-)$  or $(\loom, \loom^+, \loom^-)$. We say that $Q,Q'$ are \emph{$\pi_1(M)$-equivalent} if there is $g \in \pi_1(M)$ such that  $g(Q) = Q'$ under the $(\Psi,\Lambda)$-action of $\pi_1(M)$. Otherwise we say that $Q, Q'$ are \emph{$\pi_1(M)$-inequivalent}. 
	\end{definition}
	
	For instance, the set fundamental tetrahedron rectangles  can be described as a maximal collection of $\pi_1(M)$-inequivalent tetrahedron rectangles in $(\loom, \loom^+, \loom^-)$. 
	
	Given a $\gamma$-tetrahedron rectangle $R_\gamma$ and a $\gamma$-face rectangle $Q_\gamma\subset R_\gamma$ we call the pair $(R_\gamma, Q_\gamma)$ a \emph{$\gamma$-(tetrahedron, face) pair}. If $R_\gamma$ is fundamental we say that $(R_\gamma, Q_\gamma)$ is fundamental.  Furthermore, we say that two $\gamma$-(tetrahedron, face) pairs are \emph{solely face equivalent} if $Q_\gamma$ and~$Q_\gamma'$ are $\pi_1(M)$-equivalent  and moreover $Q_\gamma$ is a tall $\gamma$-face subrectangle of $R_\gamma$ if and only if $Q_\gamma'$ is a wide $\gamma$-face subrectangle of $R_\gamma'$. The terminology refers to the fact that, under the last condition, the $\pi_1(M)$-equivalence between $Q_\gamma$ and $Q_\gamma'$ does not extend to a $\pi_1(M)$-equivalence between $R_\gamma$ and~$R_\gamma'$ (although these rectangles can nonetheless be $\pi_1(M)$-equivalent).
	
	By Construction \ref{constr:loom:to:veering} and Proposition~\ref{prop:parent:from:intermediate:cover}, for every face pairing $\phi$ of $\V_\gamma$ one can find a pair of solely face equivalent $\gamma$-(tetrahedron, face) pairs that induce~$\phi$. Conversely, every pair of solely face equivalent $\gamma$-(tetrahedron, face) pairs induce some face pairing of $\V_\gamma$.

	\begin{proposition}\label{prop:finite:veering:parent}
	Tetrahedra of the veering parent $\V_\gamma$ are in a bijective correspondence with the $\pi_1(M)$-equivalence classes of fundamental $\gamma$-tetrahedron rectangles in $(\loom, \loom^+, \loom^-)$. Every face pairing of $\V_\gamma$ can be represented by a pair of solely face equivalent fundamental $\gamma$-(tetrahedron, face) pairs, and every such pair represents a face pairing of~$\V_\gamma$.
	\end{proposition}
	\begin{proof}
	Let $t$ be a tetrahedron of $\V_\gamma$. 
	By Proposition \ref{prop:parent:from:intermediate:cover}, there is a $\pi_1(M)$-equivalence class of $\gamma$-tetrahedron rectangles in $(\loom, \loom^+, \loom^-)$ that corresponds to $t$. Let $R_\gamma$ be any $\gamma$-tetrahedron rectangle in that class. Lemma \ref{lem:everything:in:tet:rect} implies that $R_\gamma$ is contained in a tetrahedron rectangle $R$. By the definition of fundamental tetrahedron rectangles, there is $g \in \pi_1(M)$ such that $g(R)$ is fundamental. Then $g(R_\gamma)$ is a fundamental $\gamma$-tetrahedron rectangle that represents the $\pi_1(M)$-equivalence class corresponding to $t$. 
	
	Conversely, a $\pi_1(M)$-equivalence class of fundamental $\gamma$-tetrahedron rectangles is contained in a $\pi_1(M)$-equivalence class of all $\gamma$-tetrahedron rectangles in $(\loom, \loom^+, \loom^-)$. Thus it determines one tetrahedron of $\V_\gamma$.
	
	As explained immediately before the statement of the proposition, every face pairing of $\V_\gamma$ can be represented by a pair $\lbrace (R_\gamma, Q_\gamma), (R_\gamma', Q_\gamma')\rbrace$ of solely face equivalent \mbox{$\gamma$-(tetrahedron, face) }pairs. There are $g, g' \in \pi_1(M)$ such that $(g(R_\gamma), g(Q_\gamma))$ and $(g'(R_\gamma'), g'(Q_\gamma'))$ are fundamental. Since the action of $\pi_1(M)$ on $(\loom,\loom^+,\loom^-)$ preserves the relevant combinatorial structure, in particular $\liftGamma$ and the distinction between tall and wide $\gamma$-face subrectangles,  $(g(R_\gamma), g(Q_\gamma))$ and  $(g'(R_\gamma'), g'(Q_\gamma'))$ are solely face equivalent. Thus the face pairing induced by $\lbrace (R_\gamma, Q_\gamma), (R_\gamma', Q_\gamma')\rbrace$ can be represented by a pair  $\lbrace (g(R_\gamma), g(Q_\gamma)), (g'(R_\gamma'), g'(Q_\gamma'))\rbrace$  of solely face equivalent fundamental $\gamma$-(tetrahedron, face) pairs.
	
	The fact that every pair of solely face equivalent $\gamma$-(tetrahedron, face) pairs, fundamental or not, induces a face pairing of $\V_\gamma$ follows from Construction \ref{constr:loom:to:veering} and Proposition~\ref{prop:parent:from:intermediate:cover}.
	\end{proof}
	
	We refer the reader to Appendix \ref{appendix:examples} for several examples of veering parents. Each example is presented in a way consistent with the description of the veering parent given in Proposition \ref{prop:finite:veering:parent}. Specifically, we list the \mbox{$\pi_1(M)$-equivalence} classes of fundamental \mbox{$\gamma$-tetrahedron} rectangles, representing the tetrahedra of the veering parent, and pairs of solely face equivalent $\gamma$-(tetrahedron, face) pairs, representing its face pairings.
	
	\subsection{Determining $\pi_1(M)$-equivalence classes of $\gamma$-simplex rectangles}\label{subsec:equiv:classes}
	
By Proposition \ref{prop:finite:veering:parent}, to construct the veering parent $\V_\gamma$ it suffices to:
	\begin{enumerate}
		\item Fix a collection $\FTRect$ of fundamental tetrahedron rectangles in $(\loom,\loom^+,\loom^-)$.
		\item For each $R \in \FTRect$, find all points of $\liftGamma$ contained in $R$ and thus all fundamental $\gamma$-tetrahedron rectangles.
		\item Determine $\pi_1(M)$-equivalence classes of fundamental $\gamma$-tetrahedron rectangles and their $\gamma$-face subrectangles.
	\end{enumerate}
	
	In the drilling algorithm, each of these steps must be carried out using only the combinatorics of $\V$. The most challenging part is determining $R\cap\liftGamma$ for each $R\in\FTRect$. We discuss and develop the  tools needed for this in Sections \ref{sec:flow:cycles}, \ref{sec:continents}, \ref{sec:orders}, and \ref{sec:adapted:set}. However, already at this stage  we can explain how to determine the $\pi_1(M)$-equivalence classes of fundamental $\gamma$-simplex rectangles, assuming that $R\cap\liftGamma$ is known for every $R\in\FTRect$. 
	
	First, it follows directly from Construction \ref{const:finite:veering} that the action of $\pi_1(M)$ on simplex rectangles in $(\loom, \loom^+, \loom^-)$ can be recovered from face pairings of $\V$. More precisely, we can make the following observation. 
	
	\begin{observation}\label{obs:identification:equivalences}
		Let $R, R' \in\FTRect$. Let $t, t'$ be the tetrahedra of $\V$ admitting elevations to $\univV$ that correspond to $R$ and $R'$, respectively, via Construction~\ref{constr:loom:to:veering}.
		Then simplex subrectangles $Q \subset R$, $Q' \subset R'$ are $\pi_1(M)$-equivalent if and only if the corresponding subsimplices of $t, t'$ are identified in $\V$. 
	\end{observation}
	
	Determining if two $\gamma$-simplex rectangles contained in some elements of $\FTRect$ are $\pi_1(M)$-equivalent is a bit more involved. We give a criterion for this in Lemma \ref{lem:equivalent:rectangles}, after introducing necessary notation and terminology.

	Let $R$ be a simplex or a $\gamma$-simplex rectangle in  $(\loom, \loom^+, \loom^-)$. Fix the cardinal directions in $R$ according to Convention \ref{convention:cardinal:directions}. 
	Recall that $\liftGamma$ is the set of points in $\loom$ that are the projections of the preimage of $\gamma$ under the covering map $\widetilde{M}\rightarrow M$.
	Let $\Ideal(R)$ be the set of cusps and punctures in the boundary of $R$ and the punctures of $R \bez \liftGamma$. The set $\Ideal(R)$ is nonempty, because $R$ must have at least two ideal points in the boundary. By Proposition~\ref{prop:orbit:space}~\eqref{orbit_space:no:accumulation:points}, $\liftGamma$ is discrete and does not accumulate on the cusps of  $(\loom, \loom^+, \loom^-)$. Thus $\Ideal(R)$ must be finite. 
	Suppose that it has $k$ elements. 
	Since, by Proposition \ref{prop:orbit:space}~\eqref{orbit_space:one:periodic}, no two elements of $\Ideal(R)$ lie on the same leaf of $\loom^+$, there is a unique way to label the elements of $\Ideal(R)$ by elements of $\lbrace 1,2,\ldots, k \rbrace$ such that the label of $x \in \Ideal(R)$ is smaller than the label of $y \in \Ideal(R)$ if and only if $x$ is strictly to the west of~$y$. We call this labelling the \emph{west-east labelling on $\Ideal(R)$}. The \emph{east–west labelling} on $\Ideal(R)$ can be defined analogously, by labelling the elements of $\Ideal(R)$ by elements of $\lbrace 1,2,\ldots, k \rbrace$ from the eastmost to the westmost.


	Let $Q_\gamma$ be a $\gamma$-simplex rectangle, and let $R$ be any simplex rectangle in $(\loom, \loom^+, \loom^-)$ that contains $Q_\gamma$. Then $\Ideal(Q_\gamma) \subset \Ideal(R)$.   Set $\eta^\mathrm{WE}(R, Q_\gamma)$ to be the ordered tuple of labels of the elements of $\Ideal(Q_\gamma)$, where the label of each element is  inherited from the west-east labelling on $\Ideal(R)$, arranged in increasing order. We call $\eta^{\mathrm{WE}}(R,Q_\gamma)$ the \emph{west-east signature of $Q_\gamma$ in~$R$}; see Figure \ref{fig:subrectangles} for an example. 
	 We analogously define the \emph{east-west signature} $\eta^{\mathrm{EW}}(R,Q_\gamma)$ \emph{of $Q_\gamma$ in $R$}, for which the labels are determined by the east-west labelling on $\Ideal(R)$. 
	
	\begin{figure}[h]
		\centering
		\labellist
		\hair 2pt \small
		\pinlabel $1$ [r] at 2 267
		\pinlabel $2$ [b] at 142 372
		\pinlabel $3$ [tr] at 174 144
		\pinlabel $4$ [t] at 249 1
		\pinlabel $5$ [br] at 283 236
		\pinlabel $6$ [l] at 372 114
		
		\endlabellist
		\includegraphics[width = 0.25\textwidth]{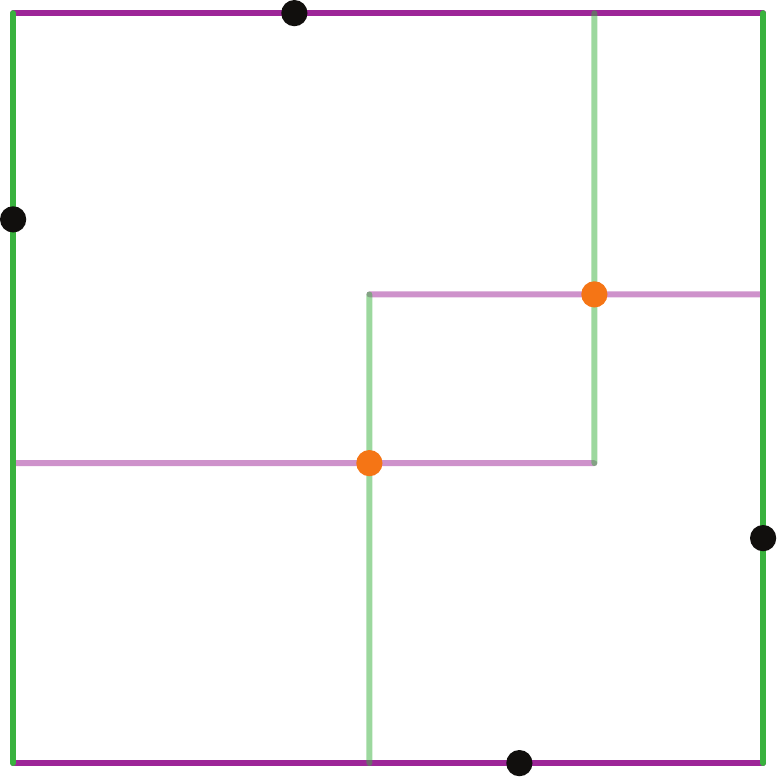}
		\caption{A tetrahedron rectangle $R$ in  $(\loom, \loom^+, \loom^-)$ with the west-east labelling on $\Ideal(R)$. The cusps of $R$ are indicated by black dots, and the elements of $\liftGamma$ by orange dots. The tetrahedron $R$ contains two $\gamma$-tetrahedron rectangles with west-east signatures $(1,2,3,5)$, $(3,4,5,6)$, eight $\gamma$-face rectangles with west-east signatures $(1,2,3)$, $(1,2,5)$, $(1,3,5)$, $(2,3,5)$, $(3,4,5)$, $(3,4,6)$, $(3,5,6)$, $(4,5,6)$, and eleven $\gamma$-edge rectangles with west-east signatures $(1,2)$, $(1,3)$, $(1,5)$, $(2,3)$, $(2,5)$, $(3,5)$, $(3,4)$, $(3,6)$, $(4,5)$, $(4,6)$, $(5,6)$.}
		\label{fig:subrectangles}
	\end{figure}
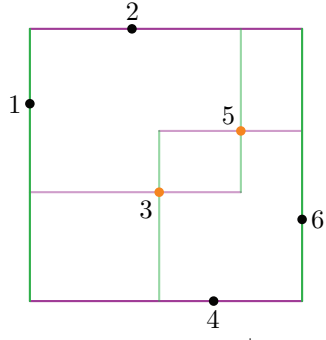

	\begin{definition}
		Fix the cardinal directions in all rectangles in $(\loom, \loom^+, \loom^-)$ according to Convention \ref{convention:cardinal:directions}. 
		We say that $g \in \pi_1(M)$ \emph{acts with a twist} on $(\loom, \loom^+, \loom^-)$ if there is a rectangle $R$ in $(\loom, \loom^+, \loom^-)$ such that $g$ sends the west side of $R$ to the east side of $g(R)$. 
	\end{definition}
	
	For instance,  if the conjugacy class of $g\in\pi_1(M)$ corresponds to the free homotopy class of a twisted orbit of $\Psi$, then $g$ acts with a twist on $(\loom,\loom^+,\loom^-)$.


	\begin{lemma}\label{lem:equivalent:rectangles}
		Let $Q_\gamma, Q_\gamma'$ be $\gamma$-simplex rectangles in $(\loom, \loom^+, \loom^-)$. Then $Q_\gamma, Q_\gamma'$ are $\pi_1(M)$-equivalent if and only if they are contained in simplex rectangles $R, R'$ respectively, such that  with $R' = g(R)$ for some $g \in \pi_1(M)$ and 
		\begin{equation}\label{eqn:signatures}
		\eta^{\mathrm{WE}}(R', Q_\gamma') = \begin{cases}
		\eta^{\mathrm{WE}} (R,Q_\gamma)& \text{if $g$ does not act with a twist on } (\loom, \loom^+, \loom^-)\\
		\eta^{\mathrm{EW}}(R,Q_\gamma) & \text{otherwise.}
		\end{cases}
		\end{equation}
	\end{lemma}
	\begin{proof}
	 Let \mbox{$g \in \pi_1(M)$}. Recall that $g$ sends cusps of $\loom$ to cusps of $\loom$, elements of $\liftGamma$ to elements of $\liftGamma$, and preserves each of the foliations $\loom^\pm$. Therefore for every simplex rectangle $R$ and every $\gamma$-simplex subrectangle $Q_\gamma \subset R$ we have 
	 \begin{equation}\label{eqn:signatures:general}
			\eta^{\mathrm{WE}}(g(R), g(Q_\gamma)) = \begin{cases}
			\eta^{\mathrm{WE}} (R,Q_\gamma)& \text{if $g$ does not act with a twist on } (\loom, \loom^+, \loom^-)\\
			\eta^{\mathrm{EW}}(R,Q_\gamma) & \text{otherwise.}
			\end{cases}
	\end{equation}
	Moreover, two $\gamma$-simplex subrectangles of the same simplex rectangle are equal whenever they have the same west-east (and hence east-west) signature.
	It follows that if $Q_\gamma, Q_\gamma'$ are $\gamma$-simplex rectangles contained in simplex rectangles $R, R'$, respectively, such that $R =g(R')$ and moreover \eqref{eqn:signatures} holds, then $Q_\gamma' = g(Q_\gamma)$. Hence $Q_\gamma, Q_\gamma'$ are \mbox{$\pi_1(M)$-equivalent}.
	
	For the converse, suppose that $Q_\gamma'=g(Q_\gamma)$ for some $g\in\pi_1(M)$. By Lemma \ref{lem:everything:in:tet:rect}, there is a tetrahedron rectangle $R$ containing $Q_\gamma$. Then $R' = g(R)$ is a tetrahedron rectangle containing $Q_\gamma'$. Thus equality \eqref{eqn:signatures:general}  gives \eqref{eqn:signatures}.
	\end{proof}
	
	Observation \ref{obs:identification:equivalences} and Lemma \ref{lem:equivalent:rectangles} provide a practical way to determine the \linebreak \mbox{$\pi_1(M)$-equivalence} classes of the fundamental $\gamma$-simplex rectangles. Let $Q_\gamma$ and $Q_\gamma'$ be two such rectangles. By Lemma \ref{lem:equivalent:rectangles}, $Q_\gamma, Q_\gamma'$ may be $\pi_1(M)$-equivalent only if they are contained in $\pi_1(M)$-equivalent simplex rectangles $R$, $R'$, respectively. Observation~\ref{obs:identification:equivalences} allows us to determine whether such simplex rectangles exist. If they do not, then $Q_\gamma$ and $Q_\gamma'$ are not $\pi_1(M)$-equivalent. If they do, it remains to compare the west-east signatures of $Q_\gamma$ in $R$ and of $Q_\gamma'$ in $R'$. Specifically, $Q_\gamma$ and $Q_\gamma'$ are $\pi_1(M)$-equivalent if and only if their signatures in $R$ and $R'$, respectively, are related as prescribed by Lemma~\ref{lem:equivalent:rectangles}. Although Lemma \ref{lem:equivalent:rectangles} refers also to the east-west signature, observe that if
	$\eta^{\mathrm{WE}}(R,Q)=(i_1,i_2,\ldots,i_l)$
	for some integers $2\leq l\leq k$ and $1\leq i_1<i_2<\cdots<i_l\leq k$, then
	\[
	\eta^{\mathrm{EW}}(R,Q)
	=
	(k+1-i_l,k+1-i_{l-1},\ldots,k+1-i_1).
	\]
	Thus, whether \eqref{eqn:signatures} holds can be determined using the west-east signatures  alone.

	\section{Combinatorial representation of orbits} \label{sec:flow:cycles}
	
	In this section, as in the previous sections, $\Psi$ is a transitive pseudo-Anosov flow on a closed three-manifold~$N$, $\Lambda$ is a perfect set for $\Psi$, and $\V$ is the veering triangulation of $M = N\bez \Lambda$ determined by the pair $(\Psi,\Lambda)$. A closed orbit $\gamma$ of $\Psi$ that is not in $\Lambda$ determines a closed and discrete set of points $\liftGamma$ in the loom space $(\loom, \loom^+, \loom^-)$ determined by $(\Psi, \Lambda)$. 
	By Proposition \ref{prop:finite:veering:parent}, to construct the veering parent $\V_\gamma$ of $\V$ one needs to be able to locate the points of $\liftGamma$ inside (finitely many) fundamental tetrahedron rectangles in $(\loom, \loom^+, \loom^-)$.
	In this section, we follow Landry-Minsky-Taylor \cite[Section 6]{LMT_flow} to explain how to find  these points using only the combinatorics of $\V$.

	\subsection{Flow cycles and flow lines}\label{subsec:flow:cycles:and:lines}

	As in Section~\ref{subsec:taut}, we say that a tetrahedron~$t$ of a taut triangulation is immediately above an edge $e$ of that triangulation if $e$ is the bottom diagonal of $t$. 	We use this terminology to recall the notion of \emph{flow cycles} of $\V$, introduced in \cite[Section 4.3]{LMT}.

	\begin{definition}\label{defn:flow:cycle} \cite[Section 2]{LMT_flow}
		A sequence $\zeta = (t_1, t_2, \ldots, t_k)$ of tetrahedra in $\V$ is a \emph{flow cycle} if for each $i \in \{1,2, \ldots, k\}$, the tetrahedron $t_{i+1}$  is either immediately above the top diagonal of $t_i$ or immediately above one of the equatorial edges of $t_i$ that have  different color than the top diagonal of $t_i$ (the indices are taken modulo~$k$).
		The integer $k$ is called the \emph{length} of $\zeta$. We say that $\zeta$ is \emph{primitive} if there is no integer $0<l < k$ such that $t_i = t_{i+l}$ for every $i \in \lbrace 1,2, \ldots, k\rbrace$.
	\end{definition}

	\begin{definition}\label{def:flow:line}
		A bi-infinite sequence $(t_i)_{i\in \zz}$ of tetrahedra of $\univV$ is a \emph{flow line} if for every $i \in \zz$  the tetrahedron $t_{i+1}$  is either immediately above the top diagonal of $t_i$ or immediately above one of the equatorial edges of $t_i$ that has a different color than the top diagonal of $t_i$. 
	\end{definition}

	Given a flow cycle $\zeta$, an \emph{elevation} of $\zeta$ is  a flow line which projects to $\zeta$ under the covering map $\univV \rightarrow \V$.  Since $\zeta$ may pass through the same tetrahedron multiple times, this terminology differs slightly from the usual notion of an elevation of an object understood as a connected component of its preimage.
	
	Given a flow cycle $\zeta = (t_1, t_2, \ldots, t_k)$, let $\zeta_i$ denote a simple, unknotted path in $t_i$ connecting the midpoint of the bottom diagonal of $t_i$ to the midpoint of the bottom diagonal of $t_{i+1}$ (indices taken modulo~$k$). The concatenation $\zeta_1 \zeta_2 \cdots \zeta_k$ is a loop in $M$ and hence determines a conjugacy class $\lbrack \zeta \rbrack$ in $\pi_1(M)$. For each elevation $\eleZeta$ of $\zeta$ there is an element $g_{\eleZeta} \in \lbrack \zeta \rbrack $ that acts by translation on $\eleZeta$. We say that  $g_{\eleZeta}$ \emph{stabilizes}~$\eleZeta$.

	\begin{proposition}\label{prop:nontrivial}
		For every elevation $\eleZeta$ of a flow cycle $\zeta$, the element $g_{\eleZeta} \in \pi_1(M)$ is nontrivial.
	\end{proposition}
	\begin{proof}
		The concatenation $\zeta_1\zeta_2\cdots\zeta_k$ can be isotoped so that for each tetrahedron $t_i$, where  $i=1,2,\ldots, k$, the path $\zeta_i$ links with the union of the equatorial edges of $t_i$. Thus $\zeta_1\zeta_2\cdots\zeta_k$ is vertical with respect to the two-skeleton of $\V$; see \cite[Definition 2.1]{SchleimSeg-essential}. The statement follows from the fact that every vertical loop is homotopically nontrivial \cite[Theorem~3.2]{SchleimSeg-essential}.
	\end{proof}
	
	More generally, we say that $g \in \pi_1(M)$ \emph{stabilizes} a flow line $\rho = (t_i)_{i \in \mathbb{Z}}$ if there exists $k \in \mathbb{N}$ such that $g(t_i) = t_{i+k}$ for all $i \in \mathbb{Z}$. We denote the maximal subgroup of $\pi_1(M)$ that stabilizes $\rho$ by $\Stab(\rho)$. If $\Stab(\rho)$ is nontrivial we say that $\rho$ is \emph{periodic}. Clearly, this can happen only if $\rho$ is an elevation of a flow cycle.

	Construction~\ref{constr:loom:to:veering} establishes a bijective correspondence between simplices of $\univV$ and simplex rectangles in $(\loom, \loom^+, \loom^-)$. See also Section \ref{subsec:loomification}, where we define the loomification. This allows us to view flow lines as sequences of tetrahedron rectangles in $(\loom, \loom^+, \loom^-)$; see Definition \ref{def:flow:line:loom} below. 
	
	\begin{definition}\label{def:flow:line:loom}
		Let  $\rho = (t_i)_{i\in \zz}$ be a flow line in $\univV$. Set $R_i = \loomif(t_i)$ for every $i \in \zz$.  We call the sequence $(R_i)_{i\in \zz}$ a \emph{flow line} in $(\loom, \loom^+, \loom^-)$. 
	\end{definition}
	It follows from Definition \ref{def:flow:line} and Figure \ref{fig:tet:rects} that if $(R_i)_{i\in \zz}$ is a flow line in $(\loom, \loom^+, \loom^-)$ then for every $i \in \zz$ the tetrahedron rectangle $R_{i+1}$ is either the maximal vertical expansion of the tallest edge rectangle of $R_i$ or the maximal vertical expansion of an edge rectangle of~$R_i$ that is disjoint from the tallest edge rectangle of $R_i$; see Figure \ref{fig:flow:line}. This immediately implies the following observation.
	
	\begin{observation}\label{obs:tets:get:thinner}
		If $(R_i)_{i \in \zz}$ is a flow line in $(\loom, \loom^+, \loom^-)$, then
		\begin{itemize}
			\item for every $i \in \zz$, the rectangle $R_{i+1}$ intersects both horizontal sides of $R_i$;
			\item for all $i<j<k$, we have $R_i \cap R_k \subset R_i \cap R_j$.
		\end{itemize}
	\end{observation}

	\begin{figure}[h]
		\includegraphics[scale=0.25]{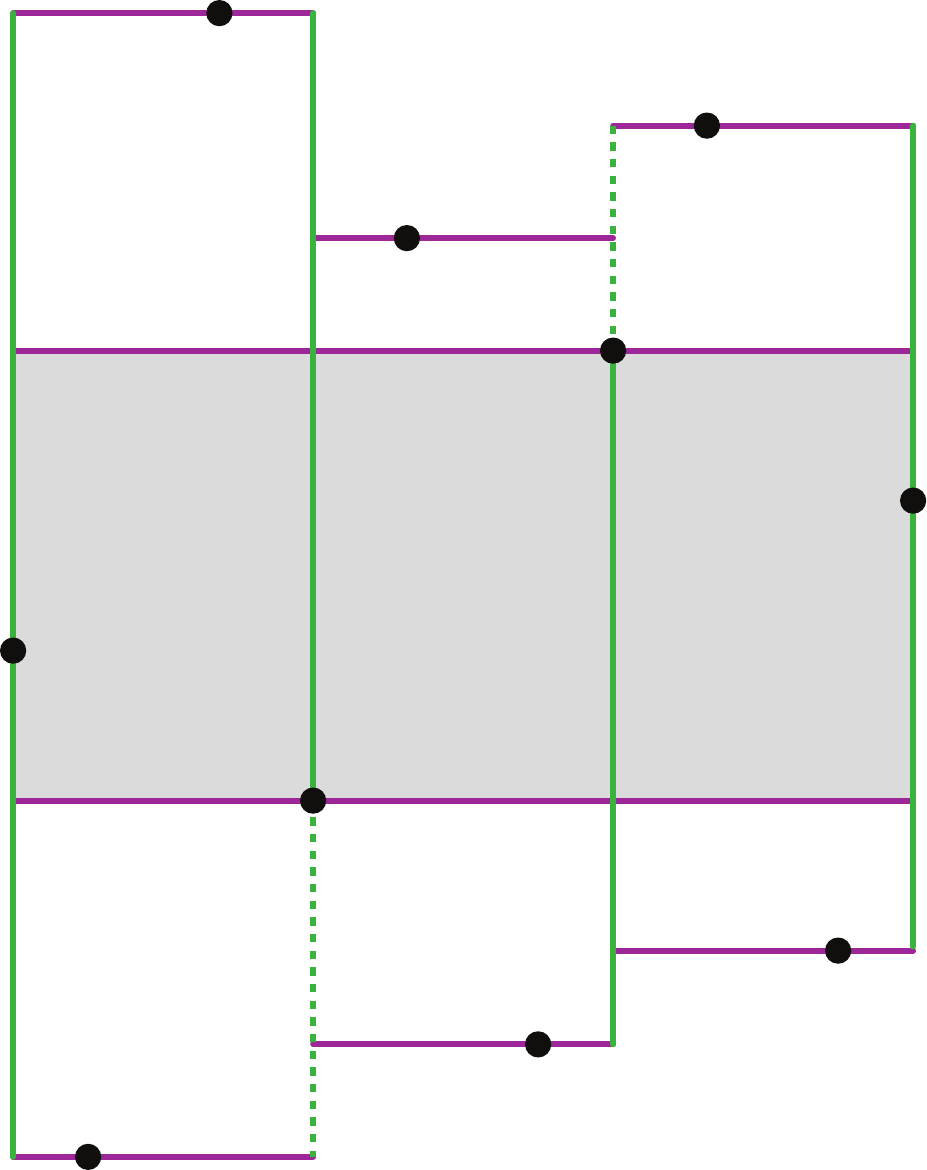}
		\caption{Relationship between $R_{i+1}$ and $R_i$ when $(R_i)_{i \in \zz}$ is a flow line. The tetrahedron rectangle $R_i$ is shaded gray. There are three options for $R_{i+1}$: it can be either the maximal vertical expansion  of the tallest edge subrectangle of $R_i$ (in the middle) or the maximal vertical expansion of an edge subrectangle of $R_i$ that is disjoint from the tallest edge rectangle of $R_i$ (on the sides). The dotted segments indicate where the expanded rectangles do not share the corresponding subsets of their sides.}
		\label{fig:flow:line}
	\end{figure}
	
	Throughout the paper, we do not distinguish notationally between a flow line viewed as a sequence of tetrahedra and a flow line viewed as a sequence of tetrahedron rectangles, and we freely pass between these interpretations. 
	In particular, we may write that a sequence $\eleZeta = (R_i)_{i\in \zz} $ of tetrahedron rectangles is an elevation of a flow cycle $\zeta = (t_1, \ldots, t_k)$, without explicitly mentioning the loomification map $\loomif(\cdot)$.

	\subsection{Flow lines and orbits of the flow}\label{subsec:flow:lines:and:orbits}
	Recall that the points of $\loom$ correspond to the orbits of the lift of the flow $\Psi\bez \Lambda$ to the universal cover of $M$, where  $M = N \bez\Lambda$.
	The terminology that appears in Definitions \ref{defn:flow:cycle} and~\ref{def:flow:line} is motivated by the following result of Landry-Minsky-Taylor.
	
	\begin{theorem}\cite[Fact 4.4 and Proposition 6.12]{LMT_flow}\label{thm:flow:lines}
		Let $(R_i)_{i \in \zz}$ be a flow line in $(\loom, \loom^+, \loom^-)$. Then either all rectangles $R_i$ share a common cusp of $\loom$ or there exists $x \in \loom$ such that
		\begin{equation}\label{eqn:orbit:in:loom}
		\bigcup\limits_{g \in \pi_1(M)} \bigcap\limits_{i\in \zz} g(R_i) = \bigcup\limits_{g \in \pi_1(M)} g(x).\end{equation}
		Conversely, for every $x\in \loom$ there is a flow line $ (R_i)_{i \in \zz}$ in $(\loom, \loom^+,\loom^-)$ such that \eqref{eqn:orbit:in:loom} holds. \qed
	\end{theorem}
	In other words, flow lines in the sense of Definition \ref{def:flow:line} encode orbits of $\widetilde{\Psi}$. 
	
	To simplify notation, we denote the intersection $\bigcap\limits_{i \in \zz} R_i$ of all rectangles that make up a flow line $\rho = (R_i)_{i \in \zz}$ by~$\bigcap \rho$. If $\bigcap \rho = x$ for some $x \in \loom$, we say that $\rho$ \emph{represents} $x$. Otherwise,  all rectangles~$R_i$ share a common cusp~$c$ of $\loom$. Then we say that $\rho$ represents~$c$. In the compactified loom space $\completedLoom$ the intersection $\bigcap \rho$  is the infinite-prong singularity corresponding to $c$. 
	

	\begin{lemma}\label{lem:stabilizers}
		Let $\rho = (R_i)_{i \in \zz}$ be flow line in $(\loom, \loom^+, \loom^-)$ that represents $x \in \completedLoom$. Then, under the $(\Psi, \Lambda)$-action of $\pi_1(M)$ on $(\loom, \loom^+, \loom^-)$, we have $\Stab(\rho) \subseteq \Stab(x)$.
	\end{lemma}
	\begin{proof}
		Let $g \in \Stab(\rho)$. Since $g$ is injective, we have
		\[g(x) = g\left(\bigcap\limits_{i \in \zz} R_i\right) = \bigcap\limits_{i \in \zz} g(R_i) = \bigcap\limits_{i \in \zz} R_i = x,\]
		Thus $g \in \Stab(x)$. 
	\end{proof}
	It follows from Lemma \ref{lem:stabilizers} that every periodic flow line represents either a cusp or a periodic point in~$\loom$.  
	Landry-Minsky-Taylor proved that $\Stab(\rho)$ is a subgroup of index at most two in $\Stab(x)$.

	\begin{proposition}\cite[Theorem 6.1 (4)]{LMT_flow}\label{prop:square}
		Suppose that $x \in \loom$ is periodic and represented by a flow line $\rho = (R_i)_{i\in \zz}$. Let $g_x$ be the generator of the stabilizer of $x$. Then every element of $\langle g_x^2 \rangle$ stabilizes $\rho$. \qed
	\end{proposition}
	
	For almost all closed orbits $\gamma$ of $\Psi\bez \Lambda$, the stabilizer of any point from $\liftGamma$ is equal to the stabilizer of a flow line representing this point. 
	The only exceptions are twisted orbits~$\gamma$ for which a primitive flow cycle $\zeta$ satisfying \[\bigcup\limits_{g \in \pi_1(M)} \bigcap g(\eleZeta) = \liftGamma\] is homotopic to $\gamma^2$ instead of $\gamma$ \cite[Theorem~6.1(4)]{LMT_flow}. In this case, we get distinct translates of the same flow line that represent the same point of $\liftGamma$. In the drilling algorithm such translates have to be recognized and \emph{uniquified}; see Section \ref{subsec:fellow-travel}.

	\section{Continents}\label{sec:continents}

	In this section, we discuss convexity in the loom space, define \emph{continents}, and prove that the loomification of any continent $\C$ in $\univV$ is a convex disk in $\loom = \loomif(\univV)$; see Proposition \ref{prop:convex:disk}. This result is used in Section \ref{sec:orders} to derive a circular order on the set consisting of the cusps of $\C$ together with the \emph{exit points of continental cusp leaves} of~$\C$.
	This circular order is further used to determine the horizontal order on certain stable continental cusp leaves and the vertical order on certain unstable continental cusp leaves; see Corollary \ref{cor:from:circ:to:total}.

	\subsection{Geodesics and convex hulls in loom spaces}\label{subsec:geodesics}

	In this subsection, we discuss geodesics and convex hulls in loom spaces. These notions are used in Proposition \ref{prop:convex:disk} which states that the loomification of a continent is a convex disk.
	The original source for this material is \cite[Section~6]{SchleimSegLoom}. 
	
	Let $(\loom, \loom^+, \loom^-)$ be a loom space. By a \emph{leaf segment} in $(\loom,\loom^+,\loom^-)$ we mean an oriented segment of a leaf of $\loom^\pm$ that is closed, except possibly at an endpoint corresponding to a cusp. The orientation of a leaf segment need not agree with the orientation of the leaf.
	We denote the closure of leaf segment $\alpha$ in the compactified loom space $\completedLoom$ by $\overline{\overline{\alpha}}$. 
	
	\begin{definition}
		Let $k\geq 1$. For $i=1,2,\ldots, k$ let $\alpha_i$ be an oriented leaf segment in $(\loom,\loom^+,\loom^-)$. 
		We say that $\alpha = (\alpha_1, \ldots, \alpha_k)$ is a \emph{polygonal path} in in $(\loom,\loom^+,\loom^-)$ if for every \mbox{$i\in \lbrace 1,2,\ldots, k-1\rbrace$} the terminal endpoint of $\overline{\overline{\alpha}}_i$ is the initial endpoint of $\overline{\overline{\alpha}}_{i+1}$.  
	\end{definition}
	
	Polygonal paths can connect two points, two cusps, or a point and a cusp. 
	
	\begin{definition}
		Let $\alpha = (\alpha_1, \ldots, \alpha_k)$ be a polygonal path in $(\loom, \loom^+, \loom^-)$. \\We say that $\alpha$  is a \emph{geodesic} if 
		\begin{itemize}[topsep=0pt]
			\item $\bigcup\limits_{i=1}^k \alpha_i$ intersects every leaf of $\loom^\pm$ in at most one connected component.
			\item For $i=1,\ldots,k-1$, no terminal segment of $\alpha_i$ is the reverse of some initial segment of $\alpha_{i+1}$.
		\end{itemize}
	\end{definition}

	Denote the set of cusps of $\loom$ by $\Omega$. We say that a subset $A \subset \loom$ \emph{meets} a cusp $c \in \Omega$ if the closure of $A$ in $\completedLoom$ contains the singularity corresponding to $c$, and denote the set of cusps met by $A$ by $\Omega(A)$ .

	The compactified loom space $\completedLoom$ can be naturally identified with $\loom \cup \Omega$. Nevertheless, we consider geodesics to be subsets of $\loom$, not of $\completedLoom$, keeping in mind that they may meet cusps in the above sense. 
	\begin{definition}
		Let $A \subset \loom$. 
		We say that $A$ is \emph{convex} if for any $x,y \in A \cup \Omega(A)$  every geodesic connecting  $x$ and $y$ is contained in the closure $\overline{A}$ of $A$ in $\loom$. 
		
	\end{definition}
	The most basic examples of convex sets in $\loom$ are rectangles. More complicated convex sets can be constructed using \emph{hulls}.
	
	\begin{definition} \cite[Definition 6.19]{SchleimSegLoom}
		Let $X \subset \loom \cup \Omega$. The \emph{hull} of $X$ is the set
		\[\hull(X) = \lbrace p \in \loom  \ | \ p \text{ lies on a geodesic between some } x,y \in X \cup \Omega(X)\rbrace.\]
	\end{definition}
	
	For instance, if $x,y$ are corners (material or ideal) of a rectangle $R\subset \loom$, then $\mathrm{int}(\hull(\lbrace x, y \rbrace)) =  R$ (recall from Definition \ref{defn:rectangle} that rectangles are open). 
	
	\begin{proposition}\cite[Lemma 6.21]{SchleimSegLinks} \label{prop:SS:convex:hull:is:convex}
		For any finite subset $X \subset \loom \cup \Omega$,  $\hull(X)$ is convex. \qed
	\end{proposition}

In the proof of Lemma \ref{lem:adapted:set}, we will use the following result which gives a sufficient condition for the existence of a geodesic passing through some point or cusp of $\loom$ that connects two points or cusps of $\loom$.
	\begin{lemma}\cite[Lemma 6.17]{SchleimSegLoom} \label{lem:geodesic:between}
		Let $p, q, x \in \loom \cup \Omega$. If $p$ and $q$ lie in the closures of distinct, non-adjacent orthants at $x$ then there is a geodesic from $p$ to $q$ that passes through~$x$.
	\end{lemma}

	\subsection{Landscapes}

	Let $t$ be a tetrahedron of a taut triangulation $\T$, and let $f, f'$ be distinct faces of $t$  which meet along an edge $e$ of $t$. We realize this intersection with a dihedral angle of $\pi$ if $f$ and $f'$ are either both top faces of $t$ or they are both bottom faces of $t$. Otherwise, we realize this intersection  with a dihedral angle of 0. By Condition~(2) of  Definition~\ref{def:taut}, this smooths the two-skeleton of $\T$ into a (non-generic) branched surface~$\HH$ whose branch locus coincides with the one-skeleton of $\T$. We call $\HH$ the \emph{horizontal branched surface} associated to $\T$. 

	A small regular neighborhood $N(\HH)$ of~$\HH$ in the manifold $|\T|$ underlying $\T$ can be foliated by properly embedded intervals that are transverse to the faces of $\T$. We call $N(\HH)$ a \emph{tie neighborhood} of~$\HH$, and the transverse intervals that make up the foliation its \emph{ties}. We say that a surface $S$  embedded in $|\T|$ is \emph{carried} by~$\T$ if it can be isotoped into a tie neighborhood of $\HH$ so that it is there transverse to its ties. 
	
	Using this terminology, we now define \emph{landscapes}, first introduced by Schleimer and the second author in \cite[Section~4.4]{SchleimSegLinks}. 
	
	\begin{definition}\cite[Definition 4.5]{SchleimSegLoom}\label{def:landscape}
		Let $\widetilde{\T}$ be a taut triangulation of $\rr^3$. A \emph{landscape} in~$\widetilde{\T}$ is a union of faces of $\univT$ that forms a connected, embedded surface  carried by $\univT$.
	\end{definition}

	
	We denote the set of cusps of a landscape $L$ by $\Omega(L)$. Note that every edge of $L$ has at most two incidences with faces of $L$. An edge with exactly one such incidence is called a \emph{coastal edge} of $L$.

	\begin{lemma} \cite[Lemma 4.6]{SchleimSegLoom} \label{lem:landscape:disk} 	
		Let $L$ be a landscape in a taut triangulation $\univT$ of $\rr^3$.
		\begin{enumerate}[topsep=0pt]
			\item The interior of $L$ is an open disk.
			\item Any two distinct vertices of $L$ are distinct vertices of $\univT$. 
			\pushQED{\qed} 
			\qedhere
			\popQED\end{enumerate}
	\end{lemma}

 Our goal in Proposition \ref{prop:convex:disk} is to show that the loomification of any landscape in a veering triangulation $\univV$ is a convex subset of the loom space associated to $\univV$. As a first step, we prove below that the closure of the loomification of a landscape is the convex hull of its cusps.

	\begin{proposition}\label{prop:loomif:of:landscape}
		Let $L$ be a finite landscape in a veering triangulation $\univV$. Then \[\overline{\loomif(L)} = \hull(\Omega(L)).\] 
	\end{proposition}
	\begin{proof} The smallest possible landscape consists of a single triangle. In this case, the loomification is a face rectangle and hence its closure is the convex hull of the three cusps of the landscape.
		
		Suppose that for some $n \in \mathbb{N}$, $n\geq 3$, the equality $\overline{\loomif(L_n)} = \hull(\Omega(L_n))$ holds for every landscape $L_n$ with $n$ cusps.
		
		Let $L_{n+1}$ be a landscape with  $n+1$ cusps. Since, by Lemma \ref{lem:landscape:disk}, $L_{n+1}$ determines a finite ideal triangulation of a disk with finitely many points removed from its boundary, the dual graph $\Gamma_{n+1}$ of $L_{n+1}$ is a finite tree. The tree $\Gamma_{n+1}$ has at least two vertices, and thus admits a vertex $v$ of degree 1. Let $f$ be the face of $L_{n+1}$ that is dual to $v$. Then the union of all faces of $L_{n+1}$ except $f$ is connected, and therefore a landscape. Since $f$ is connected to the rest of $L_{n+1}$ only along one edge, this smaller landscape has exactly one less cusp than $L_{n+1}$. We denote it by $L_n$.
		We may assume that the set of cusps of  $L_{n+1}$ and $L_n$ are $\Omega_{n+1} = \lbrace c_i \rbrace_{i=1}^{n+1}$ and $\Omega_n = \lbrace c_i \rbrace_{i=1}^n$, respectively.

		Let $F_{n+1}$ be the set of triangles of $L_{n+1}$. For a triangle $g$ of $L_{n+1}$ let $E(g)$ be the set of edges of $g$. Since $\overline{\loomif(g)}$ is the union of $\overline{\loomif(e)}$ for $e \in E(g)$, and $\overline{\loomif(e)}$ is the convex hull of the cusps of $e$, we have
		\begin{align*}
		\overline{\loomif(L_{n+1}) }&= \bigcup\limits_{g \in F_{n+1}} \overline{\loomif(g)} = \\
		&= \bigcup\limits_{g \in F_{n+1}}  \bigcup\limits_{e \in E(g)} \overline{\loomif(e)} =\\
		& = \bigcup\limits_{g \in F_{n+1}}  \bigcup\limits_{e \in E(g)} \hull(\text{cusps of e}) \subseteq \\
		& \subseteq \bigcup\limits_{1\leq i < j \leq n+1} \hull(\lbrace c_i, c_j\rbrace) = \hull(\Omega_{n+1}).
		\end{align*}
		By the induction hypothesis $\hull(\lbrace c_i, c_j \rbrace) \subseteq \overline{\loomif(L_n)} \subset \overline{\loomif(L_{n+1})}$ whenever \mbox{$1\leq i < j \leq n$}. Therefore to show that $\hull(\Omega_{n+1}) \subseteq \overline{\loomif(L_{n+1})}$ it suffices to show that $\hull(\lbrace c_i, c_{n+1}\rbrace ) \subseteq \overline{\loomif(L_{n+1})}$ for every $1\leq i \leq n$. 
		
		Let $e$ be the edge of $L_{n+1}$ that separates $L_n$ from $f$. Assume that $c_1$ and $c_n$ are the cusps of $e$. Then $\hull(c_1, c_{n+1})$ and $\hull(c_n, c_{n+1})$ are contained in $\overline{\loomif(f)}$ and hence in $\overline{\loomif(L_{n+1})}$.

		Let $\delta$ be a geodesic that connects $c_i$ to $c_{n+1}$ for some $2\leq i \leq n-1$. Every cusp leaf of $\loom^{\pm}$ separates $\loom$, thus the complement of $\loomif(e)$ in $\loom$ consists of two disjoint sets: one which meets  all cusps of $\Omega_n$, and the other that meets $\lbrace c_1, c_n, c_{n+1}\rbrace$. It follows that~$\delta$ passes through $\loomif(e)$. 
		Subdivide $\delta$ into two subpaths $\delta_1, \delta_2$ meeting at a point $x\in \loomif(e)$. Since $\delta$ is a geodesic and $\loomif(e)$ is the interior of $\loomif(L_n) \cap \loomif(f)$, the first subpath $\delta_1$ is contained in $\overline{\loomif(L_n)}$, and the second subpath $\delta_2$ is contained in $\overline{\loomif(f)}$. Thus $\delta \subseteq \overline{\loomif(L_n)} \cup \overline{\loomif(f)} = \overline{\loomif(L_{n+1})}$. This finalizes the proof.
	\end{proof}

	\begin{corollary}\label{cor:loomification:convex}
		The loomification of a finite landscape in $\univV$ is convex in $\loom = \loomif(\univV)$. 
	\end{corollary}
	\begin{proof}
		By Proposition \ref{prop:loomif:of:landscape}, the loomification of a finite landscape is the convex hull of its cusps. 	Its convexity therefore follows from Proposition \ref{prop:SS:convex:hull:is:convex}.
	\end{proof}

	\subsection{Continents} \label{subsec:continents}

	Continents are a special type of subtriangulations of taut triangulations.
	The first property that they satisfy is \emph{face-connectedness}.
	
	\begin{definition}\label{def:face:connected}
		Let $\T' = (T', \Phi')$ be a subtriangulation of a triangulation $\T = (T, \Phi)$. We say that $\T'$ is \emph{face-connected} if for any pair of tetrahedra $t, t' \in T'$ there is a path contained in $\T'$ that connects $t$ to~$t'$ and is transverse to the two-skeleton of~$\T'$. 
	\end{definition}

	%
	%
	%
	
	Now, consider a landscape $L$ in a taut triangulation $\widetilde{\T}$ of $\rr^3$. Given a tetrahedron $t$ of $\widetilde{\T}$ we say that $L$ is \emph{immediately above} (respectively, \emph{immediately below}) $t$ if both top faces (respectively, both bottom faces) of $t$ belong to $L$. 
	
	\begin{definition} \label{def:layering}
		Let $\T'$ be a finite face-connected subtriangulation of a taut triangulation $\widetilde{\T}$ of $\rr^3$. Let $n$ be the number of tetrahedra of $\T'$. A \emph{layering} of $\T'$ is a collection $(L_i)_{i=0}^n$ of pairwise distinct landscapes  in $\widetilde{\T}$ such that:
		\begin{enumerate}
			\item For every tetrahedron $t$ of $\T'$ there is a unique $0 \leq i \leq n$ such that $t$ lies immediately above $L_i$ and immediately below $L_{i+1}$.
			\item For every $0 \leq i < n$ there is a tetrahedron of $\T'$ immediately above $L_i$.
		\end{enumerate}
		We say that $\T'$ is \emph{layered} if it admits a layering. 
	\end{definition}

	Finally, a \emph{continent} in  $\widetilde{\T}$  is a finite subtriangulation $\C$ of $\widetilde{\T}$ that is face-connected and layered \cite[Section 4.14]{SchleimSegLinks}. 
	
	We denote the set of cusps of a continent $\C \subset \univT$ by $\Omega(\C)$. If $(L_i)_{i=0}^n$ is a layering of~$\C$ then every landscape $L_i$ has the same set of coastal edges. We call them the \emph{coastal edges} of $\C$. The landscapes $L_0$ and $L_n$ are respectively called the \emph{upper} and \emph{lower landscapes} of $\C$.

	Crucially, continents exist. For instance, a single tetrahedron in $\univT$ is a continent. The union of two tetrahedra in $\univT$ that share a face is also a continent. More generally, there are certain conditions under which  adding a tetrahedron to a continent yields another continent; see \cite[Lemma 4.24]{SchleimSegLinks}. In the veering case, Schleimer and the second author proved that there exists a collection $(\C_n)_{n \in \mathbb{N}}$ of continents in $\univV$ such that $\C_n \subset \C_{n+1}$ and  $\univV = \bigcup\limits_{n \in \mathbb{N}} \C_n$ \cite[Proposition 5.4]{SchleimSegLinks}. The collection $(\C_n)_{n \in \mathbb{N}}$ is called a \emph{continental exhaustion} of $\univV$. Here, we will not need a continental exhaustion, but only the following consequence of its existence.

	\begin{proposition}\label{prop:continents:exists} \cite[Proposition 5.4]{SchleimSegLinks}.
		Let $X$ be a finite collection of tetrahedra of $\univV$. Then there exists a continent $\C \subset \univV$ that contains every tetrahedron from $X$. \qed
	\end{proposition}

	The process of constructing a continent that contains a given finite set of tetrahedra relies on a procedure, called \emph{channelization}, that was introduced in \cite[Section 5.19]{SchleimSegLinks}. We refer the reader to the original source for the details of the construction. 
	
	
	\begin{proposition}\label{prop:convex:disk}
		Let $\C$ be a nonempty continent in $\univV$. Then $\overline{\loomif(\C)} \subset \loom$ is a convex, finite-sided closed disk with finitely many boundary points removed.
	\end{proposition}
	\begin{proof}
		Let $(L_i)_{i=0}^n$ be a layering of $\C$. By Definition \ref{def:layering}, $\loomif(L_i) = \loomif(L_{i+1})$ for every $0\leq i < n$ and thus $\loomif(L_0) = \loomif(L_n) = \loomif(\C)$. Therefore the convexity of $\overline{\loomif(\C)}$ follows from Corollary \ref{cor:loomification:convex}. 
		We also get $\overline{\loomif(\C)} = \hull(\Omega(\C))$. 
		
		Schleimer and the second author showed in \cite[Lemma 6.33]{SchleimSegLoom} that for any finite, non-empty, face-connected subtriangulation of $\univV$, the convex hull of the cusps of its tetrahedron rectangles is a finite-sided closed disk minus finitely many boundary points. Since $\C$ is face-connected, this description applies to $\hull(\Omega(\C))$ and thus to  $\overline{\loomif(\C)}$.
	\end{proof}
	
	\section{The circular, horizontal, and vertical orders}\label{sec:orders}

	As already indicated in Section \ref{subsec:loomification}, the drilling algorithm described in Section \ref{sec:algorithm} relies on our ability to compare the relative cardinal positions of certain cusps of $\loomif(\univV)$. In this section we explain how to leverage the combinatorics of continents to extract precisely this information.
	
	\subsection{The loom-compatible circular order and partial cardinal orders}\label{subsec:linear:orders}
	Let $\C$ be a continent in $\univV$. We say that a cusp leaf $\ell$ of $\loom^\pm$ is a \emph{continental cusp leaf} (with respect to~$\C$) if $\ell$ originates at a cusp of $\C$ and intersects $\loomif(\C)$. Since, by Proposition~\ref{prop:convex:disk}, $\loomif(\C)$ is convex and connected, any leaf intersecting $\loomif(\C)$ splits it into two components. In particular, a continental cusp leaf $\ell$ meets  $\partial \loomif(\C)$ in a single point $\epsilon(\ell)$, called the \emph{(continental) exit point} of $\ell$. We denote the set of all exit points of all stable (respectively, unstable) continental cusp leaves for $\C$ by $\E^+(\C)$ (respectively, $\E^-(\C)$) and set $\E(\C) = \E^+(\C) \cup \E^-(\C)$.
	
	Let $h$ an orientation-preserving homeomorphism  between $\loomif(\C)$ and an open two-dimensional disk. Such a homeomorphism exists by Proposition~\ref{prop:convex:disk}. Fill in the cusps of $\loomif(\C)$ and extend $h$ to a homeomorphism $\overline{\overline{h}}$ between two closed disks. 
	Let  \mbox{$\kappa: \Omega(\C) \cup \E(\C) \hookrightarrow S^1$} be the restriction of $\overline{\overline{h}}$ to $\Omega(\C)\cup \E(\C)$. For $a,b,c \in \Omega(\C) \cup \E(\C)$ set
	\begin{equation*}
	\Theta(a,b,c) = \begin{cases}
	0 & \text{ iff } |\lbrace a,b,c \rbrace| < 3, \\
	1 & \text{ iff } (\kappa(a), \kappa(b), \kappa(c)) \text{ are in counterclockwise cyclic order on } S^1,\\
	-1 &\text{ iff } (\kappa(a), \kappa(b), \kappa(c)) \text{ are in clockwise cyclic order on } S^1.
	\end{cases}\end{equation*}
	Then $\Theta: (\Omega(\C) \cup \E(\C))^3 \rightarrow \lbrace -1, 0, 1\rbrace$ is a circular order on $\Omega(\C) \cup \E(\C)$. We refer to~$\Theta$ as the~\emph{loom-compatible} circular order on $\Omega(\C) \cup \E(\C) $. 
	Its restriction to $\Omega(\C)$ is compatible with the taut structure of $\univV$ in the sense of \cite[Definition~2.19]{SchleimSegLinks}.

	Given a continental cusp leaf $\ell$ of $\C$, let $\iota(\ell)$ be the origin cusp of  $\ell$. As before, let $\epsilon(\ell)$ be the continental exit point of $\ell$. Define the \emph{side function} $\Theta_\ell$ of $\ell$ by setting, for any $c \in \Omega(\C)$,
	\begin{equation}\label{function:on:cusps}
	\Theta_\ell(c) = \Theta(\iota(\ell), \epsilon(\ell), c).
	\end{equation}
	Below we show that the side functions encode how their corresponding continental cusp leaves separate the cusps of $\C$. This provides a link between the loom-compatible circular order and a partial horizontal order on stable continental cusp leaves, as well as a partial vertical order on unstable continental cusp leaves; see Corollary \ref{cor:from:circ:to:total}.
	
	\begin{lemma}\label{lem:theta:l:function} 
		Let $\C$ be a continent in $\univV$. 
		\begin{enumerate}
			\item For any continental cusp leaf $\ell$ in $\loomif(\C)$ and $c,d \in \Omega(\C)$
			\[\Theta_\ell(c)\cdot \Theta_\ell(d) = \begin{cases} 
			0 & \text{ if either } c \text{ or } d \text{ is the origin of } \ell,\\
			1 & \text{ if } c \text{ and } d \text{ are (strictly) on the same side of } \ell,\\
			-1 &\text{ if } \ell \text{ (strictly) separates } c \text{ and } d. 
			\end{cases}\]
			\item For any continental cusp leaf $\ell$ in $\loomif(\C)$ the function $\Theta_\ell(\cdot)$ is constant on a subset $D \subset \Omega(\C)$ if and only if either every $d \in D$ lies (strictly) on the same side of $\ell$ or $D = \lbrace \iota(\ell) \rbrace$.
			\item Let $\ell, \ell'$ be distinct continental cusp leaves in $\loomif(\C)$ originating at the same cusp \mbox{$c \in \Omega(\C)$} and belonging to the same foliation. Denote the unique connected component of $\loomif(\C) \bez (\ell \cup \ell')$ whose boundary intersects both $\ell$ and $\ell'$ by $J$. Then for any $d \in \Omega(\C)$
			\[\Theta_\ell(d) \cdot \Theta_{\ell'}(d) = 
			\begin{cases}
			0 & \text{ if } c = d,\\
			1 & \text{ if } c \neq d \text{ and } d \text{ does not meet } J,\\
			-1 &\text{ if } c \neq d \text{ and } d \text{ meets } J.
			\end{cases}
			\]
		\end{enumerate}
	\end{lemma}
	\begin{proof}
		By definition \eqref{function:on:cusps}, $\Theta_\ell(c) = 0$ if and only if $c = \iota(\ell)$. 
		
		Since a continental cusp leaf is a geodesic that intersects the interior of $\loomif(\C)$ and, by Proposition \ref{prop:convex:disk}, $\loomif(\C)$ is a convex disk, there are precisely two connected components of \mbox{$\loomif(\C) \bez \ell$}. 
		Orient $\ell$ away from its origin, and denote the connected component that is to the right of $\ell$, relative to that orientation, by $J_R$ and the connected component that is to the left of $\ell$ by $J_L$. 
		Observe that every cusp $c \in \Omega(\C)$ other than $\iota(\ell)$ meets the interior of either $J_R$ or $J_L$. Moreover, if $c \neq \iota(\ell)$ then 
		\[\Theta_\ell(c) = 
		\begin{cases}
		1 & \text{ if } c \text{ meets the interior of } J_L,\\
		-1 & \text{ if } c \text{ meets the interior of }  J_R.
		\end{cases}\]
		This geometric interpretation of the value of $\Theta_\ell(c)$ imply all three statements of the lemma.
	\end{proof}
	
	The word `strictly' that appears in the statement of Lemma \ref{lem:theta:l:function} (1) refers to neither $c$ nor $d$ being equal to $\iota(\ell)$. When $\Theta_\ell(c) \neq \Theta_\ell(d)$ and $c\neq d$ we say that $c$ and $d$ are \emph{weakly separated} by $\ell$. This allows for either $c$ or $d$, but not both at the same time, to be the origin of $\ell$.

Let $P$ be a rectangle in $\loomif(\C)$, equipped with the cardinal directions specified in Convention~\ref{convention:cardinal:directions}. The stable continental cusp leaves intersecting $P$ are naturally ordered from west to east, while the unstable continental cusp leaves intersecting $P$ are naturally ordered from south to north. By Part~(1) of Lemma~\ref{lem:theta:l:function}, these orders can be read off from~$\Theta$; see Corollary~\ref{cor:from:circ:to:total}.

	\begin{corollary}\label{cor:from:circ:to:total}
		Let $\ell_1, \ell_2$ be stable (respectively, unstable) continental cusp leaves that intersect an  edge rectangle $P \subset \loomif(\C)$. Denote the cusp of $P$ that meets the west side (respectively, south side) of $P$ by $w$ (respectively, $s$). Then $\ell_1$ is  west (respectively, south) of $\ell_2$ if and only if 
		\[\pushQED{\qed}
		\Theta_{\ell_1}(w)\cdot\Theta_{\ell_1}(\iota(\ell_2)) = -1 \hspace{1cm} (\text{respectively, } \Theta_{\ell_1}(s)\cdot\Theta_{\ell_1}(\iota(\ell_2)) = -1). \qedhere  		\pushQED{\qed}\]
\end{corollary}

	We never need to order a set of leaves that contains both stable and unstable leaves. Therefore to use Corollary \ref{cor:from:circ:to:total} it suffices to independently determine $\Theta^+$, the restriction of $\Theta$ to $\Omega(\C) \cup \E^+(\C)$ and $\Theta^-$, the restriction of $\Theta$ to $\Omega(\C) \cup \E^-(\C)$.

	\subsection{Determining the natural circular order} 
	
	In Corollary \ref{cor:from:circ:to:total} we observed that the loom-compatible circular order $\Theta$ on $\Omega(\C) \cup \E(\C)$ can be used to determine the horizontal and vertical orders on the set of the continental cusp leaves passing through a fixed simplex rectangle in $\loomif(\C)$. 
	However, we have not explained how to determine $\Theta$ in the situation when we only have direct access to $\C$, and not $\loomif(\C)$. Recall from Section~\ref{subsec:loomification} that given a simplex $s$ of $\univV$ we can only determine $\loomif(s)$ up to the horizontal and vertical orders on the set of cusps of $\loomif(s)$. This means that $\C$ might not encode complete information about the relative positions of cusps of $\loomif(\C)$; see Figure \ref{fig:non:convex}. In this subsection we explain how to resolve this issue.
	
	\begin{figure}[h]
		\centering
		\begin{minipage}{0.3\textwidth}
			\centering
			\labellist
			\hair 2pt \small
			\pinlabel $c_2$ [r] at 2 167
			\pinlabel $c_1$ [br] at 42 315
			\pinlabel $c_3$ [tr] at 78 53
			\pinlabel $c_0$ [b] at 151 353
			\pinlabel $c_4$ [t] at 225 15
			\pinlabel $c_6$ [bl] at 263 278
			\pinlabel $c_5$ [l] at 336 200
			\pinlabel $\epsilon$ [t] at 151 18
			\endlabellist
			\includegraphics[width=\linewidth]{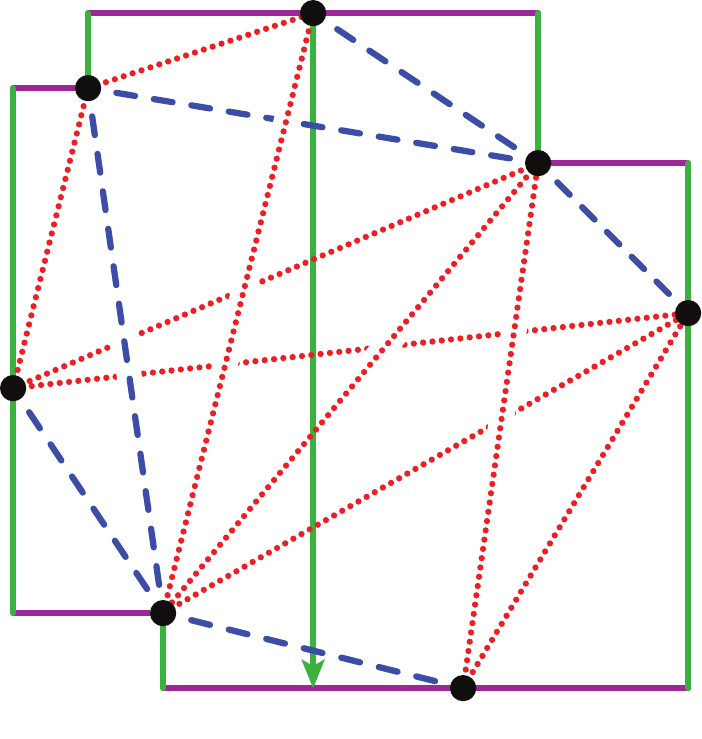}
			
		\end{minipage}
		\qquad\qquad
		\begin{minipage}{0.3\textwidth}
			\centering
			\labellist
			\hair 2pt \small
			\pinlabel $c_2$ [r] at 2 167
			\pinlabel $c_1$ [br] at 42 315
			\pinlabel $c_3$ [tr] at 78 53
			\pinlabel $c_0$ [b] at 214 353
			\pinlabel $c_4$ [t] at 152 15
			\pinlabel $c_6$ [bl] at 263 278
			\pinlabel $c_5$ [l] at 336 200
			\pinlabel $\epsilon$ [t] at 214 18
			\endlabellist
			\includegraphics[width=\linewidth]{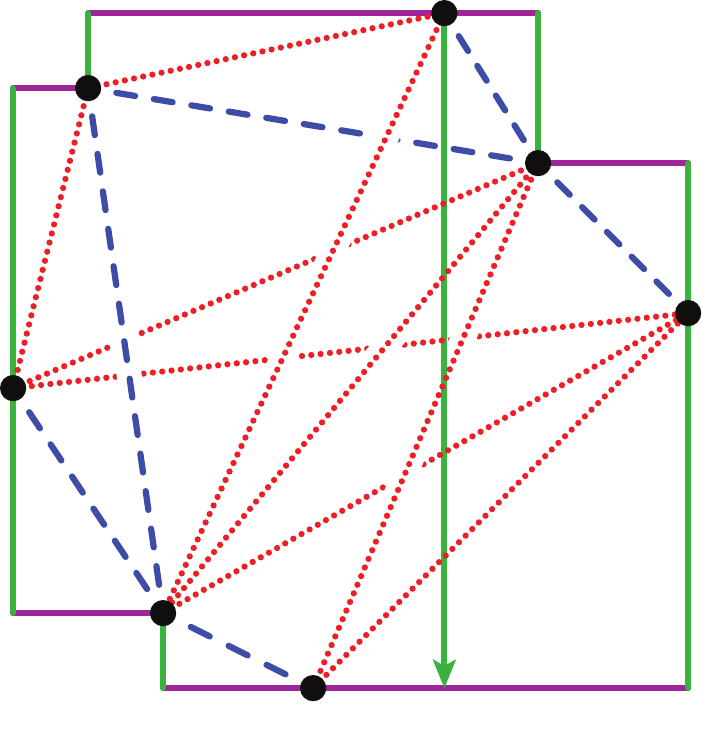}
			
		\end{minipage}
		\caption{Two possible loomifications of the same veering continent that indicate different circular orders on the set of cusps and exit points of the continental cusp leaves. On the left, the stable continental cusp leaf originating at the cusp $c_0$ exits the loomification between cusps $c_3$ and $c_4$; on the right --- between cusps $c_4$ and~$c_5$.}
		\label{fig:non:convex}
	\end{figure}
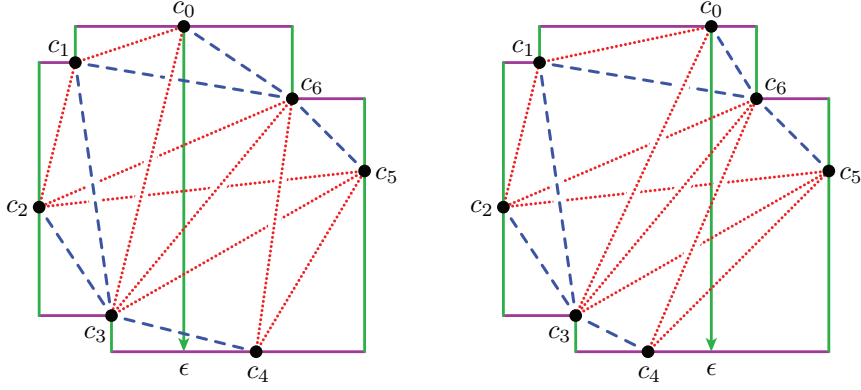
	
	The restriction of $\Theta$ to $\Omega(\C)$ is clear: by Lemma \ref{lem:landscape:disk}, the upper landscape of $\C$ is a disk with finitely many boundary points, corresponding to the cusps of~$\C$, deleted. Furthermore, by Lemma \ref{lem:landscape:disk} (2), any two of its vertices correspond to different cusps of~$\loom$. Thus the circular order on the set $\Omega(\C)$ of cusps of $\C$ can be determined by following the boundary of the upper landscape of $\C$.
	
	Ordering the exit points of the continental cusp leaves is a bit more involved. The exit point  of a continental cusp leaf $\ell$ lies between two cusps of $\C$ that are consecutive in the restriction of $\Theta$ to $\Omega(\C)$. We call these cusps the \emph{bracketing cusps} for $\ell$. 
	
	\begin{lemma}\label{lem:determining:circular:order}
		The circular order $\Theta^+ : (\Omega(\C) \cup \E^+(\C))^3 \to \lbrace-1,0,1\rbrace$ can be determined using only the following data:
		\begin{itemize}
			\item the restriction of $\Theta^+$ to triples of points from $(\Omega(\C))^3$, and
			\item the bracketing cusps of each stable continental cusp leaf.
		\end{itemize}
		Analogous claim holds for $\Theta^- : (\Omega(C) \cup \E^+(\C)) \rightarrow \lbrace -1, 0, 1\rbrace$ with the stable continental cusp leaves replaced by the unstable continental cusp leaves.
	\end{lemma}
	\begin{proof}
We first make a few general observations about orthants and stable cusp leaves. Let~$c$ be a cusp of~$\loom$ and let $\ell$ be a stable $c$-leaf. Since rectangles are open (see Definition~\ref{defn:rectangle}), every edge rectangle that intersects $\ell$ intersects precisely two adjacent orthants at $c$. In particular,\\
($\ast$) an edge rectangle with ideal corner at $c$ is disjoint from every stable $c$-leaf.

Moreover,\\
($\ast\ast$) the closure of each orthant at $c$ intersects precisely one stable $c$-leaf.

Since $\Omega(\C)$ partitions $\partial\loomif(\C)$ into nonempty intervals, the circular order $\Theta^+$ is completely determined by its restriction $\Theta^+_\Omega$ to $(\Omega(\C))^3$ together with the linear order induced on the exit points contained in each interval. We now explain how to determine the latter assuming that $\Theta^+_\Omega$ and the bracketing cusps of each stable continental cusp leaf are known.

Let $\ell_1, \ell_2$ be distinct stable continental cusp leaves with the same pair of bracketing cusps $a,b$. For $i=1,2$, let $c_i$ and $\epsilon_i$ be the cusp from which $\ell_i$ originates and the exit point of $\ell_i$, respectively. Set $P = \hull(\lbrace a,b\rbrace)$. Then $P$ is an edge rectangle; see \cite[Lemma 6.36]{SchleimSegLoom}.

Since both $\ell_1$ and $\ell_2$ intersect $P$, and $a$ and $b$ are ideal corners of $P$, $(\ast)$ implies that neither $c_1$ nor $c_2$ is equal to $a$ or $b$. Moreover, since $P$ intersects two adjacent orthants at $c_1$ meeting $\ell_1$ and two adjacent orthants at $c_2$ meeting $\ell_2$, the assumption that $\ell_1 \neq \ell_2$ together with $(\ast\ast)$ gives $c_1\neq c_2$. Thus  $a,c_1,c_2$ are three distinct cusps of $\loom$.

Since furthermore distinct leaves of~$\loom^+$ do not intersect, we get the equality
\[
\Theta^+(a,\epsilon_1,\epsilon_2)
=
\Theta^+(b,\epsilon_1,\epsilon_2)
=
-\Theta^+(a,c_1,c_2).
\]
Observe that the last expression $-\Theta^+(a,c_1,c_2)$ involves only the cusps of $\C$, while the first expression $\Theta^+(a,\epsilon_1,\epsilon_2)$ determines the linear order on $\lbrace \epsilon_1,\epsilon_2\rbrace$ induced by $\Theta^+$. Therefore the claim follows.
	\end{proof}

	\subsection{Finding the bracketing cusps}\label{subsec:convexity}
	By Lemma \ref{lem:determining:circular:order}, the loom-compatible circular order~$\Theta$ is completely determined by its restriction to the triples of cusps of $\C$ together with the information on the bracketing cusps of each continental cusp leaf.
Unfortunately, the bracketing cusps for continental cusp leaves cannot always be read off from the combinatorics of $\C$ alone; see Figure \ref{fig:non:convex}. In this section we discuss an additional property that a continent may satisfy and prove that it ensures that the bracketing cusps of any continental cusp leaf can be determined unambiguously.

	\begin{definition}\label{def:convex:continent}
		We say that a continent $\C \subset \univV$ is \emph{convex} if for every tetrahedron rectangle $R \subset \loomif(\C)$ there is a tetrahedron $t$ of $\C$ with $\loomif(t) = R$.
	\end{definition}
	
	The notion of convexity in Definition~\ref{def:convex:continent} agrees with \cite[Definition~5.17]{SchleimSegLinks}, where it is defined in terms of properties of certain train tracks embedded in the upper and the lower landscape of $\C$. 
	
	\begin{remark} \label{rem:convexity:overlap}
		We caution the reader not to confuse the convexity of a continent with the convexity of its loomification. Indeed, in Proposition \ref{prop:convex:disk} we proved that the loomification of any continent is convex.  We decided to retain the original terminology from \cite[Definition~5.17]{SchleimSegLinks}, despite the fact that it creates this terminological overlap, because introducing yet another term could create even more confusion in the literature.
	\end{remark}
	
	\begin{example} The veering continent illustrated in Figure \ref{fig:non:convex} is not convex: it misses a tetrahedron~$t$ with vertices $c_0, c_3, c_4, c_6$. 
		Since we do not know the color of the top diagonal of $t$, we cannot determine the horizontal order of $\{c_0, c_4\}$ which is needed to establish the bracketing cusps for the stable continental cusp leaf originating at $c_0$.\end{example}

	In Lemma \ref{lem:bracketing:cusps} we prove that convexity of a continent is sufficient to determine, without any ambiguity, the bracketing cusps for any stable continental cusp leaf. We begin by introducing the necessary terminology and establishing several auxiliary lemmas.

%
	
	We say that an edge $e$ of a face $f$ of $\univV$ is the \emph{widest edge} of $f$ if it is the bottom diagonal of the tetrahedron $t$ that is immediately above $f$. 
	This terminology is consistent with Definition \ref{def:taller:etc}, because $\loomif(e)$ is the widest edge rectangle among all edge rectangles contained in $\loomif(f)$.
	
	\begin{lemma}\label{lem:widest:edges}
		Let $L$ be the upper landscape of a convex continent $\C$ in $\univV$. Let $f, f'$ be faces of $L$  that are adjacent along an edge $e$ of $L$. Then $e$ cannot be the widest edge of both $f$ and $f'$. 
	\end{lemma}
	\begin{proof}
		If $e$ is the widest edge of both $f$ and $f'$ then $\loomif(f)\cup \loomif(f')$ is a tetrahedron rectangle contained in $\loomif(\C)$. Since $\C$ is convex, there is a tetrahedron $t$ of $\C$ such that $\loomif(t) = \loomif(f)\cup \loomif(f')$. Furthermore, $e$ is the bottom diagonal of $t$. This contradicts that $f, f'$ are faces of the upper landscape of $\C$.
	\end{proof}
	
	Let $\C$ be a convex continent in $\univV$. For each face $f$ of the upper landscape $L$ of $C$ we define a sequence $r(f) = (f_i)_i$ of faces of $L$ as follows. First, set $f_1$ to be $f$. For $i\geq1$, if the widest edge $e_i$ of $f_i$ is a coastal edge of $\C$, we are done. Otherwise, set $f_{i+1}$ to be the unique face of $L$ that is adjacent to $f_i$ along $e_i$. We call $r(f)$ the \emph{upper river} originating at $f$. 
	
	\begin{lemma}\label{lem:rivers}
		All faces of the upper river $r(f)$ are distinct.
	\end{lemma}
	\begin{proof}
		For every $i$, let $e_i$ be the widest edge of $f_i$. Then the set of the stable leaves intersecting $\loomif(f_i)$ and the set of the stable leaves intersecting $\loomif(e_i)$ are equal. 
		By Lemma \ref{lem:widest:edges}, $\loomif(e_i)$ is not the widest edge of $\loomif(f_{i+1})$. Therefore every stable leaf that intersects $\loomif(f_i)$, intersects $\loomif(f_{i+1})$, but there are stable leaves intersecting $\loomif(f_{i+1})$ that do not intersect $\loomif(f_i)$. In other words, the sets of the stable leaves that intersect $\loomif(f_i)$ form a strictly ascending family. Thus all $f_i$'s are distinct.
	\end{proof}
	
	Since  the upper landscape contains only finitely many faces, Lemma \ref{lem:rivers} implies that $r(f)$ is finite. In particular, $r(f)$ is a special case of an upper river in the sense of \cite[Definition 4.12]{SchleimSegLinks}, where the notion is defined more generally, also for non-convex continents.

	\begin{lemma}\label{lem:bracketing:cusps}
		Let $\C$ be a convex continent in $\univV$. For every stable continental cusp leaf~$\ell$ there is a unique face $f$ of the upper landscape of $\C$ such that $\ell$ intersects the loomification of every face of the upper river $r(f)$. In particular, the cusps of the widest edge of the last face of $r(f)$ are the bracketing cusps for $\ell$. 
	\end{lemma}
	\begin{proof}
		Let $\ell$ be a stable continental cusp leaf for $\C$. 
		Let $\TRect(\C, \ell)$ be the set of all tetrahedron rectangles in $\loomif(\C)$ that contain some initial segment of $\ell$. This set must be nonempty, because  $\ell$ intersects $\loomif(\C)$ that is covered by tetrahedron rectangles, and these rectangles do not have ideal corners. Each tetrahedron rectangle from $\TRect(\C, \ell)$ meets a cusp $c \in \Omega(\C)$ from which $\ell$ originates and meets another cusp along each of its vertical sides. 
		Therefore, there exists an ordering $(R_1,\ldots,R_m)$ of the elements of $\TRect(\C,\ell)$ such that $R_i$ is strictly thinner than $R_j$ whenever $1 \le i < j \le m$. Convexity of $\C$ implies that for each $1 \le i \le m$ there exists a tetrahedron $t_i \subset \C$ with $\loomif(t_i)=R_i$.

		Each $R_i$ contains exactly two face subrectangles that contain an initial segment of $\ell$. However, only one of those, denoted by $Q_i$, is strictly thinner than $R_i$. The rectangle $Q_i$ is the loomification of a top face $q_i$ of $t_i$; see Figure \ref{fig:tet:rects}.
		
		Furthermore, $\C$ contains the tetrahedron rectangle $t_i'$ immediately above $q_i$ if and only if the maximal vertical expansion of $Q_i$, which is a tetrahedron rectangle strictly thinner than $R_i$, is in $\loomif(\C)$. It follows that $q_1$ is the unique face of the upper landscape of $\C$ whose loomification contains some initial segment of $\ell$.

		Let $r(q_1) = (f_i)_{i=1}^k$ be the upper river originating from $q_1$. As observed in the proof of Lemma \ref{lem:rivers}, for every $1 \leq i \leq k-1$ every stable leaf intersecting $\loomif(f_i)$, intersects $\loomif(f_{i+1})$. Since $\ell$ intersects $q_1 = f_1$, $\ell$ intersects the loomifications of all faces of the upper river $r(q_1)$. By the definition of $r(q_1)$, the widest edge $e_k$ of $f_k$ is a coastal edge of~$\C$. Thus $\ell$ must exit $\loomif(\C)$ between the cusps of $\loomif(e_k)$. 
	\end{proof}


	A result analogous to Lemma \ref{lem:bracketing:cusps} holds for unstable continental cusp leaves. These are encoded by \emph{lower rivers}, which are sequences of triangles contained in the lower landscape of $\C$, where each successive triangle is determined by the tallest edge of its predecessor. Lower rivers satisfy properties analogous to those of upper rivers established in Lemmas~\ref{lem:widest:edges} and~\ref{lem:rivers}, and 
	the bracketing cusps of an unstable leaf $\ell$ can be recovered by following the lower river encoding $\ell$. 
	
	\begin{proposition}\label{prop:conv:cont:circ:order}
		Let $\C\subset \univV$ be a convex continent. The loom-compatible circular orders \[\Theta^\pm:(\Omega(\C) \cup \E^\pm(\C))^3 \rightarrow \lbrace -1, 0, 1 \rbrace\] are uniquely determined by the veering structure on $\C$.
	\end{proposition}
\begin{proof} We prove the claim for $\Theta^+$; the proof for $\Theta^-$ is analogous.
	By Lemma \ref{lem:determining:circular:order}, $\Theta^+$ is uniquely determined by its restriction to triples from $\Omega(\C)$ together with the data on the bracketing cusps of every stable continental cusp leaf.  The former can be read off from the boundary of the upper landscape of $\C$. The bracketing cusps of a stable continental cusp leaf $\ell$ can be found by following the upper river associated to $\ell$ as described in Lemma~\ref{lem:bracketing:cusps}.
	
	By definition, an upper river is a finite sequence $(f_i)_{i=1}^m$  of faces of the upper landscape of $\C$ such that $f_{i+1}$ is adjacent to $f_i$ along an edge $e_i$ that is 
	the bottom diagonal of the tetrahedron immediately above $f_i$. Meanwhile, Definition~\ref{def:veering:tet} of a veering tetrahedron implies that $e_i$ can be identified solely from the colors of the edges of $f_i$.
	It follows that the veering structure on $\C$ uniquely determines upper rivers of $\C$ and therefore the bracketing cusps of stable continental cusp leaves. 
	\end{proof}

By part~(1) of Lemma \ref{lem:theta:l:function}, a continental cusp leaf $\ell$ intersects a simplex rectangle $R \subset \loomif(\C)$ if and only if there are two cusps $c,d$ of $R$ such that $\Theta_\ell(c)\cdot\Theta_\ell(d)=-1$. Proposition \ref{prop:conv:cont:circ:order} ensures that by working exclusively with convex continents we can always compute the values of $\Theta_\ell(c)$ and $\Theta_\ell(d)$, and hence determine whether $\ell$ intersects~$R$. Once two stable (respectively, unstable) continental cusp leaves are known to intersect $R$, Corollary \ref{cor:from:circ:to:total} allows us to determine their horizontal (respectively, vertical) order.

	The advantage of working with convex continents is now clear. Fortunately, restricting our attention to them does not limit us in any essential way, as the following result shows.

	\begin{lemma}\cite[Lemma 5.18]{SchleimSegLinks} \label{lem:exist:convex:continent}
		Let $\C$ be a continent in $\univV$. There exists a unique convex continent $\C' \subset \univV$ such that $\C \subset \C'$ and $\Omega(\C) = \Omega(\C')$. \qed
	\end{lemma}
Combining this with Lemma \ref{prop:continents:exists} gives that any finite set  of tetrahedra of $\univV$ is contained in a convex continent.

	\begin{remark}
		Convexity of a continent is also needed for \emph{channelization}, a procedure introduced by Schleimer and the second author in \cite[Section 5.19]{SchleimSegLinks}. Channelization is key for  constructing a continent that contains a given finite set of tetrahedra of $\univV$.  Since our algorithm requires certain tetrahedra to be in the continent (see Figure \ref{fig:flowchart}), this gives another reason to work exclusively with convex continents.
	\end{remark}
	
	\section{Subset adapted to a periodic orbit}\label{sec:adapted:set}

	Recall that to construct the veering parent $\V_\gamma$ using Proposition \ref{prop:finite:veering:parent} one first has to find all points of~$\liftGamma$  inside each fundamental tetrahedron rectangle in $(\loom, \loom^+, \loom^-)$. As explained in Section~\ref{subsec:flow:lines:and:orbits}, each element of $\liftGamma$ can be realized as the intersection of infinitely many tetrahedron rectangles that make up a flow line in~$\loom$.  However, in the drilling algorithm we work only with finite subtriangulations of $\univV$ (specifically, convex continents), instead of the entire $\univV$. The drawback of this approach is that for a given convex continent~$\C$ it may happen that some fundamental tetrahedron rectangle $R$ contains some $x \in  \liftGamma$, but a flow line in $\univV$ that represents~$x$ is disjoint from $\C$. In this case, we would fail to detect $x$ in $R$ using $\C$ and thus the output of the algorithm would not be correct. 
	
	To avoid this we need to ensure that the continents that we work with are `sufficiently large'. The main difficulty here is that we do not know a priori how many points of $\liftGamma$ lie in each fundamental tetrahedron rectangle or where they are located. Therefore starting from some convex continent $\C$ in $\univV$ it is not immediately clear in which directions $\C$ should be expanded so that no relevant flow line is missed.
	
	Observe that given $x \in \liftGamma$, up to the action of the stabilizer of $x$, there are only finitely many edge rectangles that contain $x$. Indeed if that was not the case, there would be an edge rectangle that contains infinitely many elements of $\liftGamma$ which is impossible, as $\liftGamma$ is a closed and discrete subset of $\loom$. 
	The challenge is to find such a finite collection of representatives explicitly. Once it is found, it can be used to find all elements of $\liftGamma$ inside any edge rectangle in $\loom$, and thus inside any tetrahedron rectangle in $\loom$; see Lemma \ref{lem:adapted}.
	
	In this section, we use basic properties of \emph{staircases} (Definition~\ref{def:staircase}) to explicitly construct a subset of $\loom$ that is covered by finitely many edge rectangles and, for some $x \in \liftGamma$, contains, up to the action of the stabilizer of $x$, every edge rectangle containing~$x$.
	In the drilling algorithm, this set serves as a `roadmap', directing us to the regions of $\univV$ that need to be included in the continent $\C$ in order to detect all points of $\liftGamma$ inside a fixed collection of fundamental tetrahedron rectangles; see Sections \ref{subsec:adapted:continent} and \ref{subsec:translate}.

	\subsection{Subset of the loom space adapted to a periodic orbit}
	Recall from Proposition~\ref{prop:orbit:space}~\eqref{orbit_space:fixed:points} that the stabilizer $\Stab(x)$ of a periodic point $x \in \liftGamma$ is cyclic. Furthermore, Lemma \ref{lem:loom:stabilizers} implies that given any rectangle $R$ containing~$x$, the corners of $R$ are not fixed by any nontrivial element of $\Stab(x)$. Therefore, by Proposition~\ref{prop:orbit:space} \eqref{orbit_space:group_action}, we can pick a generator $\gamma_x$ of $\Stab(x)$ such that $\gamma_x(R)$ is strictly taller than $R$ and $R$ strictly wider than $\gamma_x(R)$. We call  $\gamma_x$ the \emph{positive generator} of $\Stab(x)$.

	\begin{definition}\label{def:adapted:set}
		Fix $x \in \liftGamma$. We say that a subset $A \subset \loom$ is \emph{adapted to $x$} if every edge rectangle in $\loom$ that contains $x$ has a $\langle \gamma_x \rangle$-translate in $A$. We say that $A$ is \emph{adapted to $\gamma$} if~$A$ is adapted to $x$ for some $x \in \liftGamma$.
	\end{definition}
	
	For example, the union of all edge rectangles in $\loom$ that contain $x$ is a subset of $\loom$ that is adapted to $x$.  However, since every edge rectangle contains only finitely many elements of $\liftGamma$, there are also subsets of $\loom$ that are adapted to $x$ and covered by finitely many edge rectangles.
	Our goal in that section is to explicitly construct such a subset.
	
	Subsets adapted to $\gamma$ capture all information about $\liftGamma$ in the following sense. 
	
	\begin{lemma}\label{lem:adapted}
		Suppose that $A \subset \loom$ is adapted to some $x \in \liftGamma$. Then for every edge rectangle $P \subset \loom$, and every $y \in P \cap \liftGamma$ there is $g \in \pi_1(M)$ such that $x \in g(P)$ and $g(P) \subset A$. 
	\end{lemma}
	\begin{proof}
		Let $y \in P \cap \liftGamma$. Then $y = h(x)$ for some $h \in \pi_1(M)$. Thus $x \in h^{-1}(P)$. Since $A$ is adapted to $x$, there exists $m \in \zz$ such that $\gamma_x^m h^{-1}(P) \subset A$. Given that $\gamma_x$ fixes $x$, we also have $x \in \gamma_x^m h^{-1}(P)$. Thus $g =  \gamma_x^m h^{-1}$ translates $P$ into $A$ and $y$ to $x$.
	\end{proof}

By Lemma \ref{lem:adapted}, once a finitely covered set $A$ adapted to $x$ has been found, the problem of determining $P \cap \liftGamma$ for any edge rectangle $P$ reduces to a finite search. Namely, we can iterate over the finitely many translates of $P$ contained in $A$, identify those that contain~$x$, and translate $x$ from each of them back into $P$.

	\subsection{Staircases}\label{subsec:staircases}
	
In this section we recall certain properties of $\loom$ that we use to construct a subset of $\loom$ that is adapted to $\gamma$ and covered by finitely many edge rectangles.
	
	Let $x \in \loom \cup \Omega$. Recall from Section \ref{subsec:loom} that there are either four or infinitely many orthants at $x$, depending on whether $x \in \loom$ or $x \in \Omega$. When $x \in \loom$, an orthant at $x$ is also called a quadrant. Each orthant determines a unique \emph{staircase} at $x$, defined as below and illustrated in Figure \ref{fig:staircase}.
	
	\begin{definition}\label{def:staircase}
		Let $O$ be an orthant at $x \in \loom\cup \Omega$. A \emph{staircase} at $x$ is the closure of the union of all rectangles $P \subset O$ with one corner at $x$.
	\end{definition}

	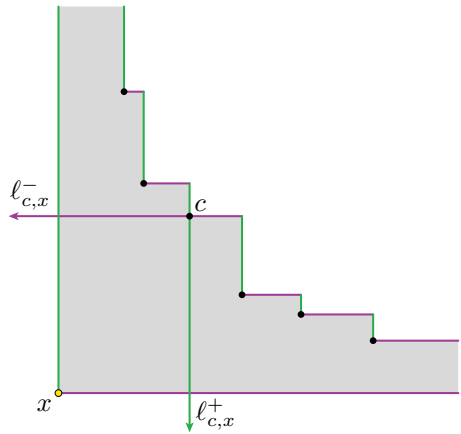
\begin{figure}[ht]
	\centering
	\labellist
	\hair 2pt \small
	\pinlabel $x$ [tr] at 65 50
	\pinlabel $c$ [bl] at 250 300
	\pinlabel $\ell^-_{c, x}$ [b] at 30 300
	\pinlabel $\ell^+_{c, x}$ [l] at 250 25
	\endlabellist
	\includegraphics[width = 0.4\textwidth]{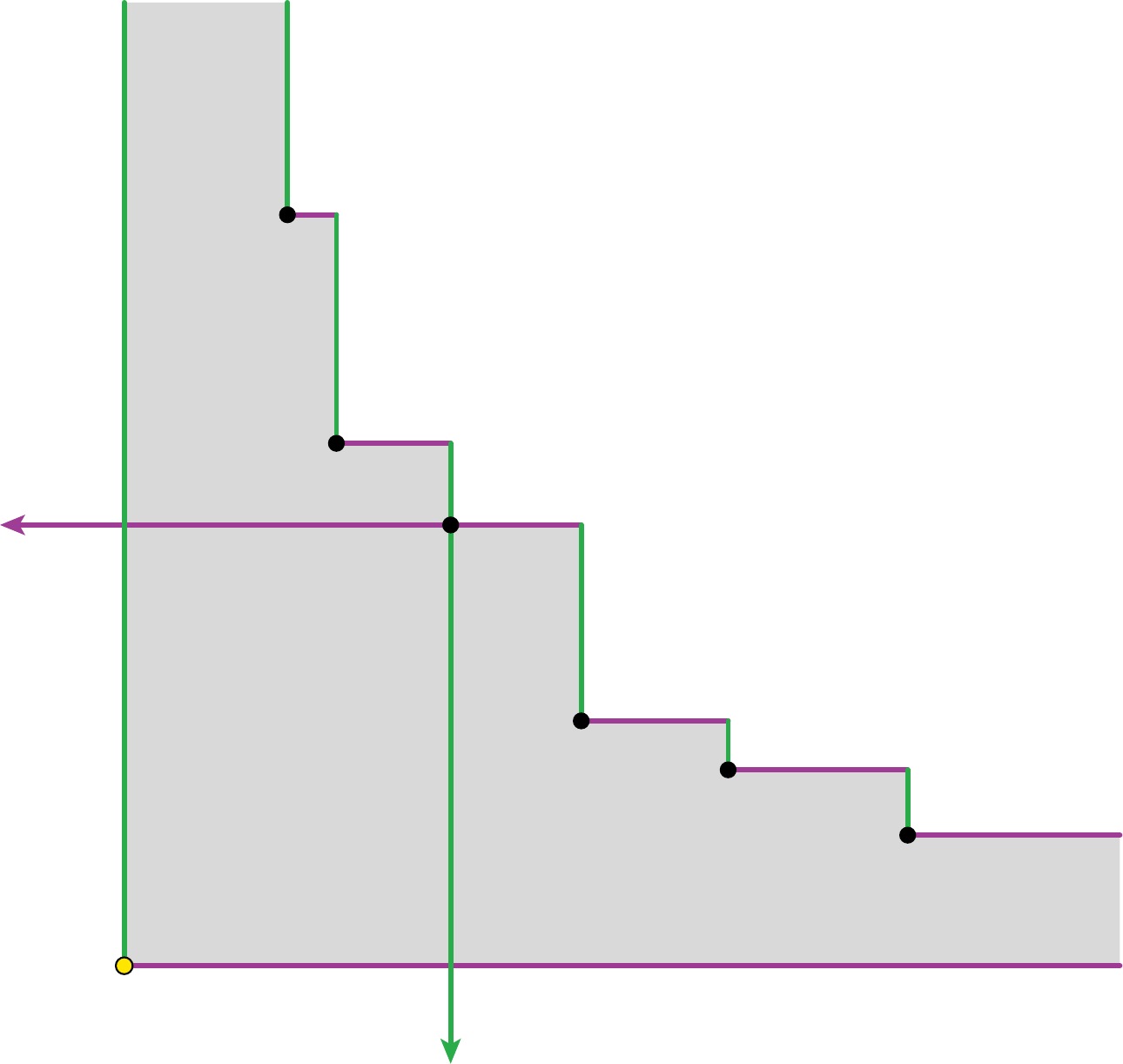}
	\caption{A part of a staircase $S$ at $x$. The cusps belonging to $\Omega_x(S)$ are represented by solid black dots. For each $c \in \Omega_x(S)$, there is precisely one stable $c$-leaf $\ell_{c,x}^+$ and precisely one unstable $c$-leaf $\ell_{c,x}^-$ intersecting the interior of $S$.}
	\label{fig:staircase}
\end{figure}
	
	
	Let $S$ be a staircase at $x \in \loom \cup \Omega$. Set $\Omega_x(S) = \Omega(S) \bez \lbrace x \rbrace$. When $x \in \loom$, $\Omega_x(S) = \Omega(S)$. Each $c \in \Omega_x(S)$ determines a unique rectangle $R_c \subset S$ with corners at $x$ and~$c$. Observe that $R_c \cap R_d \neq \emptyset$ for any $c,d \in \Omega_x(S)$. In particular, we can compare the height or width of $R_c, R_d$ in the sense of Definition~\ref{def:taller:etc}.  We define a \emph{dynamic linear order} $<_S$ on $\Omega_x(S)$ by declaring, for $c,d \in \Omega_x(S)$, that $c <_S d$ if  the rectangle $R_d$ is strictly taller than $R_c$.

	\begin{lemma}\label{lem:staircase}
		Let $S$ be a staircase at a periodic point $x \in \loom$. Denote the positive generator of the stabilizer of $x$ by $\gamma_x$, and the quadrant at $x$ that contains the interior of $S$ by $Q_S$.
		\begin{enumerate}
				\item Let $\ell^+$, $\ell^-$ be the half-leaves at $x$ that cobound $Q_S$. Arbitrarily close to $x$ there is 
				\begin{itemize}
					\item a point of $\ell^+$ that lies on a cusp leaf of $\loom^- $ meeting $\Omega(S)$, and
					\item a point of $\ell^-$ that lies on a cusp leaf of $\loom^+$ meeting $\Omega(S)$. 
				\end{itemize}
			In particular, $S$ meets infinitely many cusps of $\loom$.
			\item The stabilizer of $S$ is equal to $\langle \gamma_x^2 \rangle$, if $\gamma_x$ acts with a twist on $\loom$, and to $\langle \gamma_x \rangle$ otherwise.
			\item  The group $\langle \gamma_x^2 \rangle$ acts on $\Omega(S)$ by translations with respect to the dynamic linear order on $\Omega(S)$.
		\end{enumerate}
	\end{lemma}
	\begin{proof}
			If part (1) was false, there would be a rectangle $R \subset \loom$ that is not contained in any maximal rectangle. However, we have already proved in Lemma \ref{lem:expansions} that $\loom$ does not admit unbounded product regions. See also \cite[Lemma 4.10]{SchleimSegLoom}. 
		
		For (2) observe that any element stabilizing $S$ must fix $x$ and thus $\Stab(S) \subset \langle \gamma_x \rangle$. Conversely, for any edge rectangle $P\subset S$ and every $m \in \zz$ the rectangle $\gamma_x^m(P)$ lies in some staircase at $x$. If $\gamma_x$ does not act with a twist then $\gamma_x^m(P) \subset S$. Otherwise, $\gamma_x^{2m}(P) \subset S$. 
		
		Let $c \in \Omega(S)$. Since $R_{\gamma_x^2(c)} = \gamma_x^2(R_c)$, the group $\langle \gamma_x^2 \rangle$ preserves the linear order $<_S$ on $\Omega(S)$. Furthermore, by the assumption that $\gamma_x$ is the positive generator of $\Stab(x)$, $\gamma_x^2(R_c) = R_{\gamma_x^2(c)}$ is strictly taller than $R_c$. Therefore no nontrivial element of $\langle \gamma_x^2 \rangle$ has fixed points on $\Omega(S)$. This gives (3).
	\end{proof}

	\subsection{Finitely covered subset adapted to $\gamma$}\label{subsec:adapted:set}
	
	Let $x \in \liftGamma$. Denote the union of staircases at $x$ by $\Sigma_x$, and the set of cusps meeting $\Sigma_x$ by $\Omega_x$. For every $c \in \Omega_x$, exactly one stable $c$-leaf $\ell_{c,x}^+$ intersects the interior of $\Sigma_x$ and exactly one unstable $c$-leaf $\ell_{c,x}^-$ intersects the interior of $\Sigma_x$; see Figure \ref{fig:staircase}. By Lemma \ref{lem:staircase} (1), for any pair of distinct cusps $c,d \in \Omega_x$  there always exists an edge rectangle $P \subset \Sigma_x$ that intersects both $\ell_{c,x}^+$ and $\ell_{d,x}^+$. Thus the leaves $\ell_{c,x}^+, \ell_{d,x}^+$ can be horizontally ordered as in Section~\ref{subsec:linear:orders}. This horizontal order (say, from west to east) on $\lbrace \ell_{c,x}^+ \ | \ c \in \Omega_x \rbrace$ induces a horizontal order on $\Omega_x$. Analogously, we also get a vertical order on $\Omega_x$ induced by  the vertical order on $\lbrace \ell_{c,x}^- \ | \ c \in \Omega_x \rbrace$. Thus it makes sense to say that one cusp of $\Omega_x$ lies to the east or west of the other, and likewise to the north or south.
Each $c \in \Omega_x$ also has a well-defined horizontal and vertical position relative to $x$, which depends only on the staircase at $x$ that $c$ meets. Hence we may speak about the horizontal and vertical order on $\Omega_x \cup \lbrace x \rbrace$.

\begin{definition}\label{def:separator:cusp}
	Let $c$ be a cusp that meets some staircase $S$ at $x$. We say that a cusp $s$ is an \emph{external horizontal $(c,x)$-separator} (respectively, \emph{external vertical $(c,x)$-separator})~if
	\begin{itemize}
		\item $s$ meets a staircase at $x$ that is adjacent to $S$,
		\item $s$ is between $c$ and $x$ in the horizontal (respectively, vertical) order on $\Omega_x \cup \lbrace x \rbrace$.
	\end{itemize} See Figure \ref{fig:separators}.
\end{definition}

	\begin{figure}[h]
	\centering
	\labellist
	\hair 2pt \small
	\pinlabel $x$ [tr] at 548 526
	\pinlabel $c$ [bl] at 763 700
	\endlabellist
	\includegraphics[width = 0.5\textwidth]{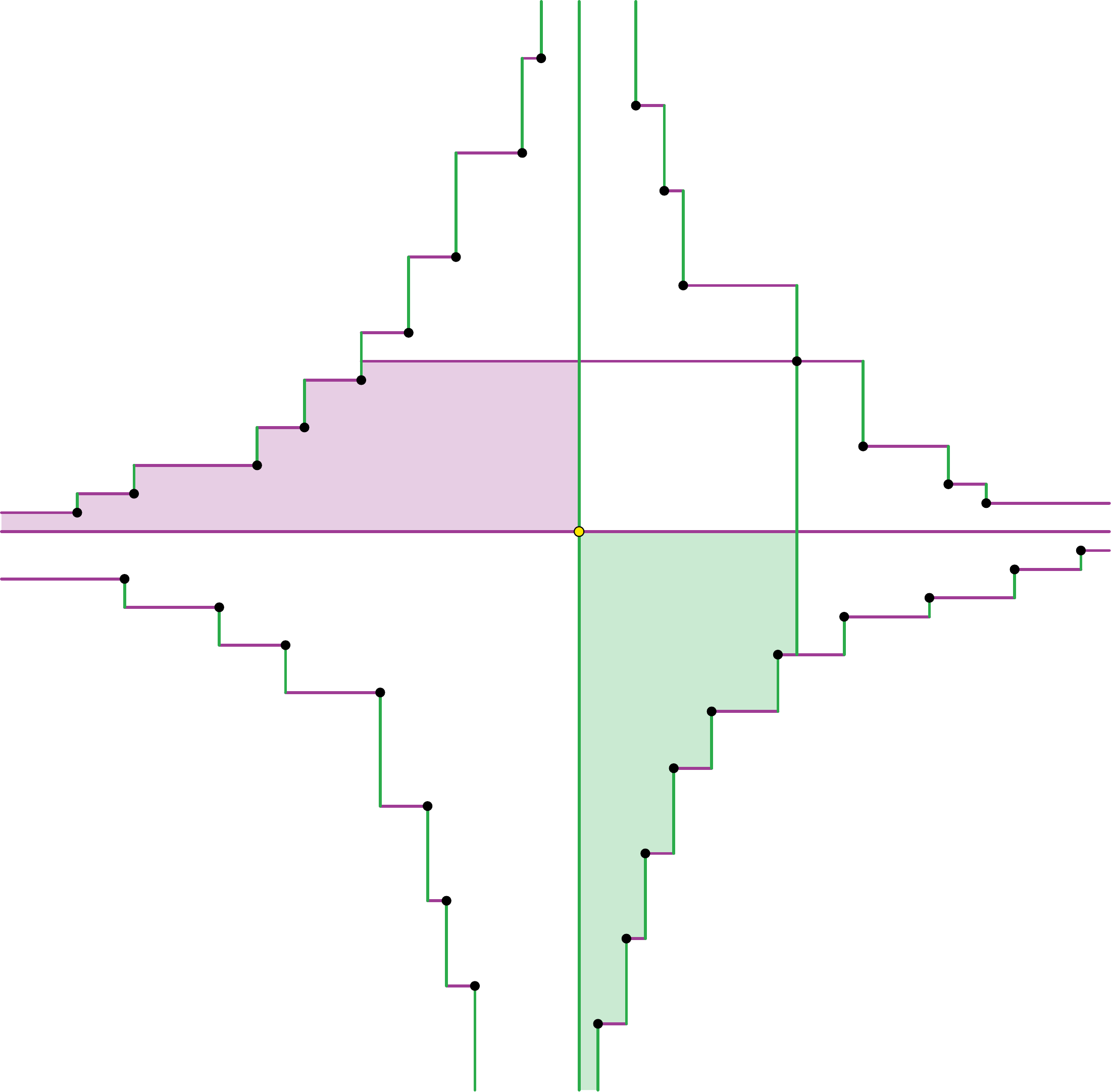}
	\caption{Every cusp meeting the region shaded green is an external horizontal $(c,x)$-separator. Every cusp meeting the region shaded purple is an external vertical $(c,x)$-separator.}
	\label{fig:separators}
\end{figure}

The adjective `external' in Definition \ref{def:separator:cusp} refers to the fact that $s$ is in a different staircase at $x$ than $c$. A cusp $d$ that meets the same staircase $S$ at $x$ as $c$, and lies further along in the dynamic linear order on $S$, could be viewed as an `internal' horizontal $(c,x)$-separator, because $\ell_{d,x}^+$ separates $c$ from $x$. However, to avoid any confusion, we will not use this terminology.

\begin{remark}\label{rem:infinitely:many:seps}
Lemma \ref{lem:staircase} (1) implies that for any $c \in \Omega_x$ there are infinitely many external horizontal $(c,x)$-separators and infinitely many external vertical $(c,x)$-separators; see Figure \ref{fig:separators}.
\end{remark}

	Denote the north-east, nort-west, south-west, and south-east staircase at $x$ by $\SNE, \SNW, \SSW$, and $\SSE$, respectively.
	Observe that every red edge rectangle containing $x$ has ideal corners in $\SSW$ and $\SNE$, while every blue edge rectangle containing $x$ has ideal corners in $\SNW$ and $\SSE$. Furthermore, $x$ is contained in infinitely many red edge rectangles, and infinitely many blue edge rectangles. 
	
	Let $a$ be a cusp that meets $\SSW$. Fix $m \in \mathbb{N}$ such that $\gamma_x^m(a)$ meets $\SSW$. By Lemma~\ref{lem:staircase}~(2), any even $m \in \mathbb{N}$ works.  Pick an external horizontal $(\gamma_x^m(a),x)$-separator  $s^+_{a,m,x}$, and an external vertical $(a, x)$-separator $s^-_{a,0,x}$; see Figure \ref{fig:hull:red:edges}. By Lemma~\ref{lem:staircase} (3), $s^+_{a,m,x}$ is also an external horizontal $(a,x)$-separator, and $s^-_{a,0,x}$ is also an external vertical  $(\gamma_x^2(a),x)$-separator.
		Set
	\[A_0= \hull\left(a, \gamma_x^m(a), s^+_{a,m,x},s^-_{a,0,x}\right).\]
		Then $A_0$ is convex by \cite[Lemma 6.21]{SchleimSegLoom} and covered by finitely many edge rectangles by \cite[Corollary 6.26]{SchleimSegLoom}.
	
	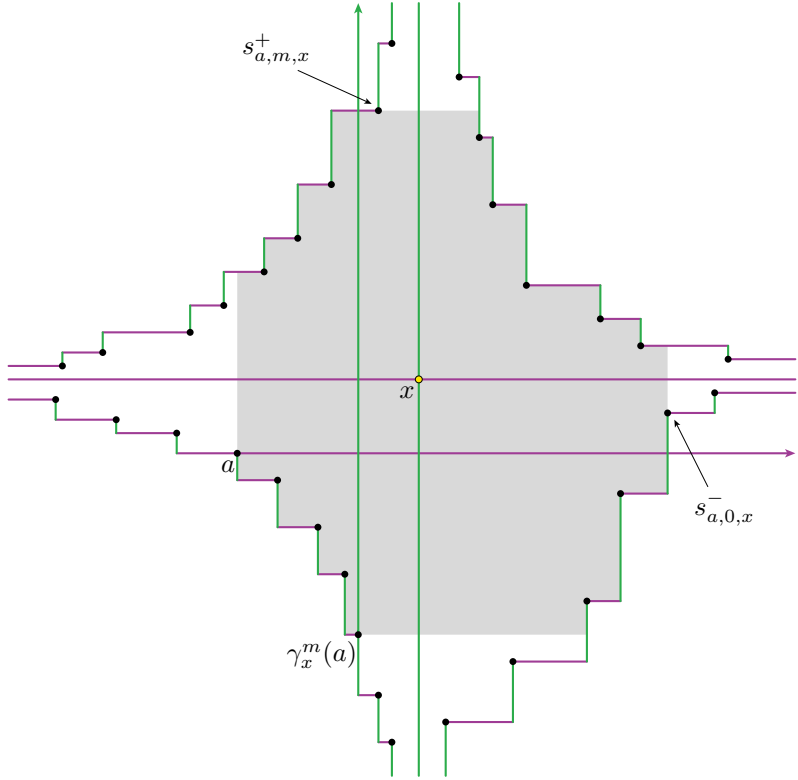
\begin{figure}[ht]
		\centering
		\labellist
		\hair 2pt \small
		\pinlabel $x$ [tr] at 550 530
		\pinlabel $a$ [tr] at 310 430
		\pinlabel $\gamma_x^m(a)$ [tr] at 473 187
		\pinlabel $s^+_{a,m,x}$ [b] at 360 950
		\pinlabel $s^-_{a,0,x}$ [t] at 960 390
		\endlabellist
		\includegraphics[width = 0.7\textwidth]{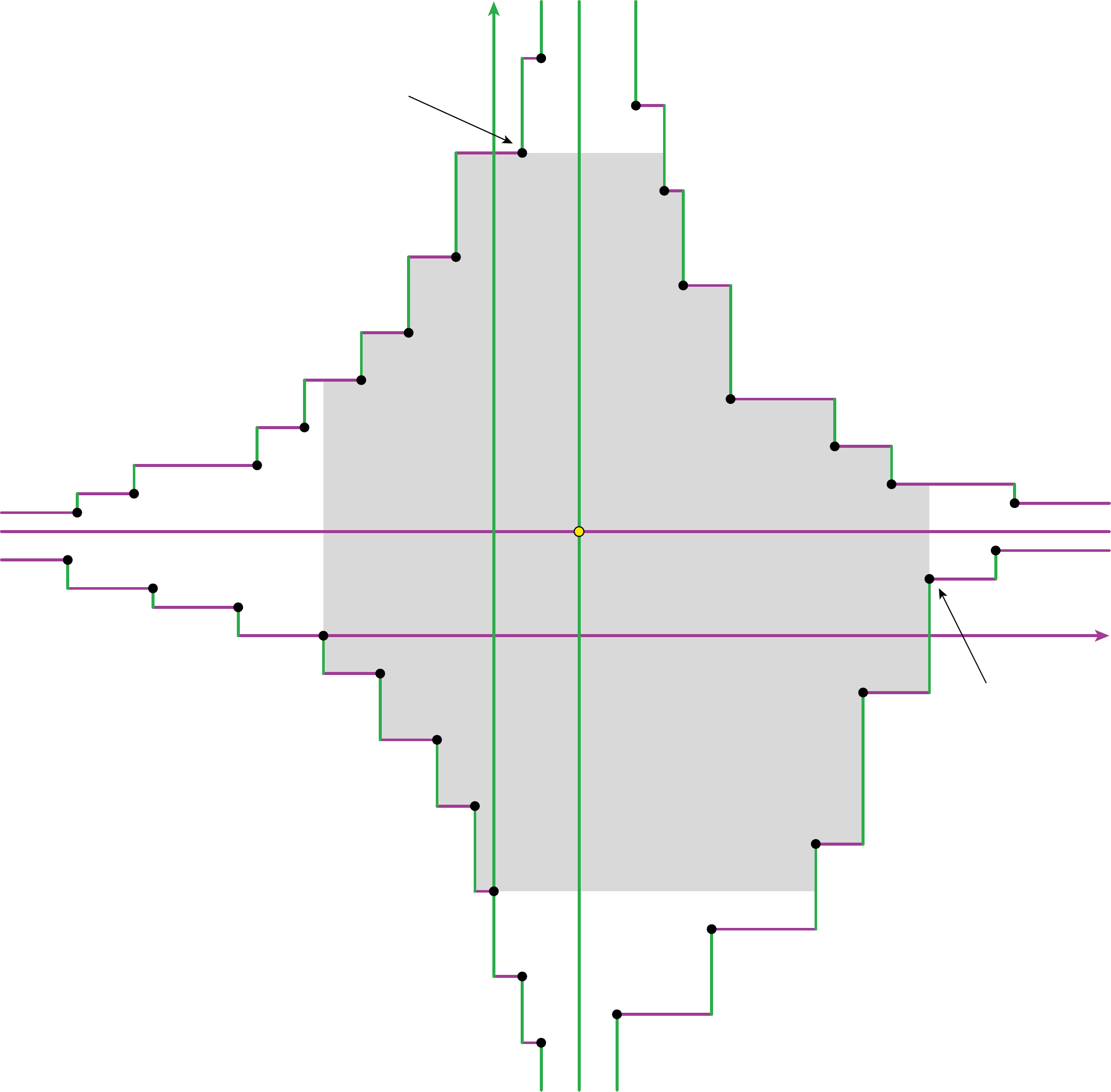}
		\caption{The set $A_0$, shaded gray, is the convex hull of $\lbrace a, \gamma_x^m(a), s^+_{a,m,x}, s^-_{a,0,x}\rbrace$. There are infinitely many possible choices of $s^+_{a,m,x}$ and $s^-_{a,0,x}$ for a given \mbox{$a \in \Omega(\SSW)$} and $m\in \mathbb{N}$ such that $\gamma_x^m(a) \in \Omega(\SSW)$; the ones indicated are those that yield the smallest possible convex hull.}
		\label{fig:hull:red:edges}
	\end{figure}

	\begin{lemma}\label{lem:red:edge:rects}
		Every red edge rectangle that contains $x$ has a $\langle \gamma_x \rangle$-translate in $A_0$.
	\end{lemma}
	
	\begin{proof}
		Let $P$ be any red edge rectangle that contains $x$. Then $P$ meets a cusp of $\SSW$ and a cusp of $\SNE$. By Lemma \ref{lem:staircase} (1) and (3), there exists $n \in \zz$ such that the unique cusp $w$ of $\gamma_x^n(P)\cap \SSW$ meets the geodesic contained in $\partial \SSW$ connecting $a$ and $\gamma_x^m(a)$. The other cusp $z$ of $\gamma_x^n(P)$ certainly meets $\SNE$. 
		Since $\gamma_x^n(P)$ is  the interior of $\hull(w,z)$ and $A_0$ is convex, to show that $\gamma_x^m(P) \subset A_0$ it suffices to show that $A_0$ meets $z$. 
		
	Observe that $A_0$ can miss $z$ only if $s^+_{a,m,x}$ is an external vertical $(z,x)$-separator or if $s^-_{a,0,x}$ is an external horizontal $(z,x)$-separator. In the first case, $z$ and $w$ are contained in (the interiors of) distinct and non-adjacent orthants at $s^+_{a,m,x}$. By  Lemma \ref{lem:geodesic:between}, there is a geodesic between $z$ and $w$ that passes through $s^+_{a,m,x}$. It follows that the interior of $\hull(z,w)$ is not an edge rectangle. This contradicts that $z$ and $w$ span $\gamma_x^n(P)$. The case when $s^-_{a,0,x}$ is an external horizontal $(z,x)^+$-separator is analogous. 
	\end{proof}

To simplify notation, set $b = s^+_{a,m,x}$. Fix an external horizontal $(b,x)$-separator $s^+_{b,0,x}$, and an external vertical $(\gamma_x^{-m}(b),x)$-separator $s^-_{b,-m,x}$. 
	Furthermore set 
	\begin{equation*}
	A = \hull\left(a, \gamma^m_x(a), b = s^+_{a,m,x}, s^-_{a, 0,x}, \gamma_x^{-m}(b), s^+_{b,0,x}, s^-_{b,-m,x}\right).\end{equation*}
	\begin{lemma}\label{lem:adapted:set}
		$A$ is adapted to $x$ and covered by finitely many edge rectangles.
	\end{lemma}
	\begin{proof}
		The fact that $A$ is covered by finitely many edge rectangles follows from the fact that the convex hull of finitely many cusps of $\loom$ meets only finitely many cusps of $\loom$  \cite[Corollary 6.26]{SchleimSegLoom}.
		
		Because $A_0 \subset A$, Lemma \ref{lem:red:edge:rects} implies that every red edge rectangle containing $x$ has a $\langle \gamma_x \rangle$-translate in $A$.
		Set
			\[A_1 = \hull(b, \gamma_x^{-m}(b), s^+_{b,0,x}, s^-_{b,-m,x}).\]
			The proof that every blue edge rectangle that contains $x$ has a $\langle \gamma_x \rangle$-translate in $A_1$ is completely analogous to the proof of Lemma \ref{lem:red:edge:rects}. Since furthermore $A_1 \subset A$, the same holds for $A$.
			As every edge rectangle is either red or blue, $A$ is adapted to~$x$.
	\end{proof}
	
		In the construction above we chose a cusp $a$ meeting the south-west staircase at $x$ as our starting point. An analogous construction can be carried out starting from any cusp meeting any of the staircases at $x$, with the roles of red and blue edges possibly reversed. To describe these variants precisely, we say that a subset $A\subset\loom$ is \emph{red-adapted to $x$} if every red edge rectangle containing $x$ has a $\langle\gamma_x\rangle$-translate in $A$. We define a set that is \emph{blue-adapted to $x$} analogously.
		
	\begin{proposition}\label{prop:adapted:general}
		Let $x \in \liftGamma$. For any $a \in \Omega_x$ and $m \in 2\mathbb{N}$, and any choice of the separators $b = s^+_{a,m,x}, s^-_{a,0,x}$, $s^+_{b,0,x}, s^-_{b,-m,x}$, the convex hull of the set
			\[
		W = \lbrace a, \gamma_x^m(a), b = s^+_{a,m,x}, s^-_{a,0,x}, \gamma_x^{-m}(b), s^+_{b,0,x}, s^-_{b,-m,x}\rbrace
		\]
		is adapted to $x$ and covered by finitely many edge rectangles.
	\end{proposition}
	\begin{proof} By Lemma \ref{lem:staircase} (2), if $m$ is even then $\gamma_x^m(a)$ must meet the same staircase at $x$ as~$a$. 
		Set $A_0 = \hull(a, \gamma_x^m(a), s^+_{a,m,x}, s^-_{a, 0, x})$. Lemma \ref{lem:red:edge:rects} establishes that if $a$ meets the south-west staircase at $x$ then $A_0$ is red-adapted to $x$. Observe that the same conclusion holds when $a$ meets the north-east staircase at $x$. On the other hand, when $a$ meets either the north-west or the south-east staircase at $x$ then $A_0$ is blue-adapted to $x$.
		
		Since $b = s^+_{a,m,x}$ is a horizontal $(\gamma_x^m(a),x)$-separator (see Definition \ref{def:separator:cusp}), it meets a staircase at $x$ that is adjacent to the staircase at $x$ that meets $a$. It follows that $A_1 =  \hull(b, \gamma_x^{-m}(b), s^+_{b,0,x}, s^-_{b,-m,x})$ is blue-adapted to $x$ when $a$ meets the south-west or the north-east staircase at $x$, and red-adapted to $x$ when $a$ meets the north-west or the south-east staircase at $x$. In either case, we get that $\hull(W)$ contains a subset that is red-adapted to $x$ and a subset that is blue-adapted to $x$. Thus $\hull(W)$ is adapted to~$x$.
\end{proof}

	\section{Drilling veering triangulations}\label{sec:algorithm}

	In this section we finally describe all steps in the construction of a veering parent of $\V$, taking as input $\V$ and a flow cycle $\zeta$ of $\V$. Recall from Theorem~\ref{thm:flow:lines} that $\zeta$ represents a closed orbit~$\gamma$ of the flow encoded by $\V$. If $\zeta$ is non-peripheral, the output is the veering parent~$\V_\gamma$. Otherwise, the algorithm simply returns $\V$; see Theorem \ref{thm:algorithm:correct}. 
	
	We begin by outlining the steps of the algorithm in the flowchart shown in Figure~\ref{fig:flowchart}.  Since throughout the algorithm we work only with the flow cycle $\zeta$, rather than with the orbit $\gamma$ that it encodes, we write $\liftZeta$ instead of $\liftGamma$ (which is justified by Theorem~\ref{thm:flow:lines}), and $\V_\zeta$ instead of $\V_\gamma$. We also speak about a continent adapted to $\zeta$ rather than to~$\gamma$, \mbox{$\zeta$-simplex} rectangles instead of $\gamma$-simplex rectangles, and   $\zeta$-(tetrahedron, face) pairs instead of \mbox{$\gamma$-(tetrahedron, face)} pairs. In this section, $k$ always denotes the length of $\zeta$.

	\begin{figure}[h]
		\centering
		\includegraphics[width=\textwidth]{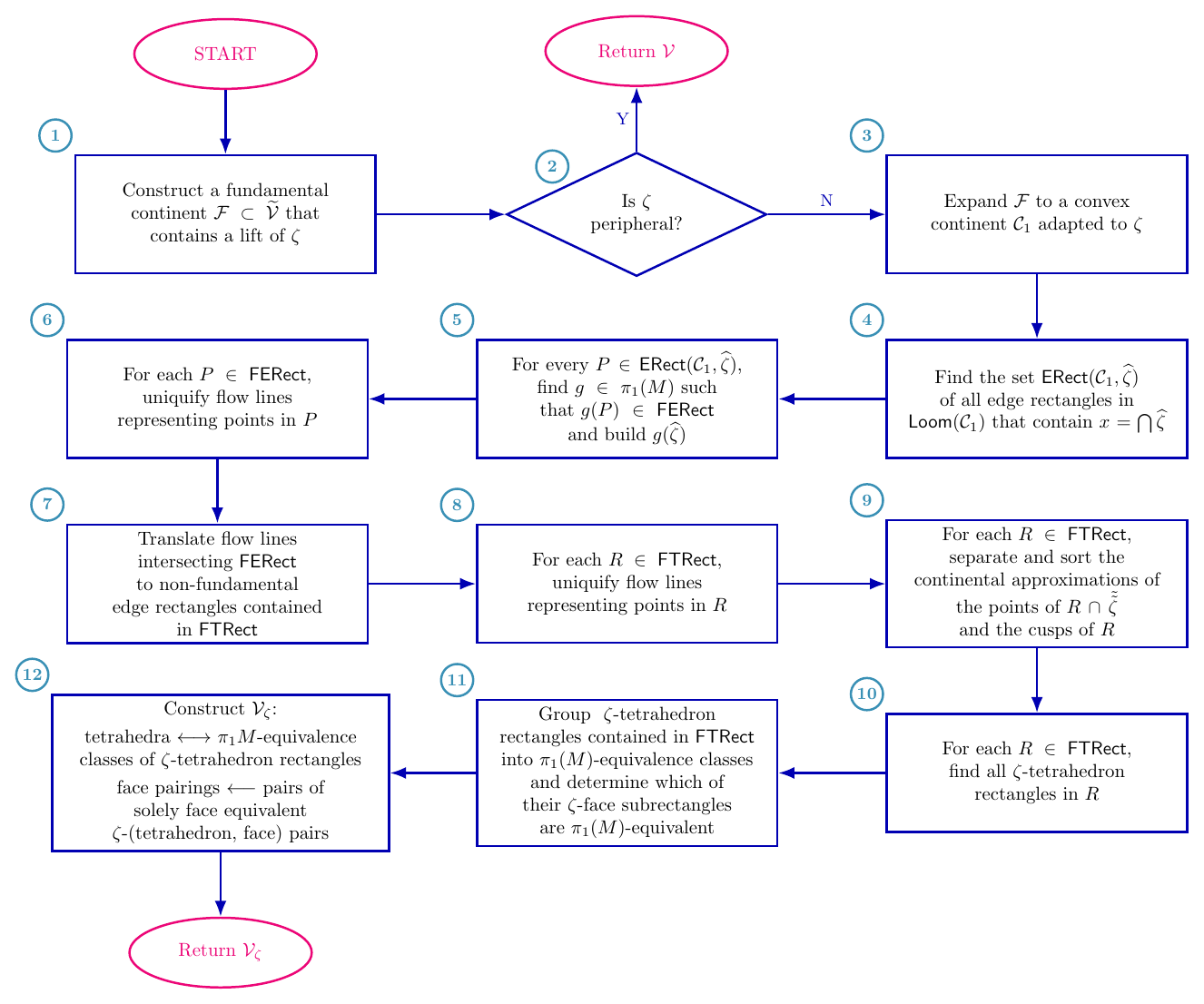}
		\caption{Steps in the construction of a veering parent of $\V$, given $\V$ and a flow cycle $\zeta$ of $\V$ as input. $\FERect$ and $\FTRect$ denote the sets of fundamental edge rectangles and fundamental tetrahedron rectangles, respectively.}
			\label{fig:flowchart}
	\end{figure}

	A \emph{fundamental continent} $\F \subset \univV$ mentioned in \textbf{Step 1}  means a continent that contains a fundamental domain for the action of $\pi_1(M)$ on $\univV$. Such a continent is guaranteed to intersect some elevation $\eleZeta$ of $\zeta$. By Theorem \ref{thm:flow:lines}, there is a periodic point $x \in \loom$ such that $x = \bigcap\eleZeta$. This is the point to which we will be adapting the continent in later stages of the algorithm. 
	
	More generally, given $y \in \liftZeta$ by $\eleZeta_y$ we denote a flow line that is an elevation of $\zeta$ and satisfies $\bigcap \eleZeta_y = y$. In particular, $\eleZeta = \eleZeta_x$. This notation is slightly  abusive, because there may be more than one  elevation of $\zeta$ that represents $y$; see Proposition \ref{prop:square} and the discussion below~it, as well as Section \ref{subsec:fellow-travel}.
	
	The loomifications of the tetrahedra of the fundamental domain contained in $\mathcal{F}$ will form the set of fundamental tetrahedron rectangles. We call the widest edges of the fundamental tetrahedron rectangles the \emph{fundamental edge rectangles}. In Figure \ref{fig:flowchart}, the sets of these simplex rectangles in $\loom$ are  denoted by $\FTRect$ and $\FERect$, respectively.
	
	The fundamental continent is repeatedly \emph{expanded} in later stages of the algorithm. By expanding a continent $\C$ we mean constructing a continent $\C' \subset \univV$ such that $\C \subset \C'$. Each expansion is uniquely determined by a finite set $X$ of tetrahedra in $\univV \bez \C$ to be included in $\C'$, and $\C'$ is the smallest convex continent containing $\C \cup X$. The existence of such a continent follows from Proposition \ref{prop:continents:exists} and Lemma \ref{lem:exist:convex:continent}. We refer the reader to \cite[Sections 5.15 and 5.19]{SchleimSegLinks} for details.
	
	Typically, every step between Step 2 through Step 9 requires an expansion of the continent. Each time we specify which additional tetrahedra are to be added, keeping in mind that the requirement to work exclusively with convex continents, discussed in Section \ref{subsec:convexity}, forces many more to be included. 
	
	We do not keep track how many times the continent is expanded, nor by how many tetrahedra. In particular, after Step 3 any continent that is an expansion of $\C_1$ is simply denoted by $\C$. 
	
	In \textbf{Step 2} we check if $\zeta$ is \emph{peripheral}. Peripheral flow cycles cannot be used to construct any nontrivial veering parent of $\V$. If $\eleZeta = (R_i)_{i\in \zz}$ is indexed so that $R_0 \subset \loomif(\F)$ then Step~2 expands $\F$  to a continent $\C_0$ that contains both $\loomif^{-1}(R_0)$ and $\loomif^{-1}(R_k)$, where~$k$ denotes the length of $\zeta$. Further details of how the check of peripherality is carried out using $\C_0$ can be found in Section~\ref{subsec:peripheral}. If $\zeta$ is peripheral, the algorithm terminates and returns $\V$ unchanged. Otherwise, we proceed to Step 3.  
	
	In \textbf{Step 3} we expand $\C_0$ to a convex continent $\C_1$ that is \emph{adapted} to $x = \bigcap \eleZeta$, meaning that $\loomif(\C_1)$ is adapted to $x$ in the sense of Definition \ref{def:adapted:set}. To accomplish this, we use the ideas introduced in Section~\ref{subsec:adapted:set}, but the precise subroutine requires further explanation. This appears in Section~\ref{subsec:adapted:continent}.

	In \textbf{Step 4} the initial adapted continent $\C_1$ is used to find the set $\ERect(\C_1, \eleZeta)$ of all edge rectangles in $\loomif(\C_1)$ that contain $x = \bigcap \eleZeta$. 
	This set is used in \textbf{Step 5} to translate the flow line $\eleZeta$ to flow lines that represent points of $\liftZeta$ contained inside the fundamental edge rectangles. The process of translating flow lines is explained in Section~\ref{subsec:translate}.


	It may happen that two distinct translates of the same flow line represent the same point of the loom space. In \textbf{Step 6} we \emph{uniquify} the flow lines representing the points of $\liftZeta$ that lie inside fundamental edge rectangles. This means keeping only one flow line  for each such point.
	The details of the procedure are described in Section~\ref{subsec:fellow-travel}.  Without uniquifying flow lines, the later Step~9 of the algorithm might fail to terminate.
	
	In \textbf{Step 7} we translate flow lines representing points of $\liftZeta$ inside the fundamental edge rectangles (found in Step 5 and uniquified in Step~6) into flow lines representing points of $\liftZeta$ inside non-fundamental edge rectangles contained in fundamental tetrahedron rectangles. These flow lines are uniquified in \textbf{Step 8}.

	In practice, we never work with the entire flow lines, but only with their restrictions to the continent currently under consideration. These restrictions are called \emph{flow intervals}. 
	Given $y \in \liftZeta$ we denote the flow interval obtained by restricting $\eleZeta_y$ to $\C$ by $\eleZeta_y^\C$.  If $\eleZeta_y^\C \neq \emptyset$  then $\bigcap \eleZeta_y^\C$ is a rectangle that contains $y$, whose sides are contained in some continental cusp leaves. We call $\bigcap \eleZeta_y^\C$ a \emph{continental approximation} of $y$. 
	
	In \textbf{Step 9} we further expand the continent so that, within each fundamental tetrahedron rectangle~$R$, the continental approximations of all points of $\liftZeta$ contained in~$R$ are sufficiently small to unambiguously determine the horizontal and vertical orders on the  set of cusps of $R$ and punctures of $R \bez \liftZeta$; see Section \ref{subsec:separating} for details. 
	These orders are then used in \textbf{Step 10} to find all $\zeta$-tetrahedron rectangles in~$R$. 
	
	In \textbf{Step 11} we group all found fundamental $\zeta$-tetrahedron rectangles and their $\zeta$-face subrectangles into $\pi_1(M)$-equivalence classes using~Lemma \ref{lem:equivalent:rectangles}. This is used in \textbf{Step 12} to construct a veering triangulation  with tetrahedra corresponding to the $\pi_1(M)$-equivalence classes of fundamental $\zeta$-tetrahedron rectangles and face pairings determined by pairs of solely face equivalent $\zeta$-(tetrahedron, face) pairs.

In the forthcoming subsections, we explain how some of the steps of the algorithm are carried out and prove that their implementations achieve the intended goals. Before doing so, we establish the correctness of the algorithm under the assumption that each step can be successfully executed.
	
\begin{theorem}\label{thm:algorithm:correct}
Given a veering triangulation $\V$ and a flow cycle $\zeta$ of $\V$, the algorithm presented in Figure~\ref{fig:flowchart} terminates and outputs
\begin{itemize}
	\item $\V$ if $\zeta$ is peripheral;
	\item the veering parent $\V_\zeta$ if $\zeta$ is non-peripheral.
\end{itemize}
\end{theorem}
\begin{proof}
It is clear from the flowchart in Figure \ref{fig:flowchart} that the algorithm terminates at the first decision node (Step~2) and outputs $\V$ whenever it detects that $\zeta$ is peripheral. Detecting peripherality is explained in Section \ref{subsec:peripheral}.

Suppose that $\zeta$ is non-peripheral. The fundamental continent $\mathcal{F}$ constructed in Step~1 must intersect at least one flow line $\eleZeta$. Set $x = \bigcap \eleZeta$.

We claim that at the end of Step 9 the obtained continent $\C$ is such that for every fundamental tetrahedron rectangle $R$:
\begin{enumerate}[label=(\thetheorem.\arabic*), leftmargin=3.2em]
	\item $\eleZeta_y^\C \neq \emptyset$ for every $y \in R\cap \liftZeta$. \label{cond:meets:all:relevant:flowlines}
	\item  No cusp leaf originating from a cusp of $R$ intersects $\bigcap \eleZeta_y^\C$ for some $y \in R \cap \eleZeta$. \label{cond:separated:from:cusp:leaves}
	\item For distinct $y,z \in R \cap \liftZeta$ no leaf intersects the closure of both $\bigcap \eleZeta_y^\C$ and $\bigcap \eleZeta_z^\C$. \label{cond:separated:from:each:other}
\end{enumerate}
	
Condition	\ref{cond:meets:all:relevant:flowlines} means that for every point of $y \in R \cap \liftZeta$ the continent $\C$ is large enough to meet a flow line $\eleZeta_y$ that represents $y$. The flow line $\eleZeta_y$ exists by Theorem~\ref{thm:flow:lines}. Since $y \in \liftZeta$, there is $g \in \pi_1(M)$ such that $y = g(x)$. Furthermore, $y$ is contained in some edge rectangle $P \subset R$. Thus $g^{-1}(P)$ is an edge rectangle that contains $x$. 

In Step 3 we build a continent $\C_1$ that is adapted to $x$. The fact that this can be achieved in finitely many steps is proved in Corollary \ref{cor:expanding:along:flowline:gives:adapted}. Since $\loomif(\C_1)$ is adapted to~$x$ and $x \in g^{-1}(P)$, there is $m \in \zz$ such that $\gamma_x^mg^{-1}(P) \subset \loomif(\C_1)$. Clearly, $x \in \gamma_x^m g^{-1}(P)$ and thus  the edge rectangle $\gamma_x^mg^{-1}(P)$ belongs to the set $\ERect(\C_1, \eleZeta)$ that we find in Step~4. 

In Step 5, the continent is expanded to meet the flow line that represents the translate of $x$ contained in the fundamental edge rectangle that is $\pi_1(M)$-equivalent to $P$. In Step~7, the continent is expanded to meet the translate of $\eleZeta$ that represents $y$. It follows that Condition \ref{cond:meets:all:relevant:flowlines} holds at the end of Step 7, and thus at the end of Step~9.

Conditions \ref{cond:separated:from:cusp:leaves} and \ref{cond:separated:from:each:other} imply that the continental approximations of the points of $R \cap \liftZeta$ are sufficiently small so that the horizontal and vertical orders on the set of cusps of $R$ and the punctures of $R \bez \liftZeta$ are completely unambiguous, despite the fact that we do not know the precise location of points of $R \cap \liftZeta$. In Figure~\ref{fig:flowchart}, the process of achieving these conditions is referred to as \emph{separating the continental approximations of $R \cap \liftZeta$ and the cusps of~$R$}. This process is carried out in Step~9, and its implementation is described in Section~\ref{subsec:separating}. While separating the continental approximations, we also sort them according to the horizontal and vertical orders on $R$. 

An important subtlety here is that distinct translates of the same flow line may represent the same point of $\liftZeta$; see Proposition \ref{prop:square} and the discussion below it. When this happens, we get distinct continental approximations of the same point of $R \cap \liftZeta$ that cannot be separated. To avoid getting stuck in an infinite loop in Step~9, we must first ensure that each point of $R \cap \liftZeta$ is represented by exactly one translate of $\eleZeta$. This is carried out in Step~6 for flow lines representing points inside fundamental edge rectangles (to avoid translating more flow lines than necessary in Step 7) and in Step~8 for flow lines representing points inside fundamental tetrahedron rectangles. In Figure \ref{fig:flowchart} this process is called \emph{uniquifying} flow lines. It consists of detecting distinct translates of $\eleZeta$ that represent the same point and discarding all but one of them; see Section \ref{subsec:fellow-travel} for details. Since flow lines are uniquified before Step 9, Step 9 terminates.

	The horizontal and vertical orders on the cusps of $R$ and punctures of $R \bez \liftZeta$, obtained in Step~9, are used in Step~10 to find all $\zeta$-tetrahedron rectangles in $R$. In Step~11 we use Lemma \ref{lem:equivalent:rectangles} to  determine the $\pi_1(M)$-equivalence classes of all found fundamental $\zeta$-tetrahedron and $\zeta$-face rectangles. By Proposition \ref{prop:finite:veering:parent} and Theorem~\ref{thm:flow:lines}, the veering triangulation whose tetrahedra are in bijective correspondence with the $\pi_1(M)$-equivalence classes of fundamental $\zeta$-tetrahedron rectangles and whose face pairings are determined by pairs of solely face equivalent $\zeta$-(tetrahedron, face) pairs, constructed in Step 12, is the veering parent $\V_\zeta$. 
\end{proof}

	\subsection{Peripheral flow cycles}\label{subsec:peripheral}

	We say that a flow cycle~$\zeta$ is \emph{peripheral} if any (and thus all) of its elevations represents a cusp of~$\loom$. In this case, $\zeta$ does not encode an orbit of the drilled flow~$\Psi \bez \Lambda$, but rather an orbit already contained in the perfect set~$\Lambda$. Naturally, we cannot obtain a nontrivial veering parent of~$\V$ by drilling out such an orbit. Therefore, we need an efficient way to detect peripheral flow cycles, allowing us to terminate the algorithm early when the input is of this type.
	
	Peripheral flow cycles can be detected without passing to the universal cover. However, since we  work with subsets of the universal cover anyway, it makes sense to make use of them to detect peripherality as well. For this, we use the element $g_{\eleZeta} \in \pi_1(M)$ associated to an elevation $\eleZeta$ of $\zeta$, introduced in Section \ref{subsec:flow:cycles:and:lines}.

	\begin{lemma}\label{lem:test:peripheral}
		A flow cycle $\zeta$ is peripheral if and only if for some elevation $\eleZeta = (R_i)_{i\in \zz}$ of $\zeta$ some cusp of $R_0$ is fixed by $g_{\eleZeta}$.
	\end{lemma}
	\begin{proof}
		Suppose that $\zeta$ is peripheral. Pick some elevation $\eleZeta$ of $\zeta$. The flow line $\eleZeta$ necessarily represents some cusp $c$ of $\loom$.
		Thus, by Lemma \ref{lem:stabilizers}, $g_{\eleZeta}$ fixes~$c$. Since all rectangles of $\eleZeta$ meet $c$, $c$ is a cusp of $R_0$ that is fixed by $g_{\eleZeta}$. 
		
		Conversely, suppose that $c$ is a cusp of $R_0$ that is fixed by $g_{\eleZeta}$. Let $k$ be the length of~$\zeta$.
		Then $R_k = g_{\eleZeta}(R_0)$ also meets $c$.
		Fix an integer $0 \leq j < k$.
		Since $\eleZeta$ is a flow line, we have $R_0 \cap R_k \subset R_0 \cap R_j$; see Observation \ref{obs:tets:get:thinner}.
		It follows that $R_j$ meets $c$.
		Because $g_{\eleZeta}(c)=c$, we deduce that $g_{\eleZeta}^m(R_j)$ meets $c$ for every $m \in \zz$.
		Now, for every $i \in \zz$ there exist integers $0 \le j < k$ and $m \in \zz$
		such that $R_i = g_{\eleZeta}^m(R_j)$.
		Therefore every rectangle of $\eleZeta$ meets $c$.
		By Theorem \ref{thm:flow:lines}, $\eleZeta$ represents  $c$ and thus $\zeta$ is peripheral.
	\end{proof}

\subsubsection{Implementation}\label{subsec:peripheral:implement}
In \textbf{Step 1} we construct a fundamental continent $\F \subset \univV$. It is guaranteed that some elevation $\eleZeta = (R_i)_{i \in \zz}$ of $\zeta$ intersects~$\F$. 	After possibly reindexing~$\eleZeta$, we may assume that $R_0 \subset \loomif(\F)$. 
In order to use Lemma \ref{lem:test:peripheral}, in \textbf{Step 2}  we expand~$\F$ to a continent $\C_0$ that contains both $t_0 = \loomif^{-1}(R_0)$ and $t_k = \loomif^{-1}(R_k)$, where $k$ denotes the length of $\zeta$. Both $t_0$ and $t_k$ project to the same tetrahedron $t$ of $\V$ under the universal covering map. By Lemma \ref{lem:test:peripheral}, $\zeta$ is peripheral if and only if $t_0$ and $t_k$ share an ideal vertex $c$ of $\univV$ such that $c$ maps to the same vertex of $t$ under both $t_0 \rightarrow t$ and $t_k \rightarrow t$. If that is the case, the algorithm terminates and returns $\V$. Otherwise, we proceed to \textbf{Step 3}.

	\subsection{Adapting the continent to $\zeta$} \label{subsec:adapted:continent}

	From now on, assume that $\zeta$ is non-peripheral and thus $x = \bigcap \eleZeta \in \loom$.
	In \textbf{Step 3}, we construct a continent that is adapted to $x$. 
	First, we define a specific type of continent expansion that we will repeatedly use throughout the algorithm.
	
	\begin{definition}\label{def:expanding:along:flowline}
		Let $\rho = (t_i)_{i \in \zz}$ be a flow line in $\univV$. For a convex continent $\C \subset \univV$, let 
		\begin{equation*}
		M = \max\lbrace i \in \zz \ | \ t_i \in \C \rbrace,\hspace{2cm}
		m = \min \lbrace i \in \zz \ | \ t_i \in \C \rbrace.
		\end{equation*}
		We say that a convex continent $\C'$ is obtained from~$\C$ by 
		\begin{itemize}
			\item \emph{expanding upwards along $\rho$}  if $\C'$ is the smallest convex continent containing \mbox{$\C \cup \lbrace t_i \rbrace_{i=M+1}^{M+N}$} for some $N \in \mathbb{N}$. If $N = 1$, we denote $\C'$ by $\C^{\rho^\uparrow}$.
			\item \emph{expanding downwards along $\rho$} if $\C'$ is the smallest convex continent containing $\C \cup \lbrace t_i \rbrace_{i=m-n}^{m}$ for some $n \in \mathbb{N}$. If $N = 1$, we denote $\C'$ by $\C_{\rho_\downarrow}$.
			\item \emph{expanding along $\rho$} if $\C'$ is obtained from $\C$ by expanding it upwards and downwards along $\rho$.
		\end{itemize}
We call the tetrahedron $t_M$ the \emph{uppermost tetrahedron} of $\C \cap \rho$, and the tetrahedron $t_m$ the \emph{lowermost tetrahedron} of $\C \cap \rho$. 
	\end{definition}
For an upward or downward expansion up to index $i \in \mathbb{Z}$, we use the notation $\C^{\rho^\uparrow}(i)$ and $\C_{\rho_\downarrow}(i)$, respectively. For instance, the convex continent $\C_0$ constructed in Step 2 from the fundamental continent $\F$ can be expressed as $\C_0 = \F^{\eleZeta^\uparrow}(k)$. This notation is used in Figure \ref{fig:flowchart:adapting}.

	\begin{lemma}\label{lem:finite:expansion:meeting:cusps}
		Let $\rho = (t_i)_{i\in \zz}$ be a non-peripheral periodic flow line in $\univV$. Set $x = \bigcap\rho$. Let $\C \subset \univV$ be a continent such that $t_0 \in \C$. 
		For any finite set~$W$ of cusps, each meeting some staircase at $x$, there exists a finite expansion $\C_W$ of~$\C$ along~$\rho$ such that $\loomif(\C_W)$ meets every cusp of $W$.
	\end{lemma}
	\begin{proof}For $i \in \zz$ let $R_i = \loomif(t_i)$. If $W = \emptyset$, $\C_W = \C$. Otherwise, let $w \in W$. Denote the staircase at $x$ that meets~$w$ by $S$. By Observation \ref{obs:tets:get:thinner} and Lemma \ref{lem:staircase} (1), there exists $N_1(w) \in \mathbb{N}$ such that the convex hull of the cusps of $\bigcup\limits_{i=0}^{N_1(w)} R_i$ meets some cusp $c_{w,1} \in \Omega(S)$ that is  in-between $x$ and $w$ in the horizontal order on $S$. Similarly, there exists $N_2(w)\in \mathbb{N}$ such that the convex hull of the cusps of $\bigcup\limits_{i=-N_2(w)}^{0} R_i$ meets some cusp $c_{w,2} \in \Omega(S)$ that is in-between $x$ and $w$ in the vertical order on $S$.
		
		Since $W$ is nonempty and finite, the set $\lbrace N_j(w)  \ | \  w \in W\rbrace $ admits a maximum for each $j=1,2$. We denote this maximum by $N_j$.
		
	 Let $\C_W$ be the smallest convex continent that contains $\C \cup \lbrace t_i \rbrace_{i=-N2}^{N_1}$. Since $\C$ contains~$t_0$, $\C_W$ is obtained from $\C$ by expanding it along $\rho$. 
		
		For every $w \in W$, $\loomif(\C_W)$ meets two cusps $c_{w,1}$ and $c_{w,2}$ lying in the same staircase~$S$ at $x$ as $w$, and such that $w$ is between $c_{w,1}$ and $c_{w,2}$ in the dynamic linear order on $\Omega(S)$. 
		Since there is a geodesic connecting $c_{w,1}$ and $c_{w,2}$ that passes through $w$, and the closure of $\loomif(\C_W)$ is convex by Proposition \ref{prop:convex:disk}, $\loomif(\C_W)$ meets~$w$. 
	\end{proof}
	
		\begin{corollary}\label{cor:expanding:along:flowline:gives:adapted}
		Let $\C_0$ be the convex continent constructed in Step 2. 
		Expanding $\C_0$ finitely many times along~$\eleZeta$ eventually yields a convex continent that is adapted to~$x$.
	\end{corollary}
	\begin{proof}
		In Proposition \ref{prop:adapted:general} we proved that there exists a finite and nonempty set $W$ of cusps of $\loom$, each meeting some staircase at $x$, such that $\hull(W)$ is adapted to $x$. 
		Thus, by Lemma \ref{lem:finite:expansion:meeting:cusps}, there exists a finite expansion $\C_W$ of $\C_0$ along $\eleZeta$ such that $\loomif(\C_W)$ meets every cusp of $W$.  Since $\overline{\loomif(\C_W)}$ is convex by Proposition \ref{prop:convex:disk},  it must contain $\hull(W)$. Therefore~$\C_W$ is adapted to $x$.
	\end{proof}
	
	It follows  that Step 3 of the algorithm presented in Figure \ref{fig:flowchart} can be achieved by expanding $\C_0$ along $\eleZeta$ until it meets a certain finite set $W$ of cusps. Recall from Proposition~\ref{prop:adapted:general} that we can choose 
	\[W = \left\lbrace a, \gamma_x^m(a), b=s^+_{a,m,x}, s^-_{a,0,x}, \gamma_x^{-m}(b), s^+_{b,0,x}, s^-_{b,-m,x} \right\rbrace,\]
	where $a \in \Omega_x$ and $m \in 2\mathbb{N}$ are  arbitrary, and, for a cusp $c \in \Omega_x$ and $i \in \zz$, $s^\pm_{c,i,x}$ is any external horizontal (+) or vertical ($-$) $(\gamma_x^i(c),x)$-separator; see Definition \ref{def:separator:cusp}.

	 Recall from Remark \ref{rem:infinitely:many:seps} that for any $c \in \Omega_x$ and $i \in \zz$ there are infinitely many choices for the cusps $s^\pm_{c,i,x}$. In our algorithm we find the most optimal ones by searching through face rectangles with ideal corner at $\gamma_x^i(c)$. To  explain the procedure, we first introduce some terminology and notation associated to face rectangles.
	 As before, denote the leaves of $\loom^\pm$ passing through $x$ by $\ell_x^\pm$.
	 
	 \begin{definition}
	 	Let $Q$ be a face rectangle. We say that a cusp leaf~$\ell$ is
	 	\begin{enumerate}
	 		\item an \emph{internal cusp leaf} of $Q$ if $\ell$ originates from a cusp of $Q$ and intersects~$Q$.
	 		\item a \emph{non-corner boundary leaf} of $Q$ if $\ell$ originates from a non-corner cusp of $Q$ and intersects the boundary of $Q$ in a leaf segment of nonzero length. 
	 	\end{enumerate}
	 \end{definition}
	 Recall from Definition \ref{defn:rectangle} that rectangles are open.  Thus each face rectangle $Q$ admits exactly one stable internal cusp leaf and one unstable internal cusp leaf. These two leaves always originate from distinct cusps of $Q$. On the other hand, $Q$ admits exactly two stable non-corner boundary leaves, both originating from the same cusp of $Q$; an analogous statement holds for unstable non-corner boundary leaves. 
	 
	 \begin{definition}\label{def:face:rect:separating}
	 	Let $c \in \Omega_x$. We say that a face rectangle $Q$ is  \emph{horizontally $(c,x)$-separating}  (respectively, \emph{vertically $(c,x)$-separating}) if $Q$ has ideal corner at $c$, $x \in Q$ and $\ell_x^+$ (respectively, $\ell_x^-$) separates the non-corner cusps of $Q$; see Figure \ref{fig:cx:separating}.
	 \end{definition}
 
 	 \begin{figure}[h]
	 \centering
	 \labellist
	 \hair 2pt \small
	 \pinlabel $c$ [tr] at 2 10
	 \pinlabel $s^+$ [tr] at 160 270
	 \endlabellist
	 \includegraphics[width = 0.2\textwidth]{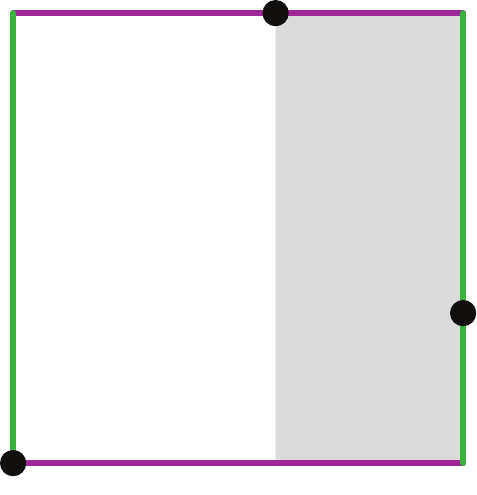}
 	\caption{Horizontally $(c,x)$-separating face rectangle. The cusp $s^+$ is the origin of the stable internal cusp leaf of the rectangle. The shaded region indicates the possible position of $x$.}
 	\label{fig:cx:separating}
 \end{figure}

Observe that if $Q$ is a horizontally $(c,x)$-separating face rectangle, for some $c \in \Omega_x$, then the origin $s^+$ of the stable internal cusp leaf of $Q$ is an external horizontal $(c,x)$-separator. 
  Indeed, if $S$ is the staircase at $x$ that meets $c$ then $s^+$ meets the staircase at $x$ that is on the same side of $\ell_x^+$ and on the other side of $\ell_x^-$ as $S$; see Figure \ref{fig:cx:separating}. Moreover, $s^+$ is between $c$ and $x$ in the horizontal order on $Q$ and thus in the horizontal order on $\Sigma_x$. Analogously, the origin $s^-$ of the unstable internal cusp leaf of a vertically $(c,x)$-separating face rectangle is an external vertical $(c,x)$-separator. 

	\begin{lemma}\label{lem:face:rectangle}
For any $c \in \Omega_x$ there is a unique horizontally $(c,x)$-separating face rectangle and a unique vertically $(c,x)$-separating face rectangle. 
	\end{lemma}
	\begin{proof}
		We prove the existence and uniqueness of a horizontally $(c,x)$-separating face rectangle. The proof of the other statement is completely analogous.
		Since $c$ and $x$ span a rectangle, there is a staircase $\sigma$ at $c$ that contains~$x$. 
		As $\ell_x^+$ is not a cusp leaf, there is a pair of cusps $a,b \in \Omega_c(\sigma)$ that are consecutive in the dynamic linear order on $\Omega_c(\sigma)$ and separated by~$\ell_x^+$. Set $Q$ to be the interior of $\hull(a,b,c)$. Then $Q$ is a face rectangle with ideal corner at $c$ and non-corner cusps $a$, $b$ that are separated by $\ell_x^+$.  Furthermore, $x \in Q$, because $Q \subset \sigma$, $x \in \sigma$ and $\ell_x^+$ intersects $Q$.
		
	The uniqueness follows from the fact that only one staircase at $c$ contains $x$, and only one pair of consecutive cusps in that staircase can be separated by $\ell_x^+$.
	\end{proof}

	 	Recall the element $g_{\eleZeta}\in \pi_1(M)$ associated to $\eleZeta$ that we defined in Section \ref{subsec:flow:cycles:and:lines}. By Lemma \ref{lem:stabilizers}, $g_{\eleZeta}$ fixes $x$. 
	 	Given any cusp $c \in \Omega_x$, and $i \in \zz$, denote the unique horizontally (respectively, vertically) $(g_{\eleZeta}^i(c),x)$-separating face rectangle by $Q^+_{c,i,x}$ (respectively, $Q^-_{c,i,x}$), and the origin of its stable (respectively, unstable) internal cusp leaf  by $s^\pm_{c,i,x}$. This notation is analogous to the one we used in Section \ref{subsec:adapted:set}, except that, by Proposition~\ref{prop:square}, $g_{\eleZeta}$ may be homotopic to  $\gamma_x^2$ instead of $\gamma_x$.

	 	\begin{lemma}\label{lem:special:face:rectangles}
	 		Let $a\in \Omega_x$. Set $b = s^+_{a,2,x}$. If $\loomif(\C)$ contains 
	 		$Q^+_{a,2, x}$,
	 		$Q^-_{a,0,x}$,
	 		$Q^+_{b,0,x}$, and $Q^-_{b,-2, x}$ then $\C$ is adapted to $x$. 
	 	\end{lemma}
 	\begin{proof}
 		By Proposition \ref{prop:square}, $g_{\eleZeta}$ is homotopic either to $\gamma_x$ or $\gamma_x^2$. In either case, Lemma~\ref{lem:stabilizers}~(2) implies that $a$ and $g_{\eleZeta}^2(a)$ lie in the same staircase at $x$. Set
 			\[W = \left\lbrace a, g_{\eleZeta}^2(a), b=s^+_{a,2,x}, s^-_{a,0,x}, g_{\eleZeta}^{-2}(b), s^+_{b,0,x}, s^-_{b,-2,x} \right\rbrace.\]
 			Then $\hull(W)$ is adapted to $x$ by Proposition \ref{prop:adapted:general}.

Suppose that $\loomif(\C)$ contains $Q^+_{a,2, x}$,
$Q^-_{a,0,x}$,
$Q^+_{b,0,x}$, and $Q^-_{b,-2, x}$. Since each $w \in W$ is a cusp of at least one of these face rectangles, $\C$ clearly meets $W$. Convexity of the closure of $\loomif(\C)$ (Proposition \ref{prop:convex:disk}) then implies that $\hull(W) \subset \overline{\loomif(\C)}$. Hence $\C$ is adapted to $x$.
 	\end{proof}

	 	In Figure \ref{fig:flowchart:adapting} we present a flowchart of an algorithm that expands the continent $\C_0$, constructed in Step 2, into a convex continent that is adapted to $x$. Lemma \ref{lem:finite:expansion:meeting:cusps} implies that this algorithm terminates in finite time. The fact that it returns a continent adapted to $x$ follows from Lemma \ref{lem:special:face:rectangles}.

	 \begin{figure}[h]
	 	\centering
	 	\includegraphics[width=\textwidth]{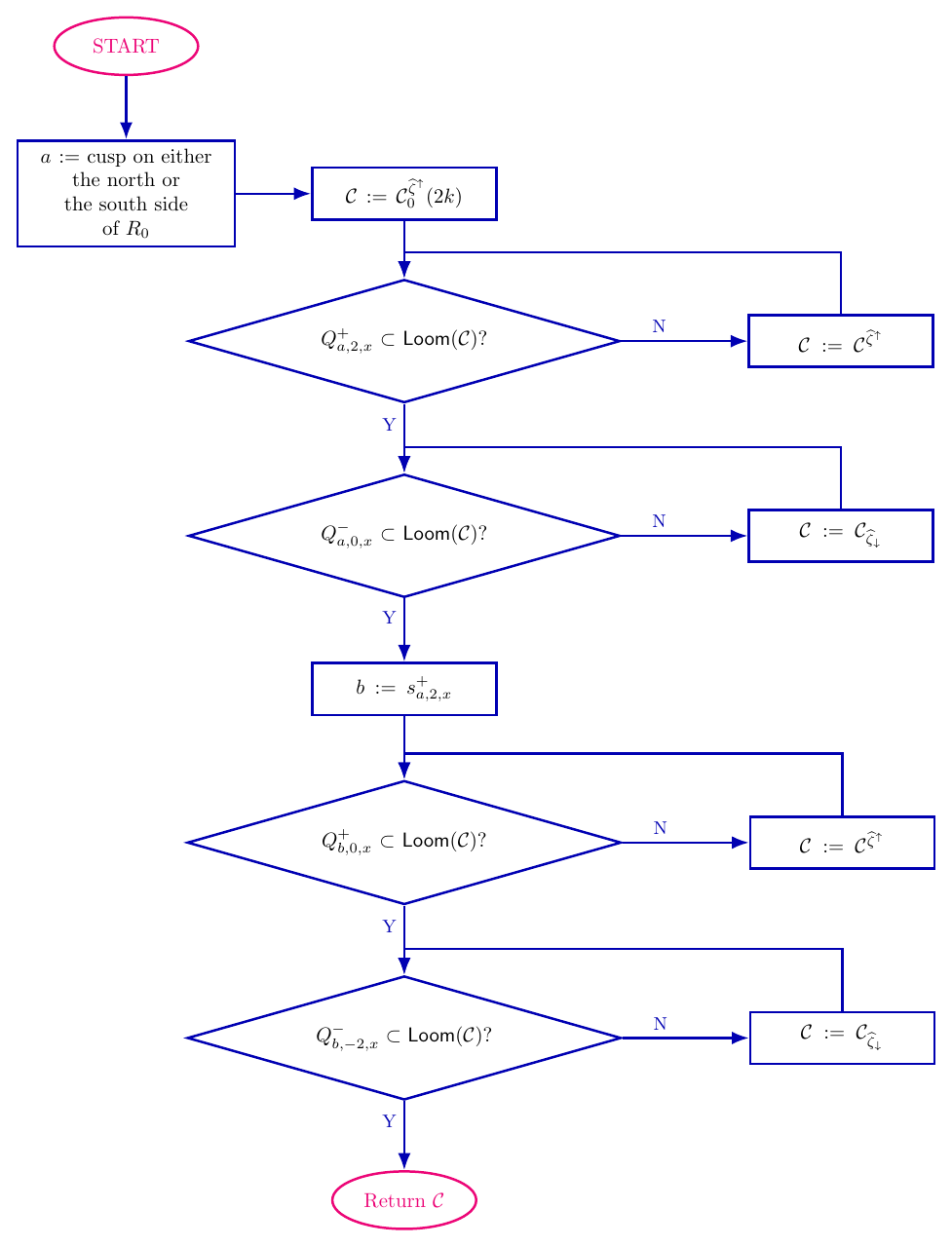}
	 	\caption{Flowchart for Step 3 of the algorithm in Figure \ref{fig:flowchart}. The input for this step is the convex continent $\C_0$ constructed in Step 2 and the flow cycle $\zeta$. The output is a convex continent adapted to $x = \bigcap \eleZeta$.}
	 	\label{fig:flowchart:adapting}
	 \end{figure}

	 It remains to explain how the checks in the decision nodes of Figure~\ref{fig:flowchart} are carried out. Since for any cusp $c \in \Omega_x$ we have
	 $Q^\pm_{c,i,x}=Q^\pm_{g_{\eleZeta}^i(c),0,x}$, the index $i$ can always be absorbed into the choice of cusp. In what follows, we therefore study only the rectangles $Q^\pm_{c,0,x}$, for which we use the simplified notation $Q^\pm_{c,x}$. 
	 
	 First, we characterize $Q^\pm_{c,x}$ using only tetrahedron rectangles making up the flow line $\eleZeta=(R_i)_{i\in\zz}$ instead of the precise position of $x$.

	\begin{lemma}\label{lem:recognizing:Q:cx}
		Let $c \in \Omega_x$. Let $Q$ be a face rectangle with ideal corner at $c$.
		Then $Q = Q^+_{c, x}$ if and only if 
		\begin{enumerate}
			\item $Q$ is contained in the same orthant at $c$ as $x$.
			\item There exists $N \in \mathbb{N}$ such that for all $n \geq N$: 
			\begin{enumerate}
				\item The stable internal cusp leaf of $Q$ separates all cusps of $R_n$ from $c$.\label{cond:lemQ:cx:separation}
				\item All cusps of $R_n$ and $c$  are on the same side of each  stable non-corner boundary cusp leaf of $Q$.\label{cond:lemQ:cx:same:side}
			\end{enumerate}
		\end{enumerate}
	An analogous characterization of $Q^-_{c, x}$ holds, with stable leaves replaced by unstable leaves, and $n \geq N$ replaced by $n \leq -N$.
	\end{lemma}
\begin{proof}
	Denote the stable internal cusp leaf of $Q$ by $\eta$ and the stable non-corner boundary leaves of $Q$ by $\beta_1, \beta_2$ 
	
	Suppose that $Q = Q^+_{c,x}$. Since $Q$ has ideal corner at $c$ and $x \in Q$, $Q$ is contained in the same orthant at $c$ as $x$. Hence (1) holds.

By Observation \ref{obs:tets:get:thinner} and Theorem \ref{thm:flow:lines},  $\left\lbrace\bigcap\limits_{i=0}^{n} R_i\right\rbrace_{n \in \mathbb{N}}$ is a strictly descending family of rectangles whose intersection is a line segment contained in $\ell_x^+$.
Since, by Definition \ref{def:face:rect:separating}, $\eta$ separates $x$ from~$c$, there exists $N_1 \in \mathbb{N}$ such that for every $n \geq N_1$ the leaf $\eta$ separates all cusps of $R_n$ from~$c$. 
Similarly, since $x$ and $c$ are on the same side of $\beta_1$ and $\beta_2$, there exists $N_2 \in \mathbb{N}$ such that for every $n \geq N_2$ all cusps of $R_n$ and $c$ are on the same side of $\beta_i$ for $i=1,2$. Setting $N = \max\lbrace N_1, N_2 \rbrace$ proves (2). 

Conversely, suppose that conditions (1) and (2) hold. We need to show that $x \in Q$ and~$\eta$ separates $x$ from $c$. Since $Q$ is a face rectangle with ideal corner at $c$, it lies in some staircase $\sigma$ at $c$. Because $x$ spans a rectangle with $c$ and, by (1), lies in the same orthant at $c$ as $Q$, it follows that $x \in \sigma$.

Condition (2) and Observation \ref{obs:tets:get:thinner} imply that $\ell_x^+$ separates $\eta$ from $\beta_1$ and $\beta_2$. Thus $\eta$ separates $c$ from $x$. Since $x \in \sigma$, we conclude that $x \in Q$. Thus $Q = Q^+_{c, x}$.
\end{proof}

Lemma~\ref{lem:recognizing:Q:cx} leads to the following iterative procedure for finding $Q^+_{c, x}$. 
\begin{procedure} \label{proc:Q:cx:search} (Iterative search for $Q^+_{c, x}$ given $c \in \Omega_x$ and a continent $\C$ meeting $c$)\hfill
\begin{enumerate}[leftmargin=3.8em, label=(\theprocedure.\arabic*),
	ref=\theprocedure.\arabic*]
	\item Find the orthant $O_{c,x}$ at $c$ that contains $x$. \label{proc:orthant}
	\item Find the (finite) collection $\Q^\C_{c,x}$ of face rectangles with ideal corner at $c$ that are contained in $O_{c, x} \cap \loomif(\C)$. \label{proc:step:face:rects}
	\item For each $Q \in \Q^\C_{c, x}$, check whether:\label{proc:step:conditions}
	\begin{enumerate}[label=(\alph*), ref=\theprocedure.\arabic{enumi}~(\alph*)]
		\item the stable internal cusp leaf of $Q$ separates $c$ from all cusps of the uppermost tetrahedron of $\C \cap \eleZeta$, and \label{cond:separates:from:uppermost}
		\item all cusps of the uppermost tetrahedron of $\C$ and $c$ lie on the same side of each stable non-corner boundary cusp leaf of $Q$.\label{cond:same:side:stable:bdy}
	\end{enumerate}
	If such a rectangle $Q$ exists, then $Q = Q^+_{c,x}$ by Lemma~\ref{lem:recognizing:Q:cx}.\\
Otherwise, expand~$\C$ upwards along $\eleZeta$ and repeat Steps \ref{proc:step:face:rects} and \ref{proc:step:conditions}.
\end{enumerate}
\end{procedure}

Since all cusps of $Q^+_{c,x}$ meet some staircase at $x$, Lemma~\ref{lem:finite:expansion:meeting:cusps} implies that after finitely many upwards expansions along $\eleZeta$ we must reach a continent whose loomification contains $Q^+_{c,x}$. Thus Procedure \ref{proc:Q:cx:search} terminates.

The first and third decision loops in the diagram on Figure \ref{fig:flowchart:adapting} are carried out using Procedure \ref{proc:Q:cx:search}. The second and fourth decision loops are analogous, with stable cusp leaves replaced by unstable cusp leaves, the uppermost tetrahedron of $\C \cap \eleZeta$ replaced by its lowermost tetrahedron, and upwards expansion of $\C$ along $\eleZeta$ replaced by downwards expansion of $\C$ along $\eleZeta$.

\subsubsection{Implementation}\label{subsubsec:implementation:adapting}
Recall from Section \ref{subsec:linear:orders} that given a continent $\C$ the loom-compatible circular order $\Theta$ on $\Omega(\C) \cup \E(\C)$ determines, for each continental cusp leaf $\ell$, a side function $\Theta_\ell: \Omega(\C) \rightarrow \lbrace -1, 0, 1 \rbrace$. By Lemma \ref{lem:theta:l:function}, $\Theta_\ell$ can be used to determine the relative positions of cusps of $\C$ with respect to $\ell$. We use these side functions in implementations of Procedure \ref{proc:Q:cx:search}, as well as Steps 4 and~9 of the algorithm.
Since, as explained in Section \ref{subsec:convexity}, in order to effectively determine $\Theta$ we need $\C$ to be convex, from Step~3 onward  it becomes crucial  to work exclusively with convex continents.

\begin{remark}\label{rem:expanding:to:continental}
We may write $\Theta_\ell(\cdot)$ for a cusp leaf $\ell$ that is not guaranteed to be continental with respect to a given continent $\C$. This notation is justified by the fact that if $\C'$ is an expansion of $\C$, then the corresponding circular order
$\Theta' : (\Omega(\C') \cup \E(\C'))^3 \to \lbrace -1, 0, 1 \rbrace$ agrees with $\Theta$ on $(\Omega(\C) \cup \E(\C))^3 $. Accordingly, by $\Theta_\ell(\cdot)$ we mean the value of $\Theta'_\ell(\cdot)$ after an appropriate expansion of $\C$. 

\end{remark}

\subsubsection*{Subprocedure (\ref{proc:orthant}): Finding the orthant at $c \in \Omega_x$ that contains $x$}
Let~$O_{c,x}$ be the orthant at~$c$ that contains $x$. To find $O_{c,x}$  we first find a simplex rectangle~$X_c$ that is guaranteed to contain $x$ and meets $c$ along the interior of one of its horizontal sides. There is a unique stable $c$-leaf $\mu$ that intersects $X_c$ and this leaf must be the stable boundary of $O_{c,x}$. Furthermore, the two unstable $c$-leaves $\nu_1, \nu_2$ intersecting $\partial X_c$ are the only two possible options for the unstable boundary of $O_{c,x}$. 

\begin{table}[H]
	\centering
	\begin{tabular}{|c|c|c|c|c|}
		\hline 
		$c$ & $g_\zeta^{2}(a)$ & $a$ & $b = s^+_{a,2,x}$ & $g_{\eleZeta}^{-2}(b)$\\ \hline
		$X_c$ & $R_{2k}$ & $R_0$ & $Q^+_{a,2,x}$ & $Q^+_{a,0,x}$ \\
\hline 

	\end{tabular}
\vspace{2pt}
	\caption{Simplex rectangle $X_c$ that contains $x$ and meets $c$ along the interior of one of its horizontal sides for $c \in \left\lbrace a, g_{\eleZeta}^{2}(a), b = s^+_{a,2,c}, g_{\eleZeta}^{-2}(b) \right\rbrace$. The stable $c$-leaf  intersecting $X_c$ forms the stable boundary of $O_{c,x}$.}
	\label{tab:orthants}
\end{table}

We list $X_c$ for  $c \in\left \lbrace a, g_{\eleZeta}^2(a), b =  s^+_{a,2,x}, g^{-2}_{\eleZeta}(b)\right\rbrace$ in Table \ref{tab:orthants}. This gives us  the stable boundary $\mu$ of $O_{c,x}$. Denote the unstable boundary of $O_{c,x}$ by~$\nu$. By Lemma~\ref{lem:theta:l:function}~(3), for any cusp $d$ other than $c$ that meets $O_{c,x}$, 
\[\Theta_{\mu}(d)\cdot\Theta_\nu(d) = -1 \text{ and }  
\Theta_{\nu_1}(d)\cdot\Theta_{\nu_2}(d) = -1.\]
Therefore we can determine which one of $\nu_1$ and $\nu_2$ is the unstable boundary $\nu$ of $O_{c,x}$ by first finding a cusp $d$ that is guaranteed to meet $O_{c,x}$ and then computing the values of $\Theta_{\nu_1}(d)$, $\Theta_{\nu_2}(d)$, and $\Theta_\mu(d)$. It remains to explain how to find the cusp $d$.

When $c \in \lbrace a, g_{\eleZeta}^2(a)\rbrace$ then $c$ meets the tallest edge subrectangle of $R_j$ for some $j \in \zz$. In this case, there is exactly one cusp $d$ that meets the tallest edge subrectangle $R_{j+1}$ and satisfies $\Theta_{\nu_1}(d) \cdot \Theta_{\nu_2}(d) = -1$; see Figure \ref{fig:d:O:c:x}. Then $d\neq c$ and $d$ meets $O_{c,x}$.

When $c \in \lbrace b, g_{\eleZeta}^{-2}(b)\rbrace$ then $c$ is the origin of the stable internal cusp leaf of $Q^+_{a,j,x}$, for some $j \in \zz$. We set $d$ to be the other non-corner cusp of $Q^+_{a,j,x}$. Then $d \neq c$ and, by Definition \ref{def:face:rect:separating}, $d$ is guaranteed to meet $O_{c,x}$.

		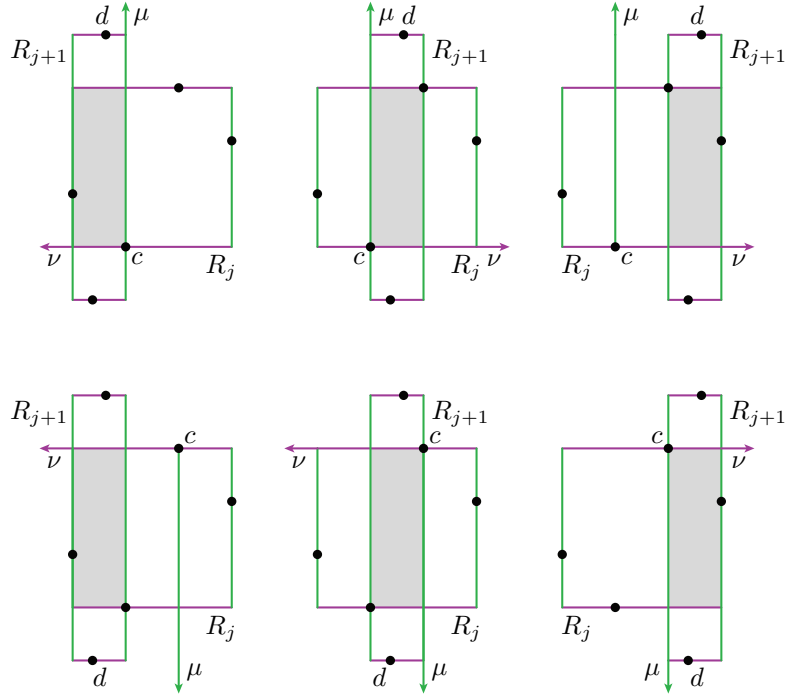
\begin{figure}[h]
	\centering
	\subfloat{
		\labellist
		\hair 3pt \small
		\pinlabel $c$ [tl] at 117 117
		\pinlabel $R_j$ [t] at 247 117
		\pinlabel $\mu$ [l] at 117 430
		\pinlabel $\nu$ [t] at 20 117
		\pinlabel $R_{j+1}$ [r] at 50 380
		\pinlabel $d$ [b] at 85 409
		\endlabellist
		\includegraphics[width = 0.2\textwidth]{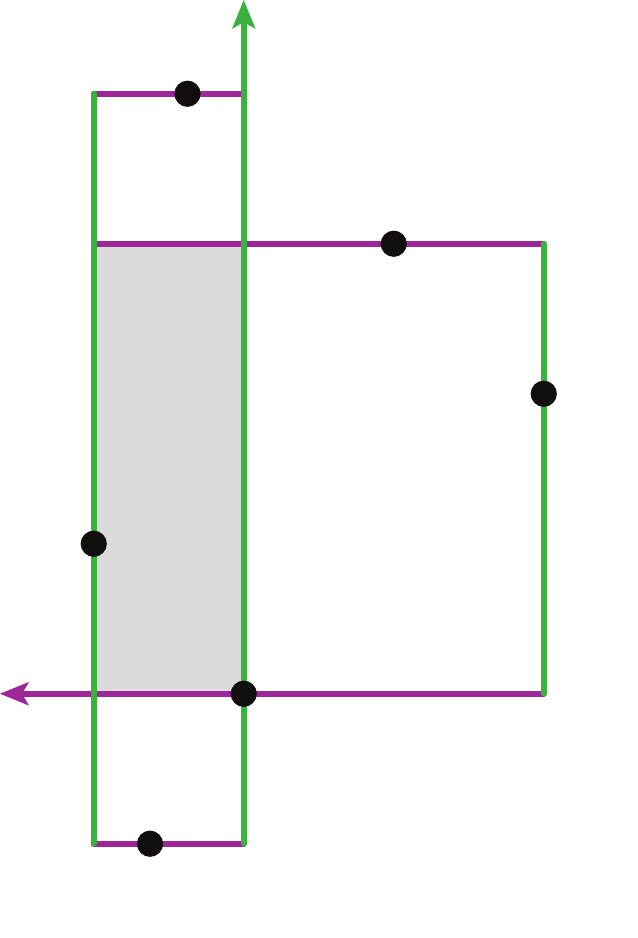}
	}
	\subfloat{
		\labellist
		\hair 3pt \small
		\pinlabel $c$ [tr] at 117 117
		\pinlabel $R_j$ [t] at 247 117
		\pinlabel $\mu$ [l] at 117 430
		\pinlabel $\nu$ [t] at 285 117
		\pinlabel $R_{j+1}$ [l] at 189 380
		\pinlabel $d$ [b] at 169 409
		\endlabellist
		\includegraphics[width = 0.2\textwidth]{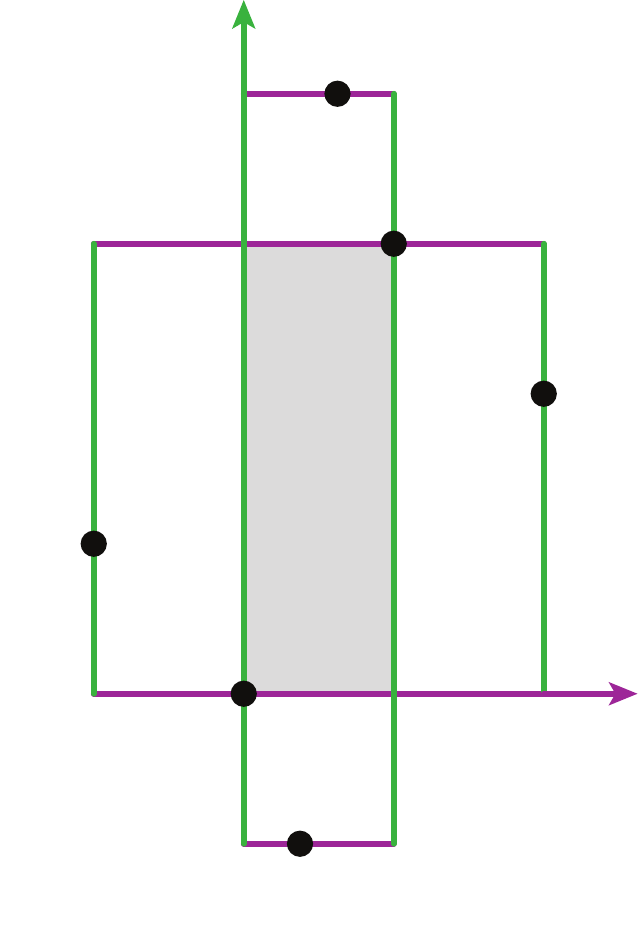}
	}
	\subfloat{
		\labellist
		\hair 3pt \small
		\pinlabel $c$ [tl] at 117 117
		\pinlabel $R_j$ [t] at 65 117
		\pinlabel $\mu$ [l] at 117 430
		\pinlabel $\nu$ [t] at 285 117
		\pinlabel $R_{j+1}$ [l] at 261 380
		\pinlabel $d$ [b] at 228 409
		\endlabellist
		\includegraphics[width = 0.2\textwidth]{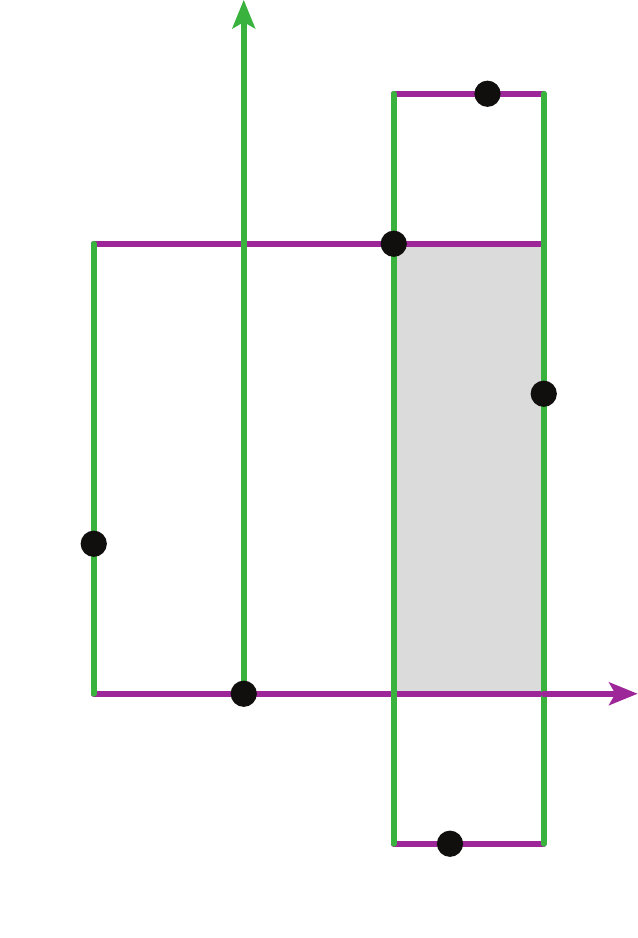}
	}
	
	\subfloat{
		\labellist
		\hair 3pt \small
		\pinlabel $c$ [bl] at 189 333
		\pinlabel $R_j$ [t] at 247 117
		\pinlabel $\mu$ [l] at 189 30
		\pinlabel $\nu$ [t] at 20 333
		\pinlabel $R_{j+1}$ [r] at 50 380
		\pinlabel $d$ [t] at 82 46
		\endlabellist
		\includegraphics[width = 0.2\textwidth]{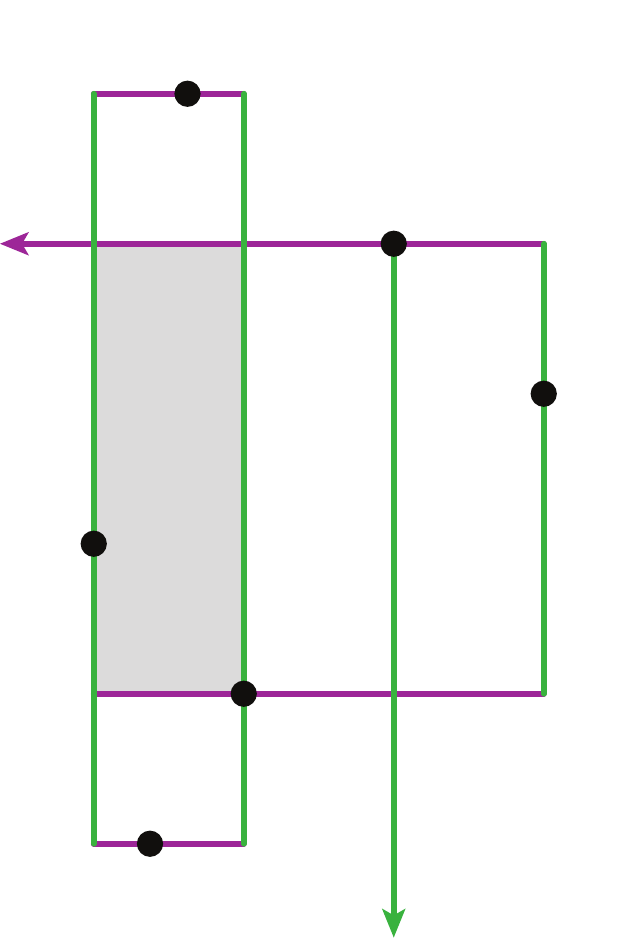}
	}
	\subfloat{
		\labellist
		\hair 3pt \small
		\pinlabel $c$ [bl] at 189 333
		\pinlabel $R_j$ [t] at 247 117
		\pinlabel $\mu$ [l] at 189 30
		\pinlabel $\nu$ [t] at 20 333
		\pinlabel $R_{j+1}$ [l] at 189 380
		\pinlabel $d$ [t] at 129 46
		\endlabellist
		\includegraphics[width = 0.2\textwidth]{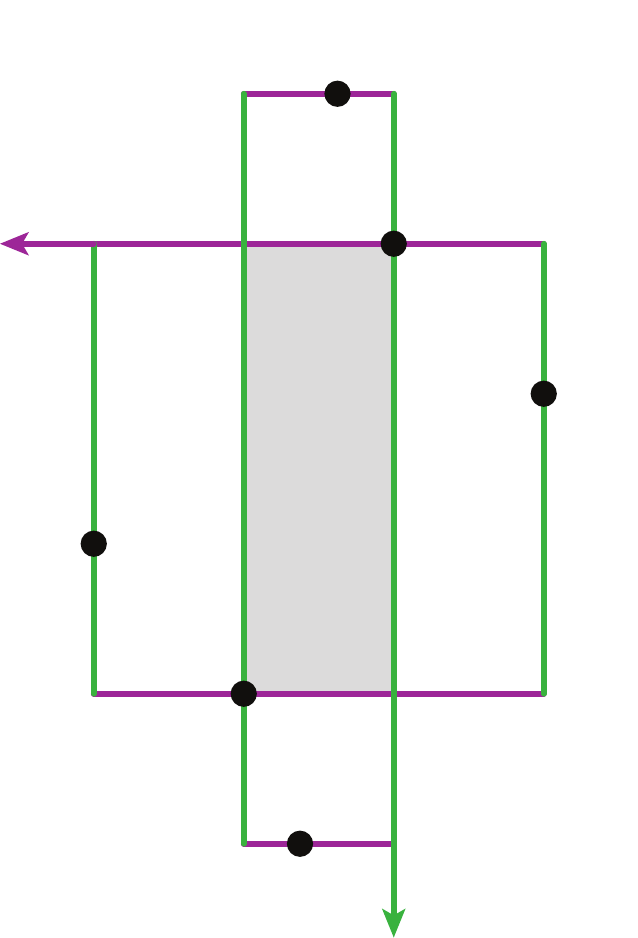}
	}
	\subfloat{
		\labellist
		\hair 3pt \small
		\pinlabel $c$ [br] at 189 333
		\pinlabel $R_j$ [t] at 65 117
		\pinlabel $\mu$ [r] at 189 30
		\pinlabel $\nu$ [t] at 285 333
		\pinlabel $R_{j+1}$ [l] at 261 380
		\pinlabel $d$ [t] at 228 46
		\endlabellist
		\includegraphics[width = 0.2\textwidth]{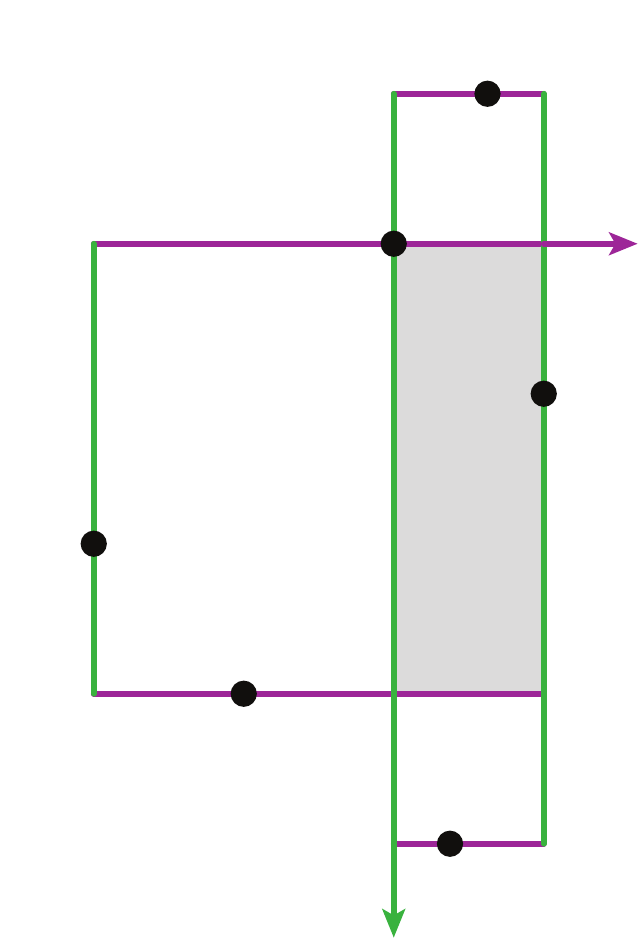}
	}

	\caption{Finding a cusp $d$ that meets  $O_{c,x}$ when $c$ meets the horizontal side of $R_j$ for some $j \in \zz$. The shaded region indicates the possible position of $x$.}
	\label{fig:d:O:c:x}
\end{figure}

\subsubsection*{Subprocedure (\ref{proc:step:face:rects}): Finding face rectangles with ideal corner at $c$ contained in~$O_{c,x}$}
Let \mbox{$\mu$, $\nu$} be the stable and the unstable boundary of $O_{c,x}$, respectively. 
Given a face rectangle~$Q$ with ideal corner at $c$ the containment $Q \subset O_{c,x}$ holds if and only if both non-corner cusps of~$Q$ lie strictly between $\mu$ and $\nu$. 
By Lemma \ref{lem:theta:l:function} (3), this is equivalent to requiring that $\Theta_{\mu}(d) \cdot\Theta_{\nu}(d) = -1$ for every non-corner cusp $d$ of $Q$.

Part \eqref{proc:step:face:rects} of Procedure~\ref{proc:Q:cx:search} is implemented by checking the above condition for each face rectangle in $\loomif(\C)$ with ideal corner at $c$ and retaining those for which it holds. Since there are only finitely many face rectangles in $\loomif(\C)$, the search is finite.

\subsubsection*{Subprocedure (\ref{proc:step:conditions}): Searching for $Q^\pm_{c,x}$} 
Let  $\Q^\C_{c,x}$ be the collection of face rectangles with ideal corner at $c$ that are contained in $O_{c,x} \cap \loomif(\C)$. This collection is found through Part \eqref{proc:step:face:rects} of Procedure~\ref{proc:Q:cx:search}. In Part \eqref{proc:step:conditions} we go through every $Q \in \Q^\C_{c,x}$  and check if any of them satisfies conditions \eqref{cond:separates:from:uppermost} and \eqref{cond:same:side:stable:bdy}.  

As in the previous two steps of Procedure \ref{proc:Q:cx:search}, the relevant conditions can be checked using the side functions $\Theta_\ell$. Let $Q \in \Q^\C_{c,x}$.
Denote the stable internal cusp leaf of $Q$ by $\eta$ and the two stable non-corner boundary cusp leaves of $Q$ by $\beta_1, \beta_2$. By Lemma \ref{lem:theta:l:function}~(1), condition \eqref{cond:separates:from:uppermost} is satisfied if and only if $\Theta_\eta(c)\cdot \Theta_\eta(d) = -1$ for every cusp $d$ of the uppermost tetrahedron $\tau$ of $\eleZeta \cap \C$. Similarly, condition \eqref{cond:same:side:stable:bdy} is satisfied if and only if $\Theta_{\beta_i}(c) \cdot \Theta_{\beta_i}(d) = 1 $ for every cusp $d$ of $\tau$ and every $i=1,2$.

If $Q$ satisfies \eqref{cond:separates:from:uppermost} and \eqref{cond:same:side:stable:bdy} then Lemma \ref{lem:recognizing:Q:cx} implies that $Q = Q^+_{c,x}$. If no $Q \in \Q^\C_{c,x}$  satisfies both these conditions, we expand the the continent upwards along~$\eleZeta$, and repeat Steps \eqref{proc:step:face:rects} and \eqref{proc:step:conditions} of Procedure~\ref{proc:Q:cx:search}.  

\begin{remark}
Since expanding the continent upwards makes the loomification of the uppermost tetrahedron of the corresponding flow interval strictly thinner (Observation~\ref{obs:tets:get:thinner}), it suffices to check the separation condition \eqref{cond:separates:from:uppermost} via the weaker condition $\Theta_\eta(c) \neq \Theta_\eta(d)$ for every cusp $d$ of the uppermost tetrahedron of $\eleZeta \cap \C$.
 This yields a slight speed-up of the algorithm.
\end{remark}

The conditions for $Q = Q^-_{c,x}$ are analogous, with the stable cusp leaves $\eta$, $\beta_1, \beta_2$ of~$Q$ replaced by its unstable cusp leaves, and the cusps of the uppermost tetrahedron of $\C \cap \eleZeta$ replaced by the cusps of its lowermost tetrahedron. Furthermore, when no $Q \in \Q^\C_{c,x}$  satisfies the conditions, the continent needs to be expanded downwards along $\eleZeta$.

		\subsection{Locating continental approximations inside edge rectangles}\label{subsec:locating}
		In \textbf{Step 4} we need to find all edge rectangles in $\loomif(\C_1)$ that contain $x = \bigcap \eleZeta$. Since we only see the flow interval $\eleZeta^{\C_1} \subset \C_1$, the exact position of $x$ is unknown. Given an edge rectangle $P \subset \loomif(\C_1)$, there are three possibilities for the relationship between the continental approximation $N_1(x) = \bigcap \eleZeta^{\C_1}$ of $x$ and $P$
		\begin{enumerate}
			\item $N_1(x) \cap P = \emptyset$,
			\item $N_1(x) \subset P$,
			\item $N_1(x) \cap P \neq \emptyset$ and $N_1(x) \not\subset P$.
		\end{enumerate} 
	
	In case (1), $x \notin P$. In case (2), $x \in P$. If case (3) holds, it is not possible to determine if $x \in P$ using only $\C_1$. This motivates the following definition.
	
	\begin{definition}\label{def:leaf:line:comparable}
		Let $\C$ be a convex continent, and let $\rho$ be a flow line such that $\rho^\C \neq \emptyset$. Given a stable (respectively, unstable) continental cusp leaf $\ell$ we say that $\rho$ and~$\ell$ are \emph{comparable} in $\C$ if all cusps of the uppermost (respectively, lowermost) tetrahedron of~$\rho^\C$ are on the same side of $\ell$.
	\end{definition}

	\begin{lemma}\label{lem:expanding:until:on:the:same:side}
		Let $\C$ be a convex continent in $\univV$, and let $\rho$ be a periodic flow line such that $\rho^\C \neq \emptyset$.  For any stable (respectively, unstable) continental cusp leaf $\ell$, expanding~$\C$ upwards (respectively, downwards) along $\rho$ eventually yields a continent in which $\rho$ and $\ell$ are comparable.
	\end{lemma}
\begin{proof}
Since $\rho$ is periodic, Proposition \ref{prop:orbit:space} \eqref{orbit_space:one:periodic} implies that $p = \bigcap \rho$ does not lie on a cusp leaf. In particular, $p$ lies strictly on one side of $\ell$. The statement therefore follows from Observation \ref{obs:tets:get:thinner}.
\end{proof}

	Lemma \ref{lem:expanding:until:on:the:same:side} implies that if case~(3) holds then expanding $\C_1$  along $\eleZeta$ eventually yields a continent for which the corresponding continental approximation of $x$ lies either entirely inside $P$ or entirely outside $P$. Then we can deduce that $x \in P$ or $x \notin P$, respectively.

		\subsubsection{Implementation}
		
	Let $P$ be an edge rectangle in $\loomif(\C_1)$. In Lemma \ref{lem:x:notin:P}, we give a sufficient condition for $x\notin P$ in terms of the \emph{boundary cusp leaves} of $P$, that is, the cusp leaves containing the sides of $P$. To use this condition to establish that $x\notin P$, a boundary cusp leaf $\ell$ of $P$ must be comparable with $\eleZeta$. If this is not yet the case in~$\C_1$ (or an expansion of $\C_1$), we continue expanding along $\eleZeta$ until comparability is achieved. Lemma \ref{lem:expanding:until:on:the:same:side} guarantees that this process terminates after finitely many steps.



	\begin{lemma}\label{lem:x:notin:P}
		Let $\ell$ be a boundary cusp leaf of an edge rectangle $P \subset \loomif(\C_1)$. Let~$\C$ be an expansion of  $\C_1$ in which $\eleZeta$ and $\ell$ are comparable.
		If $\ell$ is stable, let $c$ be any cusp of the uppermost tetrahedron of $\eleZeta^\C$. Otherwise, let $c$ be any cusp of the lowermost tetrahedron of $\eleZeta^\C$.
		Denote by $d$ the cusp of $P$ that is not the origin of $\ell$.  If $\Theta_\ell(c) \cdot\Theta_\ell(d) = -1$ then $x\notin P$.
	\end{lemma}
\begin{proof}
If $\ell$ is a stable leaf, let $\tau$ be the uppermost tetrahedron of $\eleZeta^\C$. Otherwise, let $\tau$  be the lowermost tetrahedron of $\eleZeta^\C$.  Since $\eleZeta$ and $\ell$ are comparable in $\C$, all cusps of $\tau$ are on the same side of $\ell$. Thus the value of $\Theta_\ell(c)$ does not depend on the choice of $c \in \Omega(\tau)$. Observe also that $\Theta_\ell(d)\neq 0$ and thus $\Theta_\ell(c)\cdot \Theta_\ell(d) \in \lbrace -1, 1\rbrace$.
	
 Since $x \in \loomif(\tau)$, $x$ is on the same side of $\ell$ as the fixed cusp $c$ of~$\tau$.  Therefore, by Lemma \ref{lem:theta:l:function} (1), if $\Theta_\ell(c)\cdot \Theta_\ell(d)=-1$ then $\ell$ separates $d$ from $x$. Hence $x \notin P$. 
	\end{proof}

For a boundary cusp leaf $\ell$ of $P$, the condition $\Theta_\ell(c)\cdot \Theta_\ell(d)=-1$ from Lemma~\ref{lem:x:notin:P} certifies that $\ell$ separates $x$ from the boundary cusp leaf of $P$ that is parallel to $\ell$. Using that interpretation, the algorithm for determining if $x \in P$ can be described as follows.

Let $s_1,s_2$ be the stable boundary leaves of $P$, and let $u_1,u_2$ be the unstable boundary leaves of $P$. First use Lemma \ref{lem:x:notin:P} to check whether $s_1$ separates $x$ from $s_2$. If so, then $x \notin P$. Otherwise, check whether $s_2$ separates $x$ from $s_1$. If so, then again $x \notin P$. We then perform the analogous checks for $u_1$ and $u_2$: if either $u_1$ separates $x$ from $u_2$ or $u_2$ separates $x$ from $u_1$, then $x \notin P$. If none of these four separation conditions occurs, then $x \in P$.

	\subsection{Translating flow lines}\label{subsec:translate}
In \textbf{Step 5} we \emph{translate} the initial flow line $\eleZeta$ into other flow lines that represent points of $\liftZeta$ lying inside fundamental edge rectangles. In \textbf{Step 7} we further translate these flow lines into flow lines that represent points of $\liftZeta$ lying inside non-fundamental edge rectangles contained in fundamental tetrahedron rectangles. 

As in all steps of the algorithm, we work with convex continents in $\univV$ rather than the entire $\univV$, and thus with flow intervals rather than entire flow lines. However, for a flow line~$\rho$ and $g \in \pi_1(M)$, the restriction $g(\rho)^\C$ of $g(\rho)$ to a continent $\C$ need not be a translate of the flow interval $\rho^\C$. We therefore explain the construction in terms of translating entire flow lines; the flow intervals that we actually work with are obtained by restricting the resulting flow lines to the continent under consideration. In the process, the continent may also need to be expanded to include certain tetrahedra that it does not initially contain.

Both Steps~5 and~7 involve the following situation. We are given an elevation $\eleZeta'$ of~$\zeta$, an edge rectangle $P'$ for which we have certified that $\bigcap\eleZeta'\in P'$, and a translate $P''$ of~$P'$. There is an element $g\in\pi_1(M)$ such that $P''=g(P')$. This element, however, is not given as an abstract element of $\pi_1(M)$, but rather is encoded by a path of tetrahedra in $\univV$. Specifically, let $R'$ and $R''$ be the maximal vertical expansions of $P'$ and $P''$, respectively. Both $R'$ and $R''$ are tetrahedron rectangles and thus determine tetrahedra $t'$ and $t''$ of $\univV$, respectively.  The element $g$ may be represented by a path from $t'$ to $t''$. Notice that $t'$ does not need to belong to the flow line $\eleZeta'$.

Our goal is to construct the flow line $g(\eleZeta')$. A subtlety arises from the fact that the same tetrahedron of $\univV$ may belong to several elevations of $\zeta$, so we need to keep track of the particular elevation under consideration. For this purpose, fix an enumeration of the edges of each tetrahedron of~$\V$. Then, whenever an edge $e$ of $\V$ belongs to a tetrahedron~$t$ of $\V$, it has a well-defined number in this enumeration, which we call the \emph{index of $e$ in $t$}. Lift these indices to the universal cover $\univV$.

Recall from Section~\ref{subsec:peripheral} that we fixed an indexing of the tetrahedra of the initial flow line $\eleZeta=(t_i)_{i\in\zz}$ such that $t_0\subset\F\subset\C$. Given any other elevation $\eleZeta'=(t_i')_{i\in\zz}$ of $\zeta$, we say that its indexing is \emph{compatible} with the fixed indexing of $\eleZeta$ if, for every $i\in\zz$,
\begin{itemize}
	\item $t_i'$ is a translate of $t_i$,
	\item the bottom diagonal of $t_{i+1}'$ has the same index in $t_i'$ as the bottom diagonal of $t_{i+1}$ has in $t_i$.
\end{itemize}
Every elevation $\eleZeta'$ of $\zeta$ admits an indexing compatible with that of $\eleZeta$. Moreover, if $(t_i')_{i\in\zz}$ and $(s_i')_{i\in\zz}$ are two compatible indexings of $\eleZeta'$, then Definition \ref{defn:flow:cycle}  implies that there exists an element $h$ of the stabilizer of $\eleZeta'$ such that
$s_i'=h(t_i')$ for every $i\in\zz$.

Returning to the construction of the flow line $g(\eleZeta')$, let $(t_i')_{i\in\zz}$ be an indexing of $\eleZeta'$ compatible with the fixed indexing of $\eleZeta$. Choose a path $\delta$ from $t'$ to $t_0'$, and let $t_0''$ be the terminal tetrahedron of the translate of $\delta$ starting at $t''$. Since $t''=g(t')$, we must have $t_0''=g(t_0')$. The flow line $g(\eleZeta')$ is therefore obtained by extending $t_0''$ to a flow line $(t_i'')_{i\in\zz}$ whose indexing is compatible with that of $(t_i')_{i\in\zz}$.
	
	\subsection{Uniquifying flow lines}\label{subsec:fellow-travel}
	
	In \textbf{Steps 6} and \textbf{8} we \emph{uniquify} flow lines that represent points of $\liftZeta$ contained inside the same simplex rectangle $R$. There are two aspects of this process. First, in earlier steps of the algorithm we translate flow lines representing points in various translates of $R$ to obtain all flow lines representing points of $R$. As a result, the same flow line is often constructed multiple times, arising as a translate of different flow lines. One purpose of Steps~6 and~8 is to remove these duplicate copies.
	
	Given two flow intervals $\rho_1^\C$, $\rho_2^\C$ that are elevations of $\zeta$, one can directly determine whether they are equal as sets of tetrahedra of $\C$. 
	A more efficient procedure uses the compatible indexings of elevations of $\zeta$ introduced in Section \ref{subsec:translate}. Recall that $\zeta$ has length~$k$. If $(t_i')_{i\in\zz}$ is a compatibly indexed elevation of $\zeta$, we call the residue of $i$ modulo~$k$ the \emph{$\zeta$-index} of $t_i'$. Then $\rho_1^\C=\rho_2^\C$ if and only if their uppermost tetrahedra coincide and have the same $\zeta$-index.

A more subtle aspect of uniquifing flow lines stems from the observation made below Proposition~\ref{prop:square} that distinct translates of the same flow line may represent the same point of $\loom$. Such translates give rise to different continental approximations of the same point that cannot be separated in Step~9. This would prevent the algorithm from terminating. Therefore, we must identify such flow lines earlier in the process and retain only one representative for each point of $ R \cap \liftZeta$.

	\begin{definition}
		Let $\rho, \rho'$ be distinct flow lines in $(\loom, \loom^+, \loom^-)$. We say that $\rho$, $\rho'$ \emph{fellow-travel} if  $\rho, \rho'$ represent the same point or cusp of $\loom$.
	\end{definition}

Flow lines that are elevations of different flow cycles may also fellow-travel, as Theorem~\ref{thm:flow:lines} does not guarantee the uniqueness of a flow cycle of $\V$ associated with a closed orbit $\gamma$ of $\Psi$. Indeed, a single closed orbit may be encoded by multiple distinct flow cycles,  but only finitely many \cite[Theorem~6.1(3)]{LMT_flow}. 
	
	The following lemma gives a criterion for recognizing fellow-travelling translates of the same periodic non-peripheral flow line. As in Section~\ref{subsec:flow:cycles:and:lines}, $g_{\eleZeta}$, denotes a generator of the stabilizer of $\eleZeta$.

	\begin{lemma}\label{lem:fellow:travel}
		Let $\zeta$ be a flow cycle that is non-peripheral. Let $\eleZeta, \eleZeta'$ be distinct elevations of $\zeta$. Let  $h \in \pi_1(M)$ be such that $h(\eleZeta) = \eleZeta'$. Then $\eleZeta, \eleZeta'$ fellow-travel if and only if $hg_{\eleZeta} = g_{\eleZeta}h$.
	\end{lemma}		
	\begin{proof} To simplify notation, set $g = g_{\eleZeta}$ and $g' = g_{\eleZeta'}$. Then $g'h = hg$, because $\eleZeta' = h(\eleZeta)$.
		
		Suppose that $\eleZeta$, $\eleZeta'$ fellow-travel. Since $\zeta$ is non-peripheral, there is $y \in \loom$ such that $y = \bigcap \eleZeta = \bigcap \eleZeta'$. By Lemma \ref{lem:stabilizers}, both $g$ and $g'$ fix $y$. 
		Recall from Proposition \ref{prop:orbit:space} \eqref{orbit_space:fixed:points} that the stabilizer of $y$ is infinite cyclic. Therefore $\langle g \rangle$ and $\langle g' \rangle$ are nontrivial subgroups of the same infinite cyclic group, and thus they intersect nontrivially. However, the equality $g'h=hg$ implies that $g$ and $g'$ map to the same element under the abelianization map 
		\mbox{$\pi_1(M)\rightarrow H_1(M;\zz)$}. It follows  that $g = g'$ and thus $hg = gh $. 
		
		Conversely, if $hg = gh$ then $g' = hgh^{-1} = g$. Therefore the points represented by $\eleZeta$ and~$\eleZeta'$ are both stabilized by $g$. Lemma \ref{lem:loom:stabilizers} then implies that $\eleZeta, \eleZeta'$ represent the same point of $\loom$. Thus  $\eleZeta, \eleZeta'$ fellow-travel.
	\end{proof}

\subsubsection{Implementation}

Fix a tetrahedron $t$ of $\zeta$. To decide whether two elevations $\eleZeta$,~$\eleZeta'$ of $\zeta$ fellow-travel, choose lifts $\widehat{t}\subset\eleZeta\cap\C$, $\widehat{t}'\subset\eleZeta'\cap\C$ of $t$, and a path $\alpha$ in $\C$ from $\widehat{t}$ to $\widehat{t}'$. Let $\beta_1$ be the path obtained by first following~$\alpha$ and then moving one cycle up along~$\eleZeta'$, and let $\beta_2$ be the path obtained by first moving one cycle up along $\eleZeta$ and then following the translate $g_{\eleZeta}(\alpha)$ of $\alpha$. The continent $\C$ may need to be expanded to contain both $\beta_1$ and $\beta_2$. We claim that $\eleZeta, \eleZeta'$ fellow-travel if and only if the terminal tetrahedra of $\beta_1$ and $\beta_2$ agree; this is the condition that we actually check in the algorithm.

To justify the claim, let $h\in\pi_1(M)$ be the element represented by $\alpha$, so that \mbox{$\widehat{t}'=h(\widehat{t})$}.
Since the translation by one cycle along $\eleZeta'$ corresponds to applying $g_{\eleZeta'} = hg_{\eleZeta}h^{-1}$, the terminal tetrahedron of~$\beta_1$ is $g_{\eleZeta'}h(\widehat{t}) = hg_{\eleZeta}(\widehat{t})$. On the other hand, the terminal tetrahedron of $\beta_2$ is $g_{\eleZeta}(\widehat{t'}) = g_{\eleZeta}h(\widehat{t})$.
Thus the terminal tetrahedra of $\beta_1$ and $\beta_2$ coincide if and only if $hg_{\eleZeta}(\widehat{t})=g_{\eleZeta}h(\widehat{t})$. 
Since $\pi_1(M)$ acts freely on the tetrahedra of $\univV$ (see Construction \ref{const:finite:veering}) this is equivalent to
$hg_{\eleZeta}=g_{\eleZeta}h$.
By Lemma~\ref{lem:fellow:travel}, this holds if and only if $\eleZeta$ and $\eleZeta'$ fellow-travel.

	\subsection{Separating and sorting flow intervals}\label{subsec:separating}

	After \textbf{Step 8} we are guaranteed that for each fundamental tetrahedron rectangle $R$ and for every point $y \in R \cap \liftZeta$ we have precisely one flow interval $\eleZeta_y^\C$ whose corresponding flow line $\eleZeta_y$ represents $y$. As in Section \ref{subsec:locating}, after possibly expanding $\C$ along these flow lines, we can ensure that each continental approximation $N_\C(y) = \bigcap \eleZeta_y^\C$ is contained in $R$. We further aim to obtain a continent $\C$ with the following properties:
	\begin{itemize}[leftmargin=3.2em]
		\item[\ref{cond:separated:from:cusp:leaves}] For every $y \in R \cap \liftZeta$, no cusp leaf originating at a cusp of $R$ intersects $N_\C(y)$. 
		\item[\ref{cond:separated:from:each:other}] For every $y,z \in R \cap \liftZeta$, no leaf intersects both $\overline{N_\C(y)}$ and $\overline{N_\C(z)}$. 
	\end{itemize}

These properties were previously mentioned in the proof of Theorem \ref{thm:algorithm:correct}.

	If $\C$ satisfies \ref{cond:separated:from:cusp:leaves} and \ref{cond:separated:from:each:other}, we say that $\C$ \emph{separates the flow lines intersecting $R$ and the cusps of $R$}.  
	In this case, the horizontal and vertical orders on 
	$\Omega(R) \cup \lbrace p_y \ | \ y \in R \cap \liftZeta\rbrace $, induced by the cardinal directions in~$R$, do not depend on the choice of $p_y \in N_\C(y)$. Thus we can unambiguously determine the horizontal and vertical orders on the set $\Ideal(R)$  of cusps of $R$ and punctures of $R \bez \liftZeta$, without knowing the exact position of any $y \in R \cap \liftZeta$.
	This is sufficient and necessary to find all $\zeta$-tetrahedron rectangles contained in $R$ in \textbf{Step~10}.

	The existence of a continent that separates the flow lines intersecting $R$ and the cusps of~$R$ follows from the ideas that we introduced in Section~\ref{subsec:locating}.
	First observe that Lemma~\ref{lem:expanding:until:on:the:same:side} implies that by expanding $\C$ along the flow lines $\eleZeta_y$, for $y \in R \cap \liftZeta$, we eventually obtain a continent in which these flow lines are comparable with the cusp leaves originating from the cusps of $R$. This gives \ref{cond:separated:from:cusp:leaves}. 
	Furthermore, building on Definition~\ref{def:leaf:line:comparable}, we define what it means for two flow lines to be comparable in~$\C$.
	
	\begin{definition}
	Let $\C$ be a convex continent. Let $\rho_1, \rho_2$ be flow lines such that $\rho_i^\C \neq \emptyset$ for each $i \in \lbrace 1, 2\rbrace$.
	We say that $\rho_1, \rho_2$ are \emph{comparable} in $\C$ if 
	\begin{itemize}
		\item For $i,j \in \lbrace 1, 2\rbrace$, with $i\neq j$, the flow line $\rho_i$ is comparable in $\C$ with every stable boundary cusp leaf of the uppermost tetrahedron of $\rho_j^\C$,
		\item For $i,j \in \lbrace 1, 2\rbrace$, with $i\neq j$, the flow line  $\rho_i$ is comparable in $\C$ with every unstable boundary cusp leaf of the lowermost tetrahedron of $\rho_j^\C$.
	\end{itemize}
	\end{definition}
Lemma \ref{lem:expanding:until:on:the:same:side} implies that if $\rho_1, \rho_2$ are periodic, do not fellow-travel, and intersect $\C$ then expanding $\C$ along $\rho_1$ and $\rho_2$ eventually yields a continent in which $\rho_1$ and $\rho_2$ are comparable. 
Since the flow lines $\eleZeta_y$, for $y \in R \cap \liftZeta$, satisfy all these conditions relative to the continent $\C$ constructed up Step 9, expanding $\C$ along these flow lines eventually gives \ref{cond:separated:from:each:other}.

\subsubsection{Implementation}

Let $\C$ be a convex continent for which \ref{cond:meets:all:relevant:flowlines}, \ref{cond:separated:from:cusp:leaves} and \ref{cond:separated:from:each:other} holds. It can be  constructed as described in the final paragraphs of Section \ref{subsec:locating}. Here we focus on ordering the points of $R \cap \liftZeta$ for a fundamental tetrahedron rectangle $R \subset \loomif(\C)$. Once again, we use the side functions $\Theta_\ell:\Omega(\C) \rightarrow \lbrace -1, 0, 1\rbrace$ introduced in \eqref{function:on:cusps}. Fix the cardinal directions in $R$ according to Convention \ref{convention:cardinal:directions}. Let $\ell_i^\pm$ be the stable/unstable cusp leaves originating from the cusps of $R$ and intersecting~$R$, ordered so that $\ell_1^+$ is west of $\ell_2^+$ and $\ell_1^-$ is south of $\ell_2^-$. Furthermore, denote the west and the south cusp of $R$ by $w$ and $s$, respectively. 

For $y \in R\cap \liftZeta$, let $c$ be any cusp of the uppermost tetrahedron of $\eleZeta_y^\C$. Then the position of $y$ relative to the cusp leaves $\ell_i^+$ is uniquely determined by the triple $(\Theta_{\ell_1^+}(c), \Theta_{\ell_2^+}(c), \Theta_{\ell_1^+}(w))$. Observe that the last entry is independent of $y$. Its role is to detect whether $\ell_1^+$ originates from the south or the north cusp of $R$, which in turn affects the interpretation of the first two entries; see Lemma \ref{lem:theta:l:function}. The choice of the cusp $c$ does not affect the triple, because $\eleZeta_y$ and $\ell_i^\pm$ are comparable in $\C$. 
Analogously, the position of $y$ relative to the cusp leaves $\ell_i^-$ is uniquely determined by the triple $(\Theta_{\ell_1^-}(d), \Theta_{\ell_2^-}(d), \Theta_{\ell_1^-}(s))$, where $d$ is any cusp of the lowermost tetrahedron of~$\eleZeta_y^\C$. It follows that the side functions of the continental cusp leaves can be used to determine which of the nine subrectangles of $R$ cut out by the leaves $\ell_i^\pm$ contains $y$.

To compare the relative horizontal position of distinct $y, z \in R \cap \liftZeta$ it suffices to compute $\Theta_{\ell^+}(c)\cdot \Theta_{\ell^+}(w)$, where $c$ is any cusp of the uppermost tetrahedron of $\eleZeta_z^\C$ and $\ell^+$ is any stable boundary cusp leaf originating from a cusp of the uppermost tetrahedron of $\eleZeta_z^\C$. By Lemma~\ref{lem:theta:l:function}~(1), this product records whether $c$ and $w$ are on the same side of $\ell^+$. Since $\eleZeta_y$ and $\eleZeta_z$ are comparable in~$\C$, the value of $\Theta_{\ell^+}(c)\cdot \Theta_{\ell^+}(w)$ informs about whether $y$ and $w$ are on the same side of $\ell$. This uniquely determines the relative horizontal position of $y$ and $z$.

Analogously, the relative vertical position of $y$ and $z$ is uniquely determined by the product $\Theta_{\ell^-}(d)\cdot \Theta_{\ell^-}(s)$, where $d$ is any cusp of the lowermost tetrahedron of $\eleZeta_y^\C$ and $\ell^-$ is  any unstable boundary cusp leaf originating from a cusp of the lowermost tetrahedron of $\eleZeta_z^\C$.

	It follows that once a continent separating the flow lines intersecting $R$ and the cusps of $R$ has been constructed, determining the horizontal and vertical orders on $\Ideal(R)$ reduces to computing $\Theta_\ell(c)$ for finitely many cusps $c$ of $\C$ and finitely many continental cusp leaves $\ell$. Consequently, these orders can be determined in finite time.

	\section{Applications to pseudo-Anosov flows} \label{sec:applications}
	
	In this section, we present the results of several computational experiments performed using the drilling algorithm described in Section~\ref{sec:algorithm}. More detailed data from these experiments is available in the `Drilling data' section of the veering census \cite{VeeringCensus}.  We also propose a number of conjectures motivated by the obtained results.

	\subsection{Layered parents of non-layered veering triangulations}\label{subsec:layered:parents}

			In Definition \ref{def:layering} we used landscapes to define layered subtriangulations of a taut triangulation of $\rr^3$. This notion can be adapted to finite taut triangulations as follows. 
			
			\begin{definition}\label{def:layered:tri}
				Let $(\T, \alpha)$ be a finite taut triangulation of a three-manifold $M$. We say that $(\T, \alpha)$ is \emph{layered} if $(\T, \alpha)$ carries a fiber of some fibration of $M$ over the circle.
			\end{definition}
	In particular, finite taut layered triangulations exist only on three-manifolds that are fibered over the circle. 
	For a proof that Definition \ref{def:layered:tri} is an adaptation of Definition \ref{def:layering} see \cite[Proposition 4.5]{Landry-stable-loops}. We refer the reader to \cite[Section~2.3]{Parlak-mutations} and \cite[Theorem 5.15]{LMT} for alternative characterizations of layeredness.
	
	Let $\V$ be a finite layered veering triangulation of $M$. 
	Denote by $\Psi$ a transitive pseudo-Anosov flow on some Dehn filling $N$ of $M$, and by $\Lambda$ a perfect set for $\Psi$ such that $\V = \V(\Psi, \Lambda)$. 
	The horizontal branched surface of $\V$ is positively transverse to $\Psi \bez \Lambda$ by \cite[Theorem 5.1]{LMT_flow}. Thus if $S$ is carried by~$\V$ then $S$ is positively transverse to $\Psi \bez \Lambda$. If~$S$ is furthermore a fiber of a fibration of $M$ over the circle, then the drilled flow \mbox{$\Psi \bez \Lambda$} is orbit equivalent  to the suspension flow of a pseudo-Anosov homeomorphism of $S$.
	
	
	Among all pseudo-Anosov flows, suspension flows are by far the best understood. Their key advantage is that the dynamics of the entire flow is encoded by the underlying pseudo-Anosov homeomorphism.
	It was therefore of interest to try to restrict a general pseudo-Anosov flow to a suspension flow on a complement of finitely many closed orbits. Fried proved that this is always possible for transitive Anosov flows \cite[Theorem~1]{Fried_drill}. Brunella further extended this result to transitive pseudo-Anosov flows \cite[Theorem~1]{Brunella_surfaces}. 

	In \cite{Tsang-branched-surfaces} Tsang used the machinery of veering triangulations to refine the results of Fried and Brunella by obtaining explicit bounds on the number of orbits that need to be drilled out, their combinatorial length, and the Euler characteristic of the resulting \emph{Birkhoff section} (the fiber after drilling; see \cite[Definition 2.14]{Tsang-Birkhoff}). 
	
\begin{theorem}\cite[Theorem 6.3]{Tsang-Birkhoff} \label{thm:tsang:layered:parent}
	Let $\V$ be finite a veering triangulation constructed from a transitive pseudo-Anosov flow $\Psi$ and a perfect set $\Lambda$ for $\Psi$. Suppose that $\V$ has~$n$ tetrahedra.
	Then there exists a closed orbit $\gamma$ of $\Psi \bez \Lambda$, whose square is represented by a flow cycle $\zeta$ of $\V$ of length at most $2\cdot10^6n^{12}$, and such that the veering parent $\V_\gamma$ is layered. \qed
\end{theorem}
	The condition on the length of the flow cycle representing $\gamma^2$ instead of $\gamma$ arises due to the fact that a closed orbit of $\Psi$ is guaranteed to be represented by a flow cycle only after taking its square; see the discussion at the end of Section \ref{sec:flow:cycles} or \cite[Theorem 6.1(4)]{LMT_flow}.

	On one hand, Theorem~\ref{thm:tsang:layered:parent} is extremely efficient: it shows that a layered veering parent can be constructed by drilling out just one additional orbit of the flow. On the other hand, the upper bound on the length of that orbit is impractically large. 
	
	Computational problems in low-dimensional topology are often much easier on many individual examples than the general theoretical bounds might suggest. We therefore used the drilling algorithm described in Section~\ref{sec:algorithm} to study the problem of constructing layered parents of non-layered veering triangulations experimentally. Our approach is independent of Tsang's methods and instead relies on an exhaustive brute-force search. For each non-layered veering triangulation in the veering census, we enumerated all flow cycles of length 1 and checked whether any of them yielded a layered parent. If not, we repeated the process with all flow cycles of length at most~2, and so on. We present the results in Table \ref{tab:layered:parents} and summarize them in Experimental Result \ref{experimental:layered:parents}.

	\begin{table}[h]
		\centering
		\setlength{\tabcolsep}{3pt}
		\begin{tabular}{|c|c|c|c|c|c|c|c|c|c|c|c|c|c|}
			\hline
			$n$ & $\NL_n$ & $\NL_n^1$ & $\NL_n^2$ & $\NL_n^3$ & $\NL_n^4$ & $\NL_n^5$ & $\NL_n^6$ & $\NL_n^7$ & $\NL_n^8$ & $\NL_n^9$ & $\NL_n^{10}$ & $\NL_n^{11}$ & $\NL_n^{12}$ \\
			\hline
			5 & 4 & 4 & 0 & 0 & 0 & 0 & 0 & 0 & 0 & 0 & 0 & 0 & 0 \\ \hline
			6 & 13 & 13 & 0 & 0 & 0 & 0 & 0 & 0 & 0 & 0 & 0 & 0 & 0 \\ \hline
			7  & 24 & 24 & 0 & 0 & 0 & 0 & 0 & 0 & 0 & 0 & 0 & 0 & 0\\ \hline
			8  & 60 & 59 & 0 & 1 & 0 & 0 & 0 & 0 & 0 & 0 & 0 & 0 & 0 \\ \hline
			9  & 120 & 118 & 0 & 0 & 1 & 1 & 0 & 0 & 0 & 0 & 0 & 0 & 0 \\ \hline
			10 & 253  & 244 & 5 & 1 & 1 & 0 & 1 & 1 & 0 & 0 & 0 & 0 & 0 \\ \hline
			11  & 492 & 474 & 9 & 2 & 6 & 0 & 0 & 1 & 0 & 0 & 0 & 0 & 0 \\ \hline
			12  & 1034 & 960 & 43 & 9 & 15 & 4 & 2 & 1 & 0 & 0 & 0 & 0 & 0 \\ \hline
			13  & 2075 & 1888 & 113 & 26 & 29 & 10 & 3 & 3 & 2 & 1 & 0 & 0 & 0 \\ \hline
			14 & 4263 & 3668 & 368 & 83 & 71 & 43 & 14 & 6 & 4 & 5 & 1 & 0 & 0 \\ \hline
			15 & 8786 & 7363 & 906 & 200 & 127 & 118 & 36 & 14 & 8 & 12 & 2 & 0 & 0 \\ \hline
			16 & 18157 & 14392 & 2391 & 581 & 297 & 281 & 114 & 45 & 27 & 18 & 7 & 2 & 2 \\ \hline
		\end{tabular}
		\caption{Shortest flow cycles of non-layered veering triangulations with up to 16 tetrahedra that yield a layered veering parent. For $n \in \mathbb{N}$,  by $\NL_n$ we denote the number of non-layered veering triangulations with exactly $n$ tetrahedra. This data can be found in the veering census \cite{VeeringCensus}. Since there are no non-layered veering triangulations with fewer than five tetrahedra, the rows are indexed starting from $n=5$. Furthermore, for $n$ as above and $k \in \mathbb{N}$, by $\NL_n^k$ we denote the number of non-layered veering triangulations with exactly $n$ tetrahedra for which the shortest flow cycle yielding a layered veering parent has length $k$. For veering triangulations with up to 16 tetrahedra $k$ never exceeds 12.}
		\label{tab:layered:parents}
	\end{table}

\Needspace{5\baselineskip}
\begin{experimental}\label{experimental:layered:parents} \hfill \nopagebreak
	\begin{enumerate}
		\item 
		Every non-layered veering triangulation with $n\leq 16$ tetrahedra admits a flow cycle of length at most $n-3$ that yields a layered veering parent.
		
		
		\item Among all non-layered veering triangulations with at most $16$ tetrahedra, approximately $82.78\%$ admit a flow cycle of length $1$ that yields a layered veering parent. Approximately $93.66\%$ admit such a flow cycle of length at most $2$. \qed
	\end{enumerate}
\end{experimental}

As in Table \ref{tab:layered:parents}, for $n,k \in \mathbb{N}$ let $\NL_n^k$ be the number of non-layered veering triangulations with exactly $n$ tetrahedra for which the shortest flow cycle yielding a layered veering parent has length $k$. Furthermore, set
\[M_n = \max \lbrace k \in \mathbb{N} \ | \ \NL_n^k > 0\rbrace. \]
The observed growth of $M_n$ with $n$ is very slow compared to the growth allowed by the theoretical upper bound of $10^6n^{12}$ that appears in Theorem~\ref{thm:tsang:layered:parent}. Of course, as the number of tetrahedra increases, more complicated combinatorial configurations may arise, potentially requiring substantially longer flow cycles to obtain a layered veering parent. Thus our experimental data for $n \leq 16$ may not be representative of the general behaviour. However, Tsang himself remarked that the methods used to derive the theoretical bound are unlikely to yield a sharp estimate; see \cite[Section~7.2]{Tsang-Birkhoff}.

Motivated by the data in Table~\ref{tab:layered:parents}, we propose the following conjecture.
	\begin{conjecture}\label{conj:layered}
Every non-layered veering triangulation with $n$ tetrahedra admits a layered veering parent obtained by drilling out a single flow cycle of length at most $n$.
	\end{conjecture}

A careful analysis of the layered veering parents produced by our brute-force search may lead to new methods for constructing Birkhoff sections of pseudo-Anosov flows that are substantially different from the existing methods described in \cite{Fried_drill, Brunella_surfaces, Tsang-Birkhoff}.
	
	\subsection{Geometric parents of non-geometric veering triangulations}
	
	Every three-manifold that admits a veering triangulation is hyperbolic \cite[Theorem 1.5]{veer_strict-angles}. However, the combinatorics of the triangulation is often not compatible with the hyperbolic metric in the following sense.
	
	\begin{definition}
	Let $\T$ be an oriented ideal triangulation of a hyperbolic three-manifold~$M$. We say that $\T$ is \emph{geometric} if every tetrahedron of $\T$ can be equipped with a hyperbolic structure of strictly positive volume so that the face pairings are realized by hyperbolic isometries and the resulting metric on the interior of $M$ is complete.
	\end{definition}
	
	One can determine whether $\T$ is geometric by solving the \emph{gluing and completeness equations} associated to $\T$; see \cite[Section~3.1]{FKP-survey-hyperbolic}.
	
	The first known example of a non-geometric veering triangulation appears in \cite[Section~7]{Hodgson_non-geom}. Later, Futer-Taylor-Worden proved that layered veering triangulations are generically non-geometric \cite[Theorem 1.1]{FTW-random-veering}. It is expected that the same is true for non-layered veering triangulations.


There are 931 non-geometric veering triangulations in the veering census, 220 of which are layered. We used the drilling algorithm described in Section \ref{sec:algorithm} and SnapPy \cite{snappea} to determine if they have geometric veering parents. We present the results in Table \ref{tab:geometric:parents} and summarize them in Experimental Result \ref{experimental:geometric}.

	\begin{table}[h]
	\centering
	\setlength{\tabcolsep}{3pt}
	\begin{tabular}{|c|c|c|c|c|c|}
		\hline
		$n$ & $\NN_n$ & $\NN_n^1$ & $\NN_n^2$ & $\NN_n^3$ & $\NN_n^4$ \\
		\hline
		9  & 1 & 1 & 0 & 0 & 0  \\ \hline
		10 & 3 & 3 & 0 & 0 & 0  \\ \hline
		11  & 9 & 8 &1 & 0 & 0 \\ \hline
		12  & 22 & 18 & 3 & 1 & 0  \\ \hline
		13  & 52 & 40 & 12 & 0 & 0 \\ \hline
		14 & 110  & 79 & 31 & 0 & 0  \\ \hline
		15 & 232 & 153 & 77 & 2 & 0\\ \hline
		16 &  502 & 311 & 184 & 5 & 2 \\ \hline
	\end{tabular}
	\caption{Shortest flow cycles of non-geometric veering triangulations with up to 16 tetrahedra that yield a geometric veering parent. For $n \in \mathbb{N}$,  let $\NN_n$ be the number of non-geometric veering triangulations with exactly $n$ tetrahedra. This data can be found in the veering census \cite{VeeringCensus}. Since there are no non-geometric veering triangulations with fewer than nine tetrahedra, the rows are indexed starting from $n=9$. Furthermore, for $n$ as above and $k \in \mathbb{N}$, by $\NN_n^k$ we denote the number of non-geometric veering triangulations with exactly $n$ tetrahedra for which the shortest flow cycle yielding a geometric veering parent has length~$k$. For veering triangulations with up to 16 tetrahedra~$k$ never exceeds~4.}
	\label{tab:geometric:parents}
\end{table}
	
	\begin{experimental}\label{experimental:geometric}
		Every non-geometric veering triangulation with up to 16 tetrahedra admits a flow cycle $\zeta$ such that the veering parent $\V_\zeta$ is geometric. \qed
	\end{experimental}

	In all the cases that we examined, the shortest flow cycle required to obtain a geometric veering parent was quite short: its length was at most $n/3$, where $n$ denotes the number of tetrahedra of the initial triangulation. Furthermore, in most of these cases we found several distinct short flow cycles that yield geometric veering parents. This motivates the following conjecture.
	
	\begin{conjecture}\label{conj:geometric}
		Every transitive pseudo-Anosov flow can be encoded by a geometric veering triangulation.
	\end{conjecture}

The relationship between the dynamics of a (drilled) pseudo-Anosov flow on a hyperbolic three-manifold and its hyperbolic geometry is poorly understood beyond coarse geometric properties \cite{Fenley_QG_flows, Huang-Taylor-hyperbolic}. The appeal of the above conjecture lies in the possibility of making the two more compatible with each other after drilling out finitely many closed orbits.

	\subsection{Orbits that are not ambiently isotopic to the freely homotopic geodesics}\label{subsec:orbit:geodesic}
By a theorem of Zeghib \cite[Theorem A]{Zeghib-hyperbolic}, a closed hyperbolic 3-manifold cannot admit a one-dimensional foliation of class $C^0$ with geodesic leaves.  Consequently, the orbits of a pseudo-Anosov flow $\Psi$ on a hyperbolic 3-manifold cannot form a geodesic foliation. It is therefore natural to ask whether a given closed orbit of $\Psi$ is geodesic and, if not, how far it is from being one. For instance, Fenley proved that orbits of $\Psi$ are quasi-geodesics if and only if no closed orbit of $\Psi$ is non-trivially freely homotopic to itself and there is a uniform upper bound on the number of closed orbits of $\Psi$ in any given free homotopy class \cite[Main Theorem]{Fenley_QG_flows}.

Here we use our drilling algorithm to study  if the orbits of $\Psi$ differ from geodesics topologically. Since we work with veering triangulations, we consider closed orbits of $\Psi \bez \Lambda$ for some perfect set $\Lambda$ for $\Psi$. Under certain conditions, the implementation of our algorithm allows us to determine whether such an orbit is ambiently isotopic to the unique geodesic in its free homotopy class using the following approach.

Let $M$ be a three-manifold that admits a finite veering triangulation $\V$.  Let $\gamma$ be a closed orbit of the drilled pseudo-Anosov flow underlying~$\V$. Then $\gamma$ is represented by a flow cycle $\zeta$ of $\V$; see Theorem \ref{thm:flow:lines}. Moreover, since $M$ is hyperbolic \cite[Theorem~1.5]{veer_strict-angles}, there is a unique closed geodesic $\gamma^\ast$ in $M$ that is freely homotopic to $\gamma$. Observe that $\gamma$ and~$\gamma^\ast$ are not ambiently isotopic if 
 $M_\gamma = M \bez \gamma$ and $M_\gamma^\ast = M \bez \gamma^\ast$ are not homeomorphic.
 %
 %

The manifold $M_\gamma$ is always hyperbolic because it is triangulated by the veering parent $\V_\zeta = \V_\gamma$. The manifold $M_\gamma^\ast$ is hyperbolic provided that  $\gamma^\ast$ is embedded \cite[Theorem~1.2.1]{Kojima-hyperbolic-drilling}.  In this case, the Mostow–Prasad rigidity theorem \cite[Theorem~15.3.5]{Martelli} implies that it suffices to show that $M_\gamma$ and $M_\gamma^\ast$ are not isometric to conclude that  $\gamma$ and $\gamma^\ast$ are not ambiently isotopic. If we have any ideal triangulation of $M_\gamma^\ast$, this can be checked using  SnapPy \cite{snappea}.

To construct such a triangulation, we first use $\zeta$ to find a word representing the homotopy class of $\gamma$ in some generating set of $\pi_1(M)$. We  then attempt to drill out the corresponding geodesic $\gamma^\ast$ using the \texttt{drill\_word} function in SnapPy \cite{snappea}. The current implementation of this computation may fail for various reasons, including when~$\V$ is not geometric or when SnapPy fails to establish that $\gamma^\ast$ is embedded. However, when it succeeds, it produces an ideal triangulation $\T_\zeta^\ast = \T_\gamma^\ast$ of~$M_\gamma^\ast$. 

\begin{experimental}\label{experimental:geodesic}
	Every drilled pseudo-Anosov flow underlying a geometric veering triangulation  with up to 10 tetrahedra admits a closed orbit  that is not ambiently isotopic to the unique geodesic in its free homotopy class. \qed
\end{experimental}

The above result is obtained by finding, for every geometric veering triangulation $\V$ with at most 10 tetrahedra, a flow cycle $\zeta$ such that the construction of $\T_\zeta^\ast$ succeeds and the manifolds triangulated by $\V_\zeta$ and $\T_\zeta^\ast$ are not isometric. Table \ref{tab:geodesic} records the lengths of the flow cycles found in this way.

	\begin{table}[h]
	\centering
	\setlength{\tabcolsep}{3pt}
	\begin{tabular}{|c|c|c|c|c|c|c|c|c|c|c|c|}
		\hline
		$n$ & $\NG_n$ & $\NG_n^1$ & $\NG_n^2$ & $\NG_n^3$ & $\NG_n^4$ & $\NG_n^5$ & $\NG_n^6$ & $\NG_n^7$ & $\NG_n^8$ & $\NG_n^9$ & $\NG_n^{10}$ \\
		\hline
		2 &  2& 0 & 0 & 0 & 0 & 2 & 0 & 0 & 0 & 0 & 0 \\\hline 
		3 & 3	&0  & 0 & 1 &  1 &  1 &  0 &  0 &  0 &  0 &  0\\\hline 
		4 & 12 & 0 &  0 &  2 &  5 &  3 &  2 &  0 &  0 &  0 &  0\\\hline 
		5 & 20 &0 & 0 & 4 & 10 & 5 & 1 & 0 & 0 & 0 & 0\\\hline 
		6 & 50 & 0 & 1 & 5 & 24 & 14 & 4 & 0 & 2 & 0 & 0\\\hline 
		7 & 85 & 0 & 1 & 9 & 47 & 22 & 6 & 0 & 0 & 0 & 0\\\hline 
		8 & 202 & 0 & 1 & 15 & 98 & 61 & 21 & 4 & 1 & 0 & 1\\\hline 
		9 & 354 &  0 & 1 & 22 & 170 & 122 & 35 & 4 & 0 & 0 & 0\\\hline 
		10& 742 & 0 & 1 & 37 & 304 & 286 & 92 & 15 & 5 & 1 & 1\\\hline 
	\end{tabular}
	\caption{Shortest flow cycles of geometric veering triangulations with up to~10 tetrahedra representing closed orbits that are not ambiently isotopic to the geodesics in their free homotopy classes. For $n \in \mathbb{N}$,  by $\NG_n$ we denote the number of geometric veering triangulations with exactly $n$ tetrahedra. This data can be found in the veering census \cite{VeeringCensus}. Since there are no veering triangulations with fewer than two tetrahedra, the rows are indexed starting from $n=2$. Furthermore, for $n$ as above and $k \in \mathbb{N}$, by $\NG_n^k$ we denote the number of geometric veering triangulations with exactly $n$ tetrahedra for which the shortest flow cycle giving an orbit not freely homotopic to a geodesic has length $k$. For veering triangulations with up to 10 tetrahedra $k$ never exceeds 10.}
	\label{tab:geodesic}
\end{table}

Experimental Result \ref{experimental:geodesic} motivates  the following conjecture.

\begin{conjecture}\label{conj:geodesic}
	Every  (possibly drilled) pseudo-Anosov flow on a hyperbolic three-manifold admits closed orbits that are not ambiently isotopic to the unique geodesic in their respective free homotopy classes.
\end{conjecture}

The intuition behind Conjecture \ref{conj:geodesic} is that, since the orbits of $\Psi$ do not follow geodesics exactly, the accumulated deviation along sufficiently long and complicated orbits should eventually prevent the orbit from being ambiently isotopic to the geodesic that is freely homotopic to it. This is somewhat similar in spirit to the so-called Frisch-Wasserman-Delbr\"uck conjecture which states that a sufficiently long simple closed curve in $\rr^3$ is almost surely knotted; see \cite[Theorem 1]{Diao-random-knots} for the proof of this conjecture for curves that are Gaussian random polygons.

Experimental Result~\ref{experimental:geodesic} more directly suggests a conjecture concerning only drilled pseudo-Anosov flows. Extending it to pseudo-Anosov flows on closed hyperbolic three-manifolds is not straightforward even when the orbits that are drilled out are short. This is because the relationship between the hyperbolic structures before and after drilling is currently unknown when drilling along orbits of the flow rather than geodesics.
Nonetheless, it seems reasonable to expect such a relationship to exist, and the infinitely many possible choices of a perfect set $\Lambda$  for a given pseudo-Anosov flow $\Psi$ make it plausible that one can choose $\Lambda$ so that some closed orbit of $\Psi\bez\Lambda$ is not ambiently isotopic to its geodesic representative either in $N\bez\Lambda$ or in $N$.

\begin{example}[Figure-eight knot exterior] \label{ex:figure:eight:geodesic}
	Let $M$ be the figure-eight knot complement. Then $M$ is a hyperbolic once-punctured torus bundle. Let $\V$ be the veering triangulation of~$M$ that encodes the suspension flow $\Psi$ on $M$. Then $\V$ has two tetrahedra \cite[Example~3.2]{SchleimSegLinks} and, by Experimental Result \ref{experimental:geodesic}, admits flow cycles representing closed orbits of $\Psi$ that are not ambiently isotopic to the geodesic representatives of their free homotopy classes.
	
	For one such cycle $\zeta$, Figure~\ref{fig:fig8_geodesic} shows the fundamental tetrahedron rectangles for $\V$  with the points represented by the elevations of $\zeta$ marked. Let $\gamma$ be the corresponding closed orbit of $\Psi$.
	The manifold~$M_\gamma$ underlying the veering parent $\V_\gamma$ appears in the snappy census as \texttt{o9\_40888}, while the manifold $M_\gamma^\ast$ obtained by drilling along the geodesic $\gamma^\ast$ appears in the SnapPy census as \texttt{L14n33639}.  Their volumes are approximately 8.11 and 10.96, respectively. In particular, they are certainly not isometric and hence they are not homeomorphic. Further details about the veering parent $\V_\gamma$ can be found in Appendix~\ref{appendix:geodesic}.

	\begin{figure}[h]
		\includegraphics[scale=1.5]{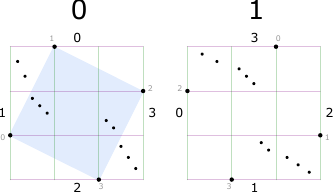}
		\caption{Drilled tetrahedron rectangles for a veering triangulation $\V$ of the figure-eight knot complement and a flow cycle $\zeta$ of $\V$ of length 5 that represents a closed orbit that is not ambiently isotopic to the geodesic in its homotopy class. The shaded square corresponds to the once-punctured torus fiber.}
		\label{fig:fig8_geodesic}
	\end{figure}
	
	The shaded square $S$ contained in the tetrahedron rectangle $R$ on the left of Figure~\ref{fig:fig8_geodesic} can be identified with a once-puntured torus fiber of $M$. Recall that the positions of the points of $R \cap \liftGamma$ shown in Figure~\ref{fig:fig8_geodesic} are determined correctly only up to their horizontal and vertical orders relative to one another and to the cusps of $R$. This means that a point that appears to lie in $S \cap \liftGamma$ may in fact lie outside $S$.
	However, the north-west red edge subrectangle  of~$R$ is $\pi_1(M)$-equivalent to the south-east edge subrectangle of~$R$, and the south-west blue edge subrectangle of $R$ is $\pi_1(M)$-equivalent to the north-east blue edge subrectangle of~$R$. Therefore moving a point of $R \cap \liftGamma$ out of $S$ across the north-west boundary of $S$ necessarily moves some other point of $R \cap \liftGamma$ into $S$ across the south-east boundary of $S$, and similarly for moving points across other sides of $S$. Thus we can conclude that $\liftGamma$ intersects $S$ in exactly five points.
	It follows that $\gamma$ is the suspension of a period-5 point of the monodromy $\varphi$ of the fibration of $M$. Since  $\varphi$ is known to be induced by the linear map 
	\[A = \begin{bmatrix}
	2 & 1\\ 1 & 1
	\end{bmatrix},\]
	it is even possible to determine which period-5 orbit of $\varphi$ suspends to $\gamma$, up to the ambiguity arising from self orbit equivalences of $\Psi$. We leave this computation to the interested reader.

\end{example}

Layne Hall brought to our attention that Sally Kuhlmann had also found period-$5$ orbits of the suspension flow on the figure-eight knot complement that are not ambiently isotopic to the geodesics in their respective free homotopy classes, using a completely different technique \cite[Appendix~D]{Kuhlmann}. An advantage of our implementation of the drilling algorithm is that it allows us to carry out this comparison for longer periodic orbits and for other pseudo-Anosov flows, including those that are not suspension flows.

	\subsection{A conjecture of Ghys}\label{subsec:ghys}
	In this section we define the \emph{veering ancestry graph} and discuss its relationship with a conjecture of Ghys concerning Anosov flows with orientable invariant foliations.
	Recall that if $\Psi$ is a pseudo-Anosov flow $\Psi$ on a closed 3-manifold $N$ and $\Gamma\subset \Per(\Psi)$  then $\Psi \bez \Gamma$ denotes the restriction of $\Psi$ to $N \bez \Gamma$. 

	\begin{conjecture}[Ghys] \label{conj:Ghys}
		If $\Psi, \Psi'$ are transitive Anosov flows with orientable invariant foliations then there are finite collections $\Gamma \subset \Per(\Psi)$, $\Gamma' \subset \Per(\Psi')$ such that the restricted flows $\Psi \bez \Gamma$ and $\Psi' \bez \Gamma'$ are orbit equivalent.
	\end{conjecture}
	
	If some finite drillings of flows are orbit equivalent we also say that the flows are \emph{almost orbit equivalent}. The above conjecture of Ghys is currently open in its full generality. However, we already know that it holds for  suspensions of toral automorphisms.
	
	\begin{theorem}\cite[Theorem A]{dehornoy-almost-equiv} \label{thm:dehornoy}
		Any two Anosov suspensions with orientable invariant foliations are almost orbit equivalent. \qed
	\end{theorem}
	
	Our goal is to restate Conjecture \ref{conj:Ghys} in the language of veering triangulations. 
	
	\begin{definition}
		The \emph{veering ancestry graph} is a directed graph $\A$ defined as follows.
		\begin{itemize}
			\item The vertices of $\A$ correspond to finite veering triangulations.
			\item There is an edge of $\A$ from a veering triangulation $\V_1$ to a veering triangulation~$\V_2$ if and only if~$\V_2$ is obtained from $\V_1$ via a drilling along a single flow cycle. 
		\end{itemize}
	\end{definition}

\begin{remark}\label{remark:parent:terminology}
	With the above convention on the orientation of the edges of $\A$, a veering parent of $\V$ is in fact a descendant of $\V$ in the graph-theoretic sense. Our use of the term \emph{parent} clearly does not refer to this orientation. Instead, it follows  terminology commonly used in the study of  Dehn surgery, where a $3$-manifold obtained by drilling another $3$-manifold along a simple closed curve is called a \emph{surgery parent} of that manifold; see for instance \cite{Adams-Dehn-filling}.
\end{remark}
	
	The graph $\A$ is clearly infinite, and has infinite out-degree at every vertex. However, some vertices have no incoming edges; for instance, this holds for all vertices corresponding to veering triangulations with a single ideal vertex.
	
	The graph $\A$ is also disconnected. For instance, veering triangulations encoding flows with transversely orientable invariant foliations (called \emph{edge-orientable} veering triangulations; see \cite[Section 3.4]{parlak-alex}) cannot be drilled to veering triangulations that encode flows with invariant foliations that are not tranversely orientable, and vice versa. 
	
	Furthermore, even two edge-orientable veering triangulations need not be connected by a path in $\A$. This can be seen by analyzing the boundary data induced by a veering structure on a triangulation. Specifically, every veering triangulation determines a decomposition of each boundary component of the underlying manifold into an even number of annuli, called \emph{ladders}.
	We refer the reader to \cite[Section 2]{explicit_angle} for a precise definition. The curves that bound the ladders are called \emph{ladderpole curves}. If the Dehn filling slope along which we fill a boundary component $B$ of a veering triangulation $\V$ intersects the collection of ladderpole curves of $B$ $k$ times then the core curve of the filling solid torus becomes a \mbox{$\frac{k}{2}$-prong} singularity of the pseudo-Anosov flow constructed from~$\V$ and $s$ via Theorem~\ref{thm:AT}. Thus if $\V$ encodes an Anosov flow, each of its boundary components has either two or four ladders. Furthermore, two ladders can occur only if the invariant foliations of the flow are not transversely orientable. Thus a veering triangulation encodes an Anosov flow with orientable invariant foliations if and only if it is edge-orientable and has four ladders in each boundary component.
	
	Moreover, our drilling procedure introduces only boundary components with two or four ladders, depending on whether the orbit that we drill is twisted or not. It follows that if a veering triangulation has a boundary component with $2k$-ladders for some $k\geq 3$ then it cannot be connected  in the veering ancestry graph to a veering triangulation that does not admit a boundary component with $2k$-ladders. Consequently, $\A$ must have infinitely many components. 

	In Theorem \ref{thm:connected:cpts} we prove that the connected components of $\A$ bijectively correspond to almost orbit equivalence classes of transitive pseudo-Anosov flows.
	First, we introduce necessary notation and terminology. Given vertices $v,w$ of a directed graph $G$, we write $v \rightsquigarrow w$ if there exists a directed path from $v$ to $w$. A vertex $v$ of $G$ is \emph{confluent} if for all pairs $v_1, v_2$ such that $v \rightsquigarrow v_1$ and $v \rightsquigarrow v_2$ there exists a vertex $w$ of $G$ such that $v_1 \rightsquigarrow w$ and $v_2 \rightsquigarrow w$. We say that $G$ is \emph{confluent} if every vertex of $G$ is confluent. 
	
	\begin{lemma}\label{lem:confluence}
		$\A$ is confluent.
	\end{lemma}
	\begin{proof}
		Suppose that $\V_1, \V_2$ are two drillings of a veering triangulation $\V$. By Theorem~\ref{thm:AT}, there is a transitive pseudo-Anosov flow $\Psi$ and a perfect set $\Lambda$ for $\Psi$ such that $\V = \V(\Psi, \Lambda)$. Furthermore, for $i=1,2$ there are $\Gamma_i\subset \Per(\Psi)$ such that $\V_i = \V(\Psi, \Lambda \cup \Gamma_i)$. The union $\Lambda \cup \Gamma_1 \cup \Gamma_2$ is a perfect set for $\Psi$, and thus there is a veering triangulation $\mathcal{W} = \V(\Psi, \Lambda \cup \Gamma_1 \cup \Gamma_2)$. Since $\mathcal{W}$ is a drilling of both $\V_1$ and $\V_2$, we get directed paths $\V_1 \rightsquigarrow \mathcal{W}$ and $\V_2 \rightsquigarrow \mathcal{W}$. 
	\end{proof}
	Let $\beta(\V)$ be the number of ideal vertices of a veering triangulation $\V$, and let $\tau(\V)$ be the number of tetrahedra of $\V$. 
	Assigning to each vertex $\V$ of $\A$ a pair $(\beta(\V),\tau(\V))\in\mathbb{N}^2$ endows $\A$ with a bi-grading; see Figure \ref{fig:drilling:graph} that illustrates a subgraph of~$\A$ arranged in a way that reflects this bi-grading. 
	
		\begin{figure}[h]
		\includegraphics[scale=0.34]{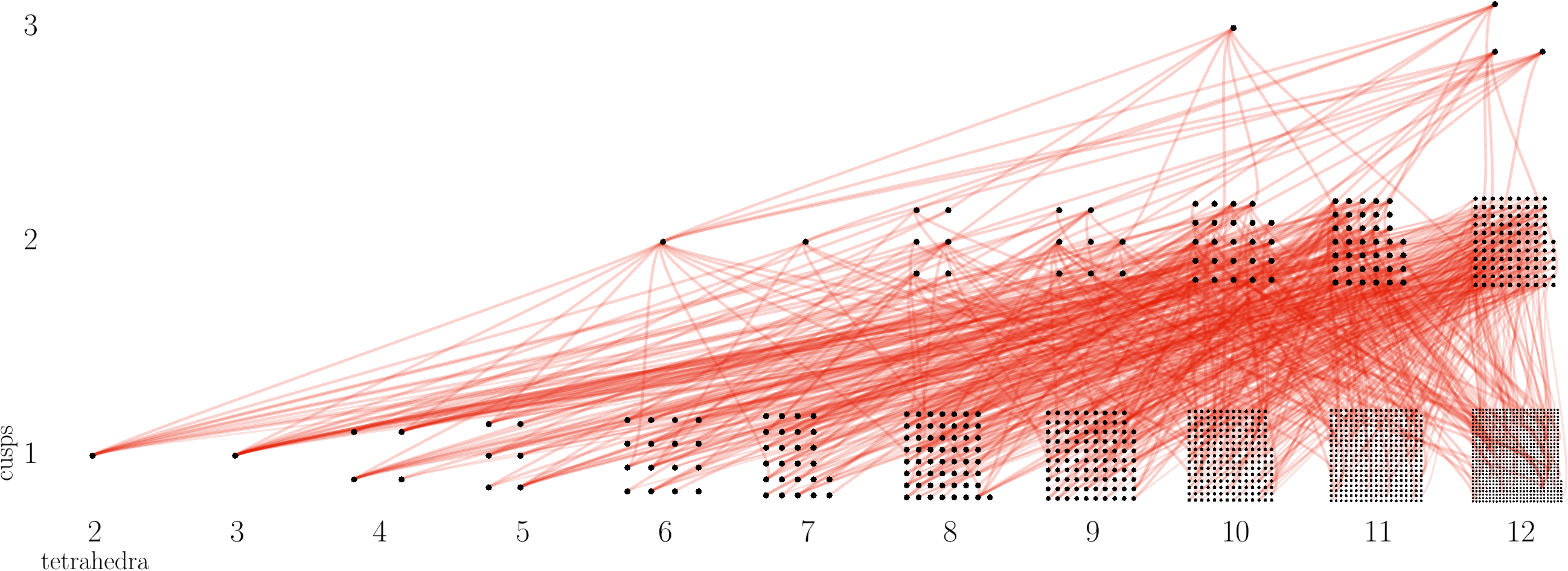}
		\caption{A subgraph of the veering ancestry graph. Each vertex represents an edge-orientable veering triangulation with at most 12 tetrahedra and exactly four ladders in each boundary component. Each edge represents a drilling along a flow cycle of length at most 7 connecting two such triangulations. Vertices are arranged according to the number of tetrahedra (horizontal axis) and the number of ideal vertices (vertical axis) of the corresponding triangulations. Since edges are oriented from triangulations with fewer cusps to those with more cusps, they all point upwards. We therefore omit the orientations from the figure.}
		\label{fig:drilling:graph}
	\end{figure}
	
	Observe that if there is an edge in $\A$ from $\V$ to $\V'$, then $\beta(\V')=\beta(\V)+1$. In other words, the edge orientation records the direction in which $\beta(\cdot)$ increases. In contrast, there is no such constraint for $\tau$;  depending on the edge, $\tau(\V')$ may be larger than, smaller than, or equal to $\tau(\V)$.
	
	Let $\delta = (\V_1, \V_2, \ldots, \V_n)$ be a (not necessarily directed) path in $\A$. Then for every integer $1 < i \leq n$, either there is an edge of $\A$ from $\V_{i-1}$ to $\V_i$ or an edge from $\V_i$ to $\V_{i-1}$. For $1<i<n$ we say that $\V_i$ is an \emph{internal local minimum} for $\delta$ if $\beta(\V_{i+1}) = \beta(\V_{i-1}) = \beta(\V_i) + 1$. That is, both $\V_{i-1}$ and $\V_{i+1}$ are single drillings of $\V_i$. Each internal local minimum $\V_i$ of $\delta$ determines two maximal directed subpaths of $\alpha$: one beginning with $(\V_i,\V_{i-1},\ldots)$ and the other beginning with $(\V_i,\V_{i+1},\ldots)$. We call the terminal endpoints of these subpaths the \emph{local peaks of $\delta$ around} $\V_i$. 
	
	\begin{theorem}\label{thm:connected:cpts}
		Let $\Psi, \Psi'$ be transitive pseudo-Anosov flows, and let $\Lambda, \Lambda'$ be arbitrary perfect sets for $\Psi, \Psi'$, respectively. Set $\V = \V(\Psi, \Lambda)$ and $\V' = \V(\Psi', \Lambda')$. Then $\Psi, \Psi'$ are almost orbit equivalent if and only if $\V, \V'$ are in the same connected component of $\A$. 
	\end{theorem}
	\begin{proof}
		Let $N, N'$ be the closed three-manifolds that support $\Psi, \Psi'$,  respectively. Suppose that $\Psi, \Psi'$ are almost orbit equivalent. Then there are $\Gamma \subset \Per(\Psi)$, $\Gamma' \subset \Per(\Psi')$ for which there exists an orbit equivalence $h: (N\bez \Gamma) \rightarrow (N' \bez \Gamma')$. Set
		\begin{gather*}
		\Delta =\Lambda \cup  \Gamma\cup h^{-1}(\Lambda' \bez \Gamma')\\
		\Delta' = \Lambda'  \cup \Gamma' \cup h(\Lambda \bez \Gamma).
		\end{gather*}
		Then $\Delta$, $\Delta'$ are perfect sets for $\Psi, \Psi'$, respectively. Set $\W = \V(\Psi, \Delta)$ and $\W' = \V(\Psi', \Delta')$. Since $\W$ is a veering ancestor of $\V$, and $\W'$ is a veering ancestor of $\V'$, there are directed paths in $\A$ from $\V$ to $\W$, and from $\V'$ to $\W'$. 
		
		Observe that $h$ restricts to an orbit equivalence between $\Psi \bez \Delta$ and $\Psi' \bez \Delta'$. Therefore it induces a $\pi_1(N - \Delta)$-equivariant homeomorphism between the loom space associated to $(\Psi, \Delta)$ and the loom space associated to $(\Psi', \Delta')$. It follows that $\W = \W'$; see Construction \ref{const:finite:veering}. Consequently, we get a path from $\V$ to $\V'$ that passes through $\W$. 
		
		Conversely, suppose that $\V, \V'$ are connected by a path $\delta$ in $\A$. By Lemma \ref{lem:confluence}, for each internal local minimum $\mathcal{W}$ of $\delta$, the local peaks of $\delta$ around $\W$ have a common ancestor. This means that we can replace $\delta$ by a path $\delta'$ that also connects $\V$ to $\V'$ but has strictly less internal local minima. After repeating this procedure finitely many times, we eventually get a path between $\V$ and $\V'$ that does not have any internal local minima. Such a path splits into a directed subpath $\V \rightsquigarrow\W$ and the reverse of the directed subpath $\V' \rightsquigarrow\W$, for some $\W \in \A$.  This means that $\W$ is a common veering ancestor of $\V$ and~$\V'$. Hence we can express $\W$ as both $\V(\Psi, \Lambda \cup \Gamma)$ and $\V(\Psi', \Lambda' \cup \Gamma')$ for some $\Gamma \subset \Per(\Psi)$, $\Gamma' \subset \Per(\Psi')$.
		
		Let $\Phi$ and $\Phi'$ be the transitive Anosov flows obtained from $\W$ via Theorem~\ref{thm:AT}, using the slopes that recover $N$ and $N'$, respectively. Then $\Phi$ and $\Phi'$ agree on the complement of orbits corresponding to the cores of the Dehn filling solid tori. 
		Furthermore, by Theorem~\ref{thm:tri:flow:id}, $\Phi$ is orbit equivalent to $\Psi$ and $\Phi'$ is orbit equivalent to $\Psi'$.  Consequently, $\Psi \bez (\Lambda \cup \Gamma)$ and $\Psi' \bez (\Lambda' \cup \Gamma')$ are orbit equivalent. 
	\end{proof}
	
	By Theorem \ref{thm:connected:cpts}, Conjecture \ref{conj:Ghys} can be restated as follows.
	
	\begin{conjecture}[Ghys, restated] \label{conj:veering:graph:component}
		Any two edge-orientable veering triangulations with four ladders in each boundary component are in the same connected component of $\A$. 
	\end{conjecture}
	
	We used the drilling algorithm described in Section \ref{sec:algorithm} to check if experimental data supports the above conjecture or indicates that it may be false. Specifically, we constructed a subgraph  $\A^{+}_{16,4,7}$  of $\A$ whose vertices correspond to the following veering triangulations:
	\begin{enumerate}
		\item all edge-orientable veering triangulations with at most sixteen tetrahedra and four ladders in each boundary component;
		\item for each veering triangulation $\V$ from (1) and each non-peripheral flow cycle $\zeta$ of $\V$ of length at most seven, the veering parent $\V_\zeta$.
	\end{enumerate}
	Furthermore, $\A^{+}_{16,4,7}$ contains every edge of $\A$ that corresponds to a drilling $\V \to \V_\zeta$ for $\V$ as in (1) and a flow cycle $\zeta$ of $\V$ of length at most seven. Note that  24,343 vertices of $\A^+_{16,4,7}$ correspond to veering triangulations of type (1).
	
	\begin{experimental}\label{exp:ancestry:experiments}
		\hfill
		\begin{enumerate}[label=(\Roman*)] 
			\item $\A^{+}_{16,4,7}$  has 8,767,114 vertices and 2247 connected components.
			\item The largest connected component of $\A^{+}_{16,4,7}$  has 8,212,337 vertices. 
			\item Every connected component of $\A^{+}_{16,4,7}$ has a vertex corresponding to a veering triangulation encoding an Anosov suspension flow.\qed
		\end{enumerate}
	\end{experimental}
	
	The fact that the vast majority (93.67 \%) of vertices of $\A^{+}_{16,4,7}$  lie in the same connected component of $\A$ is a strong indication that all of them do. That is indeed the case.
	
	\begin{corollary}
		All vertices of $\A^{+}_{16,4,7}$  lie in the same connected component of $\A$.
	\end{corollary}
	\begin{proof}
		Let $C_1, C_1$ be any connected components of $\A^{+}_{16,4,7}$. 
		By Part (III) of Experimental Result~\ref{exp:ancestry:experiments},  for $i\in \lbrace 1,2 \rbrace$, $C_i$ has a vertex $\V_i$ corresponding to a veering triangulation encoding some Anosov suspension flow $\Psi_i$. Since all veering triangulations in $\A^{+}_{16,4,7}$  are edge-orientable,~$\Psi_i$ must have transversely orientable invariant foliations. Therefore Theorem~\ref{thm:dehornoy}  implies that $\Psi_1$ and $\Psi_2$ are almost orbit equivalent. By Theorem \ref{thm:connected:cpts}, $\V_1, \V_2$ are connected by a path in $\A$. Consequently, $C_1$ and $C_2$ lie in the same connected component of $\A$. 
	\end{proof}

	Finally, we introduce new terminology that will be used to generalize the combinatorial version of Ghys' conjecture to all edge-orientable veering triangulations.

	\begin{definition}
		Let $\V$ be a veering triangulation. The \emph{ladder type} of $\V$ is a sequence $(l_k)_{k \in \mathbb{N}}$, such that for every $k \in \mathbb{N}$ the term $l_k$ is equal to the number of ideal vertices of~$\V$ admitting exactly $2k$ ladders.
	\end{definition}
	
		We say that a ladder type $(l_k)_{k \in \mathbb{N}}$ is \emph{even} if $l_{2i-1} = 0$ for every $i \in \mathbb{N}$. Furthermore, two ladder types $(l_k)_{k \in \mathbb{N}}$ and $(l'_k)_{k \in \mathbb{N}}$ are said to be \emph{tail equivalent} if $l_k = l'_k$ for all $k \ge 3$.

	\begin{question}[Conjecture of Ghys generalized]\label{question:ghys:generalized}
		Let $\A^+$ be the subgraph of $\A$ with vertices corresponding to all finite edge-orientable veering triangulations and edges corresponding to all possible single drillings between them.  Are connected components of $\A^+$ in a bijective correspondence with the tail equivalence classes of even ladder types?
	\end{question}
	\section{Further questions} \label{sec:questions}
	
	In this section, we discuss several problems that are closely related to the drilling algorithm but lie outside the scope of this paper.
	
	\subsection{Running time}\label{subsec:q:time}
	
	Although Section~\ref{sec:algorithm} gives a complete description of the drilling algorithm, we do not analyze its computational complexity. Determining it would require estimates for the number of tetrahedra in the adapted continent $\C_1$ constructed in Step~3, and for both the number and the size of the continent expansions performed during \mbox{Steps~4 -- 9} of the algorithm (step numbers as in Figure~\ref{fig:flowchart}). These in turn rely on a detailed understanding of the process of channelization introduced by Schleimer and the second author in \cite[Section~5.19]{SchleimSegLinks}.
	
	\begin{question}
		Given a veering triangulation $\V$ with $n$ tetrahedra and a flow cycle $\zeta$ of length $k$, how does the running time of the algorithm in Figure~\ref{fig:flowchart} depend on $n$ and $k$?
	\end{question}

Experimentally, the algorithm is feasible for small veering triangulations and relatively short flow cycles. In particular, we used it to construct the subgraph $\A^+_{16,4,7}$ of the veering ancestry graph, mentioned in Experimental Result~\ref{exp:ancestry:experiments}, which contains almost nine million triangulations. However, this computation only required drilling flow cycles of length at most~7 from veering triangulations with at most~16 tetrahedra. Understanding the computational complexity of the algorithm and identifying potential improvements would be particularly important if there were interesting applications that require drilling long flow cycles.
	\subsection{Combinatorics of veering parents}\label{subsec:q:combinatorics}
		We distinguish two types of veering tetrahedra: \emph{toggles} (also called \emph{hinges} in the literature; see, for example, \cite[Section 2.1]{LMT_flow} or \cite[Section 2]{explicit_angle}), whose diagonals have different colors, and \emph{fans}, whose diagonals have the same color. Taking the colors of the diagonals into account gives four types of veering tetrahedra: toggles with red bottom diagonal, toggles with blue bottom diagonal, and analogously for fans; see Figure \ref{fig:tet:rects}.
		
	Recall from Definition \ref{defn:rectangle} that a flow cycle~$\zeta$ of a veering triangulation $\V$ can pass through a tetrahedron $t$ of $\V$ in three different ways: from the bottom diagonal to the top diagonal, or from the bottom diagonal to either of the two equatorial edges whose color is opposite to that of the top diagonal. We say that $\zeta$ goes \emph{straight up} through $t$ in the first case, and \emph{sideways} through $t$ in the latter two cases. 

	Let $\V_\zeta$ be obtained from $\V$ by drilling $\zeta$. Recall from Proposition \ref{prop:finite:veering:parent} that the tetrahedra of $\V_\zeta$ correspond to the $\pi_1(M)$-equivalence classes of $\zeta$-tetrahedron rectangles contained in the fundamental tetrahedron rectangles of~$\V$. Furthermore, their face identifications are induced by those of the tetrahedra of $\V$. Thus the combinatorics of $\V_\zeta$ is determined by the positions of the points of $\liftZeta$ within the fundamental tetrahedron rectangles relative to their cusps. These positions, in turn, depend on various combinatorial features of $\zeta$ and~$\V$, such as:
	\begin{itemize}
		\item the degrees of the edges of $\V$,
	\item how often $\zeta$ goes straight up through tetrahedra of each type,
	\item how often $\zeta$ goes sideways through tetrahedra of each type, 
	\item how often $\zeta$ switches between going straight up and going sideways.
	\end{itemize}
	
Beyond knowing that these features affect the combinatorics of $\V_\zeta$, we currently have little quantitative understanding of this dependence. This leads to the following question.

\begin{question}\label{question:combinatorics}
Let $\zeta$ be a flow cycle of a finite veering triangulation $\V$. To what extent do combinatorial features of $\V$ and $\zeta$, such as those listed above, determine combinatorial features of $\V_\zeta$?
\end{question}

Of particular interest would be determining, or at least bounding, the numbers of tetrahedra of each type in~$\V_\zeta$ without having to construct $\V_\zeta$ explicitly. Such bounds could help determine whether a given veering triangulation can arise as a parent of another. They could also be useful for studying the existence of veering parents with a prescribed toggle to fan ratio.
For instance, one could ask whether every veering triangulation admits a parent consisting only of toggle tetrahedra. Our experiments do not suggest that this is always the case. If it was, Conjecture \ref{conj:geometric} would follow, since, by \cite[Lemma 4.1]{explicit_angle}, veering triangulations without any fans are geometric.

\subsection{Decreasing the number of tetrahedra}\label{further:descreasing}

In Appendix \ref{app:ex:decreasing}, we present an example of a veering parent with fewer tetrahedra than the original veering triangulation. Drillings that decrease the number of tetrahedra are also visible in Figure \ref{fig:drilling:graph}, where they correspond to edges with negative slope. We suspect that, at least in some cases, this phenomenon is related to \emph{vertical surgery} introduced by Tsang in \cite[Section~4.2]{Tsang-branched-surfaces}.

A vertical surgery can be performed along any \emph{vertical surgery curve} of a veering triangulation. Such curves are closely related to flow cycles; compare the local characterization in \cite[Proposition 2.11]{Parlak-vanishing} with the flow graph of a veering triangulation shown in \cite[Figure 4]{LMT_flow}. Performing vertical surgery on a veering triangulation $\V$ produces another veering triangulation $\V'$ that has more tetrahedra than $\V$  \cite[Proposition 4.8]{Tsang-branched-surfaces}. Moreover, by choosing sufficiently large surgery coefficient, the number of tetrahedra of~$\V'$ can be made arbitrarily large.

Conjecturally, vertical surgery combinatorializes certain cases of \emph{Goodman--Fried surgery} along periodic orbits of pseudo-Anosov flows; see  \cite{Fried_drill, Goodman} for different descriptions of Goodman-Fried surgery. This would suggest that there exist flow cycles $\zeta'$ of $\V'$ and $\zeta$ of $\V$ such that drilling them yields the same veering parent. 
In particular, if the surgery coefficient yielding $\V'$ from $\V$ was large enough, drilling~$\zeta'$ would have to result in a veering triangulation with fewer tetrahedra than $\V'$. This picture is only conjectural and  may not fully explain the phenomenon of drilling decreasing the number of tetrahedra.

\begin{question}
	Suppose that $\V'$ admits a veering parent with fewer tetrahedra than $\V'$. Does it follow that $\V'$ can be obtained by vertical surgery from a veering triangulation $\V$ with fewer tetrahedra than $\V'$?
\end{question}

By Proposition \ref{prop:tsang:perfect:set}, if $\Psi$ is a transitive pseudo-Anosov flow with $m$ singular orbits then any minimal perfect set $\Lambda$ for $\Psi$ has either $m$ (when $m \geq 1$ and $\Psi$ does not have perfect fits) or $m+1$ elements (when $m=0$ or  $\Psi$ has perfect fits). While choosing a minimal perfect set minimizes complexity at the level of the flow, the example discussed in Appendix \ref{app:ex:decreasing} shows that it does not necessarily minimize the complexity of the obtained triangulation. This gives another natural combinatorial invariant of a transitive pseudo-Anosov flow: the minimum number of tetrahedra in a veering triangulation encoding $\Psi$ via Construction \ref{const:finite:veering}. We call that number the \emph{veering complexity} of $\Psi$.

	\begin{question}
		Can the veering complexity of a transitive pseudo-Anosov flow be bounded in terms of other topological or dynamical invariants of the flow?
	\end{question}

\subsection{Surgery and drilling in the hyperbolic and veering settings}\label{subsec:q:surgery}
There is an interesting contrast between the geometric and veering behavior under surgery and drilling. Specifically, let $\zeta$ be a flow cycle of a veering triangulation $\V$ that is ambiently isotopic to a vertical surgery curve $c$ of $\V$. Denote the manifold underlying the veering parent $\V_\zeta$ of $\V$ by $M_\zeta$, the veering triangulation obtained from $\V$ by $\frac{1}{k}$-vertical surgery along~$c$ by~$\V_k$, and the manifold underlying $\V_k$ by $M_k$.

By \cite[Theorem~6.29]{Purcell-knots}, the manifolds $M_k$ converge geometrically to $M_\zeta$ as $k\to\infty$. In contrast, the combinatorial complexity of the triangulations $\V_k$ grows without bound, instead of stabilizing to that of $\V_\zeta$. However, all the `additional' tetrahedra introduced by vertical surgery are fans of the same color; see \cite[Section~4.2]{Tsang-branched-surfaces}.

\begin{question}
Does the combinatorics of $\V_k$, for sufficiently large $k$, inform about that of~$\V_\zeta$?
\end{question}

Since the change from $\V$ to $\V_k$ is local and well-understood, understanding the relationship between the combinatorics of $\V_k$ and $\V_\zeta$ could provide a way to approach Question \ref{question:combinatorics}.

	\subsection{Recognizing veering ancestors and connectedness in the veering ancestry graph}\label{subsec:q:veering:graph}
	
	The veering census \cite{VeeringCensus} collects information about all veering triangulations with up to 16 tetrahedra. In view of Theorem \ref{thm:AT}, one might hope to interpret it as a census of pseudo-Anosov flows. Unfortunately, many non-isomorphic veering triangulations encode the same flow. This is, for instance, the case for pairs of veering triagulations that can be connected by a directed path in the veering ancestry graph.
	At the moment, we have no way of recognizing such pairs, other than searching for a directed path between them by brute force. Ruling out the existence of such a path is also difficult, except in the obvious cases when the triangulations have the same number of cusps, are not tail equivalent (as defined in Section \ref{subsec:ghys}), or when only one of the triangulations is edge-orientable. 
	
 The goal of trying to determine how many orbit inequivalent pseudo-Anosov flows are represented in the veering census motivates the following question.
	
\begin{question}\label{question:recognizing:directed:path}
Can it be combinatorially recognized if two tail equivalent veering triangulations with different numbers of cusps and the same edge-orientability type are connected by a directed path in the veering ancestry graph?
\end{question}

By Theorem \ref{thm:connected:cpts}, veering triangulations lying in the same connected component of the veering ancestry graph generally encode flows that are only almost orbit equivalent. They may, however, also encode orbit equivalent flows. This happens when the Dehn filling data imposed by the two triangulations on their common veering ancestor are compatible. By this we mean that, up to orbit equivalence of the drilled flow underlying the common ancestor, whenever both triangulations require a Dehn filling of the same boundary component, the corresponding filling slopes  agree. There is no restriction on a boundary component that is filled in only one of the triangulations.
	
Properties preserved under both drilling and filling can provide obstructions to connectedness in the veering ancestry graph. So far, the only obstructions we know come from edge-orientability and tail equivalence class. Question \ref{question:ghys:generalized}, which generalizes Ghys' conjecture, asks if different tail equivalence classes are the only obstruction for two edge-orientable veering triangulations to lie in the same connected component of the veering ancestry graph. More generally, we ask:

\begin{question}
	Is there an invariant of veering triangulations that characterizes the connected components of the veering ancestry graph?
\end{question}
	
	\subsection{Filling veering triangulations}\label{subsec:q:filling}

	Let $M$ be a three-manifold with a veering triangulation $\V$, let $b$ be a boundary component of $M$, and let $s$ be a slope on $b$. Denote by $M(s)$ the manifold obtained from $M$ by Dehn filling $b$ along $s$. A \emph{veering filling of $\V$ along $s$} is a veering triangulation $\V(s)$ of $M(s)$ such that drilling $\V(s)$ along the core of the filling solid torus (in the sense of Section \ref{sec:parents}) recovers $\V$.
	
	While topological Dehn filling is always possible, veering filling is much more restrictive. The obstructions come both from the conditions that appear in the definition of a perfect set for a pseudo-Anosov flow (Definition \ref{def:perfect:set}) and from the requirement that the geometric intersection number of the filling slope with the ladderpole curves induced by $\V$ on $b$ must be at least~4 \cite[Theorem~5.1]{Tsang-Agol}. 
	
	Suppose that $\V$ has at least two boundary components,  induces at most four ladders on $b$, and~that $s$ intersects the ladderpole curves of $b$ at least four times. Then, by Theorem~\ref{thm:AT}, there exists a transitive pseudo-Anosov flow $\Psi$ and a perfect set $\Lambda$ for $\Psi$ such that
	\begin{itemize}
	\item $\V = \V(\Psi,\Lambda)$;
	\item the manifold underlying $\Psi$ is a closed Dehn filling of $M$ with filling slope $s$ on $b$;
	\item $|\Lambda| \geq 2$ and there is a non-singular orbit $\gamma_b \in \Lambda$ corresponding to the boundary component $b$ of $M$.
	\end{itemize}
	
	The difficulty in deciding whether a veering filling $\V(s)$ of $\V$ exists lies in determining whether $\Psi$ has no perfect fits relative to $\Lambda \bez {\gamma_b}$. This is generally a nontrivial task as it is not completely understood how Dehn filling of a given slope affects the orbit space of a drilled flow; see \cite{Bonatti-surgeries} for a discussion of how Goodman--Fried surgery affects the orbit spaces of certain Anosov flows and \cite{Hall-perfect-fits} for an algorithm for deciding whether a pseudo-Anosov flow has perfect fits. 
	
\begin{question}
	Is there a combinatorial criterion for when a veering filling $\V(s)$ exists, given a boundary component $b$ of the 3-manifold underlying $\V$ and a slope $s$ on $b$?
\end{question}
	
	Beyond deciding whether $\V(s)$ exists, there is a separate problem of constructing $\V(s)$ explicitly. This is especially interesting from the point of view of exploring the veering ancestry graph, as an explicit filling construction would allow us to move both forwards and backwards along its directed edges.
	
	\begin{question}
		Is there a direct combinatorial procedure for constructing $\V(s)$, whenever it exists?
	\end{question}
	
	\bibliographystyle{abbrv}
	\bibliography{mybib}
	
	\newpage

	\appendix

	\section{Examples}\label{appendix:examples}
	
	In this appendix we discuss in detail three examples of veering parents, with the information for each example organized in a table. We begin by explaining how to read these tables.

	\subsection{Explanation of table entries}\label{subsec:tables}
	
	\subsubsection*{Convention} 
	Unless otherwise stated, indices arising from the computational description of a triangulation start from 0.
	\subsubsection{Taut signatures}
	Every taut triangulation can be described by a string of the following form 
	\begin{equation*}\label{string}
	\texttt{[isoSig]\underline{ }[string of 0's, 1's, and 2's]}.
	\end{equation*}
	The \texttt{isoSig} part refers to the \emph{isomorphism signature} of the triangulation, introduced by Burton in \cite[Section 3]{Burton_isoSig}.  They proposed a systematic way to order all tetrahedra in the triangulation and their vertices, starting from a choice of an initial tetrahedron and one of the $4!$ possible numberings of its vertices by the elements of $\lbrace0,1,2,3\rbrace$. If the triangulation has $n$ tetrahedra, this gives $24n$ different \emph{canonical labellings}, each of which can be encoded by a string of letters and numbers. The isomorphism signature is the lexicographically smallest of those strings. Thus two triangulations have the same isomorphism signature if and only if they are combinatorially isomorphic --- that is, differ only by reordering their tetrahedra and renumbering their vertices. 
	
	If the triangulation has $n$ tetrahedra, the numerical string has length $n$. It specifies which edges of each tetrahedron (with vertices enumerated by the lexicographically smallest canonical labelling) are its diagonals. 
	Specifically, the string has 0, 1, or 2 at position $0\leq i \leq n-1$ if the vertices of the diagonal edges of the $i$-th tetrahedron are respectively $\lbrace 0,1 \rbrace$ and $\lbrace 2,3 \rbrace$, $\lbrace 0,2 \rbrace$ and $\lbrace 1,3 \rbrace$, or $\lbrace 0,3 \rbrace$ and  $\lbrace 1,2 \rbrace$. Because the labels are propagated throughout the triangulation according to a fixed rule, this information suffices to encode the taut structure up to reversing the coorientation of all faces.
	
	The canonical labelling associated to the isomorphism signature gives us a fixed order on the tetrahedra of the triangulation and on their vertices. With this labelling, some tetrahedra may be negatively oriented. For each such tetrahedron, we swap the labels of the two vertices of its lexicographically smaller diagonal. After this change, all tetrahedra are positively oriented. We refer to the resulting labelling as the \emph{oriented isoSig-labelling} of the triangulation. This labelling is used in all pictures available in the veering census~\cite{VeeringCensus}.
	
	As mentioned in Section \ref{subsec:veer:tri}, every taut triangulation admits at most one veering structure. Thus, if a taut triangulation is veering, its taut signature uniquely encodes not only the taut structure (up to reversing the coorientations of all faces), but also the associated veering structure.

	\subsubsection{Encoding flow cycles} \label{subsec:labelling}
	Let $\V$ be a veering triangulation with $n$ tetrahedra.
	Let~$\zeta$ be a flow cycle of $\V$ of length $k \in \mathbb{N}$. According to  Definition \ref{defn:flow:cycle}, $\zeta$ is a sequence $(t_1, t_2, \ldots, t_k)$ of tetrahedra of $\V$ such that for each $i \in \{1,2, \ldots, k\}$, the tetrahedron $t_{i+1}$  is either immediately above the top diagonal of $t_i$ or immediately above one of the equatorial edges of $t_i$ that has a different color than the top diagonal of $t_i$, where the indices are understood modulo~$k$. 
	
	We encode a flow cycle by a sequence of pairs of integers as follows. We enumerate the tetrahedra of~$\V$ and their vertices  according to the oriented isoSig-labelling.
	Any edge~$e$ of a tetrahedron~$t$ connects two vertices $i,j$ of $t$ for some $0 \leq i < j \leq 3$. We label~$e$ with the lexicographic index of $\lbrace i,j \rbrace$ among all two-element subsets of $\lbrace 0,1,2,3 \rbrace$. For example, edge 0 of $t$ is the edge of $t$ that connects vertices 0 and 1 of $t$, and edge 4 of $t$ is the edge of $t$ that connects vertices 1 and 3 of $t$.

	With these numbering conventions, a flow cycle $\zeta = (t_1, t_2, \ldots, t_k)$ can be uniquely represented by a sequence of pairs $((a_i, b_i))_{i=1}^k$, where $0 \leq a_i \leq n-1$ is the index of the tetrahedron $t_i$ in the triangulation $\V$, and $0 \leq b_i \leq 5$ is the lexicographic index of the edge of $t_i$ through which $\zeta$ passes from $t_i$ to $t_{i+1}$. Conversely, to verify if a sequence $((a_i, b_i))_{i=1}^k$ of pairs of integers satisfying $0 \leq a_i \leq n-1$ and $0 \leq b_i \leq 5$ encodes a flow cycle it suffices to check, for each $i=1,2,\ldots, k$, if the $b_i$-th edge of the $a_i$-th tetrahedron of $\V$ is the bottom diagonal of the $a_{i+1}$-st tetrahedron of $\V$.

	\subsubsection{$\gamma$-tetrahedron rectangles inside fundamental tetrahedron rectangles}\label{subsub:zeta-rect}
	 Let $(\loom, \loom^+, \loom^-)$ be the loom space such that $\loomif(\univV) = \loom$. We represent each fundamental tetrahedron rectangle in $(\loom, \loom^+, \loom^-)$ by a square $(0,1)^2$ whose closure has cusps at $\left(0, \frac{p}{3}\right)$, $\left(1, \frac{q}{3}\right)$, $\left(\frac{r}{3}, 0\right)$, $\left(\frac{s}{3},1\right)$, where $p,q,r,s \in \lbrace 1,2 \rbrace$ and $p\neq q$, $r\neq s$; see an example in Figure \ref{fig:fig8_tet_rects}.

	\begin{figure}[h]
		\includegraphics[scale=1]{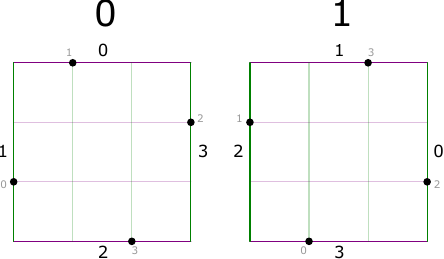}
		\caption{Square representatives of the fundamental tetrahedron rectangles in the loom space $(\loom, \loom^+, \loom^-)$ associated to the 2-tetrahedron veering triangulation of the figure-eight knot complement $M$. The black dots represent the cusps, and the small gray numbers next to them indicate the corresponding vertex numbers of the tetrahedron. 
			There are four $\pi_1(M)$-equivalence classes of face subrectangles, labelled from 0 to 3. The label of the equivalence class of a face subrectangle $Q \subset R$ appears next to the side of $R$ that is also a side of $Q$.  }
		\label{fig:fig8_tet_rects}
	\end{figure}

	Let $\zeta$ be a flow cycle of $\V$ that is non-peripheral and let $\gamma$ be the corresponding periodic orbit of a flow encoded by $\V$. Fix a fundamental tetrahedron rectangle $R$ in $(\loom, \loom^+, \loom^-)$. Recall that after Step~9 of the drilling algorithm (Figure \ref{fig:flowchart}), each point $x \in R \cap \liftGamma$ is approximated by a small rectangle $N(x) \subset R$ containing $x$. Furthermore, no cusp leaf of $\loom^\pm$ originating at a cusp of $R$ intersects $N(x)$ for any $x \in R \cap \liftGamma$, and no leaf of $\loom^\pm$ intersects  $N(x), N(x')$ for distinct  $x,x' \in R \cap \liftGamma$. 
	Therefore the horizontal and vertical orders on the set consisting of the cusps of $R$ and one representative point from each approximating rectangle are independent of the choice of representative points. 
	
	In our pictures, each rectangular approximation is represented by a single point whose position records only the position of the corresponding rectangle relative to the other approximating rectangles in $R$ and the cusps of $R$. For simplicity, we position these points as follows.
	The leaves of $\loom^+$ originating at the cusps of $R$ and intersecting~$R$ determine three vertical strips $V_i\subset R$, $i\in\lbrace0,1,2\rbrace$, while the leaves of $\loom^-$ originating at the cusps of $R$ and intersecting~$R$ determine three horizontal strips $H_j\subset R$, $j\in\lbrace0,1,2\rbrace$. Under the identification $(0,1)^2 \approx R$, these subrectangles correspond to
	\[
	V_i=\left(\frac{i}{3},\frac{i+1}{3}\right)\times(0,1)
	\text{ \ \ \ and  \ \ \ }
	H_j=(0,1)\times\left(\frac{j}{3},\frac{j+1}{3}\right).
	\]
	
	Set $v_ i = |V_i \cap \liftGamma|$ and $h_j = |H_j \cap \liftGamma|$. 
	For $k \in \lbrace 1, 2, \ldots, v_i\rbrace$, we represent the $k$-th point of $V_i \cap \liftGamma$, with respect to the horizontal order on $V_i$,  by a point with the $x$-coordinate equal to $\frac{i}{3} + \frac{k}{3(v_i+1)}$. Similarly, for $l \in \lbrace 1,2,\ldots, h_j\rbrace$, we represent the $l$-th point of $H_j \cap \liftGamma$, with respect to the vertical order on $H_j$, by a point whose $y$-coordinate is $\frac{j}{3} + \frac{l}{3(h_j+1)}$.
	Since for every point of $x \in R \cap \liftGamma$ there are  $i,j\in \{0,1,2\}$ such that $x \in H_i \cap V_j$,  this convention uniquely determines a collection of points in $(0,1)^2$ whose horizontal and vertical orders agree with those of the corresponding points of $R \cap \liftGamma$.
	
	In Figure~\ref{fig:fig8_drilled} we present our approximation of the points of $\liftGamma$ inside the fundamental tetrahedron rectangles associated to the veering triangulation of the figure-eight knot complement, with respect to the period-$5$ orbit $\gamma$ of its suspension flow discussed in Example~\ref{ex:figure:eight:geodesic}. 
	We also indicate all the resulting fundamental $\gamma$-tetrahedron rectangles. 

	\begin{figure}[h]
		\includegraphics[scale=1.3]{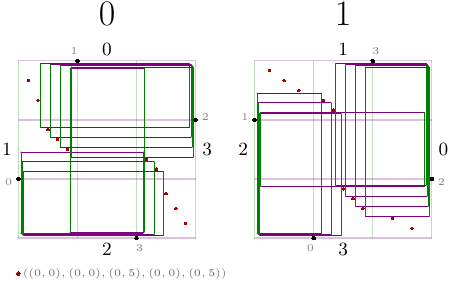}
		\caption{Our representation of the points of $R\cap\liftGamma$ where $\gamma$ is the orbit represented by the flow cycle $((0, 0), (0, 0), (0, 5), (0, 0), (0, 5))$. The resulting \mbox{$\gamma$-tetrahedron} rectangles are indicated with a slight inward offset from their actual boundaries to improve readability of the figure.}
		\label{fig:fig8_drilled}
	\end{figure}

	\subsubsection{Describing $\gamma$-simplex rectangles}
		For $i\in \lbrace 0, 1, \ldots, n-1\rbrace$ let $R_i$ be the fundamental tetrahedron rectangle in $(\loom, \loom^+, \loom^-)$ corresponding to the tetrahedron of~$\V$ with index $i$. Let $Q_\gamma\subset R_i$ be a $\gamma$-simplex rectangle. Enumerate the ideal points of $R_i\bez\liftGamma$, including its cusps, from left to right, starting from $0$. The rectangle $Q_\gamma$ meets $2\leq k\leq 4$ of these points. Let $(i_1,\ldots,i_k)$ be the tuple of their labels, arranged in increasing order. We call $i(i_1,\ldots,i_k)$ the \emph{LR-signature} of $Q_\gamma$.
		 
		 The tuple  $(i_1,\ldots,i_k)$ is similar to the west-east signature introduced in Section~\ref{subsec:equiv:classes}. For the purpose of describing examples we use the left-to-right ordering because, when  cardinal directions are determined by Convention~\ref{convention:cardinal:directions} and a veering triangulation is not edge-orientable, its fundamental tetrahedron rectangles cannot be drawn consistently with west on the left and south at the bottom. Thanks to the left-to-right convention we do not have to indicate in our pictures the cardinal directions  for each fundamental tetrahedron rectangle.
		 
		Since our pictures of fundamental tetrahedron rectangles already contain vertex, face, and tetrahedron numbers, we do not display the labels coming from the left-to-right ordering on the ideal points of $R_i \bez \liftGamma$. These labels can be easily recovered by enumerating all ideal points of $R_i\bez\liftGamma$ from left to right, starting from $0$.
	
	\subsubsection{Equivalence classes of $\gamma$-tetrahedron and pairs of solely face equivalent $\gamma$-(tetrahedron, face) pairs}

	Let $M$ be the manifold underlying $\V$. As in Section~\ref{subsec:finite:parent}, by fundamental $\gamma$-simplex rectangles we mean $\gamma$-simplex rectangles contained in fundamental tetrahedron rectangles. By Proposition~\ref{prop:finite:veering:parent}, the tetrahedra of the veering parent~$\V_\gamma$ are in bijective correspondence with the $\pi_1(M)$-equivalence classes of fundamental \mbox{$\gamma$-tetrahedron} rectangles. In our tables, we enumerate these equivalence classes starting from $0$ and we list the LR-signatures of all fundamental \mbox{$\gamma$-tetrahedron} rectangles belonging to each of them. For each equivalence class, we choose the fundamental $\gamma$-tetrahedron rectangle with the lexicographically smallest LR-signature as the preferred representative of the class.  We furthermore mark an LR-signature with a superscript $\rho$ if the $\pi_1(M)$-equivalence identifying the corresponding rectangle with the preferred representative reverses the left-to-right order. An unmarked LR-signature indicates that this order is preserved.

	Recall from Section \ref{subsec:finite:parent} that a $\gamma$-(tetrahedron, face) pair is a pair $(R_\gamma, Q_\gamma)$ such that~$R_\gamma$ is a $\gamma$-tetrahedron rectangle and $Q_\gamma$ is a $\gamma$-face subrectangle of $R_\gamma$. We call such a pair fundamental if $R_\gamma$ is fundamental. Furthermore, a pair $\{(R_\gamma,Q_\gamma),(R_\gamma',Q_\gamma')\}$ of  \mbox{$\gamma$-(tetrahedron, face)} pairs is solely face equivalent if $Q_\gamma, Q_\gamma'$ are $\pi_1(M)$-equivalent, but this equivalence does not extend to a $\pi_1(M)$-equivalence between $R_\gamma$ and $R_\gamma'$. By Proposition \ref{prop:finite:veering:parent}, every face pairing of $\V_\gamma$ can be represented by a pair of fundamental solely face equivalent $\gamma$-(tetrahedron, face) pairs. In the tables we record the information about these pairs as follows.

	Let $Q_\gamma$ be a $\gamma$-face subrectangle of a fundamental $\gamma$-tetrahedron rectangle $R_\gamma$. Suppose that $i(i_1,i_2,i_3)$  is the LR-signature of $Q_\gamma$. 
	Recall from Lemma \ref{lem:ploom}  \eqref{ploom:face:rectangles} that there is a  \mbox{$\gamma$-tetrahedron} rectangle $R_\gamma' \neq R_\gamma$ such that $Q_\gamma \subset R_\gamma'$. If $R_\gamma'$ is fundamental, then the pairs $(R_\gamma, Q_\gamma)$ and $(R_\gamma', Q_\gamma)$ are solely face equivalent. The LR-signatures of $R_\gamma, R_\gamma'$ must be of the form
	$i (i_1,i_2,i_3,i_4)^{\mathrm{sort}}$, $i(i_1,i_2,i_3,i_4')^{\mathrm{sort}}$, for some distinct nonnegative integers $i_4,i_4'$,
	where $(a_1,\ldots,a_k)^{\mathrm{sort}}$ denotes the tuple obtained by arranging the entries of $(a_1,\ldots,a_k)$ in increasing order.
	We record the face pairing induced by the pair $\lbrace (R_\gamma, Q_\gamma), (R_\gamma', Q_\gamma)\rbrace$ by $\lbrace i(i_1,i_2,i_3); i_4, i_4'\rbrace$. 
	
	If $R_\gamma'$ is not fundamental then, by Lemma~\ref{lem:everything:in:tet:rect}, $R_\gamma'$ is $\pi_1(M)$-equivalent to a fundamental $\gamma$-tetrahedron rectangle $R_\gamma''$ containing a $\gamma$-face subrectangle $Q_\gamma''$ such that the pairs $(R_\gamma, Q_\gamma)$ and $(R_\gamma'', Q_\gamma'')$ are solely face equivalent.  If $Q_\gamma''$ is contained in two fundamental $\gamma$-tetrahedron rectangles, we record the corresponding face pairing of $\V_\gamma$  using the convention described in the previous paragraph. Otherwise, we record the face pairing of~$\V_\gamma$ induced by the pair $\lbrace(R_\gamma, Q_\gamma), (R_\gamma'', Q_\gamma'')\rbrace $ by the pair $\lbrace i(i_1,i_2,i_3), j(j_1,j_2,j_3)\rbrace$, where $j (j_1, j_2, j_3)$ is the LR-signature of $Q_\gamma''$. 	If this equivalence reverses the left-to-right order, we mark the lexicographically larger LR-signature with a superscript~$\rho$. Note that in this case there is no need to specify the fundamental $\gamma$-tetrahedron rectangles containing $Q_\gamma$ and $Q_\gamma''$, because each of them is contained in only one such rectangle.


\subsubsection{Verifying the table entries}
		Now we explain how the reader can verify the data recorded in Tables~\ref{tab:drilling:fig8:sister}, \ref{tab:drilling:m004}, and \ref{tab:drilling:decreasing}. Each of these tables contains a picture of punctured fundamental tetrahedron rectangles, with the equivalence classes of face rectangles indicated in the same way as in Figure~\ref{fig:fig8_tet_rects}. These equivalence classes  are key for determining equivalence classes of fundamental $\gamma$-simplex rectangles.
		
		For a concrete illustration, we focus on the example presented in Table~\ref{tab:drilling:m004}. The \mbox{$\pi_1(M)$-equivalence} between the face rectangles labelled 0 induces the following bijection between the punctures they contain, labelled according to their left-to-right order in tetrahedron rectangles 0 and 1, respectively:
		\[
		\begin{pmatrix}
		0 & 1 & 2 & 3& 4 & 5 & 6 & 8 & 9 & 13\\
		4 & 5 & 6 & 7& 8 & 9 & 10& 11& 12 & 13
		\end{pmatrix}.
		\]
		
		This gives $\pi_1(M)$-equivalences between the following pairs of fundamental $\gamma$-tetrahedron rectangles: $0(2, 3, 6, 13)$ and $1(6, 7, 10, 13)$ (equivalence class 5), $0(3, 4, 6, 13)$ and \linebreak $1(7, 8, 10, 13)$ (equivalence class 6), $0(4, 5, 6, 13)$ and $1(8, 9, 10, 13)$ (equivalence class 7), and $0(5, 6, 8, 13)$ and $1(9, 10, 11, 13)$ (equivalence class 8). The remaining equivalence classes of fundamental $\gamma$-tetrahedron rectangles that contain more than one element are induced by face rectangles labelled 2.

		The $\gamma$-face rectangle $0(0,7,8)$ is contained in fundamental $\gamma$-tetrahedron rectangles $0(0,5,7,8)$ and $0(0, 7, 8, 9)$. This gives a pair of solely face equivalent $\gamma$-(tetrahedron, face) pairs recorded as $\lbrace 0(0,7,8);5,9\rbrace$. 

		The $\gamma$-face rectangle $0(0,5,8)$ is contained in only one fundamental $\gamma$-tetrahedron rectangle $0(0,5,7,8)$. The other $\gamma$-tetrahedron rectangle containing $0(0,5,8)$ spans west of the fundamental tetrahedron rectangle 0. Thus we map $0(0, 5, 8)$ through the \mbox{$\pi_1(M)$-equivalence} arising from the face rectangle labelled 0. This gives $\gamma$-face rectangle $1(4, 9, 11)$. Since it is not contained in any fundamental \mbox{$\gamma$-tetrahedron} rectangle, we need to trace further $\pi_1(M)$-equivalences.
		The $\pi_1(M)$-equivalence between the two face rectangles labelled 3 identifies $1(4, 9, 11)$ with $0(7,10,11)$. Again, the resulting $\gamma$-face rectangle is not contained in any fundamental \mbox{$\gamma$-tetrahedron} rectangle.
		Applying the $\pi_1(M)$-equivalence between the two face rectangles labelled 2 identifies $0(7,10,11)$ with $1(4,7,8)$. 
		The $\pi_1(M)$-equivalence between the two face rectangles labelled 3 identifies $1(4,7,8)$ with $0(7,8,9)$. This  $\gamma$-face rectangle is contained in the $\gamma$-tetrahedron rectangle $0(0,7,8,9)$. Moreover, while $0(0,5,8)$ is shorter than $0(0, 5, 7, 8)$, $0(7,8,9)$ is thinner than $0(0, 7, 8, 9)$. Thus we get a pair of solely face equivalent $\gamma$-(tetrahedron, face) pairs recorded as $\lbrace 0(0,5,8), 0(7,8,9) \rbrace$.
		
		One can use the information about the $\pi_1(M)$-equivalences classes of fundamental \mbox{$\gamma$-tetrahedron} rectangles and pairs of solely face equivalent $\gamma$-(tetrahedron, face) pairs to construct the veering parent $\V_\gamma$ as follows. 
		Suppose that the $i$-th equivalence class is represented by $j(j_1,j_2,j_3,j_4)$. Associate to it a tetrahedron  whose vertices are numbered from 0 to 3 so that vertex $k$ corresponds to $j_{k+1}$. Since $0(0,5,7,8)$ and $0(0,7,8,9)$ represent equivalence classes 1 and 2, respectively, the pair $\lbrace 0(0,7,8);5,9\rbrace$ discussed above gives the face pairing $1(0,2,3)\leftrightarrow 2(0,1,2)$, while the pair $\lbrace 0(0,5,8),0(7,8,9)\rbrace$ gives the face pairing $1(0,1,3)\leftrightarrow 2(1,2,3)$. Proceeding in this way with all the listed pairs of solely face equivalent $\gamma$-(tetrahedron, face) pairs recorded in the table yields a complete set of face pairings on the tetrahedra corresponding to the equivalence classes of $\gamma$-tetrahedron rectangles. This information can be entered into Regina~\cite{regina} to obtain the desired triangulation $\V_\gamma$ and verify the listed isomorphism signature. The full taut signature can be verified using a Python package Veering \cite{VeeringGitHub}.

	\newpage
	\subsection{Drilling a short cycle from the figure-eight knot sister}\hfill 
	
	\begin{table}[H]
		\begin{center}
			\renewcommand{\arraystretch}{1.3}
			\begin{tabular}{|c|}
				\hline
				\textbf{Initial taut signature} \\ \hline 
				\texttt{cPcbbbdxm\_10} \\[0.5em] \hline
				
				\textbf{Name of the manifold} \\ \hline
				figure-eight knot sister exterior, \texttt{m003, otet02\_00000} \\[0.5em] \hline
				
				\textbf{Flow cycle} \\ \hline
				$((0, 2))$ \\ \hline
				
				\textbf{Fundamental $\gamma$-tetrahedron rectangles} \\ \hline
				{\includegraphics[scale=1]{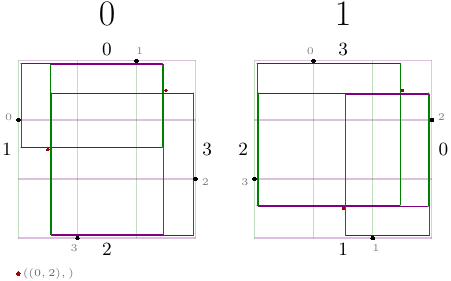}
					\put(0,140){}}  
				\\ \hline
				
				\textbf{Equivalence classes of fundamental $\gamma$-tetrahedron rectangles} \\ \hline
				
				
				\begin{tabular}{cl} 
					0:&  \{0(1, 2, 3, 4)\}\\
					1:& \{0(1, 2, 4, 5)\}\\ 
					2:& \{1(0, 1, 2, 4)\}\\ 
					3:& \{1(0, 2, 4, 5)\}\\ 
					4:& \{0(0, 1, 3, 4), 1(2, 3, 4, 5)$^\rho$\}\\
				\end{tabular}
				\\ \hline
				\textbf{Pairs of solely face equivalent $\gamma$-(tetrahedron, face) pairs} \\ \hline
				

			\begin{tabular}{l c l}
				\{0(1, 2, 4); 3, 5\} &&  \{0(1, 2, 5),	1(0, 1, 2)$^\rho$\} \\
				\{0(2, 3, 4), 1(0, 1, 4)\} && \{1(0, 2, 4); 1, 5\}\\
				\{0(1, 4, 5), 1(1, 2, 4)$^\rho$\} && \{0(1, 3, 4); 0, 2\}\\
				\{0(0, 3, 4), 0(1, 2, 3)$^\rho$\} && \{0(2, 4, 5), 1(0, 4, 5)\}\\
				\{0(0, 1, 4), 1(2, 4, 5)\}$^\rho$ && \{0(0, 1, 3), 1(0, 2, 5)\}\\
				\end{tabular}
\\\hline
%
				\textbf{Taut signature of the veering parent} \\ \hline
				\texttt{fLLQcbeddeehhbghh\_01110}\\ \hline
				
				\textbf{Names of the drilled manifold} \\ \hline
				Whitehead link sister/P(-2,3,8) exterior, \texttt{ m125, L13n5885, ooct01\_00000}\\ \hline
				
			\end{tabular}
		\end{center}
		
		\vspace{0.2cm}
		\caption{Drilling  \texttt{cPcbbbdxm\_10}   along $((0, 2))$.}
		\label{tab:drilling:fig8:sister}
	\end{table}
	\newpage 
		\subsection{Drilling an orbit that is not ambiently isotopic to the homotopic geodesic} \hfill
	\label{appendix:geodesic}
	\begin{table}[H]
		\begin{center}
			\renewcommand{\arraystretch}{1.2}
			\begin{tabular}{|c|}
				\hline
				\textbf{Initial taut signature} \\ \hline 
				\texttt{cPcbbbiht\_12}  \\[0.5em] \hline
				
				\textbf{Name of the manifold} \\ \hline
				figure-eight knot exterior, \texttt{m004,4\_1, K2\_1, K4a1, otet02\_00001} \\[0.5em] \hline
				
				\textbf{Flow cycle} \\ \hline
				$((0, 0), (0, 0), (0, 5), (0, 0), (0, 5))$ \\ \hline
				
				\textbf{Fundamental $\gamma$-tetrahedron rectangles} \\ \hline
				{\includegraphics[scale=1]{cPcbbbiht_12_00-00-05-00-05_fixed_orientations}
					\put(0,145){}}  
				\\ \hline
				
				\textbf{Equivalence classes of fundamental $\gamma$-tetrahedron rectangles} \\ \hline

					\begin{tabular}{c l} 
					0:&  \{0(5, 6, 7, 8)\}\\
					1:& \{0(0, 5, 7, 8), 1(0, 3, 4, 5)\}\\ 
					2:& \{0(0, 7, 8, 9), 1(0, 4, 5, 6)\}\\ 
					3:& \{0(0, 7, 9, 10), 1(0, 4, 6, 7)\}\\ 
					4:& \{1(0, 6, 7, 13)\}\\ 
					5:& \{0(2, 3, 6, 13), 1(6, 7, 10, 13)\}\\ 
					6:& \{0(3, 4, 6, 13), 1(7, 8, 10, 13)\}\\ 
					7:& \{0(4, 5, 6, 13), 1(8, 9, 10, 13)\}\\ 
					8:& \{0(5, 6, 8, 13), 1(9, 10, 11, 13)\}\\ 
									\end{tabular}

				\\ \hline
				\textbf{Pairs of solely face equivalent $\gamma$-(tetrahedron, face) pairs} \\ \hline

				\begin{tabular}{l c l c l}
				\{0(0, 7, 8); 5, 9\} && \{0(5, 7, 8); 0, 6\} && \{0(0, 7, 9); 8, 10\}\\
				\{0(0, 7, 10), 0(6, 7, 8)\} && \{0(0, 9, 10), 1(0, 6, 7)\} && \{0(0, 5, 8), 0(7, 8, 9)\}\\
				\{0(5, 6, 8); 7, 13\} && \{0(0, 8, 9), 0(7, 9, 10) \} && \{0(0, 5, 7), 1(0, 6, 13)\}\\
				\{0(2, 3, 13), 1(6, 7, 13)\} && \{0(2, 3, 6), 0(4, 5, 13)\} && \{0(2, 6, 13), 0(5, 6, 7)\}\\
				\{0(3, 6, 13); 2, 4\} && \{0(4, 6, 13); 3, 5\} && \{0(5, 6, 13); 4, 8\}\\
				\{0(3, 4, 6), 0(5, 8, 13)\} && \{0(6, 8, 13), 1(0, 7, 13)\} && \{0(3, 4, 13), 0(4, 5, 6)\}
				\end{tabular}
				
%
%
%
%
				
				%
				%
				%
				%

				\\\hline 
				
				\textbf{Taut signature of the veering parent} \\ \hline
				\texttt{jLLwQLQbeefgehiiixxxaaxxxcv\_102211100} \\ \hline
				
				\textbf{Name of the drilled manifold} \\ \hline
				\texttt{o9\_40888} \\ \hline
				\textbf{Comments} \\ \hline
				The encoded orbit is not ambiently isotopic to the geodesic in its homotopy class\\ \hline
			\end{tabular}
		\end{center}
		
		\vspace{0.2cm}
		\caption{Drilling  \texttt{cPcbbbiht\_12}  along $((0, 0), (0, 0), (0, 5), (0, 0), (0, 5))$.}
		\label{tab:drilling:m004}
	\end{table}
	
	\newpage
	\subsection{Decreasing the number of tetrahedra}\hfill \label{app:ex:decreasing} 
	
	\begin{table}[H]
		\begin{center}
			\renewcommand{\arraystretch}{1.3}
			\begin{tabular}{|c|}
				\hline
				\textbf{Initial taut signature} \\ \hline 
				\texttt{kLAMzwAkbcbdefgiijjhhwhhhhjsaw\_2112121221}  \\[0.5em] \hline
				
				\textbf{Name of the manifold} \\ \hline
				unknown \\[0.5em] \hline
				
				\textbf{Flow cycle} \\ \hline
				$((0, 3))$ \\ \hline
				
				\textbf{Fundamental $\gamma$-tetrahedron rectangles} \\ \hline
				{\includegraphics[scale=0.75]{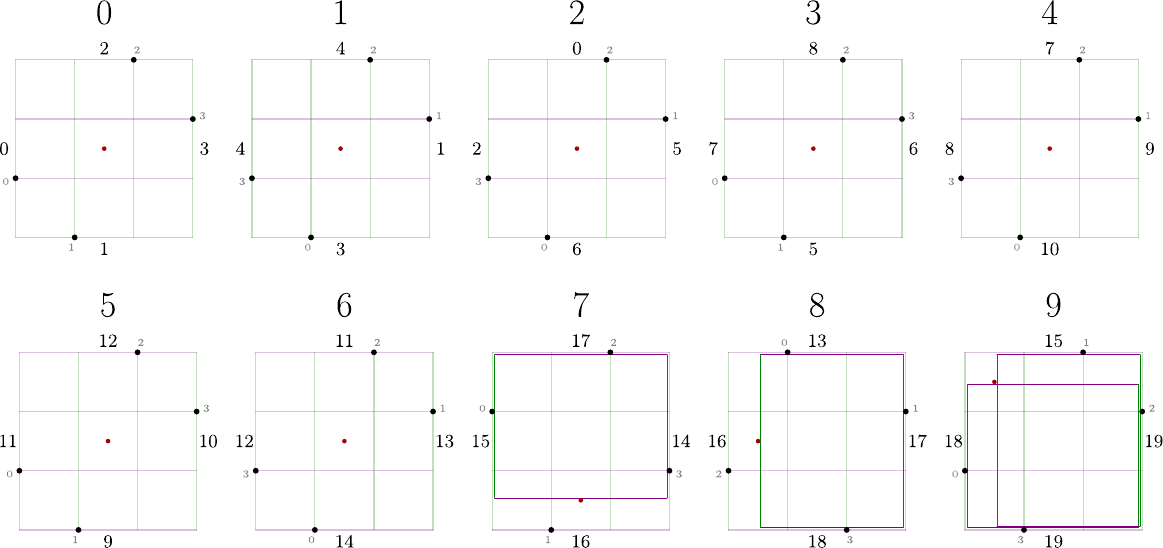}
					\put(0,200){}}  
				\\ \hline
				
				\textbf{Equivalence classes of fundamental $\gamma$-tetrahedron rectangles} \\ \hline
				




				\begin{tabular}{cl} 
				0: & \{7(0, 2, 3, 4)\}\\
				1: & \{8(1, 2, 3, 4)\} \\
				2: & \{9(0, 1, 2, 4)\} \\
				3: & \{9(1, 2, 3, 4)\} 
				\end{tabular}
				\\ \hline
				\textbf{Pairs of solely face equivalent $\gamma$-(tetrahedron, face) pairs} \\ \hline
				\begin{tabular}{l c l}
					\{7(2, 3, 4), 8(1, 2, 4)\} && \{7(0, 3, 4), 8(2, 3, 4)$^\rho$\}\\
					\{7(0, 2, 4), 9(0, 1, 2)$^\rho$\} && \{8(1, 3, 4), 9(1, 2, 3)\} \\
					\{7(0, 2, 3), 9(0, 1, 4)$^\rho$\} && \{9(1, 2, 4); 0, 3\}\\
					\{8(1, 2, 3), 9(1, 3, 4)\} && \{9(0, 2, 4), 9(2, 3, 4)$^\rho$\}
					\end{tabular}

				\\\hline 
				\textbf{Taut signature of the veering parent} \\ \hline
				\texttt{eLMkbcddddedde\_2100}\\ \hline
				
				\textbf{Name of the drilled manifold} \\ \hline
				\texttt{m203}, $6^2_2$ (or \texttt{L6a2}) exterior,  \texttt{otet04\_00001}\\ \hline
				
				\textbf{Comments} \\ \hline
				Decreasing the number of tetrahedra under drilling\\
				Layered parent of a non-layered (non-measurable) triangulation \\ \hline
			\end{tabular}
		\end{center}
		
		\vspace{0.2cm}
		\captionsetup{margin={-1mm,0mm}}
		\caption{ Drilling  \texttt{kLAMzwAkbcbdefgiijjhhwhhhhjsaw\_2112121221}   along $((0, 3))$.}
		\label{tab:drilling:decreasing}
	\end{table}


\end{document}